\documentclass[a4paper,11pt,leqno]{article}
\usepackage{enumitem}
\usepackage[utf8]{inputenc}
\usepackage[T1]{fontenc}
\usepackage{amsthm}
\usepackage{amssymb}
\usepackage{mathrsfs}
\usepackage{amsmath}
\usepackage{amsfonts}
\usepackage[autostyle]{csquotes}
\MakeOuterQuote{"}
\usepackage[english]{babel}
\usepackage{mathtools}
\usepackage{graphicx}
\usepackage{float}
\usepackage{dsfont}
\usepackage[a4paper]{geometry}
\usepackage{url}
\usepackage{appendix}
\usepackage{constants}
\usepackage{stackrel}

\usepackage{hyperref}
\hypersetup{linktocpage,
  colorlinks   = true, 
  urlcolor     = blue, 
  linkcolor    = blue,
  citecolor   = blue 
}

\newtheorem{The}{Theorem}[section]
\newtheorem{Lemme}[The]{Lemma}
\newtheorem{Prop}[The]{Proposition}
\newtheorem{Cor}[The]{Corollary}
\theoremstyle{definition}

\theoremstyle{remark}
\newtheorem{Rk}[The]{Remark}

\title{\normalsize 
\textbf{FIRST PASSAGE PERCOLATION, LOCAL UNIQUENESS FOR INTERLACEMENTS AND CAPACITY OF RANDOM WALK}}\author{Alexis Prévost\thanks{University of Geneva,
			Section of Mathematics,
			24, rue du Général Dufour,
			1211 Geneva, Switzerland.
			\\Email: \url{alexis.prevost@unige.ch}}
	}
\date{}
\makeatletter
\newcommand{\E}{\mathbb{E}}

\newcommand{\R}{\mathbb{R}}
\newcommand{\Z}{\mathbb{Z}}
\newcommand{\N}{\mathbb{N}}

\newcommand{\G}{\mathcal{G}}

\newcommand{\A}{\mathcal{A}}

\renewcommand{\P}{\mathbb{P}}
\newcommand{\eps}{\varepsilon}

\newcommand{\I}{{\cal I}}
\newcommand{\V}{{\cal V}}
\renewcommand{\L}{\mathcal{L}}

\newcommand{\calC}{\mathcal{C}}

\renewcommand{\phi}{\varphi}
\renewcommand{\tilde}{\widetilde}
\renewcommand{\hat}{\widehat}
\renewcommand{\epsilon}{\varepsilon}

\@addtoreset{equation}{section}

\newconstantfamily{c}{symbol=c}

\newcommand{\B}{\mathcal{B}}
\makeatother

\usepackage{color}
\definecolor{Red}{rgb}{1,0,0} %<<<2 Colors
\definecolor{Blue}{rgb}{0,0,1}
\definecolor{Olive}{rgb}{0.41,0.55,0.13}
\definecolor{Yarok}{rgb}{0,0.5,0}
\definecolor{Green}{rgb}{0,1,0}
\definecolor{MGreen}{rgb}{0,0.8,0}
\definecolor{DGreen}{rgb}{0,0.55,0}
\definecolor{Yellow}{rgb}{1,1,0}
\definecolor{Cyan}{rgb}{0,1,1}
\definecolor{Magenta}{rgb}{1,0,1}
\definecolor{Orange}{rgb}{1,.5,0}
\definecolor{Violet}{rgb}{.5,0,.5}
\definecolor{Purple}{rgb}{.75,0,.25}
\definecolor{Brown}{rgb}{.75,.5,.25}
\definecolor{Grey}{rgb}{.7,.7,.7}
\definecolor{Black}{rgb}{0,0,0}

\begin{document}
\maketitle

\vspace{-0.8cm}
\begin{abstract}
\centering
\begin{minipage}{0.75\textwidth}
\vspace{0.3cm}
The study of first passage percolation (FPP) for the random interlacements model has been initiated in \cite{AndPre}, where it is shown that on $\Z^d$, $d\geq3$, the FPP distance is comparable to the graph distance with high probability. In this article, we give an asymptotically sharp lower bound on this last probability, which additionally holds on a large class of transient graphs with polynomial volume growth and polynomial decay of the Green function. When considering the interlacement set in the low-intensity regime, the previous bound is in fact valid throughout the near-critical phase.  In low dimension, we also present two applications of this FPP result: sharp large deviation bounds on local uniqueness of random interlacements, and on the capacity of a random walk in a ball. 
\end{minipage}
\end{abstract}

\vspace{0.2cm}

\section{Introduction}
\label{sec:intro}
First passage percolation (FPP) has been a central topic in probability theory since its introduction in the 1960s by Hammersley and Welsh, and we refer to \cite{ADH17} for a recent survey concerning independent percolation. In the context of models with long-range correlations on $\Z^d$, $d\geq3$, it was shown in \cite{AndPre} that the associated time constant is positive under appropriate conditions.  A prominent example of such a long-range correlated model is random interlacements, which was introduced in \cite{MR2680403} to investigate simple random walk related problems, see for instance \cite{MR2520124,MR2561432,MR2838338,MR3563197}, and to which the method from \cite{AndPre} can be applied, see Theorem~2.2 and Proposition~4.9 therein. The law of the interlacement set $\I^u\subset \Z^d$, $u>0$, is characterized by the identity
\begin{equation}
\label{eq:defIuintro}
    \P(\I^u\cap K=\varnothing)=\exp(-u\mathrm{cap}(K))\text{ for all finite sets }K\text{ and }u>0,
\end{equation}
where $\mathrm{cap}(K)$ denotes the capacity of the set $K$, see \eqref{eq:defequicap}. The existence of a random set $\I^u$ that satisfies \eqref{eq:defIuintro} was established in \cite{MR2680403}, and we refer to Section~\ref{sec:preliminaries} for a more detailed description of its construction. We will only need in this introduction that $\I^u$ can be realized as the union of the traces on $\Z^d$ of infinitely many random walks, and in particular always contains a.s.\ an infinite connected component, which is in fact a.s.\ unique by \cite[Theorem~3.3]{MR3076674}. We further denote by $\V^u\stackrel{\textnormal{def.}}{=}(\I^u)^c$ the vacant set of interlacements, and write $B(N)$ for the ball centered at $0$ and with radius $N$ for the Euclidean distance on $\Z^d$.

Random interlacements is an intriguing percolation model, in particular when considering its vacant set. On $\Z^d$, $d\geq3$, precise asymptotic bounds have been obtained in the study of the associated disconnection event in \cite{li2014lower,MR3602841,nitzschner2020solidification,MR4607724}, and of the associated one-arm event and two-point function in dimension three in \cite{GosRod}. Corresponding results for the Gaussian free field are available in \cite{MR3417515,GRS21}, and for general Gaussian fields on $\R^d$ in \cite{MuiSev}.  These bounds are different and stronger than the corresponding ones for independent percolation, and suggest that random interlacements exhibit a different (near)-critical behavior in low dimension.

The natural FPP distance associated to $\V^u$ counts the minimal number of times a path starting in $0$ visits $\I^u$ before leaving $B(N)$, see \eqref{eq:defduN}. It was established in \cite{AndPre} that this distance is linear in $N$ with high probability throughout the subcritical regime of $\V^u$. In this article, we obtain precise asymptotic bounds for the probability that this FPP distance in dimension three is sublinear, which differs from the case of independent weights. Furthermore, in the low intensity regime $u$ close to $0$, we prove similar asymptotic results for the FPP distance associated to the complement of the "sausage" of width $R$ around $\I^u$. Here, $R=Cu^{-\frac1{d-2}}$ for large $C$, ensuring that the complement of this sausage is subcritical (similarly as $\V^u$, the complement of $\I^u$, was chosen subcritical before). Notably, we significantly enhance this last result  by providing effective bounds throughout the near-critical regime, instead of only asymptotically as $N\rightarrow\infty$, and in any dimension $d\geq3$, instead of only $d=3$. In particular, it exhibits $u^{-\frac1{d-2}}$ as the typical length scale for the FPP of random interlacements in the low-intensity regime.

All our findings actually never use the symmetries of the lattice $\Z^d$,  and in fact hold on a larger class of transient graphs, which are not necessarily transitive. More precisely, we consider $\G=(G,\lambda)$ a transient weighted graph, where $G$ is a countable infinite set, and $\lambda=(\lambda_{x,y})_{x,y\in{G}}\in{[0,\infty)}^{G\times G}$ is a sequence of weights. We further mainly assume that $\G$ has uniformly bounded weights, polynomial volume growth and polynomial decay of the Green function, and refer to Section~\ref{sec:preliminaries} for a more precise description of these assumptions, see also \eqref{eq:standingassumptionintro}. In the context of percolation for random interlacements, a very similar class of graphs was first introduced in \cite{DrePreRod2}, see also \cite{MR2891880} for a close but weaker setup, and then further studied in \cite{DrePreRod5}. This class of graphs encompasses many Cayley and fractal graphs with polynomial volume growth, and we refer to \cite[(1.4)]{DrePreRod2} for concrete examples. For clarity in presenting our results in this section, we will only focus on the case where $G$ is the square lattice $\Z^d$ in dimension $d\geq3$, always implicitly endowed with unit weights. We refer the reader to the end of the introduction for guidance on locating the corresponding results for more general graphs.

Before stating precisely the aforementioned FPP results, we first focus on interesting applications. The first application concerns the near-critical two-point function associated to the Gaussian free field on the cable system on $\Z^d$, $d\in{\{3,4\}}$, and is explained in \cite[Theorem~1.2]{DrePreRod8}. We derive two further interesting applications in this article, concerning respectively local uniqueness for random interlacements, and the capacity of random walk.

\subsection{Local uniqueness for random interlacements}

In our main result, we study at which scale the interlacement set $\I^u$, which is infinite and connected, starts to have the typical properties of a supercritical percolation cluster. Since the capacity of $B(u^{-\frac1{d-2}})$ is of order $u^{-1}$, see \eqref{eq:assumptionZd} and \eqref{eq:capball}, it is clear from \eqref{eq:defIuintro} that one typically starts to observe $\I^u$ at scale $u^{-\frac1{d-2}}$. However, it is less clear whether  $\I^u$ is already really supercritical, in the sense that there is at this scale a unique very large connected component. To assess this for $u>0$ and $N\in{\N}$, we introduce the local uniqueness event \text{LocUniq}$_{u,N}$ for the interlacement set $\I^u$ as the event that every two vertices in $\I^u\cap B(N)$  are connected by a path in $\I^u\cap B(2 N)$; see also \eqref{eq:deflocaluniqnointro} with $\lambda=2$ and $x=0$ therein. Controlling the probability of this event is essential when using intricate renormalization schemes to create highways along which "good" events for a connected component of interlacements occur, see for instance \cite{MR3024098,DrePreRod,DrePreRod2,DrePreRod5}. 

An upper bound on the probability of the event $\text{LocUniq}_{u,N}^c$, where $\text{LocUniq}_{u,N}^c$ is the complement of the event $\text{LocUniq}_{u,N}$, was first obtained in \cite[Proposition~1]{MR2819660} for a fixed $u$ and in any dimension, and the dependency on $u$ was first made explicit in \cite[Lemma~3.2]{DrePreRod}. When $d\in{\{3,4\}}$, this bound, as well as the dependency on $u$, was further improved in \cite[Theorem~5.1]{DrePreRod5}, where it is shown that for $u$ small enough
\begin{equation}
\label{eq:oldboundlocaluniq}
    \P\big(\text{LocUniq}_{u,N}^c\big)\leq \begin{dcases}
        C\exp\Big(-\big({cuN}\big)^{1/3}\Big)&\text{ if }G=\Z^3,\\
         C\exp\Big(-\Big(\frac{c\sqrt{u}N}{\log(N)}\Big)^{2/5}\Big)\!\!\!&\text{ if }G=\Z^4.
         \end{dcases}
\end{equation}
The interest of this last bound is that it converges to $0$ when $uN^{d-2}\rightarrow\infty$, up to logarithmic corrections in dimension four, which seems to be the optimal scaling due to its link with the correlation length of the Gaussian free field on the cable system as explained above \cite[Corollary~5.2]{DrePreRod5}. However, the exact function that appears in the exponential of \eqref{eq:oldboundlocaluniq} seems rather arbitrary. One can in fact improve this function as an consequence of our FPP results.

\begin{The}
\label{the:localuniquenessintro}
Fix some $\eta\in{(0,1)}$. There exist constants $0<c, C<\infty$, depending on $\eta$, such that for all $N\in{\N}$ and $u\in{(0,1/2]}$, if $G=\Z^3$ and $uN\geq C$
\begin{equation}
\label{eq:boundlocaluniqintro}
     \exp\Big(-\frac{(1+\eta)\pi Nu}{3\log(Nu)}\Big)\leq \P\big(\textnormal{LocUniq}_{u,N}^{c}\big)\leq
        \exp\Big(-\frac{(1-\eta)\pi Nu}{3\log(Nu)}\Big);
\end{equation}
whereas if $G=\Z^4$ and $\sqrt{u}N\geq C\log(1/u)$
\begin{equation}
\label{eq:boundlocaluniqintro4}
     \exp\big(-{C\sqrt{u}N}\big)\leq \P\big(\textnormal{LocUniq}_{u,N}^{c}\big)\leq
         \exp\Big(-\frac{c\sqrt{u}N}{\log(1/u)}\Big).
\end{equation}
\end{The}

Theorem~\ref{the:localuniquenessintro} precisely characterizes the speed at which $\I^u$ becomes a genuine supercritical connected component. This characterization is optimal when $d=3$ and includes a logarithmic adjustment when $d=4$. Note that $\eta$ plays no role in the bounds \eqref{eq:boundlocaluniqintro4} when $d=4$, and the constants therein therefore do not depend on the choice of $\eta$. One can also lower bound the probability of $\text{LocUniq}_{u,N}^c$ by $\exp(-cu^{\frac1{d-2}}N)$ for all $d\geq4$, similarly as in  \eqref{eq:boundlocaluniqintro4}, but this bound is not expected to be sharp anymore when $d\geq5$, even after logarithmic corrections. It would be interesting to derive sharp bounds in this case, and we refer to Remark~\ref{rk:endlocuniq},\ref{rk:sharpalpha<2nu}) for details. In a somewhat related fashion, it is shown in \cite{hernandez2023chemical}  that the time constant associated to the chemical distance of $\I^u$ is of order  $u^{-\frac12}$ for all $d\geq5$. This typical length is the one corresponding to mean-field behavior, whereas the typical length $u^{-1}$ appearing in \eqref{eq:boundlocaluniqintro} for $d=3$ is not mean-field. Some logarithmic corrections to mean-field behavior are usually expected in the critical dimension $d=4$, which explains the additional logarithmic factors appearing in \eqref{eq:boundlocaluniqintro4}.

Many bounds similar to the ones appearing in Theorem~\ref{the:localuniquenessintro}, but for different events or fields, have been recently obtained in \cite{MR3602841,GRS21,DrePreRod5,MR4499015,MuiSev,GosRod}. Theorem~\ref{the:localuniquenessintro} is actually different and more refined than  all these previous bounds in at least three ways. Most of the previous bounds were only valid for a fixed value of the parameter as $N\rightarrow\infty$, whereas Theorem~\ref{the:localuniquenessintro} holds throughout the "near-critical" regime $N\geq cu^{-\frac1{d-2}}$, with $\log$ corrections in dimension four. Therefore, Theorem~\ref{the:localuniquenessintro} can be interpreted as a result on the critical window of the random interlacement set which exhibits $u^{-\frac1{d-2}}$ as its typical length scale  for $d\in{\{3,4\}}$. Moreover, all these previous bounds were either the same in all dimension, which is not the case for Theorem~\ref{the:localuniquenessintro} as previously explained, or only known in dimension three, whereas Theorem~\ref{the:localuniquenessintro} is true for $d\in{\{3,4\}}$. This suggests in particular that a similar extension of the results from \cite{GRS21,GosRod} in dimension four should be true. Finally, when the dependency in $u$ was explicit, all the previous bounds concerning interlacements were of the form $\exp(-u g(N))$ for some function $g$, which naturally appears from \eqref{eq:defIuintro}, or of the form $\exp(-u^2g(N))$ when considering a Gaussian field, which naturally appears from Gaussian large deviation bounds or the isomorphism \cite{MR2892408} with random interlacements. Here, $u$ denotes the difference between the parameter one considers and the critical parameter for percolation. On the contrary, the bounds in Theorem~\ref{the:localuniquenessintro} are the first bounds of the type $\exp(-f(uN^{d-2}))$ for some non-linear function $f$.

Let us also mention that Theorem~\ref{the:localuniquenessintro} can be used to slightly improve the results from \cite[Theorem~1.7 and Proposition~6.1]{DrePreRod5} on critical exponents for the Gaussian free field on the cable system in dimension $d\in\{3,4\}$, since these results were previously proved using \eqref{eq:oldboundlocaluniq}, and we refer to Remark~\ref{rk:endlocuniq},\ref{rk:betterctegffcablesystem}) for details.

\subsection{Capacity of random walk}

Our second application of the FPP results for interlacements concerns the probability that a random walk on $\Z^d$ stays confined inside a region of "size" comparable to $\mathcal{R}_{N,p}$ before exiting $B(N)$, where 
\begin{equation}
\label{eq:RNp}
    \mathcal{R}_{N,p}\stackrel{\textnormal{def.}}{=}([0,N]\times[-p,p]^{d-1})\cap\Z^d
\end{equation}
 denotes the "tube" of length $N$ and width $2p$. We denote by $(X_k)_{k\in{\N_0}}$ the random walk on $\Z^d$ with unit weights, started in $x$ under $\P_x$, see \eqref{eq:pxy} for a more precise definition, and let $T_A\stackrel{\textnormal{def.}}{=}\inf\{k\in{\N}:X_k\notin{A}\}$ be the first exit time of $A\subset G$. Let us also abbreviate $\overline{X}_A\stackrel{\textnormal{def.}}{=}\{X_k\}_{0\leq k< T_{A}}$ for the set of points hit by the random walk before exiting $A$. A natural way to measure the size of a random walk is in terms of its capacity, and we are thus interested in the probability that the capacity of $\overline{X}_{B(N)}$ is smaller than the capacity of $\mathcal{R}_{N,p}$. Upper bounds on this probability were previously obtained in any dimension in \cite[Lemma~6]{MR2819660}, using ideas from \cite[Section~4]{MR2889752}, and then later improved in dimension three and four in \cite[Lemma~5.3]{DrePreRod5}. Noting that the capacity of $\mathcal{R}_{N,p}$ is smaller than $CN/\log(N/p)$ in dimension three, and than $CNp$ in dimension four, see Remark~\ref{rk:capRnp}, these bounds imply that
\begin{equation}
\label{eq:oldboundcapacity}
    \P_0\Big(\mathrm{cap}(\overline{X}_{B(N)})\leq \mathrm{cap}\big(\mathcal{R}_{N,p}\big)\Big)\leq\begin{dcases}
        C\Big(\frac{N}{p}\Big)^{-c}&\text{ if }G=\Z^3,\\
         C\exp\Big(-\big(\frac{cN}{p\log(N)}\big)^{1/2}\Big)\!\!\!&\text{ if }G=\Z^4.
         \end{dcases}
\end{equation}
A lower bound on the left-hand side of \eqref{eq:oldboundcapacity} is the probability that the random walk remains confined in $\mathcal{R}_{N,p}$ before exiting $B(N)$, which can easily be lower bounded by $\exp(-cN/p)$, see Lemma~\ref{lem:exitviatubes} and below \eqref{eq:illustrationcapacitynointro}. The discrepancy between this lower bound and \eqref{eq:oldboundcapacity} warrants further investigation. The proof of the inequality \eqref{eq:oldboundcapacity} essentially uses a bound on the typical value of $\mathrm{cap}(\{X_{1},\dots,X_k\})$, $k\in{\N}$, combined with the strong Markov property and the tail of $T_{B(N)}$. We present here a new method to bound the probability in \eqref{eq:oldboundcapacity} which relies on FPP bounds for interlacements, and yields the following result.

\begin{The}
\label{the:boundcapacityintro}
Fix some $\eta\in{(0,1)}$. There exist constants $0<c\leq C<\infty$, depending on $\eta$, such that for all $N,p\in{\N}$ with $p\leq N$,
\begin{equation}
\label{eq:illustrationcapacityintro}
     c\exp\Big(-\frac{CN}{p}\Big)\leq \P_0\Big(\mathrm{cap}(\overline{X}_{B(N)})\leq \mathrm{cap}\big(\mathcal{R}_{N,p}\big)\Big)\leq\begin{dcases}
        C\exp\Big(-\Big(\frac{cN}{p}\Big)^{1-\eta}\Big)&\text{ if }G=\Z^3,\\
         C\exp\Big(-\frac{cN}{p\log(p)^2}\Big)\!\!\!&\text{ if }G=\Z^4.
     \end{dcases} 
\end{equation}
\end{The}
The terms in the exponential in the upper and lower bound of \eqref{eq:illustrationcapacityintro} are roughly of the same order, which, in view of the discussion below \eqref{eq:oldboundcapacity}, suggests that remaining confined in $\mathcal{R}_{N,p}$ might be a typical kind of event that occurs when the capacity of $\overline{X}_{B(N)}$ is smaller than the capacity of $\mathcal{R}_{N,p}$. The reason why there is a discrepancy between the upper and lower  bound in \eqref{eq:illustrationcapacityintro} in dimension three is because three is the critical dimension for the capacity of $\mathcal{R}_{N,p}$, see for instance Lemma~\ref{lem:capatube} below, which in particular requires very strong bounds to achieve exponential decay, see Theorem~\ref{the:boundcapacitynointro}. This was also the reason why the bound \eqref{eq:oldboundcapacity} from \cite{DrePreRod5} was particularly weak in dimension three. In dimension four, this discrepancy comes from the fact that this is the critical dimension for the capacity of $\overline{X}_{B(N)}$, see  for instance \cite[Lemma~2.4]{DrePreRod2}. In fact, in intermediate dimension strictly between three and four and for a diffusive walk, one can obtain matching upper and lower bounds (up to constants) for the probability in \eqref{eq:illustrationcapacityintro}, see \eqref{eq:illustrationcapacitynointro}. When $d>4$,  the lower bound in \eqref{eq:illustrationcapacityintro} is still satisfied, but we do not expect it to be sharp anymore, see Remark~\ref{rk:endlocuniq},\ref{rk:sharpalpha<2nu}) for details.

Note that contrary to a lot of the existing literature, see for instance \cite{ChangYinshan2017Toot,AsselahAmine2018Cotr,AsselahAmine2019COTR,asselah2020deviations,dembo2022capacity,adhikariokada}, we study the capacity of $\overline{X}_{B(N)}$, i.e.\ the walk before exiting $B(N)$, and not its range $\{X_k,k\leq N\}$, and it would be interesting to also obtain large deviation bounds on the capacity of the latter in low dimension, which would however probably be of different order. Note that both questions are interesting in terms of intersection of random walks: bounds on the capacity of $\overline{X}_{B(N)}$ can be useful to compute the probability that two random walks intersect in a given ball, whereas bounds on the capacity of $\{X_k,k\leq N\}$ can be useful to compute the probability that two random walks intersect before a given time. Another possibility is to study the capacity of $\{X_k,k\geq0\}\cap B(N)$ instead of $\overline{X}_{B(N)}$, and in fact Theorem~\ref{the:boundcapacityintro} remains true for this change, see Remark~\ref{rk:overdefXbar}. We also refer to \cite{li2022large} for large deviations bounds on the capacity of random interlacements in a box. 

In \cite{MR2819660,DrePreRod,DrePreRod5}, the upper bounds \eqref{eq:oldboundcapacity} (or weaker version of these bounds) have been used to control the probability of local uniqueness for random interlacements. In fact, we will similarly use the better bounds from Theorem~\ref{the:boundcapacityintro} to prove the upper bounds from Theorem~\ref{the:localuniquenessintro} on local uniqueness for random interlacements. Some additional work to obtain these better bounds is however required compared to \cite{DrePreRod5}, see Remark~\ref{rk:endlocuniq},\ref{rk:differentlocaluniq}) as to why, as well as the end of the introduction for an outline of the proof. Interestingly, even if the discrepancy in dimension three in Theorem~\ref{the:boundcapacityintro} may seem stronger than in dimension four, one deduces from it optimal bounds on the local uniqueness event in Theorem~\ref{the:localuniquenessintro} in dimension three, whereas the log adjustments from Theorem~\ref{the:boundcapacityintro} carry other to Theorem~\ref{the:localuniquenessintro} in dimension four. This is due to the fact that the bounds from Theorem~\ref{the:boundcapacityintro} are in fact much stronger in dimension three than four when bounding the probability that the capacity of $\overline{X}_{B(N)}$ is smaller than $N^{d-2}/t$ ($CN^{d-2}$ being the maximal value it can reach), rather than bounding the probability that this capacity is smaller than the (dimension-dependent) capacity of the tube $\mathcal{R}_{N,p}$, see Theorem~\ref{the:boundcapacitynointro}.

\subsection{First passage percolation}

We now highlight a few first passage percolation results of interest, which are presented here on $\Z^d$ and in a simplified context compared to the rest of the article for the reader's convenience. We refer to Theorems~\ref{the:FPP}, \ref{the:limitVunointro} and~\ref{the:FPPnointro} for complete statements on more general graphs. We are first interested in the following FPP distance 
\begin{equation}
\label{eq:defduN}
d_{u}(N)\stackrel{\textnormal{def.}}{=}\inf_{\gamma:0\leftrightarrow B(N)^c}|\text{range}(\gamma)\cap\I^u|,
\end{equation}
where the infimum is taken over all nearest neighbor paths $\gamma\subset \Z^d$ from $0$ to $B(N)^c$, and $\text{range}(\gamma)$ is the set of vertices visited by $\gamma$. This FPP distance was proved in \cite{AndPre} to typically grow linearly in $N$  for all $u>u_{**}$, where $u_{**}$ is the critical parameter above which $\V^u$ is strongly non-percolating, see \eqref{eq:defu**}. In other words, any path from $0$ to $B(N)^c$ will intersect $\I^u$ at least $cN$ times with high probability for any $u>u_{**}$. In the series of articles \cite{sharpnessRI1,sharpnessRI2,sharpnessRI3}, it was recently proven that on $\Z^d$, $d\geq3$, the phase transition associated to $\V^u$ is sharp. In other words, $u_{**}$ is actually equal to the critical parameter $u_*$ associated to the percolation of $\V^u$, that is the largest $u$ such that $\V^u$ contains an infinite cluster with positive probability; see also \cite{MR4568695} for a similar result for the Gaussian free field. Combining this with \cite[Theorem~2.2 and Proposition~4.9]{AndPre}, one deduces that for all $\eta>0$, $u>u_*$, $s\leq c$ and $N\geq 1$ 
\begin{equation}
\label{eq:previousbounddu}
    \P(d_{u}(N)\leq sN))\leq \begin{dcases}
        C\exp\Big(-\frac{cN}{\log(N)^{4(1+\eta)}}\Big)&\text{ if }G=\Z^3,\\
         C\exp\big(-{cN}\big)\!\!\!&\text{ if }G=\Z^d,\,d\geq4
         \end{dcases}
\end{equation}
for some constants $0<c,C<\infty$ depending on $u$, $\eta$ and $d$. The bound \eqref{eq:previousbounddu} is optimal up to constants (depending on $u$) by the FKG inequality when $d\geq4$, but not necessarily when $d=3$. Our first main FPP result is to obtain the exact asymptotic of the probability in \eqref{eq:previousbounddu} when $d=3$.

\begin{The}
\label{the:limitVuintro2}
Let $G=\Z^3$. For all $u>u_{*}$,  
\begin{equation}
\label{eq:limitVuintro}
    \lim\limits_{s\searrow0}\lim\limits_{N\rightarrow\infty}\frac{\log(N)\log(\P(d_{u}(N)\leq sN))}{N}= -\frac{\pi(\sqrt{u}-\sqrt{u_{*}})^2}{3}.
\end{equation}
\end{The}

The lower bound in \eqref{eq:limitVuintro} is proved in \cite{GosRod}, and the main interest of Theorem~\ref{the:limitVuintro2} is thus to show that this lower bound is sharp. Theorem~\ref{the:limitVuintro2} improves \eqref{eq:previousbounddu} by both identifying the leading order in $N$ of the probability in \eqref{eq:previousbounddu} as $\exp(-cN/\log(N))$, as well as by making the constant $c=c(u)$ appearing in this exponential explicit for $N$ large and $s$ small. Note that the left-hand side of \eqref{eq:limitVuintro} is equal to $0$ for $u<u_*$ by \cite[Proposition~2.1]{AndPre}, and also  presumably for $u=u_*$. It is not clear if one can find an explicit value for the left-hand side of \eqref{eq:limitVuintro} for a fixed $s>0$, that is without taking the limit $s\searrow0$. This value would be in any case different from the right-hand side of \eqref{eq:limitVuintro}, as can be seen using the inequality $\P(d_{u}(N)\leq sN)\geq \P(d_{u}(N(1-s/2))\leq sN/2)$.

When $s=0$, the event $d_u(N)\leq sN$ exactly corresponds to the probability of connecting $0$ to $B(N)^c$ in  $\mathcal{V}^u$. In particular, independently of the current work, the upper bound in \eqref{eq:limitVuintro} but for $s$ fixed equal to $0$ was obtained in \cite{GosRod} as one of its results, see also \cite[Theorem~1.1]{GRS21} or \cite[Theorem~1.11]{MuiSev} for similar bounds for the Gaussian free field, and a class of Gaussian fields with covariance between $x$ and $y$ decaying essentially as $|x-y|^{-\nu}$, $\nu\leq1$, respectively. We refer to Remark~\ref{rk:decayVu} for more information on the additional challenges that one must overcome to obtain the upper bound in \eqref{eq:limitVuintro} on the FPP distance compared to the bound from \cite{GosRod} on the probability to  connect $0$ to $B(N)^c$ in  $\mathcal{V}^u$.

The FPP distance from \eqref{eq:defduN} is the most natural choice associated to the percolation of $\V^u$, as it counts the minimal number of times a path starting in $0$ intersects $\I^u=(\V^u)^c$ before reaching $B(N)^c$. It is however only interesting when $u>u_*$, since whenever $\V^u$ percolates, $d_u(N)$ becomes eventually constant with respect to $N$. For a fixed $u$ which can be chosen arbitrarily small, one way to generalize  \eqref{eq:defduN} in a non-trivial FPP distance is to count the minimal number of times a path starting in $0$ intersects the "sausage" $B(\I^u,R)$ before reaching $B(N)^c$, where $B(\I^u,R)$ is the set of points within distance $R$ from $\I^u$. If $uR^{d-2}$ is large enough, then $B(\I^u,R)^c$ does not percolate. This can be demonstrated by noting that $\P(0\in{B(\I^u,R)})$ can be chosen arbitrarily close to $1$ when $uR^{d-2}$ is large in view of \eqref{eq:defIuintro} and since the capacity of a ball of radius $R$ is of order $R^{d-2}$, see \eqref{eq:capball}, and combining this observation with the decoupling inequalities from Lemma~\ref{pro:softlocaltimes}. In particular, the associated FPP distance is not eventually constant anymore. Moreover, the method from Theorem~\ref{the:limitVuintro2} still applies, and so this FPP distance increases linearly with high probability, as discussed in Remark~\ref{rk:endthmintro},\ref{rk:EvsE'}).

It is actually more natural, and useful for later purposes, to  count the whole ball of radius $R$ around $x\in{\I^u}$ only once in the definition of the FPP distance. When considering the  $|\cdot|_{\infty}$ distance for the sausage, which essentially does not change the problem up to rescaling $R$, this can be achieved by restricting the FPP distance to translations of the set $(-\tfrac{R}{2},\tfrac{R}2]^d$ by vertices in the renormalized lattice $R\Z^d=\{Rx:\,x\in{\Z^d}\}$. In other words, one counts the minimal number of $x\in{R\Z^d}$ such that $\I^u$ intersects $x+(-\tfrac{R}{2},\tfrac{R}2]^d$ (which form a partition of $\Z^d$), for all $x$ along a path on $R\Z^d$ from $0$ to $B(N)^c$. Here, we say that $(x_1,\dots,x_m)$ is a path on $R\Z^d$ if $x_i/R$ and $x_{i+1}/R$ are neighbors in $\Z^d$ for all $i<m$. One can then ask with which probability is this renormalized FPP distance proportional to the graph distance on $R\Z^d$, and we now introduce the corresponding event
\begin{equation}
\label{eq:defEintro}
    E=E_{s,R,u}\stackrel{\textnormal{def.}}{=}\left\{\begin{gathered}
        \text{there exists a path in }R\Z^d\text{ from }0\text{ to }B(N)^c\text{ which intersects at}
        \\\text{most }\tfrac{sN}{R}\text{ different }x\in{R\Z^d}\text{ such that }\I^u\cap \big(x+(-\tfrac{R}{2},\tfrac{R}2]^d\big)\text{ is typical}
    \end{gathered}\right\}.
\end{equation}
What is meant by $\I^u\cap \big(x+(-\tfrac{R}{2},\tfrac{R}2]^d\big)$ is typical is that it satisfies a certain property (that one is free to choose) which is increasing and occurs with high enough probability (but not with probability one), even after decreasing the parameter $u$ by some given multiplicative factor. For instance, if one chooses the property that $\I^u\cap \big(x+(-\tfrac{R}{2},\tfrac{R}2]^d\big)\neq\varnothing$, which is typical when $uR^{d-2}$ is large as explained before, the event $E$ says that the FPP distance introduced above \eqref{eq:defEintro} is smaller than $sN/R$, and we are thus interested in the probability of this event. For any $\eta\in{(0,1)}$ and $u>0$, if $R=\zeta u^{\frac1{d-2}}$ for some $\zeta\geq C$, $s\leq c$ and $N\geq Cu^{\frac1{d-2}}$, then we show that
\begin{equation}
\label{eq:boundonPEintro}
\begin{split}
       \exp\left(-\frac{\pi(1+\eta)Nu}{3\log(Nu)}\right)&\leq \P(E)\leq 
    \exp\left(-\frac{\pi(1-\eta)Nu}{3\log(Nu)}\right)\text{ if }G=\Z^3,
    \\        \exp\left(-CNu^{\frac{1}{d-2}}\right)&\leq \P(E)\leq 
     \exp\left(-cNu^{\frac{1}{d-2}}\right)\text{ if }G=\Z^d, d\geq4,
\end{split}
\end{equation}
where the constants $0<c,C<\infty$ depend only on $\zeta$, $\eta$ and $d$. The bounds in \eqref{eq:boundonPEintro} are valid when $\I^u\cap \big(x+(-\tfrac{R}{2},\tfrac{R}2]^d\big)$ is typical with sufficiently high probability. We state \eqref{eq:defEintro} and \eqref{eq:boundonPEintro} in a rather loose way here for the sake of readability, and refer the reader to Theorem~\ref{the:FPPnointro} for a more precise statement, see in particular below \eqref{eq:defgzZd} for a precise definition of "typical" in \eqref{eq:defEintro}.

On $\Z^3$, \eqref{eq:boundonPEintro} can be loosely interpreted as a version of Theorem~\ref{the:limitVuintro2} associated with the percolation of the complement $B(\I^u,R)^c$ of the "sausage" around $\I^u$  instead of the complement $\V^u$ of $\I^u$. The intuitive reason why $u_*$ appears in \eqref{eq:limitVuintro} but not in \eqref{eq:boundonPEintro} is that the critical parameter associated to the percolation of $B(\I^u,R)^c$ can be made arbitrarily small when $uR^{d-2}$ is large enough, and thus absorbed in the factor $\eta$ in \eqref{eq:boundonPEintro}.  The bounds in \eqref{eq:boundonPEintro} are moreover much stronger than the result \eqref{eq:limitVuintro} since they hold for $Nu$ large enough, instead of $N\rightarrow\infty$, and can thus be interpreted as a result on the critical window of the random interlacement set.  In addition, on $\Z^d$, $d\geq4$, the bounds in \eqref{eq:boundonPEintro} are much stronger than the ones from \eqref{eq:previousbounddu}, as the dependency of the constants in the exponential is explicit in $u$. Theorems~\ref{the:localuniquenessintro} and~\ref{the:boundcapacityintro} are applications of the bound \eqref{eq:boundonPEintro}, which are obtained when considering in \eqref{eq:defEintro} that $\I^u\cap \big(x+(-\tfrac{R}{2},\tfrac{R}2]^d\big)$ is typical when its capacity is roughly comparable to that of $x+(-\tfrac{R}{2},\tfrac{R}2]^d$. Using \cite[Lemma~3]{MR2819660}, one can easily see that the expected value of the capacity of a random walk (that one can see as being part of $\I^u$) in a box is in fact much smaller than the capacity of this box if  $d\geq5$. The reason we restrict ourselves to dimensions three and four in Theorems~\ref{the:localuniquenessintro} and~\ref{the:boundcapacityintro} is thus to make sure that the previous event is indeed typical.

The bounds \eqref{eq:boundonPEintro} moreover imply that the time constant  for the FPP associated to the percolation of $B(\I^u,R)^c$ is proportional to $u^{\frac1{d-2}}$ for $\zeta=uR^{d-2}$ large enough, whereas the dependency of the time constant in $u$ was not clear in \eqref{eq:limitVuintro}.  Proving this result alone would however be much easier than obtaining the sharp bounds from \eqref{eq:boundonPEintro}, see Proposition~\ref{pro:apriori}, and could actually be deduced on $\Z^d$ from \cite{AndPre} (but not on more general graphs), see Remark~\ref{rk:apriori},\ref{rk:previousFPPbounds}). 

Finally, our proof of Theorems~\ref{the:limitVuintro2} and~\eqref{eq:boundonPEintro} is in fact not restricted to the FPP problem associated to the percolation of $B(\I^u,R)^c$, $R\geq0$, and one can in fact choose any other FPP distance which depend on interlacements in a monotonic fashion, see Theorem~\ref{the:FPP}. Moreover, our method can be extended when considering FPP distances between two points $x$ and $y$ at distance $N$, instead of between $0$ and $B(N)^c$ as in \eqref{eq:defduN}, see Remark~\ref{rk:endthmintro},\ref{rk:2ptfunction}), or alternatively between $B(N)$ and $B(2N)^c$, see Remark~\ref{rk:endthmintro},\ref{rk:annuluscrossing}). It is also likely that our strategy can be applied to other processes such as the Gaussian free field, see Remark~\ref{rk:endthm},\ref{rk:extensiontoGFF}).

\subsection{Sketch of the proofs}

Let us now briefly comment on the proof of our results. We refer to the paragraphs above Section~\ref{sec:coarsegraining} for some explanation on the proof of our FPP results Theorem~\ref{the:limitVuintro2} and \eqref{eq:boundonPEintro}, which is based on ideas from \cite{MR3602841,GRS21} suitably extended from percolation to first passage percolation, and that we will additionally generalize from $\Z^d$ to the more general class of graphs we consider in this article. On the way, we improve and partially simplify the previous techniques, and draw for instance the attention of the reader to Proposition~\ref{pro:entropy} (see Remark~\ref{rk:decayVu},\ref{rk:noapriori}) for details on its interest), which is an improvement of the coarsening of crossing paths result from \cite[Proposition~4.3]{GRS21}, and Proposition~\ref{pro:softlocaltimes} (see the paragraph below this proposition as to why it simplifies our proof), which is an improved version of the decoupling of interlacements result from \cite[Proposition~5.1]{MR3602841}.

We are now going to describe how the upper bounds in Theorems~\ref{the:localuniquenessintro} and \ref{the:boundcapacityintro} can be deduced from \eqref{eq:boundonPEintro}, as well as how to prove the corresponding lower bounds, and we focus on the case $d=3$ for simplicity, the case $d=4$ being very similar with some additional logarithmic corrections. The main step in the proof of both these results is the following asymptotic for the probability that a random walk avoids the interlacement set before exiting a ball of size $N$ around its starting point: on $\Z^3$ and when $Nu$ is large enough,
\begin{equation}
    \label{eq:interetmarcheintro}
         \P\otimes\P_x(\mathcal{I}^u\cap \overline{X}_{B(x,N)}=\varnothing)
        \approx\exp\left(-\frac{\pi Nu}{3\log(Nu)}\right),
\end{equation}
where $\approx$ roughly means that one can prove bounds similar to the ones from \eqref{eq:boundonPEintro}, and we refer to Lemma~\ref{lem:hittinginter} for a more precise statement. Note that a similar probability but for the intersection of two independent interlacements has been studied in \cite[Theorem~1.2]{li2022large}. In order to prove \eqref{eq:interetmarcheintro}, we use \eqref{eq:boundonPEintro}, where $\I^u\cap \big(x+(-\tfrac{R}{2},\tfrac{R}2]^d$ is said to be typical in \eqref{eq:defEintro} if its capacity is larger than $cR$. This event is indeed typical since the capacity of a random walk is larger than $cR$ with high probability by \cite[Lemma~5.3]{DrePreRod5}. This implies that, except on an event with probability of the same order as the right-hand side of \eqref{eq:interetmarcheintro},  any path from $0$ to $B(N)^c$  must hit $x+(-\tfrac{R}{2},\tfrac{R}2]^d$ for at least $cNu$ different $x\in{R\Z^d}$ such that $x+(-\tfrac{R}{2},\tfrac{R}2]^d$ contains a cluster of interlacements with capacity at least $cR$. In other words, except on the same event as before, a random walk started in $x$ will be at distance less than $R$ from a random interlacement cluster with capacity larger than $CR$ at least $cNu$ times before exiting $B(x,N)$, and since each time it has independently a constant probability to intersect the interlacement cluster and $cNu$ is much smaller than $Nu/\log(Nu)$, this yields the upper bound in \eqref{eq:interetmarcheintro}.

To show that the left-hand side of \eqref{eq:interetmarcheintro} can also be lower bounded by its right-hand side, one simply asks that the walk remains confined in the tube $\mathcal{R}_{N,p},$ see \eqref{eq:RNp} and that the interlacement set $\I^u$ does not hit this tube, and one can conclude after optimizing the choice of $p=p(N,u)$. In view of \eqref{eq:defIuintro}, this optimal choice of $p$ depends on the capacity of this tube, which is computed in Lemma~\ref{lem:capatube} and is of order $N/\log(N/p)$. We refer to \cite[Proposition~6.1]{DrePreRod5} for a proof using similar ideas in a slightly different context.  

Let us now explain the link between Theorem~\ref{the:localuniquenessintro} and \eqref{eq:interetmarcheintro}. The lower bound in \eqref{eq:boundlocaluniqintro} can be intuitively deduced from \eqref{eq:interetmarcheintro} as follows: if a random walk can avoid the interlacement set in $B(x,\lambda N)$ with probability at least the right-hand side of \eqref{eq:interetmarcheintro}, then any trajectory of random interlacements (which consist of doubly-infinite random walks) hitting $B(x,N)$ should be able to also avoid all the other trajectories of random interlacements in $B(x,\lambda N)$, which contradicts local uniqueness, with the same probability. Note that actually for technical reasons, we do not deduce the lower bound in \eqref{eq:boundlocaluniqintro} from \eqref{eq:interetmarcheintro}, but the proof ideas are very similar. However, the upper bound in \eqref{eq:boundlocaluniqintro} cannot be directly deduced from the upper bound in \eqref{eq:interetmarcheintro} since it is possible that every trajectory of random interlacements hitting $B(x,N)$ intersect some other trajectory of random interlacements in $B(x,\lambda N)$, without having local uniqueness (this is for instance the case if there are two big clusters of random interlacements in $B(x,\lambda N)$, each containing many trajectories).

In order to be able to connect any two trajectories of random interlacements together, we will actually first prove the upper bounds in Theorem~\ref{the:boundcapacityintro}, which are direct consequences of \eqref{eq:interetmarcheintro}. Indeed, by  \eqref{eq:defIuintro} we have for all $x\in{G}$, $u>0$ and $N\geq1$
\begin{equation}
\label{eq:linkcapRI}
    \E_x\big[\!\exp(-u\mathrm{cap}(\overline{X}_{B(x,N)})\big]=\P\otimes\P_x(\mathcal{I}^u\cap \overline{X}_{B(x,N)}=\varnothing),
\end{equation}
and so \eqref{eq:interetmarcheintro} gives bounds on the Laplace transform of $\mathrm{cap}(\overline{X}_{B(x,N)})$. Using Chernov bounds, the upper bound in \eqref{eq:illustrationcapacityintro} follows easily. For the corresponding lower bound, one simply asks that the random walk remains confined in $\mathcal{R}_{N,p}$, which occurs with probability $\exp(-cN/p)$, see Lemma~\ref{lem:exitviatubes} and below \eqref{eq:illustrationcapacitynointro}.

It remains to explain how to deduce the upper bound in \eqref{eq:boundlocaluniqintro} from Theorem~\ref{the:boundcapacityintro}, which uses a new strategy different from \cite{DrePreRod5}, see Remark~\ref{rk:endlocuniq},\ref{rk:differentlocaluniq}). One essentially considers any two trajectories $Y$ and $Z$ of random interlacements hitting $B(x,N)$, and computes the probability that some trajectory of random interlacements hitting $Y$ in $B(x,(1-\eta)^2 \lambda N)$ will intersect in $B(x,\lambda N)$ another trajectory of random interlacements hitting $Z$ in $B(x,(1-\eta)^2\lambda N)$. The number of trajectories of random interlacements hitting $Y$ or $Z$ in $B(x,(1-\eta)^2\lambda N)$ is of order $u$ times their capacity as explained above \eqref{eq:defRI}, which is larger than $u((1-\eta)^2\lambda-1)N/t$ with a probability which can be deduced from Theorem~\ref{the:boundcapacityintro}. Moreover, one can easily show that any two of these $cuN/t$ trajectories started in $B(x,(1-\eta)^2\lambda N)$ will intersect each other with constant probability in $B(x,\lambda N)$, see \eqref{eq:interRW}, which essentially finishes the proof of Theorem~\ref{the:localuniquenessintro} up to constants after optimizing on $t$. One actually needs to use additional trajectories in order to obtain the exact constant $\pi/3$ in the exponential on the right-hand side of \eqref{eq:boundlocaluniqintro}, and we refer to Remark~\ref{rk:endlocuniq},\ref{rk:differentlocaluniq}) for details.

\subsection{Outline of the rest of the article}

We now describe how this article is organized. Section 2 introduces the precise
framework, including the set of standing assumptions we will make on the graph $\G$, which we list here for the reader's orientation:
\begin{equation}\label{eq:standingassumptionintro}
\eqref{eq:intro_sizeball},\eqref{eq:intro_Green},\eqref{eq:intro_Green_asymp},\eqref{p0},\eqref{eq:condnualpha}\text{ and }\eqref{eq:intro_shortgeodesic}.
\end{equation} 
It further defines random interlacements, and derives in Lemma~\ref{lem:capatube} useful bounds on the capacity of tubes.

Section~\ref{sec:FPP} is devoted to the proof of Theorem~\ref{the:FPP}, which provides a general bound on first passage percolation for random interlacements. We gather the main pieces of the proof in Section~\ref{sec:coarsegraining}, see in particular Lemma~\ref{lem:coarsegraining} therein, which together with the a priori estimate from Section~\ref{sec:apriori} are assembled in Section~\ref{sec:prooftheFPP}. The main building blocks used in Section~\ref{sec:coarsegraining} are Propositions~\ref{pro:thm4.2disco}, \ref{pro:softlocaltimes} and~\ref{pro:entropy}, which are in turn proved in Section~\ref{sec:proofof3prop}. 

The purpose of Section~\ref{sec:proofofFPPintro} is to quickly explain how to deduce the upper bounds in Theorems~\ref{the:limitVuintro2} and~\eqref{eq:boundonPEintro} from the more general Theorem~\ref{the:FPP}, and to prove the matching lower bounds. It also contains their generalization to graphs satisfying \eqref{eq:standingassumptionintro}, see Theorems~\ref{the:limitVunointro} and~\ref{the:FPPnointro}. Finally, Section~\ref{sec:localuniqueness} is centered around the proof of our applications of \eqref{eq:boundonPEintro}, which starts with proving \eqref{eq:interetmarcheintro}, see Lemma~\ref{lem:hittinginter}. Then, following the strategy described above, we deduce in Theorems~\ref{the:boundcapacitynointro} and~\ref{the:localuniquenessnointro} the respective generalization of Theorems~\ref{the:localuniquenessintro} and \ref{the:boundcapacityintro} to graphs satisfying \eqref{eq:standingassumptionintro}.

We conclude this introduction with our convention regarding constants. Throughout $c$, $c'$, $C$, $C'$, \dots\ denote positive constants changing from place
to place. Numbered constants $c_1$, $C_1$, $c_2$, $C_2$, \dots\ are fixed when they first appear and do
so in increasing numerical order. All constants depend implicitly on the choice of the graph $\G$ endowed with a distance $d$
through the conditions from \eqref{eq:standingassumptionintro}, and in particular they often depend on $\alpha$ and
$\nu$. Their dependence on any other quantity will be made explicit.

\medskip

{\bf Acknowledgment:}
This article was written while the author was supported by the Swiss NSF. The author thanks Sebastian Andres for pointing out the reference~\cite[Proposition~3.8]{andres2016harnack}, Subhajit Goswami and Pierre-François Rodriguez for several useful discussions, especially on the use of the renormalization scheme illustrated in Figure~\ref{fig:renscheme}, Artem Sapozhnikov for providing the idea behind the proof of Proposition~\ref{pro:apriori}, and an anonymous reviewer for a very thorough revision.

\section{Preliminaries}
\label{sec:preliminaries}
In this section we first introduce the standing assumptions on the graphs that we will assume to hold throughout the article, then recall some useful potential theory and facts about random interlacements, and finally prove sharp bounds on the capacity of tubes, see Lemma~\ref{lem:capatube}.

Let us denote by $\G=(G,\lambda)$ the weighted graph mentioned in the introduction, where the weights $\lambda$ satisfy $\lambda_{x,x}=0$ and $\lambda_{x,y}=\lambda_{y,x}$ for all $x,y\in{G}$. We call $\{x,y\}\in{G^2}$ an edge of $\G$ if and only if $\lambda_{x,y}>0$, and we write $x\sim y$ if and only if $\{x,y\}$ is an edge of $\G$. We also assume that $\G$ endowed with the previous graph structure is locally finite, that is $\{y\in{G}:\lambda_{x,y}>0\}$ is finite for each $x\in{G}$. Furthermore, we endow $\G$ with a distance $d$, which is not necessarily the graph distance, that will also sometimes play a role and is denoted by $d_{\text{gr}}$. Let us write $B(x,R)=\{y\in{G}:\,d(x,y)\leq R\}$ for the closed ball centered at $x$ and with radius $R$, and $\partial B(x,R)=\{y\in{B(x,R)},\exists\,z\notin{B(x,R)}:\,y\sim z\}$ for its internal boundary.

Let us start with some definitions about the random walk on $\G$, which will be useful to state the conditions from \eqref{eq:standingassumptionintro}. We denote by $\P_x$ the law of the random walk $(X_k)_{k\in{\N_0}}$ on the weighted graph $\G$ starting in $x$, that is the random walk on $G$ with transition probability 
\begin{equation}
    \label{eq:pxy}
    p_{y,z}\stackrel{\textnormal{def.}}{=}\frac{\lambda_{y,z}}{\lambda_y},\text{ where }\lambda_y\stackrel{\textnormal{def.}}{=}\sum_{v\in{G}}\lambda_{y,v},
\end{equation}
from $y$ to $z$. The associated Green function is defined via 
\begin{equation*}
    g(x,y)\stackrel{\textnormal{def.}}{=}\frac{1}{\lambda_y}\E_x\Big[\sum_{k=0}^{\infty}1\{X_k=y\}\Big]\text{ for all }x,y\in{G}.
\end{equation*}

We are now ready to introduce the conditions \eqref{eq:standingassumptionintro} on the graph $\G=(G,\lambda)$, endowed with the distance $d$, which essentially correspond to the setup from \cite{DrePreRod2,DrePreRod5}, see however \eqref{eq:intro_shortgeodesic} and below for a small difference. We first assume that the graph $\G$ has regular volume growth of
degree $\alpha$ with respect to $d$, that is there exist constants  $0<\Cl[c]{csizeball}\leq \Cl{Csizeball}<\infty$ such that
\begin{equation}
\label{eq:intro_sizeball} \tag{$V_{\alpha}$}
\Cr{csizeball}r^{\alpha}\leq \lambda(B(x,r))\leq \Cr{Csizeball}r^{\alpha}\quad \text{ for all }x\in{G}\text{ and }r\geq1,
\end{equation}
where $\lambda(A)=\sum_{x\in{A}}\lambda_x$ for each $x\in{A}$ and $A\subset\subset G$. Here the notation $A\subset\subset G$ means that $A$ is a finite subset of $G$. Our second assumption is that the Green function decays polynomially fast with exponent $\nu>0$, that is there exist $\Cl[c]{cGreen}>0$ and $\Cl{CGreen}<\infty$ such that
\begin{equation}
\label{eq:intro_Green}\tag{$G_{\nu}$}
\begin{split}
&\Cr{cGreen}\leq g(x,x)\leq \Cr{CGreen}
\text{ and }  
\Cr{cGreen} d(x,y)^{-\nu}\leq g(x,y)\leq \Cr{CGreen} d(x,y)^{-\nu} \quad \text{ for all }x \neq y\in{G}.
\end{split}
\end{equation}
In particular, the graph $\G$ is transient.
We moreover assume that there are constants $\Cl[c]{cGreenasymp}>0$ and $\Cl{CGreenasymp}<\infty$ such that
\begin{equation}
\label{eq:intro_Green_asymp}
\forall\,\eta>0,\,\exists\, N>0,\,\forall\,x,y\in{G},\ d(x,y)\geq N\Longrightarrow (1-\eta)\Cr{cGreenasymp}\leq g(x,y)d(x,y)^{\nu}\leq (1+\eta)\Cr{CGreenasymp}.
\end{equation}
It is clear that \eqref{eq:intro_Green_asymp} is satisfied if $\Cr{cGreenasymp}=\Cr{cGreen}$ and $\Cr{CGreenasymp}=\Cr{CGreen}$ and so \eqref{eq:intro_Green_asymp} is not really a new assumption on $\G$ but rather a definition of the constants $\Cr{cGreenasymp}$ and $\Cr{CGreenasymp}$, which will eventually play an important role in the proof of Theorem~\ref{the:limitVuintro2} and \eqref{eq:boundonPEintro}, see also Lemma~\ref{lem:capatube} below. Denoting by $B(x,y)$ the Beta function, we further introduce constants $\Cl{Cbeta}$ and $\Cl[c]{cbeta}$ defined via
\begin{equation}
\label{eq:cbetanointro}
\Cr{Cbeta}\Cr{cGreenasymp}=\Cr{cbeta}\Cr{CGreenasymp}\stackrel{\textnormal{def.}}{=}\begin{cases}
    \frac{1}{\pi} B\left(\frac{1+\nu}{2},\frac{1+\nu}{2}\right)\cos\left(\frac{\pi\nu}{2}\right)&\text{ if }\nu<1,
    \\\frac{1}{2}&\text{ if }\nu=1.
\end{cases}
\end{equation}
These constants essentially correspond to the constant $c_{\alpha}$ from \cite[(1.9)]{MuiSev}, and will play a similar role in our main results. Note that our results will be more interesting when one can take $\Cr{Cbeta}=\Cr{cbeta}$, that is when $g(x,y)d(x,y)^{\nu}$ converges as $d(x,y)\rightarrow\infty$ uniformly in $x$ and $y$, see for instance \eqref{eq:bounddtvnointro} below, which is for instance the case on $\Z^d$ endowed with the Euclidean distance as we will explain below.

Our next assumption on $\G$ is that the weights are controlled, that is there exists $c>0$ such that 
\begin{equation}
\tag{$p_0$}
\label{p0}
    p_{x,y}\geq c\text{ for all }x\sim y\in{G}.
\end{equation}
Together, \eqref{eq:intro_sizeball}, \eqref{eq:intro_Green} and \eqref{p0} form a classical framework under which one can prove upper and lower bound heat kernel estimates when $d$ is the graph distance, see \cite{MR1853353}, and we refer to \cite[Section~3]{DrePreRod2} for extensions to general distances $d$. When $d=d_{\text{gr}}$, as explained around \cite[(1.18)]{DrePreRod5}, it follows from \cite{MR2076770} that
\begin{equation}
\label{eq:condnualpha}
    0<\nu\leq \alpha-2,
\end{equation}
and we will in fact always assume that \eqref{eq:condnualpha} is satisfied, even for general distances $d$.  Finally, our last condition is new, and assumes that there exist constants $\Cl[c]{cgeo}\in{(0,\infty)}$ and $\Cl{Cgeo}<\infty$ such that
\begin{equation}
\label{eq:intro_shortgeodesic}
\begin{array}{l}
\text{for all }x\in{G},\text{ there exists an infinite path }(x=x_0,x_1,x_2,\dots) \\[0.3em]
\text{such that }\Cr{cgeo}|k-p|-\Cr{Cgeo}\leq d(x_k,x_p)\leq \Cr{cgeo}|k-p|+\Cr{Cgeo}\text{ for all }k,p\in{\N}.
\end{array}
\end{equation}
Condition~\eqref{eq:intro_shortgeodesic} is in fact only necessary to obtain the lower bounds in our main results, and we refer to Remark~\ref{rk:whyintroshortgeodesic} for details. It essentially plays the same role as \cite[(1.16)]{DrePreRod5}, and in fact the results from \cite{DrePreRod5} would still be satisfied if one replaces \cite[(1.16)]{DrePreRod5} by \eqref{eq:intro_shortgeodesic} (one essentially needs to replace the argument around \cite[(6.12)]{DrePreRod5} by Lemma~\ref{lem:exitviatubes} below). The condition \eqref{eq:intro_shortgeodesic} would be implied by \cite[(1.16)]{DrePreRod5} if one would actually require that $d_{\mathrm{gr}}(y_k,y_p)=c_2d(y_k,y_p)$ therein, where $d_{\mathrm{gr}}$ denotes the graph distance on $\G$. In other words \eqref{eq:intro_shortgeodesic} is satisfied if $(x=x_0,x_1,x_2,\dots)$ is a geodesic for $d_{\mathrm{gr}}$ starting in $x$, which always exists by \cite[Theorem~3.1]{MR864581}, and $d=\Cr{cgeo}d_{\mathrm{gr}}$ along this geodesic, which is for instance the case on $\Z^d$ when $d$ is the Euclidean distance. In fact the previous requirement is fulfilled for all the usual examples of graphs satisfying the other conditions in \eqref{eq:standingassumptionintro}, see for instance \cite[(1.4)]{DrePreRod2} or \cite[Theorem~2]{MR2076770}, and is thus essentially no loss of generality. The reason why we require \eqref{eq:intro_shortgeodesic} instead of \cite[(1.16)]{DrePreRod5} is essentially because we want more precise control on the capacity of tubes (i.e.\ we want that the capacity of the set $\mathcal{L}_x^{P,N}$ from Lemma~\ref{lem:exitviatubes} below can be deduced from Lemma~\ref{lem:capatube}) to obtain the exact constant $\pi/3$ appearing in Theorem~\ref{the:limitVuintro2}, whereas in \cite[Theorem~1.4]{DrePreRod5} we were not paying attention to this kind of constants. We also refer to \eqref{eq:intro_shortgeodesic2} below for a stronger version of \eqref{eq:intro_shortgeodesic} that can be useful when considering the FPP distance from $0$ to a given point at distance $N$, instead of the FPP distance from $0$ to $B(N)^c$ as in Theorem~\ref{the:limitVuintro2}.

From now on, we will always implicitly assume that all the previous conditions on $\G$ listed in \eqref{eq:standingassumptionintro} are satisfied. Arguably the most interesting example is the square lattice $\Z^d$, $d\geq3$, with unit weights. Endowing $\Z^d$ with the Euclidean distance $d(x,y)=|x-y|_2$, one further has $\nu=d-2$ and one can take $\Cr{cGreenasymp}=\Cr{CGreenasymp}=\frac{d}{2}\Gamma(\frac{d}{2}-1)\pi^{-d/2}(=\frac{3}{2\pi}$ when $d=3$) in \eqref{eq:intro_Green_asymp} by \cite[Theorem~1.5.4]{MR1117680} and so by \eqref{eq:cbetanointro}
\begin{equation}
\label{eq:assumptionZd}
   \Z^d\text{ satisfies }\eqref{eq:standingassumptionintro}\text{ with } \alpha=d, \nu=d-2,\text{ and one has }\Cr{cbeta}=\Cr{Cbeta}=\frac{\pi}{3}\text{ when }d=3.
\end{equation}
The constants $\Cr{cbeta}=\Cr{Cbeta}$ are exactly the one which eventually appear in \eqref{eq:limitVuintro} on $\Z^3$, and explain the interest of the definition \eqref{eq:cbetanointro}. There are many other interesting graphs which satisfy \eqref{eq:standingassumptionintro}, see for instance \cite[(1.4)]{DrePreRod2} for concrete examples, which include Cayley graphs or fractal graphs with polynomial volume growth. In fact, by \cite[Theorem~2]{MR2076770}, see below \cite[(1.18)]{DrePreRod5} as to why, for each choice of $\alpha,\nu$ satisfying \eqref{eq:condnualpha}, there exists a graph satisfying \eqref{eq:standingassumptionintro}. It is particularly interesting that our results also apply to values of $\nu\in{(1,\alpha/2)}$ (which is never the case on $\Z^d$) for our application to critical exponents for the Gaussian free field on metric graphs in \cite{DrePreRod8}, see for instance below Corollary~1.3 therein.
 
Let us now gather some interesting consequences of our standing assumptions \eqref{eq:standingassumptionintro}. By \cite[Lemma~6.1]{DrePreRod2} there exists a family of sets $\Lambda(L)\subset G$, $L>1$, and a constant $\Cl{CLambda}<\infty$ such that for all $x\in{G}$, $L>1$ and $N\geq1$ 
\begin{equation}
\label{eq:defLambda}
    \bigcup_{y\in{\Lambda(L)}}B(y,L)=G\quad\text{ and }\quad|\Lambda(L)\cap B(x,LN)|\leq \Cr{CLambda}N^{\alpha}.
\end{equation}
For $L=1,$ we simply choose $\Lambda(L)=G,$ which also satisfies \eqref{eq:defLambda} up to changing $\Cr{CLambda}$. For instance on $\Z^d$, $d\geq 3$, one can take $\Lambda(L)=(L/\sqrt{d})\Z^d$. The sets $\Lambda(L)$ serve as a renormalized version by $L$ of the graph $G$, and will sometimes be the graph on which we study first passage percolation, similarly as in \eqref{eq:boundonPEintro}. A path $\gamma=\{\gamma_1,\dots,\gamma_N\}\subset\Lambda(L)^N$ is said to be $L$-nearest neighbor if for each $i\leq N-1$, either $L> 1$ and  there exist $x\in{B(\gamma_i,L)}$ and $y\in{B(\gamma_{i+1},L)}$ with $x\sim y,$ or $L=1$ and $\gamma_i\sim \gamma_{i+1}$. 

By \cite[Lemma~2.3]{DrePreRod2}, the conditions \eqref{eq:standingassumptionintro} moreover imply that the weights $\lambda$ are bounded, that is 
there exist constants $c>0$ and $C<\infty$ such that
\begin{equation}
\label{eq:boundonlambda}
    c\leq \lambda_{x,y}\leq \lambda_x\leq C\text{ for all }x\sim y\in{G},
\end{equation}
and that $d$ is dominated by the graph distance $d_{\rm{gr}}$ on $\G$, that is there exists $\Cl[c]{Cdvsdgr}>0$ such that
\begin{equation}
    \label{eq:dvsdgr}
    d_{\rm{gr}}(x,y)\geq \Cr{Cdvsdgr}d(x,y)\text{ for all }x,y\in{G}.
\end{equation}
As explained above \cite[(3.3)]{DrePreRod2}, one can moreover deduce from \cite[Lemma~A.2]{MR2778797} an elliptic Harnack inequality on $G$. For this purpose, recall that a function $f$ on $G$  is called harmonic in $U\subset\subset G$ if $\E_x[f(X_1)] = f(x)$ for all $x\in{U}$. There exist constants $c\in(0,1)$ and $C\in{(1,\infty)}$ such that for all $x\in{G}$, $L\geq1$, $K\geq C,$ and every non-negative functions $f$ on $G$ which are harmonic in \nolinebreak$B(x,KL)$, 
\begin{equation}
\label{EHI}
    \inf_{y\in{B(x,L)}}{f(y)}\geq c\sup_{y\in{B(x,L)}}{f(y)}.
\end{equation}
Using a chaining argument and a similar reasoning as in \cite[Proposition~3.8]{andres2016harnack} (see the arXiv version for a full proof), which is inspired by \cite[p.50]{saloff2002aspects}, one can write more explicitly  the dependency on $K$ of the constant $c$ in \eqref{EHI} as follows: there exist constants $\Cl[c]{cHarnackprecise}\in{(0,1]}$ and $\Cl{CHarnackprecise}<\infty$ such that for all $x\in{G}$, $L\geq1$, $K\geq 1$  and every non-negative functions $f$ on $G$ which are harmonic in \nolinebreak$B(x,KL)$, 
\begin{equation}
\label{EHIK}
\sup_{y\in{B(x,L)}}f(y)\leq (1+\Cr{CHarnackprecise}K^{-\Cr{cHarnackprecise}})\inf_{y\in{B(x,L)}}f(y).
\end{equation}
We now recall some results from potential theory, and first introduce some useful notation on hitting and exit times for the random walk $X$. For $A\subset\subset G$, we denote by $H_A\stackrel{\textnormal{def.}}{=}\inf\{n\geq0:X_n\in{A}\}$ the first hitting time of $A$, by $\tilde{H}_A\stackrel{\textnormal{def.}}{=}\inf\{n\geq1:\,X_n\in{A}\}$ the first return time to $A$, by $L_A\stackrel{\textnormal{def.}}{=}\sup\{k\geq0:X_k\in{A}\}$ the last exit time of $A$, and by $T_A\stackrel{\textnormal{def.}}{=}H_{A^c}$ the first exit time of $A$, where we use the convention $\inf\varnothing=+\infty$ and $\sup\varnothing=-\infty$. Let us now define the equilibrium measure and capacity of a set $A\subset\subset G$ via
\begin{equation}
    \label{eq:defequicap}
    e_A(x)\stackrel{\textnormal{def.}}{=}\lambda_x\P_x(\tilde{H}_A=\infty)1\{x\in{A}\}\text{ for all }x\in{G}\text{ and }\mathrm{cap}(A)\stackrel{\textnormal{def.}}{=}\sum_{x\in{A}}e_A(x).
\end{equation}
We also denote by $\bar{e}_A(x)=e_A(x)/\mathrm{cap}(A)$ the normalized equilibrium measure of $A$, and abbreviate $e_A(A')=\sum_{x\in{A'}}e_A(x)$ for all $A'\subset A$. Another useful definition of the capacity is via the following variational formula, see \cite[(1.61)]{MR2932978},
\begin{equation}
\label{variational}
    \mathrm{cap}(A)=\Big(\inf_{\mu}\sum_{x,y\in{A}}g(x,y)\mu(x)\mu(y)\Big)^{-1}, \quad \text{ for all } A \subset \subset G,
\end{equation}
where the infimum is among all probability measures $\mu$ supported on $A$, and is reached at $\mu=\bar{e}_A$. The main interest of the equilibrium measure is that it can be used to compute hitting probabilities via the following formula
\begin{equation}
\label{entrancegreenequi}
     \P_x(H_{A}<\infty)=\sum_{y\in{A}}g(x,y)e_{A}(y)\text{ for all }x\in{G},
\end{equation}
see for instance \cite[(1.57)]{MR2932978}. As explained above \cite[(3.11)]{DrePreRod2}, one can easily deduce from \eqref{eq:intro_Green} and \eqref{entrancegreenequi}  that there exist constants $\Cl[c]{ccapball}>0$ and $\Cl{Ccapball}<\infty$ such that
\begin{equation}
\label{eq:capball}
    \Cr{ccapball}R^{\nu}\leq \mathrm{cap}(B(x,R))\leq \Cr{Ccapball} R^{\nu}\text{ for all }x\in{G}\text{ and }R\geq1.
\end{equation}
In particular, combining \eqref{eq:intro_Green}, \eqref{entrancegreenequi} and \eqref{eq:capball},  one readily deduces that for any $R,K\geq2$ and $x\in{G}$
\begin{equation}
\label{eq:boundentrance}
\P_y(H_{B(x,R)}<\infty)\leq \frac{2^{\nu}\Cr{CGreen}\Cr{Ccapball}}{K^{\nu}}\text{ for all }y\in{B(x,KR)^c}.
\end{equation}
Moreover, for the lower bound, one notes that for any $x\in{G}$, $R,K\geq1$, $z\in{B(x,R)}$ and set $A\subset B(x,R)$  
\begin{equation*}
    \begin{split}
        \P_z(H_{A}<T_{B(x,KR)})&\geq \P_z(H_{A}<\infty)-\sup_{z'\in{\partial B(x,KR)^c}}\P_{z'}(H_{A}<\infty)
        \\&\geq \Cr{cGreen}(2R)^{-\nu}\mathrm{cap}(A)-\Cr{CGreen}((K-1)R)^{-\nu}\mathrm{cap}(A),
    \end{split}
\end{equation*}
where we used \eqref{entrancegreenequi} and \eqref{eq:intro_Green} twice in the last inequality, and so there exists a constant $\Cl{Chittinglower}<\infty$ so that uniformly in $x,R,z$ and $A$ as before
\begin{equation}
        \label{eq:probahittinglower}
        \P_z(H_{A}<T_{B(x,\Cr{Chittinglower}R)})\geq \frac{\Cr{cGreen}\mathrm{cap}(A)}{4^{\nu}R^{\nu}}.
\end{equation}

Let us now recall the definition of random interlacements on transient graphs. We denote by $W^*$  the space of doubly infinite trajectories on $G$ modulo time-shift, that is the quotient of the space $W$ of nearest neighbor paths $(w(k))_{k\in{\Z}}$ in $G$ by the equivalence relation $w\sim w'$ if and only if there exists $p\in{\Z}$ such that $w(k)=w'(k+p)$ for all $k\in{\Z}$.  Under some probability measure denoted by $\P$, the random interlacement process $\omega=\sum_{i\in{\N}}\delta_{(w_i^*,u_i)}$ is a Poisson point process on $W^*\times[0,\infty)$ with intensity measure $\nu\times\lambda$, where $\nu$ is the interlacement measure as defined in \cite[Theorem~2.1]{MR2525105} and $\lambda$ is the Lebesgue measure on $[0,\infty)$. The random interlacement process at level $u$ is then the associated point process $\omega_u\stackrel{\textnormal{def.}}{=}\sum_{i\in{\N}:u_i\leq u}\delta_{w_i^*}$ of trajectories with label at most $u$. We will not formally define $\nu$ here for simplicity, and instead we just recall how to characterize the law of the restriction of $\omega$ to trajectories hitting a finite set $A$. For $A\subset\subset G$, each trajectory in the support of $\omega$ hit $A$ only a finite number of times a.s, and we denote by $X^1,X^2,\dots$ the unique trajectories in $W$ that hit $A$ at time $0$ for the first time, and whose projection on $W^*$ are exactly the trajectories modulo time-shift in the support of $\omega$ that hit $A$, ordered by increasing label. We also denote by $N_u^A$ the number of trajectories $w_i^*$ in $\omega$ hitting $A$ and with label $u_i\leq u$. In other words, $X^1,\dots,X^{N_u^A}$ correspond modulo time-shift to the trajectories in the support of $\omega_u$ which hit $A$ and started on their first hitting time of $A$. Their joint law can be described as follows: $N_u^A$ is a Poisson random variable with parameter $u\mathrm{cap}(A)$, and $(X^k)_{k\in{\N}}$ is an independent and i.i.d.\ family of random variables such that for each $k\in{\N}$
\begin{equation}
\label{eq:defRI}
\begin{gathered}
    X_0^k\text{ has law }\bar{e}_A\text{ and, conditionally on }X_0^k,
    \\((X^k_n)_{n\in{\N_0}},(X^k_{-n})_{n\in{\N_0}})\text{ has law }\P_{X_0^k}\otimes(\P_{X_0^k}(\cdot\,|\,\tilde{H}_A=\infty)).
\end{gathered}
\end{equation}
Note that $(X^k)_{k\in{\N}}$ depends implicitly on the choice of the set $A$, which should always be clear from context. The total time spent by random interlacements in $x\in{G}$ before, or strictly before, level $u>0$ are defined by
\begin{equation}
\label{eq:defLxu}
    L_{x,u}\stackrel{\textnormal{def.}}{=}\sum_{i\in{\N}:\,u_i\leq u}\sum_{n\in{\Z}}1\{w_i^*(n)=x\}\text{ and }  L_{x,u}^-\stackrel{\textnormal{def.}}{=}\sum_{i\in{\N}:\,u_i< u}\sum_{n\in{\Z}}1\{w_i^*(n)=x\}
\end{equation}
and the local times of random interlacements are then defined by
\begin{equation}
    \label{eq:localtimes}
    \ell_{x,u}\stackrel{\textnormal{def.}}{=}\frac{1}{\lambda_x}\sum_{k=1}^{L_{x,u}}\mathcal{E}_{x}^{k}\text{ and }\ell_{x,u}^-\stackrel{\textnormal{def.}}{=}\frac{1}{\lambda_x}\sum_{k=1}^{L_{x,u}^-}\mathcal{E}_{x}^{k},
\end{equation}
where $(\mathcal{E}_{x}^{k})_{x\in{G},k\in{\N_0}}$ is an i.i.d.\ family of Exp($1$) random variables. Note that the sums in $n\in{\Z}$ in \eqref{eq:defLxu} are well-defined since their value do not depend on the choice of the representative in the equivalent class of $w_i^*$, and that for any $A\subset\subset G$, $(\ell_{x,u})_{x\in{A}}$ have the same law as the total time in $x$, $x\in{A}$, spent by the trajectories $(X^k)_{1\leq k\leq N_u^A}$ from above \eqref{eq:defRI} but with additional exponential holding time in $x$ with parameter $\lambda_x$, $x\in{A}$. Moreover, one clearly has that $L_{x,u}=L_{x,u}^-$, and so $\ell_{x,u}=\ell_{x,u}^-$, occurs with probability one for any fixed choice of $u>0$, but we will sometimes take a random choice for $u$ so that there is a trajectory in $\omega$ with label $u$ hitting $x$   with positive probability, and then $L_{x,u}>L_{x,u}^-$. We refer to \eqref{eq:boundVuz} as to where we actually need to use the local times $\ell^-_{\cdot}$. We finally define the random interlacement set $\I^u$ as the set of points visited by a trajectory in $\omega_u$, that is $\I^u\stackrel{\textnormal{def.}}{=}\{x\in{G}:\ell_{x,u}>0\}$, and recall that the vacant set of interlacements is defined by $\V^u=(\I^u)^c$. Note that the characterization \eqref{eq:defIuintro} follows readily from the law of $N_u^A$. One can then introduce the critical parameter $u_*$ as the smallest parameter above which $\V^u$ does not percolate anymore, as well as 
\begin{equation}
\label{eq:defu**}
    u_{**}\stackrel{\textnormal{def.}}{=}\inf\big\{u>0:\,\limsup_{L\rightarrow\infty}\sup_{x\in{\Lambda(L)}}\P(B(x,L)\leftrightarrow B(x,2L)^c\text{ in }\mathcal{V}^u)= 0\big\}
\end{equation}
the critical parameter above which $\V^u$ is strongly non-percolating. Here we use the notation $A\leftrightarrow B$ in $C$, or sometimes $A\stackrel{C}{\longleftrightarrow}B$, to say that there is a connected path in $C$ starting in $A$ and ending in $B$. Note that the equality $u_*=u_{**}$ has been proved on $\Z^d$, $d\geq 3$ in \cite{sharpnessRI1,sharpnessRI2,sharpnessRI3}. On general graphs satisfying \eqref{eq:standingassumptionintro} one still knows that $u_*\leq u_{**}<\infty$  by \cite[Corollary~7.3]{DrePreRod2} and, under the additional condition (WSI) from \cite[p.12]{DrePreRod2}, we moreover have by \cite[Theorem~1.2]{DrePreRod2} that $u_{*}>0$.

\medskip

We finish this section with some useful bounds on the capacity of porous tubes. We roughly follow and simplify the strategy from \cite[Lemma~4.5]{GRS21} when $\nu=1$, from \cite[Proposition~2.11]{MuiSev} when $\nu<1$, although some additional work is required as we consider a discrete Gaussian field contrary to \cite{MuiSev}, and from \cite[Lemma~6.3]{DrePreRod2} when $\nu>1$. Let us first introduce the function
 \begin{equation}
\label{eq:defGnu}
    F_{\nu}(x,y)\stackrel{\textnormal{def.}}{=}\begin{cases}
    x^{\nu}&\text{ if }\nu<1,\\
    \frac{x}{1\vee \log(x/y)}&\text{ if }\nu=1,\\
    xy^{\nu-1}&\text{ if }\nu>1,
    \end{cases}
\end{equation}
and we abbreviate $F_{\nu}(x)=F_{\nu}(x,1)$. For clarity let us mention that the notation $\Cr{ccapunionball}=\Cr{ccapunionball}(\eta,\delta)$ used in the next lemma means that $\Cr{ccapunionball}$ is some constant that only depends only $\eta$ and $\delta$ (as well as implicitly the choice of the graph $\G$), and similarly for other constants in the rest of the article.

\begin{Lemme}
\label{lem:capatube}
    Fix $\eta,\delta\in{(0,1)}$. There exist constants $\Cl[c]{ccardAP}\in{(0,1)}$, $\Cl[c]{cconstantabovekappa}>0$ and $\Cl{CboundP}<\infty$, depending on $\eta$ and $\delta$, such that for all $x\in{G}$, $N\geq\Cr{CboundP}$, $\kappa>0$, integers $P\in{[\Cr{CboundP},N]}$, $A\subset\{1,\dots,P\}$ with $n=|A|$ satisfying $\Cr{ccardAP} P\leq n\leq P$, sets $S_i$ with $\mathrm{cap}(S_i)\geq \kappa$ for each $i\in{A}$, we have if $d(S_i,S_j)\geq (|i-j|-\delta)N/P$ for all $i\neq j\in{A}$,
    \begin{equation}
    \label{eq:capunionbound}
         \mathrm{cap}\left(\bigcup_{i\in{A}} S_i\right)\geq \left(\frac{1}{\Cr{cconstantabovekappa}\kappa P}+\frac{(1+\eta)}{\Cl[c]{ccapunionball}F_{\nu}(N,N/P)}\right)^{-1},
    \end{equation}
where $\Cr{ccapunionball}=\Cr{cbeta}$ if $\nu\leq 1$ and $\Cr{ccapunionball}=\Cr{ccapunionball}(\eta,\delta)>0$ if $\nu>1$;
 whereas if $\delta|i-j|-1/\delta\leq d(y,z)\leq (|i-j|+1/\delta)N/P$ for all $i,j\in{A}$, $y\in{S_i}$ and $z\in{S_j}$, 
 \begin{equation}
    \label{eq:capunionbound2}
        \mathrm{cap}\left(\bigcup_{i\in{A}} S_i\right)\leq \Cl{Ccapunionball}(1+\eta)F_{\nu}(N,N/P),
    \end{equation}
where $\Cr{Ccapunionball}=\Cr{Cbeta}$ if $\nu\leq 1$ and $\Cr{Ccapunionball}=\Cr{Ccapunionball}(\eta,\delta)<\infty$ if $\nu>1$.
\end{Lemme}
\begin{proof}
Abbreviate $S=\cup_{i\in{A}}S_i$, and first assume that $\nu<1$. As explained in \cite[Section~1.2.1]{MuiSev} it follows from \cite[Section II.3.13
p.163–164 and Appendix p.399–400]{landkof1972foundations} that the minimum over all probability measure $\mu$ on $[0,1]$ of $\int_0^1\int_0^1|a-b|^{-\nu}\mu(\mathrm{d}a)\mu(\mathrm{d}b)$ is reached at some probability measure with continuous density $h$ on $(0,1)$  with respect to the Lebesgue measure (in fact $h_{\nu}(a)=\Cl{chnu}(a(1-a))^{(\nu-1)/2}$ for some normalization constant $\Cr{chnu}$ but we will not need this fact), and the value of this minimum is $(\Cr{CGreenasymp}\Cr{cbeta})^{-1}=(\Cr{cGreenasymp}\Cr{Cbeta})^{-1}$, see \eqref{eq:cbetanointro}. Since the function $h_{\nu}$ is unbounded on the boundary of the interval $[0,1]$, we need to modify it slightly close to zero and one. More precisely, by a change of variable, one can fix $s=s(\eta)>0$ such that abbreviating $h_{\nu,s}(a)=h_{\nu}((1-2s)a+s)$,
\begin{equation}
\label{eq:riemann0}
\int_0^1\int_0^1|a-b|^{-\nu}h_{\nu,s}(a)h_{\nu,s}(b)\mathrm{d}a\mathrm{d}b\leq \frac{1+\eta}{\Cr{CGreenasymp}\Cr{cbeta}}\text{ and }
\int_0^1h_{\nu,s}(a)\mathrm{d}a\leq 1+\eta.
\end{equation}
Moreover, if $n$ is large enough, then since $h_{\nu,s}$ is bounded on $[0,1]$,
\begin{equation}
\label{eq:riemann2}
\begin{split}
\frac1{n^2}\sum_{\substack{i\neq j=1\\|i-j|< tn}}^n|i-j|^{-\nu}h_{\nu,s}\left(\frac{i}{n}\right)h_{\nu,s}\left(\frac{j}{n}\right)&\leq \frac{4}{1-\nu}\sup_{a\in{[0,1]}}h_{\nu,s}(a)^2t^{1-\nu}n^{-\nu}
\leq  \frac{\eta(1-\delta)^{\nu}}{\Cr{CGreen}\Cr{cbeta}n^{\nu}},
\end{split}
\end{equation}
where the last inequality hold upon fixing $t=t(\eta,s,\delta)$ small enough. Since the function $(a,b)\mapsto|a-b|^{-\nu}h_{\nu,s}(a)h_{\nu,s}(b)1_{\{|a-b|\geq t\}}$ is continuous and bounded, it follows from \eqref{eq:riemann0} that if $n$ is large enough,
\begin{equation}
\label{eq:riemann1}
 \frac1{n^2}\sum_{\substack{i,j=1\\|i-j|\geq tn}}^n|i-j|^{-\nu}h_{\nu,s}\left(\frac{i}{n}\right)h_{\nu,s}\left(\frac{j}{n}\right)\leq \frac{(1+\eta)^2}{\Cr{CGreenasymp}\Cr{cbeta}n^{\nu}}.
\end{equation}
Let us denote by $\tilde{S}_i$, $1\leq i \leq n$ an enumeration of the sets $S_i$, $i\in{A}$, which preserves the initial ordering, and let us define the measure 
$$\mu_y\stackrel{\textnormal{def.}}{=}\frac{1}{n}\sum_{i=1}^{n}h_{\nu,s}\left(\frac{i}{n}\right)\times\bar{e}_{\tilde{S}_i}(y)1\{y\in{\tilde{S}_i}\}\text{ for all }y\in{G}$$
 then by \eqref{eq:riemann0} we have that $\mu(G)\leq(1+\eta)^2$ if $n$ is large enough. For each $1\leq i\neq j\leq n$, $y\in{\tilde{S}_i}$ and $z\in{\tilde{S}_j}$, we have $d(y,z)\geq (|i-j|-\delta)N/P\geq (1-\delta)|i-j|N/P$, and in view of \eqref{eq:intro_Green_asymp} if additionally $|i-j|\geq tn$  we have $g(y,z)\leq (1+\eta)^2\Cr{CGreenasymp}(|i-j|N/P)^{-\nu}$ if $n$ is large enough (depending on $\eta$, $\delta$). Combining this with \eqref{eq:intro_Green} we obtain that if $n$ is large enough then
\begin{equation}
\label{eq:sumgmu}
\begin{split}
\sum_{y,z\in{S}}g(y,z)\mu_y\mu_z&\leq (1+\eta)^2\Cr{CGreenasymp}\Big(\frac{N}{P}\Big)^{-\nu}\frac1{n^2}\sum_{\substack{i,j=1\\|i-j|\geq tn}}^{n}|i-j|^{-\nu}h_{\nu,s}\left(\frac{i}{n}\right)h_{\nu,s}\left(\frac{j}{n}\right)
\\&+\Cr{CGreen}\Big(\frac{(1-\delta)N}{P}\Big)^{-\nu}\frac1{n^2}\sum_{\substack{i,j=1\\0<|i-j|< tn}}^{n}|i-j|^{-\nu}h_{\nu,s}\left(\frac{i}{n}\right)h_{\nu,s}\left(\frac{j}{n}\right)
\\&+\frac{1}{n^2}\sum_{i=1}^{n}h_{\nu,s}\left(\frac{i}{n}\right)^2\sum_{x,y\in{S_i}}\bar{e}_{S_i}(x)\bar{e}_{S_i}(y)g(x,y).
\end{split}
\end{equation}
In the previous equation, if $n$ is large enough, then by \eqref{eq:riemann2}, \eqref{eq:riemann1}, the sum of the two first lines can be bounded by $(1+\eta)^5P^{\nu}/(\Cr{cbeta}N^{\nu}n^{\nu})\leq (1+\eta)^6/(\Cr{cbeta}N^{\nu})$ if $n\geq \Cr{ccardAP} P$ for some constant $\Cr{ccardAP}=\Cr{ccardAP}(\eta,\delta)\in{(0,1)}$ large enough. Since $n\geq \Cr{ccardAP}P$, $h_{\nu,s}$ is bounded on $[0,1]$ and the minimum in \eqref{variational} for $A=S_i$ is reached at $\bar{e}_{S_i}$, the last line of \eqref{eq:sumgmu} is smaller than $1/(\kappa\Cr{cconstantabovekappa}P)$ for some positive constant $\Cr{cconstantabovekappa}$, and we obtain \eqref{eq:capunionbound} by \eqref{variational} up to a change of variable in $\eta$ (note also that $n$ can indeed be taken large by choosing $\Cr{CboundP}$ large enough since $n\geq \Cr{ccardAP}P\geq \Cr{ccardAP}\Cr{CboundP}$). 

We now turn to the proof of the upper bound \eqref{eq:capunionbound2}, still in the case $\nu<1$, and note that without loss of generality one can assume that $n=P$ and $A=\{1,\dots,P\}$ by monotonicity of the capacity, and that we can take $\kappa=0$ since it does not appear in \eqref{eq:capunionbound2}. We proceed by contradiction, that is we assume that for all $M\geq 1$ there exist $x_M\in{G}$, $N_M\geq M$, an integer $P_M\in{[M,N]}$, and a sequence $(S_i^{(M)})_{1\leq i\leq P^{(M)}}$ of sets such that $\delta|i-j|-1/\delta\leq d(y,z)\leq (|i-j|+1/\delta)N_M/P_M$ for all $1\leq i\neq j\leq P$, $y\in{S_i^{(M)}}$ and $z\in{S_j^{(M)}}$, not satisfying \eqref{eq:capunionbound2}.  By \eqref{variational}, there exists a probability measure $\mu^{(M)}$ on $\cup_{i=1}^{N_M}S_i^{(M)}$ so that 
\begin{equation}
\label{eq:sumSiSj}
    \sum_{i,j=1}^{P_M}\sum_{y\in{S_i^{(M)}}}\sum_{z\in{S_j^{(M)}}}g(y,z)\mu^{(M)}(y)\mu^{(M)}(z)\leq \frac{1}{\Cr{Cbeta}(1+\eta)N_M^{\nu}}.
\end{equation}
Let $\tilde{\mu}^{(M)}(a)=\sum_{i=1}^{P_M}(\sum_{y\in{S_i^{(M)}}}\mu^{(M)}(y))\delta_{i/P_M}(a)$, which is a probability measure of $[0,1]$, and hence converges weakly along a subsequence as $M\rightarrow\infty$ to some probability measure $\tilde{\mu}^{(\infty)}$. By monotone convergence, there exists $t=t(\eta)\in{(0,1)}$ such that
\begin{equation}
\label{eq:splittingforupper}
\begin{split}
    \int_0^1\int_0^1|a-b|^{-\nu}1_{|a-b|>t}\tilde{\mu}^{(\infty)}(\mathrm{d}a)\tilde{\mu}^{(\infty)}(\mathrm{d}b)&\geq     \int_0^1\int_0^1|a-b|^{-\nu}\tilde{\mu}^{(\infty)}(\mathrm{d}a)\tilde{\mu}^{(\infty)}(\mathrm{d}b)-\frac{\eta/5}{\Cr{cGreenasymp}\Cr{Cbeta}}
    \\&\geq \frac{1-\eta/5}{\Cr{cGreenasymp}\Cr{Cbeta}},
\end{split}
\end{equation}
where the last inequality follows from the discussion above \eqref{eq:riemann0}. The condition $\delta|i-j|-1/\delta\leq d(S_i^{(M)},S_j^{(M)})$ for all $i,j$ ensures that by \eqref{eq:intro_Green_asymp} if $M=M(\eta,\delta)$ is large enough, we have $g(y,z)\geq \Cr{cGreenasymp}(1-\eta/5)((|i-j|+1/\delta)N_M/P_M)^{-\nu}\geq \Cr{cGreenasymp}(1-\eta/5)^2((|i-j|N_M/P_M)^{-\nu}$ for all $1\leq i,j\leq P_M$ with $|i-j|\geq tP_M$, $y\in{S_i^{(M)}}$ and $z\in{S_j^{(M)}}$. Therefore, \eqref{eq:sumSiSj} implies that  
\begin{equation}
\label{eq:summutildeM}
        {\Cr{cGreenasymp}(1-\eta/5)^2}\sum_{\substack{i,j=1\\|i-j|\geq tP_M}}^{P_M}\Big|\frac{i}{P_M}-\frac{j}{P_M}\Big|^{-\nu}\tilde{\mu}^{(M)}\left(\left\{\frac{i}{P_M}\right\}\right)\tilde{\mu}^{(M)}\left(\left\{\frac{j}{P_M}\right\}\right)\leq \frac{1}{\Cr{Cbeta}(1+\eta)}.
\end{equation}
Since the function $|a-b|^{-\nu}1\{|a-b|>t\}$ is bounded and continuous on $[0,1]^2$, the sum on the left-hand side of \eqref{eq:summutildeM} converges along a subsequence as $M\rightarrow\infty$ to the left-hand side of \eqref{eq:splittingforupper}, which is a contradiction by the inequality $(1-\eta/5)^3\geq 1/(1+\eta)$.

Assume now that $\nu=1$. The proof of \eqref{eq:capunionbound} is similar to the case $\nu<1$, except one now takes $h_{1}(x)=1_{\{x\in{[0,1]}\}}$, $s=0$, and \eqref{eq:riemann1} is replaced by
\begin{equation*}
\frac1{n^2}\sum_{i\neq j=1}^n|i-j|^{-1}h_{1}\left(\frac{i}{n}\right)h_{1}\left(\frac{j}{n}\right)\leq \frac{2}{n^2}\sum_{i=1}^n\sum_{p=1}^np^{-1}\leq\frac{(1+\eta)\log(n)}{\Cr{CGreenasymp}\Cr{cbeta}n}
\end{equation*}
if $n$ is large enough, where we used that $\Cr{CGreenasymp}\Cr{cbeta}=1/2$ when $\nu=1$, see \eqref{eq:cbetanointro}. The bound \eqref{eq:riemann2} is replaced by
\begin{equation*}
\frac1{n^2}\sum_{\substack{i\neq j=1\\|i-j|\leq n^t}}^n|i-j|^{-1}h_{1}\left(\frac{i}{n}\right)h_{1}\left(\frac{j}{n}\right)\leq \frac{2}{n^2}\sum_{i=1}^{n}\sum_{p=1}^{n^t}p^{-1}\leq\frac{\eta(1-\delta)\log(n)}{\Cr{cbeta}\Cr{CGreen}n},
\end{equation*}
where the last inequality hold for $t=t(\eta,\delta)$ small enough. Moreover, \eqref{eq:sumgmu} still holds when replacing $s$ by $0$ and $tn$ by $n^t$ and if $n$ is large enough, then the first two lines of \eqref{eq:sumgmu} can now be upper bounded by $(1+\eta)^6\log(P)/(\Cr{cbeta}N)$ when $\Cr{ccardAP}\in{(0,1)}$ is large enough, and the third line of \eqref{eq:sumgmu} is still bounded by $1/(\kappa\Cr{cconstantabovekappa}P)$ for some positive constant $\Cr{cconstantabovekappa}$, and we can conclude by \eqref{variational}. For \eqref{eq:capunionbound2} we can take $n=P$ 
by monotonicity of the capacity. Let us fix some $z_i\in{S_i}$ for each $1\leq i\leq P$ and abbreviate $S^-=\cup_{i=P^{1-\eta}}^{P-P^{1-\eta}}S_i$. Since $e_{S}(y)\leq e_{S_i}(y)$ for each $y\in{S}$ by \eqref{eq:defequicap}, we have by \eqref{eq:capball}
\begin{equation}
\label{eq:boundcapSnu=1}
\begin{split}
    \mathrm{cap}(S)&\leq 2\Cr{Ccapball}P^{1-\eta}N/(P\delta)+\sum_{y\in{S^-}}e_{S}(y)
    \\&\leq 2\Cr{Ccapball}P^{-\eta}N/\delta+P\sup_{y\in{S^-}}\left(\sum_{i=1}^Pg(y,z_i)\right)^{-1},
\end{split}
\end{equation}
where the last inequality can be easily proven by summing \eqref{entrancegreenequi} with $A=S$ and $x=z_i$ therein over $i\in{\{1,\dots,P\}}$. Moreover, for each $y\in{S^-}$ since $\delta P^{\eta}-1/\delta\leq d(y,z_{i+P-P^{1-\eta}})\leq i(1+\eta)N/P$ and $\delta P^{\eta}-1/\delta\leq d(y,z_{P^{1-\eta}-i})\leq i(1+\eta)N/P$ if $P^{\eta}\leq i\leq P^{1-\eta}$ and $P$ is large enough (depending on $\delta$ and $\eta$), we have by \eqref{eq:intro_Green_asymp} 
\begin{equation*}
\sum_{i=1}^Pg(y,z_i)\geq \frac{2(1-\eta)(1+\eta)^{-\nu}\Cr{cGreenasymp}P}{N}\sum_{k=P^{\eta}}^{P^{1-\eta}}k^{-1}\geq\frac{(1-2\eta)(1-\eta)(1+\eta)^{-\nu}P\log(P)}{\Cr{Cbeta}N}
\end{equation*}
for $P$ large enough, where we used that $\Cr{cGreenasymp}\Cr{Cbeta}=1/2$. Moreover, on the second line of \eqref{eq:boundcapSnu=1} the first term is smaller than $\eta$ times the second term, up to taking $P\geq \Cr{CboundP}$ for some constant $\Cr{CboundP}=\Cr{CboundP}(\eta,\delta)$ large enough, and we can easily conclude after a change of variable for $\eta$.

Finally, let us assume that $\nu>1$. We simply now take $\mu_y=\tfrac1n\sum_{i=1}^n\bar{e}_{S_i}(y)$ for all $y\in{G}$ and then
\begin{equation*}
\begin{split}
    \sum_{y,z\in{G}}\mu_y\mu_zg(y,z)&\leq \frac1{n^2}\sum_{i=1}^n\sum_{y,z\in{S_i}}\bar{e}_{S_i}(y)\bar{e}_{S_i}(z)g(y,z)+\frac{\Cr{CGreen}}{n^2}\sum_{\substack{i,j=1\\i\neq j}}^n(|i-j|-\delta)^{-\nu}\left(\frac{N}{P}\right)^{-\nu}
    \\&\leq \frac{1}{\kappa n}+\frac{\Cr{CGreen}2^{\nu+1}P^{\nu}}{n^2N^{\nu}(1-\delta)^{\nu}}\sum_{i=1}^n\sum_{p=1}^np^{-\nu},
\end{split}
\end{equation*}
and since $\sum_{i=1}^{\infty}p^{-\nu}<\infty$ and $n\geq \Cr{ccardAP}P$, we easily obtain \eqref{eq:capunionbound} by \eqref{variational} for any choice of $\Cr{ccardAP}\in{(0,1)}$. For the upper bound \eqref{eq:capunionbound2}, one simply notes that $\mathrm{cap}(S_i)\leq \Cr{Ccapball}(N/(\delta P))^{\nu}$ for each $1\leq i\leq P$ by \eqref{eq:capball}, and we can conclude by subadditivity of capacity.
\end{proof}

\begin{Rk}
\label{rk:capRnp}
Lemma~\ref{lem:capatube} directly implies bounds on the capacity of the set  $\mathcal{R}_{N,p}$ from \eqref{eq:RNp} on $\Z^d$, $d\geq3$. Indeed, $\mathcal{R}_{N,p}$ is included in the union of the balls $B(kpe_1,p\sqrt{d})$, $0\leq k\leq \lceil N/p\rceil$, where $e_1=(1,0,\dots)$, and contains the union of the balls $B(kpe_1,p/4)$, $1\leq k\leq \lfloor N/p\rfloor-1.$ Therefore using Lemma~\ref{lem:capatube} with $P=\lceil N/p\rceil+1$, $\delta=(2\sqrt{d})^{-1}$, replacing $N$ by $p\lceil N/p\rceil+p$ for the upper bound, and with $P=\lfloor N/p\rfloor-1$, $\delta=1/2$, replacing $N$ by $p\lfloor N/p\rfloor-p$ for the lower bound, one obtains by monotonicity of capacity, \eqref{eq:assumptionZd} and \eqref{eq:capball} that for any $\eta\in{(0,1)}$, if $N\geq Cp$ for some $C=C(\eta)$, then
    \begin{equation}
    \label{eq:capRnp}
    \begin{split}
         \frac{(1-\eta)\pi N}{3\log(N/p)}\leq& \mathrm{cap}(\mathcal{R}_{N,p})\leq \frac{(1+\eta)\pi N}{3\log(N/p)}\text{ if }G=\Z^3\text{ and}
         \\cNp^{d-3}\leq &\mathrm{cap}(\mathcal{R}_{N,p})\leq CNp^{d-3}\text{ if }G=\Z^d,d\geq4.
         \end{split}
    \end{equation}
Note that \eqref{eq:capRnp} for $d\geq4$ is still valid when $p\leq N\leq Cp$ by \eqref{eq:capball}, subbaditivity and monotonicity of capacity, as well as when $d=3$ up to replacing $(1-\eta)$ and $(1+\eta)$ by some constants $c>0$ and $C<\infty$, and $\log(N/p)$ by $\log(N/p)\vee1$. The bounds in \eqref{eq:capRnp} were essentially already proved in \cite[Lemma~2.1]{dembo2022capacity} when $d=3$, or in \cite[Lemma~2.2]{GRS21} when $p=1$.
\end{Rk}

\section{First passage percolation}
\label{sec:FPP}
This section contains our main first passage percolation result, Theorem~\ref{the:FPP}, from which the upper bounds in Theorem~\ref{the:limitVuintro2} and~\eqref{eq:boundonPEintro} will eventually be deduced. Our main tool will be Lemma~\ref{lem:coarsegraining} below, which already contains the correct large deviation bound if one is only interested in the probability that the FPP distance is positive, see Remark~\ref{rk:decayVu},\ref{rk:easierproofc1=0}). We then show in Proposition~\ref{pro:apriori} that the FPP distance grows linearly with high (but non-explicit) probability. Combining these two results one can then easily deduce that the FPP distance grows linearly with the correct large deviation probability, see Section~\ref{sec:prooftheFPP}. 

We consider a general setup in this section for our choice of weights in the definition of the FPP distance which will let us treat all applications of interest at once without making the proof much more difficult, and which essentially includes the setup from \cite{AndPre} when considering random interlacements therein. Let us endow $[0,\infty)^K,$ $K\subset G,$ with the partial order $L=(L_x)_{x\in{K}}\leq L'=(L'_x)_{x\in{K}}$ if and only if $L_x\leq L'_x$ for each $x\in{K},$ and we will always implicitly refer to the monotonicity of functions $g:[0,\infty)^{K}\rightarrow[0,\infty)$ with respect to this partial order. Recalling the definition of $\Lambda(R)$ from \eqref{eq:defLambda}, let us fix $R\geq1$ and a family of weights
\begin{equation}
\label{eq:defbft}
\begin{gathered}
    \mathbf{t}=(t_z^u)_{z\in{\Lambda(R)},u>0}\text{ such that }t_z^u\stackrel{\textnormal{def.}}{=}g_z((\ell_{x,u})_{x\in{B(z,3R)}})\text{ for some measurable}
    \\\text{function }g_z:[0,\infty)^{B(z,3R)}\rightarrow[0,\infty)\text{ for each }z\in{\Lambda(R)}\text{ and }u>0,
    \\\text{and }(g_z)_{z\in{\Lambda(R)}}\text{ are either all increasing or all decreasing.}
\end{gathered}
\end{equation}
We then define the FPP distance $d_{\mathbf{t},u}(x,y)$ associated to $\mathbf{t}$ at level $u$ as 
\begin{equation}
\label{eq:defdistance}
    d_{\mathbf{t},u}(x,y)\stackrel{\textnormal{def.}}{=}\inf_{\gamma\subset\Lambda(R):x\leftrightarrow y}\sum_{z\in{\gamma}}t_z^u\text{ for all }x,y\in{\Lambda(R)}
\end{equation}
where the infimum is taken over all $R$-nearest neighbor paths $\gamma\subset \Lambda(R)$ from $x$ to $y.$ Note that the dependency of $d_{\mathbf{t},u}$ on $R$ is implicit, but is in practice always clear from the choice of $\mathbf{t}$. When $R=1$, $d_{\mathbf{t},u}$ is simply the usual FPP distance on $G$ associated to the weights $(t_z^u)_{z\in{G}}$, whereas if $R>1$ it is an FPP distance on the renormalized graph $\Lambda(R)$. Let us further define for each $x\in{\Lambda(R)}$ and $N\geq L\geq0$
\begin{equation}
\label{eq:defdzL}
d_{\mathbf{t},u}(x;L,N)\stackrel{\textnormal{def.}}{=}\inf_{\substack{y,z\in{\Lambda(R)}\\y\in{B(x,L)},z\in{B(x,N)^c}}}d_{\mathbf{t},u}(y,z)
\end{equation}
the length of the shortest path between $B(x,L)$ and $B(x,N)^c.$ We also abbreviate $d_{\mathbf{t},u}(x;N)=d_{\mathbf{t},u}(x;0,N)$ the length of the shortest path between $x$ and $B(x,N)^c$. In particular, the distance $d_u(N)$ from \eqref{eq:defduN} is equal to $d_{\mathbf{t},u}(0;N)$ when $R=1$ and $g_z(\ell)=1\{\ell>0\}$ in \eqref{eq:defbft}.

Moreover, throughout this section we will always assume that the levels $u,v>0$ are chosen for some $\Theta>1$ in the set
\begin{equation}
\label{eq:defcalL}
\mathcal{L}_{\Theta}\stackrel{\textnormal{def.}}{=}\begin{cases}
\{(u,v)\in{(0,\infty)}:u<v<\Theta u\}&\text{if the functions $g_z$ in \eqref{eq:defbft} are increasing,}
\\\{(u,v)\in{(0,\infty)}:v<u< \Theta v\}&\text{ otherwise}.
\end{cases}
\end{equation}
We abbreviate $\mathcal{L}=\bigcup_{\Theta=1}^{\infty}\mathcal{L}_{\Theta}$, and recall the function $F_{\nu}$ from \eqref{eq:defGnu} and below.

\begin{The}
\label{the:FPP}
Fix $\eta,\xi,\zeta>0$ and $\Theta>1$. There exist constants $\Cl[c]{ccondFPP}>0$, depending only on $\eta$, $\xi$ and $\Theta$, as well as $\Cl{cboundN}<\infty$ and $\Cl[c]{cFPP}>0,$ depending on $\eta,\xi,\zeta$ and $\Theta$, such that for all $s>0,$ $(u,v)\in{\mathcal{L}_{\Theta}}$ with $\xi\sqrt{u}\leq |\sqrt{v}-\sqrt{u}|$,  $R\in{[1,\zeta |\sqrt{v}-\sqrt{u}|^{-\frac2\nu}]}$, families of weights $\mathbf{t}$ as in \eqref{eq:defbft}  satisfying
\begin{equation}
\label{eq:assumptiontzu}
    \P(t_z^u\leq s)\leq \Cr{ccondFPP}(1\wedge u^{\frac1\nu}R)^\alpha\text{ for all }z\in{\Lambda(R)},
\end{equation}
and for all $x\in{\Lambda(R)}$ and $N\geq R$, we have
\begin{equation}
\label{eq:bounddtvfinal}
    \P\left(d_{\mathbf{t},v}(x;N)\leq \frac{\Cr{cFPP}sN}{u^{-\frac1\nu}\vee R}\right)\leq 
    \Cr{cboundN}\exp\left(-\frac{\Cl[c]{cFPP2}}{1+\eta}F_{\nu}\big(N|\sqrt{v}-\sqrt{u}|^{\frac2\nu}\big)\right),
\end{equation}
where $\Cr{cFPP2}=\Cr{cbeta}$ if $\nu\leq 1$ and $\Cr{cFPP2}=\Cr{cFPP2}(\xi,\zeta,\Theta)<\infty$ if $\nu>1$.
\end{The}

 When $\nu>1$, one can in fact remove the factor $1+\eta$ in \eqref{eq:bounddtvfinal} up to changing the constant $\Cr{cFPP2}$. The interest of the variable $\eta$ comes from the exact value of the constants $\Cr{cFPP2}=\Cr{cbeta}$ when $\nu\leq 1$, as defined in \eqref{eq:cbetanointro}. The only reason we keep the variable $\eta$ when $\nu>1$ is to be able to state \eqref{eq:bounddtvfinal} for any $\nu>0$ in a more compact form (compared for instance to \eqref{eq:boundonPEintro}), and we will actually keep using this more compact form throughout the article. We stress that the fact that the constant $\Cr{ccondFPP}$ appearing in \eqref{eq:assumptiontzu} does not depend on the choice of $\zeta$ will actually be essential in the proof of Theorems~\ref{the:localuniquenessintro}, \ref{the:boundcapacityintro} and \ref{the:limitVuintro2}. Some of the conditions of Theorem~\ref{the:FPP}  might seem superfluous at first read, such as the bound $R\leq \zeta|\sqrt{v}-\sqrt{u}|^{-\frac2\nu}$, but a condition of this type is in fact necessary, whereas it would be possible to remove the condition $\xi\sqrt{u}\leq|\sqrt{v}-\sqrt{u}|$  at the cost of additional logarithmic factors in \eqref{eq:bounddtvfinal}, see  Remark~\ref{rk:endthm},\ref{rk:removingdependenceonzeta}) for details. Note however that in \eqref{eq:boundonPEintro}, which is the main application of Theorem~\ref{the:FPP}, one will eventually take $u=\eta v,$ and so the condition $\xi\sqrt{u}\leq |\sqrt{v}-\sqrt{u}|$ is automatically satisfied when taking $\xi=(1-\sqrt{\eta})/\sqrt{\eta}$.

The condition \eqref{eq:assumptiontzu} may also seem a priori surprising, as one may expect that a condition of the type $\P(t_z^u\leq s)\leq \Cr{ccondFPP}$, for some small constant $\Cr{ccondFPP}$ not depending on $u$ and $R$, is enough to start our renormalization procedure (see \eqref{eq:bounda0} below as to where exactly condition \eqref{eq:assumptiontzu} is used in the proof). It turns out that this is not the case, as the following counterexample shows: take $s=0$, $R=1$ and $g_z(\ell_{z,u})=1\{\ell_{z,u}=0\}$, $z\in{G}$. Then $d_{t,v}(x;N)=0$ if $x\leftrightarrow B(x,N)^c$ in $\mathcal{I}^v$, which always happens when $x\in{\mathcal{I}^v}$. One can therefore lower bound the probability in \eqref{eq:bounddtvfinal} by $1-\exp(-cv)$ in view of \eqref{eq:defIuintro} and \eqref{eq:intro_Green}. In particular, \eqref{eq:bounddtvfinal} is not satisfied when $(\sqrt{v}-\sqrt{u})^2N^{\nu}$ is large enough. Note that $\P(t_z^u\leq s)=1-\exp(-u/g(z,z))\leq \Cr{ccondFPP}$ is indeed satisfied uniformly in $z\in{G}$ for $u$ small enough by \eqref{eq:intro_Green}, but that \eqref{eq:assumptiontzu} is not satisfied by \eqref{eq:condnualpha}.

Let us now explain the ideas behind the proof of Theorem~\ref{the:FPP}, which uses techniques introduced in \cite{MR3602841,GRS21}, and we focus on the case $R\approx u^{-\frac1\nu}$ as in \eqref{eq:boundonPEintro} for simplicity. The main idea is to notice that $d_{\mathbf{t},v}(x;N)\geq cc'N/R$ if any $L$-nearest neighbor path $\calC\subset\Lambda(L)$ from $x$ to $B(x,N)^c$ contains at least $c'N/L$ vertices $y\in{\calC}$ such that $d_{\mathbf{t},v}(y;L;2L)\geq cL/R$, for some constants $c,c'>0$, and then to show that the probability of this last event is larger than one minus the right-hand side of \eqref{eq:bounddtvfinal}, where $L\in[R,N]$ will be an intermediate scale suitably chosen, see \eqref{eq:choiceL}. The main difficulty to implement this strategy comes from the strong dependency between the events $d_{\mathbf{t},v}(y;L;2L)\geq cL/R$, $y\in{\calC}$, and in order to weaken this dependency we will actually assume that $d(y,y')\geq 16KL$ for any $y\neq y'\in{\calC}$, for some large enough constant $K$ suitably chosen. 

There will be two main steps to finish removing this dependency with high probability. One first bounds the total time (suitably weighted) spent by interlacements in $B(y,2L)$ for any $y\in{\calC}$, see Proposition~\ref{pro:thm4.2disco}, which is a generalization of \cite[Theorem~4.2]{MR3602841}, and which together with Lemma~\ref{lem:capatube} explains the form of the bound \eqref{eq:bounddtvfinal}. Once this total time spent by interlacements in $B(y,2L)$ for any $y\in{\calC}$ is fixed, one can replace the local times of interlacements in each of these boxes by independent copies of these local times, on some events which are independent in $y\in{\calC}$ and occur with high enough probability, at the cost of decreasing, or increasing depending on the monotonicity of the function $g_z$ in \eqref{eq:defbft}, the parameter $v$ by some sprinkling parameter, see Proposition~\ref{pro:softlocaltimes} which relies on the soft local times technique from \cite{MR3420516}. If $d_{\mathbf{t},u}(y;L;2L)\geq cL/R$ occurs with high probability and our choice of sprinkling parameter is essentially smaller than $|v-u|$, see \eqref{eq:defepsu'}, using the previous reasoning together with large deviations bounds on the sum of independent random variables, see \eqref{eq:boundonhatt}, one obtains for each path $\calC$ as before the desired bound on the probability that $d_{\mathbf{t},v}(y;L;2L)\geq cL/R$ for at least $c'N/L$ different $y\in{\calC}$.

A major obstacle to then finish the proof of Theorem~\ref{the:FPP} is that one needs to make sure that the entropy term coming from summing over all possible such paths $\calC$ from $x$ to $B(x,N)^c$ does not dominate the previous probability, that is the right-hand side of \eqref{eq:bounddtvfinal}, for a suitable choice of $L$. Precise bounds on this entropy term are derived in Proposition~\ref{pro:entropy}, which improves the bounds from \cite[Proposition~4.3]{GRS21} especially in the case $\nu\leq 1$, and we refer to Remark~\ref{rk:decayVu} for more details on its interest. It only remains to prove that $d_{\mathbf{t},u}(y;L;2L)\geq cL/R$ indeed occurs with high probability under condition \eqref{eq:assumptiontzu} (up to additional sprinkling), or in other words to prove that the FPP distance grows at least linearly with high probability for $L$ large enough (without the need for precise control on this probability contrary to \eqref{eq:bounddtvfinal}). This is done in Proposition~\ref{pro:apriori} using the perforated lattice renormalization scheme introduced in \cite[Section~2]{MR3650417}, and we refer to \cite[Remark~2.3,(iii)]{AndPre} for a description of the ideas behind this proof.

\subsection{Decoupling of interlacements and coarse-graining}
\label{sec:coarsegraining}

For each $L,K\geq 1$ and $z\in{\Lambda(L)}$ we abbreviate 
\begin{equation}
\label{eq:defCU}
    C_z^L\stackrel{\textnormal{def.}}{=}B(z,L),\ \tilde{C}_z^L\stackrel{\textnormal{def.}}{=}B(z,2L)\text{ and } \hat{C}_z^L\stackrel{\textnormal{def.}}{=}B(z,8L).
\end{equation}
We will write $C_z,$ $\tilde{C}_z$ and $\hat{C}_z$ instead of $C_z^L,$ $\tilde{C}_z^L$ and $\hat{C}_z^L$ whenever the choice of $L$ is clear from context.
Moreover, throughout the article we will consider for $L,K\geq1$
\begin{equation}
    \label{eq:defcalC}
    \mathcal{C}\subset\Lambda(L)\text{ a non-empty collection of sites with mutual distance at least }16KL,
\end{equation}
as well as
\begin{equation}
\label{eq:defSigma}
    \Sigma(\mathcal{C})\stackrel{\textnormal{def.}}{=}\bigcup_{z\in{\mathcal{C}}}\hat{C}_z^L.
\end{equation}
 For each $z\in{\Lambda(L)}$ and $u>0$ we also define 
\begin{equation}
\label{eq:defTz}
    T^z(u)\stackrel{\textnormal{def.}}{=}
    \inf\Big\{v>0:\sum_{x\in{\hat{C}_z}}\bar{e}_{\hat{C}_z}(x)\ell_{x,v}\geq u\Big\}.
\end{equation} 
Let us start with the following result, which is an adaptation of \cite[Theorem~4.2]{MR3602841} in our context and will be proved in Section~\ref{sec:proofpropthm4.2}.

\begin{Prop}
\label{pro:thm4.2disco}
Fix $\eta\in{(0,1)}.$ There exists $\Cl{cforK}=\Cr{cforK}(\eta)<\infty$ such that for all $K\geq \Cr{cforK},$ $L\geq1,$ $u,v>0$ and $\mathcal{C}$ as in \eqref{eq:defcalC}
\begin{equation}
\label{eq:thm4.2disco}
\begin{split}
    \P&\left(E_{\mathcal{C}}^{u,v}\right)
    \geq 1-\exp\left(-\frac{1}{1+\eta}\left(\sqrt{v}-\sqrt{u}\right)^2\mathrm{cap}(\Sigma(\mathcal{C}))\right),
\end{split}
\end{equation}
where
\begin{equation}
\label{eq:defeventE}
    E_{\mathcal{C}}^{u,v}\stackrel{\textnormal{def.}}{=}
    \begin{cases}
    \bigcup_{z\in{\mathcal{C}}}\left\{T^z(u)< v \right\}&\text{ if }u<v
    \\\bigcup_{z\in{\mathcal{C}}}\left\{T^z(u)> v\right\}&\text{ otherwise.}
    \end{cases}
\end{equation}
\end{Prop}

The next tool we need to prove Theorem~\ref{the:FPP} is the soft local times technique from \cite{MR3420516} to decouple the excursions of random interlacements on the sets $\hat{C}_z,$ $z\in{\Lambda(L)}.$ Recall the definition of $\ell^-_{\cdot}$ in \eqref{eq:localtimes}, and note that $(\ell^{-}_{x,T^z(u)})_{x\in{\hat{C}_z}}\neq (\ell_{x,T^z(u)})_{x\in{\hat{C}_z}}$ by \eqref{eq:defTz}. 

\begin{Prop}
\label{pro:softlocaltimes}
     Upon extending the underlying probability space, there exist $\Cl{CsoftlocaltimesK},\Cl{Csoftlocaltimes},\Cl[c]{csoftlocaltimesK}<\infty$ and $\Cl[c]{csoftlocaltimes}>0$ such that for all  $\eps\in{(0,1/2]}$, $L\geq1,$ $K\geq \Cr{csoftlocaltimesK}\eps^{-\Cr{CsoftlocaltimesK}}$ and $u_0>0$,  there exists a family of processes $(\hat{\ell}^z_{x,u})_{x\in{G},u>0}$, $z\in{G}$, each with the same marginals as $(\ell_{x,u})_{x\in{G},u>0}$, and  a family of events $(F_{z}^{u,\eps})_{u>0,z\in{G}}$, such that $((F_{z}^{u,\eps})_{u>0},\hat{\ell}^z_{\cdot})$, $z\in{\calC}$, are independent for each $\calC$ as in \eqref{eq:defcalC},  and for all $u\in{(0,u_0]}$ and $z\in{G}$
    \begin{equation}
    \label{eq:boundprobaFzueps}
        \P\left(F_{z}^{u,\eps}\right)\geq1-\Cr{Csoftlocaltimes}\exp(-\Cr{csoftlocaltimes}u\eps^2L^{\nu}),
    \end{equation}
    and
    \begin{equation}
    \label{eq:defeventF}
        F_{z}^{u,\eps}\subset
            \Big\{\hat{\ell}^z_{x,u(1-\eps)}\leq \ell_{x,T^z(u)}\text{ and }\ell_{x,T^z(u)}^-\leq \hat{\ell}^z_{x,u(1+\eps)}\text{ for all } x\in{\hat{C}_z^L}\Big\}.
    \end{equation}
\end{Prop}
Together with Proposition~\ref{pro:thm4.2disco}, Proposition~\ref{pro:softlocaltimes} lets us decouple the local times of random interlacements on $\hat{C}_z$, $z\in{\calC}$ as in \eqref{eq:defcalC}, with a probability which will be essentially dominated by the right-hand side of \eqref{eq:thm4.2disco}, see for instance \eqref{eq:boundprobaH}, and we refer to Section~\ref{sec:softlocaltimes} for a proof. It thus essentially plays the same role as \cite[Proposition~5.1]{MR3602841}, with two main differences: the bound \eqref{eq:boundprobaFzueps} is more precise than \cite[(5.15)]{MR3602841}, which will later be useful, see \eqref{eq:boundprobaH}, and instead of decoupling excursions of interlacements as in \cite{MR3602841} one directly decouples the soft local times at level $T^z(u)$. This simplifies the strategy from \cite{MR3602841} by avoiding to define different type of excursions depending on the event one considers, see \cite[(5.9)]{MR3602841} and above, allows us to write our "good" events without having to refer to excursions contrary to \cite[(3.11)-(3.13)]{MR3602841} (although considering excursions is still locally necessary for us in the proof of Proposition~\ref{pro:softlocaltimes}), and makes the parallel with the Gaussian free field proof from \cite{GRS21,MR3417515} more apparent, see Remark~\ref{rk:endthm},\ref{rk:extensiontoGFF}) for details. 

Note that if $\calC=\{z,z'\}$ with $d(z,z')\geq 16KL$, decoupling inequalities in the same spirit as Proposition~\ref{pro:softlocaltimes} have been proved in \cite[Theorem~2.4]{DrePreRod2}. The main difference between these two results is that in \cite[(2.21)]{DrePreRod2} there is an additional polynomial term in front of the exponential compared to \eqref{eq:boundprobaFzueps} (but $K$ now has to depend on $\eps$ for \eqref{eq:boundprobaFzueps} to be fulfilled), that is the bound \cite[(2.21)]{DrePreRod2} is only relevant when $u\eps^2L^{\nu}\geq c\log(L)$. The fact that the bound \eqref{eq:boundprobaFzueps} is already relevant when $u\eps^2L^{\nu}\geq C$ will actually be essential to obtain \eqref{eq:boundonPEintro} and Theorems~\ref{the:localuniquenessintro} and \ref{the:boundcapacityintro}, see Remark~\ref{rk:endthm},\ref{rk:removingdependenceonzeta}) as to why, and see Remark~\ref{rk:softlocaltimes},\ref{rk:othersoftlocaltimes}) for a version of Proposition~\ref{pro:softlocaltimes} in the spirit of \cite[Theorem~2.4]{DrePreRod2}. Another difference is that in \eqref{eq:defeventF} one considers interlacements at level $T^z(u)$, instead of $u$ in \cite[(2.21)]{DrePreRod2}, which is essential to ensure that the events $F_z^{u,\eps}$, $z\in{\calC}$, are independent, see Remark~\ref{rk:softlocaltimes},\ref{rk:roleTzu}). 

 Recalling the functions $g_{z'},$ $z'\in{\Lambda(L)},$ from \eqref{eq:defbft}, and the process $\hat{\ell}^z$ from Proposition~\ref{pro:softlocaltimes} let
\begin{equation}
\label{eq:defweighthat}
    \hat{t}_{z'}^{\,u,z}\stackrel{\textnormal{def.}}{=}g_{z'}\big((\hat{\ell}^z_{x,u})_{x\in{B(z',3R)}}\big)\text{ for all }z\in{G},z'\in{\Lambda(R)}\text{ and }u>0.
\end{equation}
We also denote by $\hat{d}_{\mathbf{t},u}^z(x,y)$ and $\hat{d}_{\mathbf{t},u}^z(x;L,2L)$ the FPP distances defined as in \eqref{eq:defdistance} and \eqref{eq:defdzL} but for $\hat{t}^{\,u,z}$ instead of $t^{u}.$ Note that the event $F_z^{u,\eps}$ from \eqref{eq:defeventF}, the weights $\hat{t}_{z'}^{u,z}$ from \eqref{eq:defweighthat} and the distance $\hat{d}_{\mathbf{t},u}^z(x,y)$ all depend on the choice of $u_0$ in Proposition~\ref{pro:softlocaltimes}, as the coupling between $\hat{\ell}^z$ and $\ell$ is only valid for $u\leq u_0$, which comes from our use of \cite[Proposition~4.4]{PreRodSou} in the proof. To simplify notation, we did not write this dependency explicitly, which shouldn't lead to any confusion as none of the constants depend on $u_0$, and so one can take $u_0$ arbitrarily large.

As explained below Theorem~\ref{the:FPP}, we will also need some control on the entropy coming from considering all the possible sets $\calC$ as in \eqref{eq:defcalC} that are possibly hit by a path from $x$ to $B(x,N)$. The following coarse-graining scheme for paths is an improved version of \cite[Proposition~4.3]{GRS21}, and will be proved in Section~\ref{sec:entropy}. Recall the constant $\Cr{cbeta}$ introduced in \eqref{eq:cbetanointro} and the function $F_{\nu}$ introduced in \eqref{eq:defGnu}.

\begin{Prop}
\label{pro:entropy}
    Fix $\eta\in{(0,1)}$. There exist $\Cl[c]{crho},\Cl[c]{CNKLp},\Cl[c]{ccapcalC}>0$,and $\Cl{CKL},\Cl{CNKL},\Cl{CcardA}<\infty$, depending only on $\eta$ and satisfying $\Cr{CNKLp}<\Cr{CNKL}$ if $\nu>1$,  such that for all $x\in{G}$, $\rho\in{(0,\Cr{crho}]}$, $L\geq1$, $K\geq2$, $N\geq \Cr{CKL}KL$ and $p\in{\mathbb{N}}$ such that
    \begin{equation}
    \label{eq:choicep}
        \begin{cases}
             \Cr{CNKLp}(N/(KL))^{\nu}\log(N/(KL))^{1-\nu}\leq p\leq \Cr{CNKL}N/(KL) &\text{ if }\nu<1,
            \\ \Cr{CNKLp}(N/(KL))\log\log(N/(KL))\log(N/(KL))^{-1}\leq p\leq \Cr{CNKL}N/(KL)&\text{ if }\nu=1,
            \\ \Cr{CNKLp}N/(KL)\leq p\leq \Cr{CNKL}N/(KL)&\text{ if }\nu>1,
        \end{cases}
    \end{equation}
    there exists a family $\mathcal{A}=\mathcal{A}_{x,\rho,p}^{L,K,N}$ of collections $\mathcal{C}$ as in \eqref{eq:defcalC} such that 
    \begin{align}
    \label{eq:cardC}
         &\text{all }\mathcal{C}\in{\mathcal{A}}\text{ have cardinality }p,
\\&\label{eq:capC}\begin{aligned}
         &\mathrm{cap}(\Sigma(\tilde{\mathcal{C}}))\geq F_{\nu}(N,KL)H_{\nu}(\tfrac{N}{KL},K,\eta)
         \text{ for all }\mathcal{C}\in{\mathcal{A}}\text{ and }\tilde{\mathcal{C}}\subset\mathcal{C}\text{ with }|\tilde{\mathcal{C}}|\geq p(1-\rho),
         \\&\text{where }H_{\nu}(x,y,\eta)\geq\Cr{ccapcalC}y^{-\nu}
         \text{ satisfies } \lim\limits_{x\rightarrow\infty}H_{\nu}(x,y,\eta)=\frac{\Cr{cbeta}}{1+\eta}\,\forall\,y\geq1,\eta>0\text{ and }\nu\leq 1,
         \end{aligned}
\\&\label{eq:gammaincludedC}   \begin{aligned}
        &\text{for any $L$-nearest neighbor path }\gamma\text{ from }C_x^L\text{ to }B(x,N-2L)^c,\,\exists\,\mathcal{C}\in{\mathcal{A}}\text{ s.t.\ }\mathcal{C}\subset\gamma,
    \end{aligned}
\\& \label{eq:cardA}      |\mathcal{A}|\leq \exp\left({\Cr{CcardA}p\log(K)}\right).
    \end{align}
\end{Prop}

Note that the entropy bound \eqref{eq:cardA} is better than the one from \cite[(4.13),(4.14)]{GRS21} on $\Z^3$ (for which $\nu=1$), or the one used below \cite[(5.3)]{MuiSev} on $\Z^d$, $d\geq2$, by a $\log$ factor. On $\Z^d$, $d\geq4$, it corresponds to the one from \cite[Proposition~4.3,ii)]{GRS21}, but Proposition~\ref{pro:entropy} has the advantage to provide a unified framework for any $d$, and in fact on the larger class of graphs that we consider here. We refer to the beginning of Section~\ref{sec:entropy} on how this is achieved. Moreover, it turns out that these improvements will simplify the implementation of our renormalization scheme, for instance by requiring weaker a priori bounds, see Remark~\ref{rk:decayVu},\ref{rk:noapriori}) for details.

Let us now explain how to combine the three previous propositions to obtain bounds on $d_{t,v}(x;N)$. For any $\eps\in{(0,1)}$ and $u>0$ we abbreviate
\begin{equation}
    \label{eq:choicevu}
    u_{\eps}\stackrel{\textnormal{def.}}{=}
    \begin{cases}
    \frac{u}{1-\eps}&\text{ if the functions $g_z$ in \eqref{eq:defbft} are increasing}
    \\\frac{u}{1+\eps}&\text{ otherwise},
    \end{cases}
\end{equation}
and recall the events $E_{\calC}^{u,v}$ from \eqref{eq:defeventE} and $F_{z}^{u,\eps}$ from \eqref{eq:defeventF}. The interest of the distances $\hat{d}_{\mathbf{t},u}^z$ from below \eqref{eq:defweighthat} is highlighted in the following result.

\begin{Lemme}
\label{lem:replacetbyhatt}
For each $L\geq R\geq1,$ $\mathbf{t}$ as in \eqref{eq:defbft}, $K\geq 100,$ $N\geq10KL$, $x\in{\Lambda(R)}$, $u_0>0$, $\eps\in{(0,1/2]},$ $(u,v)\in{\L\cap (0,u_0]}$ and $u_{\eps}$ as in \eqref{eq:choicevu} such that $(u_{\eps},v)\in{\mathcal{L}},$ $p$ satisfying \eqref{eq:choicep} and $\rho\in{(0,1)},$ letting $\mathcal{A}=\mathcal{A}_{x,\rho,p}^{L,K,N}$ from Proposition~\ref{pro:entropy},  on the event
\begin{equation}
\label{eq:goodevent}
    G^{u_{\eps},v,\eps}_{\mathcal{A}}\stackrel{\textnormal{def.}}{=}\bigcap_{\calC\in{\mathcal{A}}}\bigg(\bigcap_{\substack{\tilde{\calC}\subset \calC\\|\tilde{\calC}|\geq (1-\rho)|\calC|}}E_{\tilde{\calC}}^{u_{\eps},v}\cap \bigcap_{\substack{\tilde{\calC}\subset \calC\\|\tilde{\calC}|\geq \rho|\calC|}}\tilde{F}_{\tilde{\calC}}^{u_{\eps},\eps}\bigg),
\end{equation}
where
\begin{equation}
    \label{eq:defeventFhat}
    \tilde{F}^{u_{\eps},\eps}_{\tilde{\calC}}\stackrel{\textnormal{def.}}{=}\left\{\sum_{z\in{\tilde{\calC}}}1\{F_{z}^{u_{\eps},\eps}\}\geq \frac{|\tilde{\calC}|}{2}\right\},
\end{equation}
we have
\begin{equation}
\label{eq:replaceFPPbyiid}
    d_{\mathbf{t},v}(x;N)\geq\inf_{\mathcal{C}\in{\mathcal{A}}}\inf_{\substack{\tilde{\calC}\subset \calC\\|\tilde{\calC}|\geq \frac{\rho}{2}|\calC|}} \sum_{z\in{\tilde{\calC}}}\hat{d}_{\mathbf{t},u}^z(z;L;2L).
\end{equation}
\end{Lemme}
\begin{proof}
Let $\gamma$ be the $R$-nearest neighbor path minimizing the distance $d_{\mathbf{t},v}(x;N)$ from below \eqref{eq:defdzL}. The path $\gamma$ induces a  $L$-nearest neighbor path $\tilde{\gamma}$ starting in $B(x,L)$ and ending in $B(x,N-2L)^c$ such that for each $z\in{\tilde{\gamma}},$ $\gamma$ crosses from $C_z^L$ to $(\tilde{C}_z^L)^c$.   Let $\mathcal{C}\in{\mathcal{A}}$ be such that $\mathcal{C}\subset\tilde{\gamma},$ which exists by \eqref{eq:gammaincludedC}. First assume that $v<u$, then $v<u_{\eps}$ under our assumptions. Let $\tilde{\mathcal{C}}$ be the set of $z\in{\mathcal{C}}$ such that $T^z(u_{\eps})> v,$ then $E_{\calC\setminus \tilde{\calC}}^{u_{\eps},v}$ is not satisfied by \eqref{eq:defeventE}, and thus $|\tilde{\calC}|\geq \rho|\calC|$ on the event $G^{u_{\eps},v,\eps}_{\mathcal{A}}.$ Let also $\tilde{\calC}'\stackrel{\textnormal{def.}}{=}\{z\in{\tilde{\calC}}:F_{z}^{u_{\eps},\eps}\text{ occurs}\},$ then by \eqref{eq:goodevent} and \eqref{eq:defeventFhat} on the event $G^{u_{\eps},v,\eps}_{\mathcal{A}}$ we have $|\tilde{\calC}'|\geq |\tilde{\calC}|/2\geq (\rho/2)|\calC|.$  Moreover, by \eqref{eq:localtimes}, \eqref{eq:choicevu} and \eqref{eq:defeventF}, we have for all $z\in{\tilde{\mathcal{C}}'}$ and $x\in{\hat{C}_z}$
\begin{equation}
\label{eq:upperboundlocaltimes}
\ell_{x,v}\leq \ell_{x,T^z(u_{\eps})}^-\leq \hat{\ell}^z_{x,u_{\eps}(1+\eps)}=\hat{\ell}^z_{x,u}.
\end{equation}
Since $(g_z)_{z\in{\Lambda(L)}}$ are decreasing by \eqref{eq:defcalL}, and $d_{\mathbf{t},v}(x;L;2L)$ is measurable with respect to  $(\ell_{x,v})_{x\in{\hat{C}_z}}$ by \eqref{eq:defdzL}, \eqref{eq:defbft} and \eqref{eq:defCU} (which justifies our definition of $\hat{C}_z$), the inequality \eqref{eq:replaceFPPbyiid} follows readily from \eqref{eq:defbft},  \eqref{eq:defweighthat} and our choice of $\tilde{\gamma}$. 
If on the other hand $v>u$, then $v>u_{\eps},$ and letting $\tilde{\mathcal{C}}$ be the set of $z\in{\calC}$ such that $T^{z}(u_{\eps})< v$, we have $|\tilde{\calC}|\geq \rho|\calC|$ on the event $G^{u_{\eps},v,\eps}_{\mathcal{A}}$ similarly as before. Therefore, defining $\tilde{\calC}'$ similarly as before, we have $|\tilde{\calC}'|\geq (\rho/2)|\calC|$, and for all $z\in{\tilde{\mathcal{C}}'}$ and $x\in{\hat{C}_z}$ 
\begin{equation*}
    \ell_{x,v}\geq \ell_{x,T^z(u_{\eps})}\geq \hat{\ell}^z_{x,u_{\eps}(1-\eps)}= \hat{\ell}^z_{x,u},
\end{equation*}
and \eqref{eq:replaceFPPbyiid} follows readily since  $(g_z)_{z\in{\Lambda(L)}}$ are now increasing by \eqref{eq:defcalL}.
\end{proof}

We are now ready to state the main step in the proof of Theorem~\ref{the:FPP}. Essentially, \eqref{eq:replaceFPPbyiid} lets us replace the dependent family of weights $t^v$ by the independent family of weights $\hat{t}^{u,z},$ whose sum can easily be lower bounded, see \eqref{eq:boundonhatt}. By Lemma~\ref{lem:replacetbyhatt} this can be done on the event $G^{u_{\eps},v,\eps}_{\mathcal{A}},$ which we will show happens with high probability by combining Propositions~\ref{pro:thm4.2disco},  \ref{pro:softlocaltimes} and \ref{pro:entropy}.

\begin{Lemme}
\label{lem:coarsegraining}
Fix $\eps,\eta\in{(0,1/2]}$ and $\Theta>1$. There exist constants $\Cl[c]{ccondFPP2},\Cl[c]{CcondLN},\Cl[c]{cFPP3}>0$ and $\Cl{CbounduL},\Cl{ccondLN}<\infty,$ depending only on $\eps$, $\eta$ and $\Theta$, such that for all $r\geq 0,$ $\delta\in{(0,\Cr{ccondFPP2}]},$ $L\geq R\geq 1,$ families of weights $\mathbf{t}$ as in \eqref{eq:defbft}, $(u,v)\in{\L_{\Theta}},$  $u_{\eps}$ as in \eqref{eq:choicevu} so that $(u_{\eps},v)\in{\L_{\Theta}}$ and $u_{\eps}L^{\nu}\geq \Cr{CbounduL}$, satisfying
\begin{equation}
\label{eq:assumptiontzu2}
    \P(d_{\mathbf{t},u}(x;L;2L)\leq r)\leq\delta\text{ for all }x\in{\Lambda(L)},
\end{equation}
    and for all $N\geq L$ such that, if $\nu<1,$ 
    \begin{equation}
    \label{eq:condLNnu<1}
        \Cr{ccondLN}N^{1-\nu}|\sqrt{v}-\sqrt{u_{\eps}}|^{-2}\leq L\leq \Cr{CcondLN}\big(N^{1-\nu}|\sqrt{v}-\sqrt{u_{\eps}}|^{-2}\log(1/\delta)\big)\wedge N\big),
    \end{equation}
    if $\nu=1$
    \begin{equation}
    \label{eq:condLNnu=1}
        \Cr{ccondLN}\log(N/L)|\sqrt{v}-\sqrt{u_{\eps}}|^{-2}\leq L\leq \Cr{CcondLN}\big(\big(\log(N/L)|\sqrt{v}-\sqrt{u_{\eps}}|^{-2}\log(1/\delta)\big)\wedge N\big),
    \end{equation}
    and if $\nu>1,$
    \begin{equation}
    \label{eq:condLNnu>1}
         \zeta\stackrel{\textnormal{def.}}{=}L|\sqrt{v}-\sqrt{u_{\eps}}|^{\frac{2}{\nu}}\geq \Cr{ccondLN},
    \end{equation}
    we have for all $x\in{\Lambda(R)}$
    \begin{equation}
    \label{eq:boundFPPcond}
    \P\left(d_{\mathbf{t},v}(x;N)\leq \frac{\Cr{cFPP3}rN}{L}\right)\leq 
    3\exp\left(-\Cl[c]{cFPP4}(\sqrt{v}-\sqrt{u_{\eps}})^2F_{\nu}(N,L)\right),
\end{equation}
where $\Cr{cFPP4}(\eta)=\Cr{cbeta}/(1+\eta)$  if $\nu\leq1$ and $\Cr{cFPP4}=\Cr{cFPP4}(\eps,\Theta,\zeta)>0$ if $\nu>1$.
\end{Lemme}
\begin{proof}
    Let $u_0=u\vee v$, $\rho\in{(0,1)}$, $p$ as in \eqref{eq:choicep}, $\mathcal{A}=\mathcal{A}_{x,\rho,p}^{L,K,N}$ from Proposition~\ref{pro:entropy}, $\hat{d}_{\mathbf{t},u}^z$ as below \eqref{eq:defweighthat}, and
    \begin{equation}
    \label{eq:defeventH}
        H_{\mathcal{A}}^{u_{\eps},v,\eps}\stackrel{\textnormal{def.}}{=}G_{\A}^{u_{\eps},v,\eps}\cap\bigcap_{\calC\in{\A}}\bigcap_{\substack{\tilde{\calC}\subset\calC\\|\tilde{\calC}|\geq\frac{\rho}{2}|\calC|}}\left\{\sum_{z\in{\tilde{\calC}}}\hat{d}^z_{\mathbf{t},u}(z;L;2L)\geq\frac{\rho rp }{4}\right\}.
    \end{equation}
    Then by Lemma~\ref{lem:replacetbyhatt}  
    \begin{equation}
    \label{eq:eventHisenough}
        d_{\mathbf{t},v}(x;N)\geq \frac{\rho rp}{4}\text{ on the event }H_{\mathcal{A}}^{u_{\eps},v,\eps}.
    \end{equation}
    We now bound the probability of the event $H_{\mathcal{A}}^{u_{\eps},v,\eps},$ and start with the event appearing on the right-hand side of \eqref{eq:defeventH}.  For all $\calC\in{\A}$ and $\tilde{\calC}\subset\calC$ with $|\tilde{\calC}|\geq (\rho/2)|\calC|,$  using that $\hat{d}^z_{\mathbf{t},u}(z;L,2L),$ $z\in{\tilde{\mathcal{C}}}$ are independent random variables with the same marginals as $d_{\mathbf{t},u}(z;L,2L),$ $z\in{\tilde{\mathcal{C}}}$, it follows from \eqref{eq:cardC}, \eqref{eq:assumptiontzu2} and Bennett's inequality, see for instance \cite[Theorem~2.9]{MR3185193} with $t\geq|\tilde{\calC}|/4,$ $v\leq\delta|\tilde{\calC}|$ and $b=1$ therein, we have if $\delta\leq 1/100$ that
    \begin{equation}
    \label{eq:boundonhatt}
    \begin{split}
        \P\left(\sum_{z\in{\tilde{\calC}}}\hat{d}_{\mathbf{t},u}^z(z;L;2L)\geq\frac{\rho rp }{4}\right)
        &\geq 1-\exp\left(-\frac{\log(1/\delta)\rho p}{40}\right).
        \end{split}
    \end{equation}
Using Proposition~\ref{pro:softlocaltimes} combined with Bennets's inequality similarly as in \eqref{eq:boundonhatt} we have by \eqref{eq:cardC} and \eqref{eq:defeventFhat} that there exist constants $\Cr{CbounduL}=\Cr{CbounduL}(\eps)<\infty$ and $\Cl[c]{cboundhatF}>0$ such that if $K\geq \Cr{csoftlocaltimesK}\eps^{-\Cr{CsoftlocaltimesK}}$ and $u_{\eps}L^{\nu}\geq \Cr{CbounduL},$ then for all $\calC\in{\A}$ and $\tilde{\calC}\subset\calC$ with $|\tilde{\calC}|\geq \rho|\calC|=\rho p$
\begin{equation}
\label{eq:boundprobahatF}
    \P\big(\tilde{F}_{\tilde{\calC}}^{u_{\eps},\eps}\big)\geq 1-\exp\left(-\Cr{cboundhatF}u_{\eps}\eps^2L^{\nu}\rho p\right).
\end{equation}     
Moreover, it follows from \eqref{eq:cardA} and \eqref{eq:cardC} that there exists $\Cl{Cboundentropy}=\Cr{Cboundentropy}(K)<\infty$ such that
for all $\rho\in{(0,1)}$
\begin{equation}
    \label{eq:boundentropy}
    \sum_{\calC\in{\A}}\left|\tilde{\calC}\subset\calC:\,|\tilde{\calC}|\geq \rho|\calC| \right|\leq \exp\left({\Cr{CcardA}p\log(K)}\right)2^{p}\leq \exp\left({\Cr{Cboundentropy}p}\right).
\end{equation}
 We can finally combine \eqref{eq:thm4.2disco}, \eqref{eq:capC}, \eqref{eq:boundonhatt} and \eqref{eq:boundprobahatF} to bound the probability from \eqref{eq:defeventH}: if $\delta\leq 1/100$, $\rho\leq \Cr{crho}$, $N\geq \Cr{CKL}KL$, $K\geq \Cr{cforK}(\eta)\vee \Cr{csoftlocaltimesK}\eps^{-\Cr{CsoftlocaltimesK}}$ and $u_{\eps}L^{\nu}\geq \Cr{CbounduL}$
\begin{equation}
\label{eq:boundprobaH}
\begin{split}
\P\left((H_{\mathcal{A}}^{u_{\eps},v,\eps})^c\right)\leq \exp\left(\Cr{Cboundentropy}p\right)&\bigg(\exp\left(-\frac{1}{1+\eta}\left(\sqrt{v}-\sqrt{u_{\eps}}\right)^2F_{\nu}(N,KL)H_{\nu}(\tfrac{N}{LK},K,\eta)\right)
\\&+\exp\left(-{\Cr{cboundhatF} u_{\eps}\eps^2\rho pL^{\nu}}\right)+\exp\left(-\frac{\log(1/\delta)\rho p }{40}\right)\bigg).
\end{split}
\end{equation}
Taking $p$ the largest integer satisfying \eqref{eq:choicep},  by the lower bounds in \eqref{eq:condLNnu<1} and \eqref{eq:condLNnu=1} we have for all $\nu\leq 1$ and $\delta$, $N$ and $L$ as before
\begin{equation*}
\begin{split}
    \frac{\Cr{Cboundentropy}p}{\eta}\leq \frac{\Cr{Cboundentropy}\Cr{CNKL}N}{KL\eta}\leq \frac{\Cr{Cboundentropy}N^{\nu}(\sqrt{v}-\sqrt{u_\eps})^2}{K\eta\log(N/L)^{1\{\nu=1\}}\Cr{ccondLN}}\leq \frac{\Cr{Cboundentropy}}{K\eta\Cr{ccondLN}}(\sqrt{v}-\sqrt{u_\eps})^2F_{\nu}(N,L),
\end{split}
\end{equation*}
where we used \eqref{eq:defGnu} in the last inequality.
Taking $\rho=\Cr{crho}$ as in Proposition~\ref{pro:entropy} and $p$ as before, using the upper bounds in \eqref{eq:condLNnu<1} and \eqref{eq:condLNnu=1}, which imply in particular $p\geq \frac{\Cr{CNKL}N}{2K}\geq1$ upon taking $\Cr{CcondLN}=\Cr{CcondLN}(K)<1$ small enough, we further have if $\nu\leq1$
\begin{equation*}
\begin{split}
    \rho p\left(\Cr{cboundhatF}u_{\eps}\eps^2L^{\nu}\wedge \frac{\log(1/\delta)}{40}\right)&\geq \frac{\Cr{crho}\Cr{CNKL}N}{2K}\left(\Cr{cboundhatF}u_{\eps}\eps^2L^{\nu-1}\wedge \frac{\log(1/\delta)}{40 L}\right)
    \\&\geq \frac{\Cr{crho}\Cr{CNKL}N}{2K}\left(\Cr{cboundhatF}u_{\eps}\eps^2\Cr{CcondLN}^{\nu-1}N^{\nu-1}\wedge \frac{(\sqrt{v}-\sqrt{u_\eps})^2N^{\nu-1}}{40 \Cr{CcondLN}\log(N/L)^{1\{\nu=1\}}}\right)
    \\&\geq \frac{\Cr{crho}\Cr{CNKL}N}{2K\Cr{CcondLN}^{1-\nu}}\left(\frac{\Cr{cboundhatF}\eps^2}{\Theta}\wedge \frac{1}{40}\right)(\sqrt{v}-\sqrt{u_\eps})^2F_{\nu}(N,L),
\end{split}
\end{equation*}
where we used $u_{\eps}\geq (\sqrt{u_{\eps}}-\sqrt{v})^2/\Theta$ and \eqref{eq:defGnu} in the last inequality. Note also that $H_{\nu}(\tfrac{N}{LK},K,\eta)F_{\nu}(N,KL)\geq \Cr{cbeta}F_{\nu}(N,L)/(1+\eta)^2$ if $N/L$ is large enough (depending on $K$ and $\eta$), see \eqref{eq:defGnu} and \eqref{eq:capC}, which is the case when taking  $\Cr{CcondLN}$ large enough in \eqref{eq:condLNnu<1} and \eqref{eq:condLNnu=1}. Therefore, taking $K=K(\eps,\eta)=\Cr{cforK}(\eta)\vee \Cr{csoftlocaltimesK}\eps^{-\Cr{CsoftlocaltimesK}}$ and $p$, $\rho$ as before, one can choose $\Cr{ccondLN}=\Cr{ccondLN}(\eps,\eta)>0$ large enough and $\Cr{CcondLN}=\Cr{CcondLN}(\eps,\eta,\Theta)<\infty$ small enough, so that for $\nu\leq1$,
\begin{equation}
\label{eq:condareenoughtoboundprobaH}
\begin{split}
    \frac{\Cr{Cboundentropy}p}{\eta}&\leq \frac{\Cr{cbeta}}{(1+\eta)^3}\left(\sqrt{v}-\sqrt{u_{\eps}}\right)^2F_{\nu}(N,L)
    \\&\leq \frac{1}{1+\eta}\left(\sqrt{v}-\sqrt{u_{\eps}}\right)^2F_{\nu}(N,KL)H_{\nu}(\tfrac{N}{LK},K,\eta)\wedge \rho p\left(\Cr{cboundhatF}u_{\eps}\eps^2L^{\nu}\wedge \frac{\log(1/\delta)}{40}\right).
\end{split}
\end{equation}
 Defining $\Cr{cFPP3}=\frac{\Cr{crho} \Cr{CNKL}}{8K}$ for the previous choice of $K$, we can easily conclude when $\nu\leq 1$ by combining \eqref{eq:eventHisenough}, \eqref{eq:boundprobaH} and \eqref{eq:condareenoughtoboundprobaH}, and a change of variable for $\eta$. 

When $\nu>1,$ the second inequality in \eqref{eq:condareenoughtoboundprobaH} is no longer true in view of \eqref{eq:defGnu}. However, proceeding similarly as above \eqref{eq:condareenoughtoboundprobaH}, and for the same choices of $K$, $p$ and $\rho$, one can choose $\Cr{cFPP4}$ such that $\Cr{cFPP4}\zeta^{\nu}$ is a large enough constant, depending on $\eta$, and constants $\Cr{ccondLN}$ and $\Cr{ccondFPP2}$, depending only on $\eps$, $\eta$ and $\Theta$, so that using \eqref{eq:condLNnu>1}, \eqref{eq:defGnu} and the inequalities $\delta\leq \Cr{ccondFPP2}$, $F_{\nu}(N,KL)H_{\nu}(\tfrac{N}{LK},K,\eta)\geq cF_{\nu}(N,L)$ for some $c=c(\eps,\eta)>0$ and  $u_{\eps}\geq (\sqrt{u_{\eps}}-\sqrt{v})^2/\Theta$ we have  
\begin{equation}
\label{eq:condareenoughtoboundprobaH2}
\begin{split}
    \frac{\Cr{Cboundentropy}p}{\eta }
    &\leq \frac{\Cr{cFPP4}\left(\sqrt{v}-\sqrt{u_{\eps}}\right)^2F_{\nu}(N,L)}{1-\eta}
    \\&\leq \frac{\left(\sqrt{v}-\sqrt{u_{\eps}}\right)^2}{1+\eta}F_{\nu}(N,KL)H_{\nu}(\tfrac{N}{LK},K,\eta)\wedge\Cr{crho} p\left(\Cr{cboundhatF}u_{\eps}\eps^2L^{\nu}\wedge \frac{\log(1/\delta)}{40}\right),
     \end{split}
\end{equation}
and we can also conclude similarly as before.
\end{proof}

\begin{Rk}
\phantomsection\label{rk:decayVu}
\begin{enumerate}
    \item\label{rk:easierproofc1=0}  One can directly deduce from Lemma~\ref{lem:coarsegraining} the upper bound on the probability to connect $z$ to $B(z,N)^c$ in $\V^u$ mentioned below Theorem~\ref{the:limitVuintro2}, that is a bound similar to \eqref{eq:limitVuintro}  but for $s=0$, and in fact also on general graphs satisfying \eqref{eq:standingassumptionintro} for $\nu\leq1$. To prove this, one takes $R=1$ and $g_z(L)=1\{L>0\}$ for each $L\in{[0,\infty)}$ and $z\in{G}$. Then choose $\delta$ small enough so that for $N$ large enough, there exists $L=L(N),$ with $L(N)\rightarrow\infty$ as $N\rightarrow\infty,$ satisfying \eqref{eq:condLNnu<1} if $\nu<1$ or \eqref{eq:condLNnu=1} if $\nu=1$. Moreover, $d_{\mathbf{t},u}(z;L;2L)=0$ if and only if $B(z,L)\leftrightarrow B(z,2L)^c\text{ in }\mathcal{V}^u$, and so \eqref{eq:assumptiontzu2} for $r=0$ is satisfied for any $u>u_{**}$ and $L$ large enough, see \eqref{eq:defu**}. Therefore, \eqref{eq:boundFPPcond} (for the choice $\Theta=v/u$) is satisfied for any $v>u>u_{**}$, $\eps$ small enough, and if $N=N(u,v)$ is large enough, and so the function appearing on the left-hand side of \eqref{eq:limitVuintro} (for $s=0$ and $v$ instead of $u$) can be upper bounded by $-\Cr{cbeta}(\sqrt{v}-\sqrt{u_{\eps}})^2/(1+\eta)$. Letting $\eps,\eta\rightarrow 0$ and $u\rightarrow u_{**},$ we can conclude.
    \item \label{rk:noapriori} 
    The previous proof of the upper bound in \eqref{eq:limitVuintro} but for $s=0$ is slightly different from the corresponding ones from \cite{GRS21,MuiSev} as we do not need any a priori bound on the decay of $\P(C_z^L\leftrightarrow (\tilde{C}_z^L)^c\text{ in }\mathcal{V}^u)$ as $L\rightarrow\infty,$ only that it converges to $0,$ see in particular \cite[Proposition~5.2]{GRS21} or \cite[Proposition~5.1]{MuiSev} as to why it was necessary.  The reason we do not require this a priori is  due to our improvement by a logarithmic factor in the entropy bound from \eqref{eq:cardA}, compared to the one from \cite[(4.13)]{GRS21}, or the one used below \cite[(5.3)]{MuiSev}, and we refer to \eqref{eq:boundprobaH} as to where this improvement is used. This also means that in comparison with the method from \cite[Proposition~5.2]{GRS21}, we will prove Theorem~\ref{the:FPP} using the renormalization scheme from Lemma~\ref{lem:coarsegraining} once instead of twice (or in fact even more due to our weak a priori, see Remark~\ref{rk:apriori},\ref{rk:aprioridecay})).
    \item\label{rk:whyapriori} In Theorem~\ref{the:FPP}, we want to show that $d_{\mathbf{t},v}(x;N)$ typically grows linearly in $N$, and not only that it is positive as in Remark~\ref{rk:decayVu},\ref{rk:easierproofc1=0}). When $\nu\leq 1,$ under the hypothesis \eqref{eq:assumptiontzu2}, \eqref{eq:boundFPPcond} implies that $d_{\mathbf{t},v}(x;N)$ is of order $rN/L$ for some function $L=L(N)$ chosen so that \eqref{eq:condLNnu<1} or \eqref{eq:condLNnu=1} are satisfied. If for instance $R\geq u^{-\frac1\nu}$, one wants in \eqref{eq:bounddtvfinal} that $d_{\mathbf{t},v}$ is of order $N/R$ instead. In particular when $\nu\leq1$, one cannot directly deduce from \eqref{eq:boundFPPcond} generalizations of Theorem~\ref{the:limitVuintro2} and~\eqref{eq:boundonPEintro} to graphs satisfying \eqref{eq:standingassumptionintro}, and more generally one cannot directly deduce that the time constant associated to $d_{\mathbf{t},v}(x;N)$ is positive (when it exists). In order to conclude, one first needs to show that \eqref{eq:assumptiontzu2} is satisfied for $r\approx L/R,$ that is one needs an a priori estimate on the probability that the FPP distance grows linearly, which is the content of Section~\ref{sec:apriori}. When $\nu>1$, one only needs the a priori estimate from Section~\ref{sec:apriori} if $R$ is not of order $|\sqrt{v}-\sqrt{u}|^{-\frac2\nu}$.
    \item\label{rk:othercardinality} In the proof of Lemma~\ref{lem:coarsegraining} when $\nu\leq1$, we chose $p$ maximal so that \eqref{eq:choicep} is satisfied. In fact, one could also take $p$ minimal so that \eqref{eq:choicep} is satisfied, and obtain a new version of Lemma~\ref{lem:coarsegraining} where one  replaces the factors $N^{1-\nu}|\sqrt{v}-\sqrt{u_{\eps}}|^{-2}$ by $\log(N/L)^{\frac1\nu-1}|\sqrt{v}-\sqrt{u_{\eps}}|^{-\frac2\nu}$ in \eqref{eq:condLNnu<1}, $\log(N/L)$ by  $\log\log(N/L)$ in \eqref{eq:condLNnu=1}, and $N/L$ in the probability in \eqref{eq:boundFPPcond} by the minimal choice of $p$ in \eqref{eq:choicep}. Actually the bound on $d_{\mathbf{t},v}(x;N)$ in \eqref{eq:boundFPPcond} will be of the same order $rN^{\nu}|\sqrt{v}-\sqrt{u_{\eps}}|^2$ for a typical choice of $L$ in both version of Lemma~\ref{lem:coarsegraining}, but taking $p$ maximal in \eqref{eq:choicep} makes $L$ larger, and hence will eventually make our choice of $r$ that satisfies \eqref{eq:assumptiontzu2} larger, see \eqref{eq:choicerdelta}, which is why we did this choice in the proof of Lemma~\ref{lem:coarsegraining}. When one is not interested in the FPP distance but only in connection probabilities as in Remark~\ref{rk:decayVu},\ref{rk:easierproofc1=0}) the bounds we obtain for any choice of $p$ in \eqref{eq:choicep} are thus the same, but taking $p$ minimal allows one to take $L$ of order $|\sqrt{v}-\sqrt{u_{\eps}}|^{-\frac2\nu}$ when $\nu\leq 1$, with log corrections, which when $u_{\eps}\approx u_{**}$ is expected to correspond to the correlation length associated to $\V^u$. In other words, having information at the correlation length, with logarithmic corrections, would be enough to directly deduce the bound \eqref{eq:boundFPPcond} without requiring any a priori as in Section~\ref{sec:apriori}, which is what one intuitively expects. It would be interesting to be able to remove this logarithmic correction completely, and Remark~\ref{rk:lowercardinality} could be a first step in that direction.
\end{enumerate} 
\end{Rk}

\subsection{A priori estimate}
\label{sec:apriori}
As explained in Remark~\ref{rk:decayVu},\ref{rk:whyapriori}), in order to deduce Theorem~\ref{the:FPP} from Lemma~\ref{lem:coarsegraining}, we first need some a priori bounds on the probability that the FPP distance grows linearly. Some bounds on the FPP distance of random interlacements were obtained on $\Z^d,$ $d\geq3,$ in \cite[Theorem~2.2 and Proposition~4.9]{AndPre}, but it turns out that the methods used therein are not adapted for more general graph, see Remark~\ref{rk:apriori},\ref{rk:previousFPPbounds}) for details.  Instead, we follow the general strategy described shortly in \cite[Remark~2.3,(iii)]{AndPre}.

\begin{Prop}
\label{pro:apriori}
    Fix $\eps\in{(0,1/2]}$. There exist positive constants $\Cl[c]{ccondFPP3},\Cl[c]{cFPPpriori3}$ and $\Cl[c]{cFPPpriori},$ depending only on $\eps$, such that for all $s\geq0,$ $\delta\in{(0,\Cr{ccondFPP3}]}$, $R\geq1$, $L\geq R$, families of weights $\mathbf{t}$ as in \eqref{eq:defbft} and $(u,v)\in{\L}$  so that $|u-v|\geq \eps v,$ $uL^{\nu}\geq \Cr{cFPPpriori3}\log(1/\delta)$ and 
    \begin{equation}
\label{eq:assumptiontzu3}
    \P(t_z^u\leq s)\leq \delta (1\wedge u^{\frac1\nu}R)^{\alpha}\text{ for all }z\in{\Lambda(R)},
\end{equation}
 is satisfied,  we have for all $x\in{\Lambda(L)}$
\begin{equation}
\label{eq:apriori}
    \P\left(d_{\mathbf{t},v}(x;L;2L)\leq \frac{\Cr{cFPPpriori}sL}{(u^{-1}\log(1/\delta))^{\frac1\nu}\vee R}\right)\leq 
     f_{\eps}(\delta)
\end{equation}
for some function $f_{\eps}:[0,1)\rightarrow [0,1]$, depending only on $\eps$, and such that $f_{\eps}(\delta)\rightarrow 0$ as $\delta\rightarrow0$.

\end{Prop}
\begin{proof}
Throughout the proof, we abbreviate 
\begin{equation}
\label{eq:defun}
u_n\stackrel{\textnormal{def.}}{=}u\prod_{i=1}^n\Big(1\pm\frac{\Cl[c]{cun}\eps}{(n+1)^2}\Big)^{-1}\text{ and }K_n\stackrel{\textnormal{def.}}{=}\Cr{csoftlocaltimesK}\Big(\frac{(n+1)^2}{\eps\Cr{cun}}\Big)^{\Cr{CsoftlocaltimesK}}
\text{ for all }n\geq0,
\end{equation}
where $\Cr{csoftlocaltimesK}$ and $\Cr{CsoftlocaltimesK}$ are the constants introduced in Proposition~\ref{pro:softlocaltimes}; the plus sign corresponds to the case where the functions $g_z$ in \eqref{eq:defbft} are decreasing, and we then choose the constant $\Cr{cun}\leq1$ small enough so that $\lim_{n\rightarrow\infty}u_n\geq v$; and the minus sign to the case where they are decreasing, and  we then choose $\Cr{cun}$ so that $\lim_{n\rightarrow\infty}u_n\leq v$, which is possible in view of \eqref{eq:defcalL} and since $\eps\leq |v-u|/v$. Let $L_0=R\vee \big(\Cr{cFPPpriori3}u^{-1}\log(1/\delta)\big)^{\frac1\nu}$, where we will fix the constant $\Cr{cFPPpriori3}\geq 1$ later. For some $p=p(\eps)\in{\N}$ that we will fix later, define recursively $L_{n}=p^{2\alpha}(n+1)^{2+2\alpha}K_{n}^{\alpha}L_{n-1}$ for each $n$ such that $p^{4\alpha}(n+2)^{4+2\alpha}K_{n+1}^{\alpha}L_{n}\leq L/2$. We call $M-1$ the last $n$ as before, and let $L_{M}=L/(2p^{2\alpha}(M+1)^{2})$, which satisfies $p^{2\alpha}(M+1)^{2+2\alpha}K_M^{\alpha}L_{M-1}\leq L_M\leq p^{4\alpha}(M+1)^{4+4\alpha}K_M^{2\alpha}L_{M-1}$. We also abbreviate $L'_n=p(n+1)^{2}L_n$ for each $n\in{\{1,\dots,M\}}$. For each $u'>0,$ we also let $G_0^{z,u'}=\{d_{\mathbf{t},u'}(z;L_0;2L_0)\leq s\}$ for all $z\in{\Lambda(L_0)}$ and define recursively for all $1\leq n\leq M$
\begin{equation}
\label{eq:defGnz}
    G_{n}^{z,u'}\stackrel{\textnormal{def.}}{=}\bigcup_{\substack{x,y\in{\Lambda(L_{n-1})\cap B(z,L'_{n})}\\d(x,y)\geq 32K_{n}L'_{n-1}}}G_{n-1}^{x,u'}\cap G_{n-1}^{y,u'}\text{ for all }z\in{\Lambda(L_{n})},
\end{equation}
and
\begin{equation}\label{eq:defan}
a_n\stackrel{\textnormal{def.}}{=}\sup_{z\in{\Lambda(L_n)}}\P\big(G_{n}^{z,u_n}\big).
\end{equation}
One can show recursively on $n$ that $G_n^{z,u'}$ is measurable with respect to $B(z,2L'_{n})$ for each $n\in{\N}$ if $p\geq2$. We then have by combining a union bound and \eqref{eq:defLambda} with Proposition~\ref{pro:softlocaltimes} for $\mathcal{C}=\{x,y\},$ which is as in \eqref{eq:defcalC} for $2L'_{n}$ in place of $L$ and  $K_{n+1}$ in place of $K$ if $d(x,y)\geq 32K_{n+1}L'_{n},$ that 
\begin{equation}
\label{eq:an+1an}
    a_{n+1}\leq \Cr{CLambda}^2p^{8\alpha^2}(n+1)^{14\alpha^2}K_n^{4\alpha^2}\left(a_n^2+\Cr{Csoftlocaltimes}\exp\Big(-\frac{\Cr{csoftlocaltimes}\Cr{cun}^2u_{n+1}\eps^2L_n^{\nu}}{(n+1)^4}\Big)\right)
\end{equation}
for all $n\leq M-1$. As we now explain, this implies that for a suitable choice of the parameters
\begin{equation}
\label{eq:boundan}
    a_n\leq \frac{\Cl{cFPPpriori4}\delta\log(1/\delta)^{\frac{\alpha}{\nu}}}{\Cr{CLambda}^2p^{8\alpha^2}(n+1)^{14\alpha^2}K_n^{4\alpha^2}}\exp(-2^n)\text{ for all }0\leq n\leq M
\end{equation}
for a constant $\Cr{cFPPpriori4}\geq1$ to be chosen later. Indeed first consider the case $n=0,$  by a union bound, \eqref{eq:assumptiontzu3}, our choice of $L_0$ and \eqref{eq:defLambda}
\begin{equation}
\label{eq:bounda0}
a_0\leq \sup_{z\in{\Lambda(L_0)}}|B(z,L_0)\cap\Lambda(R)|\delta (1\wedge u^{\frac{1}{\nu}}R)^{\alpha}\leq  \Cr{CLambda}\Cr{cFPPpriori3}^{\frac{\alpha}{\nu}}\log(1/\delta)^{\frac{\alpha}{\nu}}\delta,
\end{equation}
and so \eqref{eq:boundan} is clear upon fixing $\Cr{cFPPpriori4}=\Cr{cFPPpriori4}(\eps,p,\Cr{cFPPpriori3})\geq1$ large enough. Using the bounds $L_n/L_0\geq p^{n},$ $L_0^{\nu}u\geq \Cr{cFPPpriori3}\log(1/\delta),$ $u_{n+1}\geq u/4$ by \eqref{eq:defcalL}, $\delta\leq 1/2$ and the definition of $K_n,$ we moreover have 
\begin{equation}
\label{eq:bounderroran+1}
    \Cr{Csoftlocaltimes}\exp\Big(-\frac{\Cr{csoftlocaltimes}\Cr{cun}^2u_{n+1}\eps^2L_n^{\nu}}{(n+1)^4}\Big)\leq \frac{\delta}{2\Cr{CLambda}^4p^{16\alpha^2}(n+2)^{28\alpha^2}K_{n+1}^{8\alpha^2}}\exp(-2^{n+1})
\end{equation}
upon taking $p=p(\eps)$ and $\Cr{cFPPpriori3}=\Cr{cFPPpriori3}(\eps,p)$ large enough small enough, uniformly in $n,R,u,v$ and $\delta$. Combining \eqref{eq:an+1an} and \eqref{eq:bounderroran+1} yields \eqref{eq:boundan} recursively on $n$ upon choosing $\delta\leq \Cr{ccondFPP3}$ for some constant $\Cr{ccondFPP3}=\Cr{ccondFPP3}(\eps,p)$ small enough since both $K_{n+1}/K_n$ and $(n+2)/(n+1)$ are bounded uniformly in $n$.

Let us finally show that when the FPP distance between $x$ and $y$ is small enough, then some event $G_n^z$ occurs. More precisely recalling \eqref{eq:defdzL}, we prove recursively on $n\in{\{1,\dots,M+1\}}$ that there exists a constant $\Cl{cFPPpriori2}>0$ independent of our choice of $p$ such that for all $z\in{\Lambda(L_{n-1})}$
\begin{equation}
\label{eq:smallFPPimpliesGnx}
    \left\{d_{\mathbf{t},v}(z;L_{n-1};L'_{n-1})\leq s\prod_{k=1}^{n-1}\frac{L'_{k}}{L'_{k-1}}\left(1-\frac{\Cr{cFPPpriori2}}{pk^2}\right)\right\}\subset G_{n-1}^z.
\end{equation}
When $n=1,$ \eqref{eq:smallFPPimpliesGnx} is obvious by definition of $G_0^z$ if $p\geq2$. Assume now that \eqref{eq:smallFPPimpliesGnx} is fulfilled for some $n\in{\{1,\dots,M\}}$, and that for some $z\in{\Lambda(L_{n})}$ there exist $x\in{B(z,L_{n})\cap\Lambda(R)}$ and $y\in{B(z,L'_{n})^c\cap\Lambda(R)}$ with $d_{\mathbf{t},v}(x,y)\leq s\prod_{k=1}^{n}\frac{L'_{k}}{L'_{k-1}}\left(1-\frac{\Cr{cFPPpriori2}}{pk^2}\right).$ Let $\gamma$ be the $R$-nearest neighbor path minimizing the distance $d_{\mathbf{t},v}(x,y)$ in \eqref{eq:defdistance}, and let us construct a sequence $z_0,z_1,\dots$ recursively as follows: first define $z_0\in{\Lambda(L_{n-1})}$ such that $x\in{B(z_0,L_{n-1})},$ which exists by \eqref{eq:defLambda}. Then for each $i\geq0,$ let $z_{i+1}$ be the first vertex in $\Lambda(L_{n-1})$ visited by $B(\gamma,L_{n-1}),$ and such that 
\begin{equation*}
\begin{split} 
B(z_{i+1},L'_{n-1}+R+\Cr{Cdvsdgr}^{-1})\cap B(z_{j},L'_{n-1}+R+\Cr{Cdvsdgr}^{-1})=\varnothing\text{ for all }j\leq i.
\end{split}
\end{equation*}
We stop this procedure the first time $N$ that $B(z_{N+1},L'_{n-1}+R+\Cr{Cdvsdgr}^{-1})\cap B(z,L'_{n})^c\neq\varnothing,$ and note that $d(z_{i+1},\{z_j,j\leq i\})\leq 2L'_{n-1}+2L_{n-1}+2R+3\Cr{Cdvsdgr}^{-1}$ for each $i\leq N$ by  \eqref{eq:defLambda} and \eqref{eq:dvsdgr}. In particular $$L'_{n}-(2+\Cr{Cdvsdgr}^{-1})L'_{n-1}\leq d(z_0,z_{N+1})\leq (N+2)(2L'_{n-1}+(4+3\Cr{Cdvsdgr}^{-1})L_{n-1}),$$ and so $N\geq L'_n/(2L'_{n-1}+(4+3\Cr{Cdvsdgr}^{-1})L_{n-1})-C$ for some constant $C<\infty$. Moreover, by \eqref{eq:dvsdgr}, $\gamma$ can be decomposed into disjoint $R$-nearest neighbor paths $\gamma_{2i-1},\gamma_{2i}$ starting or ending on $B(z_i,L'_{n-1})^c$ and visiting $B(z_i,L_{n-1}),$ for $1\leq i\leq N$. By our assumption on $d_{\mathbf{t},v}(x,y)$, there are at least $$2N-\frac{L'_{n}}{L'_{n-1}}(1-\frac{\Cr{cFPPpriori2}}{pn^2})\geq \frac{(\Cr{cFPPpriori2}-C)L'_{n}}{(pn^2L'_{n-1})}$$ paths $\gamma_i$ such that $t_{\gamma_i}^v\leq s\prod_{k=1}^{n-1}\frac{L'_{k}}{L'_{k-1}}\left(1-\frac{\Cr{cFPPpriori2}}{pk^2}\right),$ for some constant $C<\infty$, where we used our previous bound on $N$ and our definition of $L'_n$. In particular by \eqref{eq:smallFPPimpliesGnx} the event $G_{n-1}^{z_i}$ occurs for at least $(\Cr{cFPPpriori2}-C)L'_{n}/(2pn^2L'_{n-1})$ different $i\in{\{1,\dots,N\}}$. Since for each $x\in{\Lambda(L_{n-1})},$ by \eqref{eq:defLambda} and our choice of $L_n$ and $L'_n$,  $$|B(x,32K_{n}L'_{n-1})\cap \Lambda(L_{n-1})|\leq \Cr{CLambda}(32pK_{n}n^2)^{\alpha}\leq \frac{32^{\alpha}\Cr{CLambda}L'_{n}}{(pn^2L'_{n-1})},$$  we deduce from \eqref{eq:defGnz} that $G_{n}^z$ occurs if we choose $\Cr{cFPPpriori2}$ large enough, which  finishes the proof of \eqref{eq:smallFPPimpliesGnx}.

Let us now choose $p$ large enough, independently of $M$, so that \eqref{eq:bounderroran+1} is satisfied and $\prod_{k=1}^{M}\frac{L'_{k}}{L'_{k-1}}\big(1-\frac{\Cr{cFPPpriori2}}{pk^2}\big)\geq \Cr{cFPPpriori}L/((u^{-1}\log(1/\delta))^{\frac1\nu}\vee R)$ for some positive constant $\Cr{cFPPpriori}$, which is possible by our choice of $L_0$, $L'_0$, $L_M$ and $L'_M.$ For each $x\in{\Lambda(L)}$, if $d_{\mathbf{t},v}(x;L;2L)\leq \Cr{cFPPpriori}sL/((u^{-1}\log(1/\delta))^{\frac1\nu}\vee R)$ then by \eqref{eq:defdzL} there exist $x'\in{C_x^L\cap\Lambda(R)}$ and $y'\in{(\tilde{C}_x^L)^c\cap\Lambda(R)}$ such that $d_{\mathbf{t},v}(x',y')\leq \prod_{k=1}^{M}\frac{L'_{k}}{L'_{k-1}}\big(1-\frac{\Cr{cFPPpriori2}}{pk^2}\big)$. Letting $z\in{\Lambda(L_{M})}$ be such that $x'\in{B(z,L_{M})},$ we notice that $B(z,L'_{M})\subset B(x,L'_{M}+L_{M}+L)\subset\tilde{C}_x^L$ by \eqref{eq:defCU} and our choice of $L'_{M}=L/2$. Therefore, $d_{\mathbf{t},v}(x;L;2L)\leq \Cr{cFPPpriori}sL/((u^{-1}\log(1/\delta))^{\frac1\nu}\vee R)$ implies that $G_{M}^z$ occurs for some $z\in{\Lambda(L_{M})\cap B(x,2L)}$ by \eqref{eq:smallFPPimpliesGnx} for $n=M+1.$ We can conclude by \eqref{eq:defLambda}, a union bound on $z$, \eqref{eq:defun}, \eqref{eq:boundan}, our choice of $u_n$ and monotonicity of $u\mapsto d_{\mathbf{t},u}(x;L;2L)$.
\end{proof}

\begin{Rk}
\phantomsection\label{rk:apriori}
\begin{enumerate}
    \item \label{rk:aprioridecay} Inspecting the proof of Proposition~\ref{pro:apriori}, noting in particular that $L_n\leq L_0(p(n+1)/\eps)^{Cn}$ and using \eqref{eq:boundan} therein, one can actually bound the left-hand side of \eqref{eq:apriori} by
\begin{equation}
\label{eq:feps}
C\delta\log(1/\delta)^{\frac{\alpha}{\nu}}\exp\Big(-e^{\frac{c\log(x)}{\log\log(x)}}\Big),
\end{equation}
for some constants $C>0$  and $c=c(\eps)>0,$ where $x=L/(R\vee (u^{-1}\log(1/\delta))^{\frac1\nu})$, which decays superpolynomially fast to $0$ as $x\rightarrow\infty.$ This bound could be enough to start an iterative renormalization scheme as indicated in \cite[Remark~5.6]{GRS21}, but it turns out that thanks to our improved entropy bound \eqref{eq:cardA} we will actually never need \eqref{eq:feps}, contrary to \cite{GRS21,MuiSev}, and we refer to Remark~\ref{rk:decayVu},\ref{rk:noapriori}) for details.
\item The proof of Proposition~\ref{pro:apriori} is inspired by the perforated lattice renormalization scheme introduced in \cite[Section~2]{MR3650417}. However, the bounds on the probability of events similar to $G_n^{z,u'}$ in \eqref{eq:defGnz} used therein, which are derived in \cite[Theorem~4.1]{MR3390739}, are not explicit in their dependency on $u$ and $v$, and thus we could not directly use them here.
\item \label{rk:previousFPPbounds} On $\Z^d,$ $d\geq3,$ one could also deduce Proposition~\ref{pro:apriori} from \cite[Theorem~2.2]{AndPre} with the following setup therein: consider the renormalized lattice $R\Z^d$ instead of $\Z^d$, and take $\omega_x\stackrel{\textnormal{law.}}{=}(\ell_{y,u})_{y\in{B(x,3R)}}$ for all $x\in{R\mathbb{Z}^d}$($=\Lambda(R)$) under $\P^u$, $u\in{I\stackrel{\textnormal{def.}}{=}[R^{-(d-2)},\infty)}$. Indeed, one can use the decoupling inequality from \cite{MR3420516}  similarly as in \cite[Proposition~3.9]{AndPre}, to obtain that condition ${\bf P3}$ from \cite[Section~2.1]{AndPre} is satisfied (on $R\Z^d$ instead of $\Z^d$), and that the constants appearing therein satisfy $a_P=\chi_P=d-2$, $\eps_P=1$, and the other constants do not depend on $R$ by our previous choice of $I$. Therefore, \cite[Theorem~2.2]{AndPre} is in force,  see also \cite[Remark~2.3,(i)]{AndPre}, which implies Proposition~\ref{pro:apriori} with the additional constraint $R^{d-2}u\geq1$ therein. The actual bound from \cite[(2.11) and (2.13)]{AndPre}  lets us replace \eqref{eq:feps} by  $\exp(-cx\log(x)^{-4\cdot1\{d=3\}})$, for some constant $c>0$  which only depend on $\eps$ and $\delta$ as in Proposition~\ref{pro:apriori} by inspecting the proof. Note that this last function is much better as $x\rightarrow\infty$ than the one obtained in \eqref{eq:feps}, which is due to a choice of slowly growing renormalization scheme in \cite[Section~2.2]{AndPre}. In particular, the decay is exponential when $d\geq4$, which is enough to obtain Theorem~\ref{the:FPP} directly in this case, without using the more advanced method from Lemma~\ref{lem:coarsegraining}. On general graphs satisfying \eqref{eq:standingassumptionintro}, it is however not clear whether condition {\bf P3} from \cite[Section~2.1]{AndPre} holds or not, since the decoupling inequalities in this context from Proposition~\ref{pro:softlocaltimes} or \cite[Theorem~2.4]{DrePreRod2} require to consider two sets at distance larger than their maximal radius, which is not required on $\Z^d$ in \cite{MR3420516}. To summarize, the method from \cite{AndPre} can be applied to prove Proposition~\ref{pro:apriori} on $\Z^d$, $d\geq3$, and in fact even Theorem~\ref{the:FPP} directly if $d\geq4$, under the additional constraint $uR^{d-2}\geq1$, but general graphs require a different proof given in Proposition~\ref{pro:apriori}, which yields worst bounds on the decay of the left-hand side of \eqref{eq:apriori}, but these bounds will eventually still be enough to deduce Theorem~\ref{the:FPP} from Lemma~\ref{lem:coarsegraining}.  
\end{enumerate}
\end{Rk}

\subsection{Proof of Theorem~\ref{the:FPP}}
\label{sec:prooftheFPP}
We now combine Lemma~\ref{lem:coarsegraining} and Proposition~\ref{pro:apriori} to finish the
\begin{proof}[Proof of Theorem~\ref{the:FPP}]
Fix $\eta,\xi,\zeta>0$, $\Theta>1$ and $(u,v)\in{\L_{\Theta}}$. We also assume that $N\geq \Cr{cboundN2}|\sqrt{v}-\sqrt{u}|^{-\frac2\nu}$, for a constant $\Cl[c]{cboundN2}=\Cr{cboundN2}(\eta,\xi,\zeta,\Theta)$ to be fixed later, which can be assumed w.l.o.g.\ upon increasing the constant $\Cr{cboundN}$ in \eqref{eq:bounddtvfinal}. We let 
\begin{equation}
\label{eq:defepsu'}
\eps=\frac{\xi\eta}{6}\left(\leq \frac{|\sqrt{v}-\sqrt{u}|}{\sqrt{u}}\times\frac{\eta}{6}\right)\text{ and } u'=u(1+2\text{sign}(v-u)\eps). 
\end{equation}
Recalling the definition of $u'_{\eps}$ from \eqref{eq:choicevu}, one can easily check that the above choices imply that $|u'-u|\geq \eps u'$, $(u'_{\eps},v)\in{\L_{\Theta}}$ and $|\sqrt{u'_{\eps}}-\sqrt{v}|\geq|\sqrt{u}-\sqrt{v}|(1-\eta)$ upon assuming w.l.o.g.\ that $\eta$ is small enough (depending on $\xi$ and $\Theta$). For any $\Cr{ccondFPP}\in{(0,\Cr{ccondFPP3}]}$ to be chosen later, by Proposition~\ref{pro:apriori} (applied for $\Cr{ccondFPP}$ instead of $\delta$ and $u'$ instead of $v$), under the conditions from Theorem~\ref{the:FPP}, for any $L\geq R$ with $L^{\nu}u\geq \Cr{cFPPpriori3}\log(1/\Cr{ccondFPP})$, with $\Cr{cFPPpriori3}=\Cr{cFPPpriori3}(\eps)$ as in Proposition~\ref{pro:apriori}, we have that \eqref{eq:assumptiontzu2} is satisfied with $u'$ instead of $u$,
\begin{equation}
\label{eq:choicerdelta}
    r=\frac{\Cr{cFPPpriori}sL}{(u^{-1}\log(1/\Cr{ccondFPP}))^{\frac1\nu}\vee R}\text{ and }\delta=f_{\eps}(\Cr{ccondFPP}).
\end{equation}
We now choose 
\begin{equation}
\label{eq:choiceL}
    L=\begin{cases}
        \Cr{ccondLN}N^{1-\nu}|\sqrt{v}-\sqrt{u'_{\eps}}|^{-2}\log\big(N|\sqrt{v}-\sqrt{u'_\eps}|^{2}\big)^{1\{\nu=1\}}&\text{ if }\nu\leq  1,
        \\(\Cr{cFPPpriori3}\log(1/\Cr{ccondFPP})\Theta)^{\frac1\nu}\vee(\Theta\Cr{CbounduL})^{\frac1\nu}\vee\Cr{ccondLN}\vee\zeta)|\sqrt{v}-\sqrt{u'_\eps}|^{-\frac2\nu}&\text{ if }\nu> 1.
    \end{cases}
\end{equation}
Note that $N/L$ is larger than a function of $N|\sqrt{v}-\sqrt{u}|^{\frac2\nu}$ diverging to infinity, as well as  $L|\sqrt{v}-\sqrt{u}|^{\frac2\nu}$ when $\nu\leq1$. Therefore, if $N|\sqrt{v}-\sqrt{u}|^{\frac2\nu}$ is large enough (depending on $\eta$, $\Theta$, $\zeta$ and $\xi$), then $L$ is as assumed above \eqref{eq:choicerdelta}: $R\leq \zeta |\sqrt{v}-\sqrt{u}|^{-\frac2\nu}\leq L$ by \eqref{eq:choiceL}, and $uL^{\nu}\geq \Cr{cFPPpriori3}\log(1/\Cr{ccondFPP})\Theta|\sqrt{v}-\sqrt{u'_\eps}|^{-2}u\geq \Cr{cFPPpriori3}\log(1/\Cr{ccondFPP})$. Using \eqref{eq:choicerdelta}, one can also easily check that \eqref{eq:condLNnu<1} if $\nu<1$ and \eqref{eq:condLNnu=1} if $\nu=1$  (for $u'$ instead of $u$) are satisfied uniformly in $N\geq \Cr{cboundN2}|\sqrt{v}-\sqrt{u}|^{-\frac2\nu}$ upon taking $\Cr{cboundN2}=\Cr{cboundN2}(\xi,\eta,\zeta,\Theta)$ large enough, and $\Cr{ccondFPP}=\Cr{ccondFPP}(\xi,\eta,\Theta)\leq \Cr{ccondFPP3}$ small enough. Note also that \eqref{eq:condLNnu>1} is automatically satisfied by our choice of $L$. We also choose $\Cr{ccondFPP}=\Cr{ccondFPP}(\xi,\eta,\Theta)\leq \Cr{ccondFPP3}$ small enough so that $\delta \leq \Cr{ccondFPP2}$ in \eqref{eq:choicerdelta}. Note that we can indeed choose $\Cr{ccondFPP}$ as before not depending on $\zeta$ for all $\nu>0$ since  $\Cr{ccondFPP2},\Cr{ccondLN}$ and $\Cr{CcondLN}$  do not depend on $\zeta$ in Lemma~\ref{lem:coarsegraining}. Since $u'_{\eps}L^{\nu}\geq \Cr{CbounduL}$ and $L\leq N$ if $\Cr{cboundN2}$ is large enough, we deduce that Lemma~\ref{lem:coarsegraining} is in force. Moreover, in view of \eqref{eq:defGnu} and below we have $(\sqrt{u'_{\eps}}-\sqrt{v}\big)^2F_{\nu}(N,L)\geq (1-\eta)^2F_{\nu}\big(N|\sqrt{v}-\sqrt{u}|^{\frac2\nu})$ if $\nu<1$, $(\sqrt{u'_{\eps}}-\sqrt{v}\big)^2F_{\nu}(N,L)\geq (1-\eta)^3F_{\nu}\big(N|\sqrt{v}-\sqrt{u}|^{\frac2\nu}\big)$ if $\nu=1$ (note that $\log(N/L)\leq \log\big(N(\sqrt{v}-\sqrt{u})^{2}\big)/(1-\eta)$ upon taking $\Cr{cboundN2}$ large enough), and $(\sqrt{u'_{\eps}}-\sqrt{v}\big)^2F_{\nu}(N,L)\geq cF_{\nu}\big(N|\sqrt{v}-\sqrt{u}|^{\frac2\nu}\big)$ if $\nu>1$, for some constant $c=c(\xi,\eta,\zeta,\Theta)$. Note also that 
\begin{equation*}
    \frac{\Cr{cFPP3}rN}{L}=\frac{\Cr{cFPP3}\Cr{cFPPpriori}sN}{(u^{-1}\log(1/\Cr{ccondFPP}))^{\frac1\nu}\vee R}\geq \frac{\Cr{cFPP}sN}{u^{-\frac1\nu}\vee R}
\end{equation*}
if $\Cr{cFPP}=\Cr{cFPP}(\xi,\eta,\zeta,\Theta)>0$ is chosen small enough.
Therefore \eqref{eq:boundFPPcond} (for $u'$ instead of $u$) implies \eqref{eq:bounddtvfinal} upon the change of variable from $(1-\eta)^3/(1+\eta)$ to $1/(1+\eta)$.
\end{proof}

\begin{Rk}
\phantomsection
\label{rk:endthm}
\begin{enumerate}
\item\label{rk:extensiontoGFF} Our proof of Theorem~\ref{the:FPP} can probably be adapted when replacing local times of random interlacements by the Gaussian free field, similarly as in \cite{GRS21,MR3417515}. More precisely, denoting by $(\phi_x)_{x\in{G}}$ the Gaussian free field on $G$, let $\eta_x^z=\sum_{y\in{\partial B(z,KL)}}\P_x(X_{T_{\partial B(z,KL)}}=y)\phi_y$ and $\psi_x^z=\phi_x-\eta_x^z$ for all $x,z\in{G}$. Then one essentially wants to replace the role of $\ell_{\cdot,u}$ and $\ell_{\cdot,T^z(u)}$ (or $\ell^-_{\cdot,T^z(u)}$ depending on the monotonicity of the functions $g_z$ in \eqref{eq:defbft}) in the proof of Theorem~\ref{the:FPP} by $\phi$ and $\psi^z$ respectively. Proposition~\ref{pro:thm4.2disco} is then replaced by \cite[Corollary~4.4]{MR3417515}, which is proved therein on $\Z^d$ but whose proof could probably be adapted to graphs satisfying \eqref{eq:standingassumptionintro}.  One does not need to use Proposition~\ref{pro:softlocaltimes} in the proof of Lemma~\ref{lem:coarsegraining} anymore as $\psi^z$, $z\in{\calC}$, are independent for any $\calC$ as in \eqref{eq:defcalC}, instead of being approximated by independent random variables as in \eqref{eq:defeventF}. However, in the proof of Proposition~\ref{pro:apriori} one still needs a version of Proposition~\ref{pro:softlocaltimes} but only for sets $\calC$ as in \eqref{eq:defcalC} with $|\calC|=2$. This is proved in \cite[Theorem~2.4]{DrePreRod2} but with an additional polynomial factor in front of the exponential, which one can probably remove using similar ideas as in the proof of Proposition~\ref{pro:softlocaltimes}, but this goes beyond the scope of this article. More generally, our technique could be extended to general fields $f$ (playing the role of $\ell$ or $\phi$) which can be decomposed as the sum of a global field $g$ (playing the role of $\ell_{\cdot}-\ell_{\cdot,T^z(\cdot)}$ or $\eta^z$), and a local field $h$ (playing the role of $\ell_{\cdot,T^z(\cdot)}$ or $\psi^z$), as long as a version of Proposition~\ref{pro:thm4.2disco} for $g$ and a version of Proposition~\ref{pro:softlocaltimes} for $h$ are satisfied, which is for instance the case of the class of Gaussian fields on $\R^d$ considered in \cite{MuiSev} (although one would then additionally need to extend our method when replacing $G$ by $\R^d$).

\item \label{rk:removingdependenceonzeta}
The constants $\Cr{ccondFPP}$, $\Cr{cFPP}$ and $\Cr{cboundN}$, as well as $\Cr{cFPP2}$ if $\nu>1$, appearing in \eqref{eq:assumptiontzu} and \eqref{eq:bounddtvfinal} all depend on $\xi$ and $\zeta$, and hence possibly indirectly on $u$ and $v$, and it would be very interesting to be able to remove this dependency (note that the dependency of these constants on $\Theta$ is not a problem as long as $u$ and $v$ are of the same order, and the dependency on $\eta$ is only relevant to obtain the exact constant $\Cr{cFPP2}=\Cr{cbeta}$ if $\nu\leq 1$). Indeed, it would for instance imply a version of the upper bound in \eqref{eq:limitVuintro} valid for $N\geq c(\sqrt{v}-\sqrt{u_{**}})^{2/\nu}$, instead of a limit as $N\rightarrow\infty$, and would thus provide an upper bound for the correlation length associated to the percolation of $\V^u$. The current proof only gives that these constants increase (or decrease) at most polynomially in $\xi$ and $\zeta$, which for $\xi$ can be traced back to \eqref{eq:defepsu'}, \eqref{eq:condareenoughtoboundprobaH}, \eqref{eq:condareenoughtoboundprobaH2} and the polynomial dependency of $K$ on $\eps$ in Proposition~\ref{pro:softlocaltimes}, and for $\zeta$ on the verification of the conditions for $L$ below \eqref{eq:choiceL}.

It does not seem to be possible to remove the dependency on $\zeta$,  as if for instance $R= (\sqrt{v}-\sqrt{u})^{-\frac2\nu-a}$ for some $a>0$, then taking $N=R$, the bound \eqref{eq:bounddtvfinal} would clearly be false as $\sqrt{v}-\sqrt{u}\rightarrow0$. Note that when $\nu<1$, one could replace the condition $R\leq \zeta |\sqrt{v}-\sqrt{u}|^{-\frac2\nu}$ by the weaker condition $R\leq L$, with $L$ as in \eqref{eq:choiceL} (which is also not satisfied in the previous example), as proving $R\leq L$ was the only reason we assumed $R\leq \zeta |\sqrt{v}-\sqrt{u}|^{-\frac2\nu}$, see below \eqref{eq:choiceL}. This dependency on $\zeta$ is however intuitively not a problem as one expects for $\nu$ small enough that \eqref{eq:assumptiontzu} is satisfied when $R$ is of order $\zeta(\sqrt{v}-\sqrt{u})^{-\frac2\nu}$ for a certain constant $\zeta$ (independently of $u$ and $v$ in $\L_{\Theta}$ with $u_{**}<u<v$), when for instance taking $t_z^u=1\{B(z,R)\not\leftrightarrow B(z,2R)\text{ in }\V^u\}$ as in the proof of Theorem~\ref{the:limitVuintro2} in Section~\ref{sec:proofofFPPintro}. 

Finally, the dependency on $\xi$ of the constants in Theorem~\ref{the:FPP} comes from the choice of the sprinkling parameter $\eps$ in \eqref{eq:defepsu'}, which should not depend on $u$ and $v$ since the constants in Lemma~\ref{lem:coarsegraining} and Proposition~\ref{pro:apriori} depend on this choice of $\eps$. In both cases, the dependency of the constants on $\eps$ can be traced back to the dependency of $K$ on $\eps$ in Proposition~\ref{pro:softlocaltimes}. As explained in Remark~\ref{rk:softlocaltimes},\ref{rk:othersoftlocaltimes}), one can actually prove that \eqref{eq:boundprobaFzueps} remains valid when taking $K$ large enough (independently of $\eps$), but under the additional constraint $L^{\nu}\geq C(u\eps^2)^{-1}\log(1/(u\eps^2))$. Plugging this into the proof, one could obtain a new version of Theorem~\ref{the:FPP}  such that the constants therein would not depend on $\xi$ anymore, but with various detrimental additional logarithm terms appearing in \eqref{eq:assumptiontzu} and \eqref{eq:bounddtvfinal} due to this additional constraint on $L$ in the proof. This new version would typically be better than \eqref{eq:bounddtvfinal} when we have to take $\xi$ depending on $u,v$, that is when $(\sqrt{v}-\sqrt{u})^2$ is much smaller than $u$. This is the case in Theorem~\ref{the:limitVuintro2} but its statement would actually not be improved since the constants therein anyway already depend on $u$ via the choice of $\zeta$, see its proof in Section~\ref{sec:proofofFPPintro}. However, when $u$, $v$ and $(\sqrt{v}-\sqrt{u})^2$ are of the same order, which will be the case in the proof of \eqref{eq:boundonPEintro}, this new version of Theorem~\ref{the:FPP} would be worse by logarithmic factors, and so one would not be able to deduce Theorems~\ref{the:localuniquenessintro} and~\ref{the:boundcapacityintro} anymore, hence our choice to leave the dependency on $\xi$ in Theorem~\ref{the:FPP}.
\end{enumerate}
\end{Rk}

\section{Proof of the three intermediate propositions}
\label{sec:proofof3prop}
\subsection{Proof of Proposition~\ref{pro:thm4.2disco}}
\label{sec:proofpropthm4.2}
We mainly follow the proof of \cite[Theorem~4.2]{MR3602841}. Since we need to generalize this proof from $\Z^d$ to the class of graphs from \eqref{eq:standingassumptionintro}, and only the case $u>v$ is treated in \cite{MR3602841} (which is actually slightly simpler than the case $u<v$, see \eqref{eq:laplacetheta} and below), we present for the reader convenience a full and simplified proof here.  The main tool in this proof is an explicit formula for the generating function of the local times of random interlacements from \cite[Remark 2.4.4]{MR3050507}, see also \cite[(1.9)--(1.11)]{MR2892408}: for all $V: G \to \mathbb{R}$ with finite support and satisfying
\begin{equation}
\label{eq:lap1}
\Vert \mathbf{GV} \Vert_{\infty} < 1,  \text{ where } (\mathbf{GV})f(x) \stackrel{\textnormal{def.}}{=} \sum_{y \in \Z^d}g(x,y) V(y)f(y),
\end{equation}
and $\|\cdot\|_{\infty}$ denotes the operator norm on $\ell^{\infty}(G)\mapsto \ell^{\infty}(G),$ one has
\begin{equation}
\label{eq:lap2}
\begin{split}
{\E} \Big[ \exp \Big( \sum_{x\in G} V(x) {\ell}_{x,u} \Big) \Big] 
&\ \ = \exp\left( u \langle V, (I- \mathbf{GV})^{-1} 1 \rangle_{\ell^2(G)}\right)\\
&\Big(\stackrel{\textnormal{def.}}{=} \exp \left( u \sum_{x \in G} V(x)  \sum_{ n \geqslant 0} (\mathbf{GV})^n1(x) \right) \Big).
\end{split}
\end{equation}
We now take a suitable choice for $V,$ and show that it almost satisfies \eqref{eq:lap1}.

\begin{Lemme}
\label{lem:gVcloseto1}
    For any $\eta\in{(0,1)},$ there exists $\Cl{cchoiceV}=\Cr{cchoiceV}(\eta)<\infty$ such that for any $L\geq1,$ $K\geq \Cr{cchoiceV}$ and $\mathcal{C}$ as in \eqref{eq:defcalC}, if
    \begin{equation}
     \label{eq:defV}
        V(x)=\sum_{z\in{\calC}}e_{\Sigma(\calC)}(\hat{C}_z)\bar{e}_{\hat{C}_z}(x)\text{ for all }x\in{G},
    \end{equation}
    then for all $y\in{G}$
    \begin{equation}
    \label{eq:boundgV}
        |(\mathbf{GV}) 1(y)-\P_y(H_{\Sigma(\mathcal{C})}<\infty)|\leq \eta \P_y(H_{\Sigma(\mathcal{C})}<\infty).
    \end{equation}
\end{Lemme}

 On $\Z^d$, a version of \eqref{eq:boundgV} without the absolute value is proved in \cite[Proposition~4.1]{MR3602841}, but we will follow a different and simpler strategy. Before starting the proof of Lemma~\ref{lem:gVcloseto1}, let us collect the following useful consequences of the Harnack inequality \eqref{EHIK}. Recall the definition of $H_A$ and $L_A$ for $A\subset\subset G$ from below \eqref{EHIK}, and since these times can be infinite to simplify notation we will take $X_{-\infty}=X_{\infty}=\Delta$ for some cemetery point $\Delta\notin{G}$.

\begin{Lemme}
There exist constants $\Cl{cK4.2},\Cl{C4.2}<\infty$ such that for any $L\geq1$, $K\geq \Cr{cK4.2}$, $\mathcal{C}$ as in \eqref{eq:defcalC}, $z\in{\calC}$ and $x\in{\hat{C}_z}$
    \begin{equation}
    \label{eq:esigmavsec}
        \Big|\frac{e_{\Sigma(\calC)}(x)}{e_{\hat{C}_z}(x)}-\P_z(H_{\Sigma(\calC)\setminus \hat{C}_z}=\infty)\Big|\leq \frac{\Cr{C4.2}}{K^{(\Cr{cHarnackprecise}\wedge\nu)/2}}\P_z(H_{\Sigma(\calC)\setminus \hat{C}_z}=\infty)
    \end{equation}
and
    \begin{equation}
    \label{eq:esigmavsec2}
        \big|\P_x(X_{L_{\Sigma(\mathcal{C})}}\in{\hat{C}_z})-\P_z(H_{\Sigma(\calC)\setminus \hat{C}_z}=\infty)\big|\leq \frac{\Cr{C4.2}}{K^{\Cr{cHarnackprecise}\wedge\nu}}\P_z(H_{\Sigma(\calC)\setminus \hat{C}_z}=\infty).
    \end{equation}
\end{Lemme}
\begin{proof}
Let us abbreviate $\mathcal{B}_z=\partial B(z,\sqrt{K}L)$. By \eqref{eq:defequicap},  \eqref{eq:defcalC} and the strong Markov property we have for all $x\in{\hat{C}_z}$
\begin{equation*}
    e_{\Sigma(\calC)}(x)=\lambda_x\E_x[1\{H_{\mathcal{B}_z}<\tilde{H}_{\hat{C}_z}\}\P_{X_{H_{\mathcal{B}_z}}}({H}_{\Sigma(\calC)}=\infty)].
\end{equation*}
For any $y\in{\B_z}$ it follows from the strong Markov property that
\begin{equation*}
\begin{split}
    &|\P_{y}({H}_{\Sigma(\calC)}=\infty)-\P_z(H_{\Sigma(\calC)\setminus \hat{C}_z}=\infty)|
    \\&\leq |\P_{y}({H}_{\Sigma(\calC)\setminus \hat{C}_z}=\infty)-\P_z(H_{\Sigma(\calC)\setminus \hat{C}_z}=\infty)|+\E_y[H_{\hat{C}_z<\infty}\P_{X_{H_{\hat{C}_z}}}(H_{\Sigma(\calC)\setminus \hat{C}_z}=\infty)]
    \\&\leq CK^{-(\Cr{cHarnackprecise}\wedge\nu)/2}\P_z(H_{\Sigma(\calC)\setminus \hat{C}_z}=\infty),
    \end{split}
\end{equation*}
where we used \eqref{eq:boundentrance}, \eqref{EHIK} and twice \eqref{eq:defcalC} in the last inequality. Moreover, using \eqref{eq:boundentrance} again we have
\begin{equation*}
\begin{split}
    |\P_x(\tilde{H}_{\hat{C}_z}=\infty)-\P_x(H_{\mathcal{B}_z}<\tilde{H}_{\hat{C}_z})|&=\E_x\big[1\{H_{\mathcal{B}_z}<\tilde{H}_{\hat{C}_z}\}\P_{X_{H_{\mathcal{B}_z}}}(H_{\hat{C}_z}<\infty)\big]
    \\&\leq\frac{C}{\sqrt{K}^{\nu}}\P_x(H_{\mathcal{B}_z}<\tilde{H}_{\hat{C}_z})\leq \frac{C'}{\sqrt{K^{\nu}}}\P_x(\tilde{H}_{\hat{C}_z}=\infty)
\end{split}
\end{equation*}
for $K$ large enough. Combining the last three displays with \eqref{eq:defequicap} yields \eqref{eq:esigmavsec}. We now turn to the proof of \eqref{eq:esigmavsec2}, and first note that if $H_{\Sigma(\mathcal{C})\setminus \hat{C}_z=\infty}$ then $X_{L_{\Sigma(\mathcal{C})}}\in{\hat{C}_z}$. Therefore by the strong Markov property abbreviating $\Sigma_z(\calC)=\Sigma(\calC)\setminus \hat{C}_z$ we have for all $x\in{\hat{C}_z}$
\begin{equation*}
    \big|\P_x(X_{L_{\Sigma(\mathcal{C})}}\in{\hat{C}_z})-\P_x(H_{\Sigma_z(\calC)}=\infty)\big|=\E_x\big[1\{H_{\Sigma_z(\calC)}<\infty\}\P_{X_{H_{\Sigma_z(\mathcal{C})}}}(X_{L_{\Sigma(\calC)}}\in{\hat{C}_z})\big].
\end{equation*}
Moreover, for all $y\in{\Sigma(\mathcal{C})\setminus \hat{C}_z}$ by the strong Markov property
\begin{equation*}
    \P_y(X_{L_{\Sigma(\calC)}}\in{\hat{C}_z})=\E_y\big[1\{H_{\hat{C}_z}<\infty\}\P_{X_{H_{\hat{C}_z}}}(X_{L_{\Sigma(\calC)}}\in{\hat{C}_z})\big]\leq \frac{C}{{K^{\nu}}}\P_x(X_{L_{\Sigma(\calC)}}\in{\hat{C}_z}),
\end{equation*}
where we used \eqref{eq:boundentrance} and \eqref{EHI} in the last inequality. Combining the last two displays with \eqref{EHIK}, we can easily conclude upon taking $K$ large enough.
\end{proof}
\begin{proof}[Proof of Lemma~\ref{lem:gVcloseto1}]
    By \eqref{eq:lap1} and \eqref{eq:defV} we have for all $y\in{G}$
    \begin{equation*}
        (\mathbf{GV})1(y)=\sum_{x\in{\Sigma(\mathcal{C})}}\sum_{z\in{\calC}}g(x,y)e_{\Sigma(\mathcal{C})}(\hat{C}_z)\bar{e}_{\hat{C}_z}(x)=\sum_{z\in{\calC}}\frac{e_{\Sigma(\calC)}(\hat{C}_z)}{\mathrm{cap}(\hat{C}_z)}\P_y(H_{\hat{C}_z}<\infty),
    \end{equation*}
where we used \eqref{entrancegreenequi} and the fact that $\bar{e}_{\hat{C}_z}$ is supported on $\hat{C}_z$ in the last equality. Moreover, for any $z\in{\mathcal{C}}$ and $x\in{\hat{C}_z}$ by \eqref{eq:esigmavsec} and \eqref{eq:esigmavsec2}
\begin{equation*}
\begin{split}
    e_{\Sigma(\calC)}(\hat{C}_z)&\leq\Big(1+\frac{\Cr{C4.2}}{K^{(\Cr{cHarnackprecise}\wedge\nu)/2}}\Big)\P_z(H_{\Sigma(\calC)\setminus \hat{C}_z}=\infty)\mathrm{cap}(\hat{C}_z)
    \\&\leq \frac{1+\Cr{C4.2}{K^{-(\Cr{cHarnackprecise}\wedge\nu)/2}}}{1-\Cr{C4.2}{K^{-\Cr{cHarnackprecise}\wedge\nu}}}\P_x(X_{L_{\Sigma(\mathcal{C})}}\in{\hat{C}_z})\mathrm{cap}(\hat{C}_z)
\end{split}
\end{equation*}
for $K$ large enough. Therefore there exists $\Cr{cchoiceV}=\Cr{cchoiceV}(\eta)\in{[\Cr{cK4.2},\infty)}$ such that if $K\geq \Cr{cchoiceV}$ then
\begin{equation*}
    (\mathbf{GV})1(y)\leq (1+\eta)\sum_{z\in{\calC}}\sum_{x\in{\hat{C}_z}}\P_y(X_{H_{\hat{C}_z}}=x)\P_x(X_{L_{\Sigma(\mathcal{C})}}\in{\hat{C}_z})=(1+\eta)\P_y(H_{\Sigma(\mathcal{C})}<\infty)
\end{equation*}
by the strong Markov property. The lower bound can be proved similarly.
\end{proof}

Proposition~\ref{pro:thm4.2disco} now follows from \eqref{eq:lap2} and Lemma~\ref{lem:gVcloseto1}.
\begin{proof}[Proof of Proposition~\ref{pro:thm4.2disco}]
Let $V$ as in \eqref{eq:defV}, and let us first assume that $v<u,$ then for any $\theta>0$ such that $\|\theta\mathbf{GV}\|_{\infty}<1$
\begin{equation}
\label{eq:boundprobaE}
\begin{split}
    \P((E_{\mathcal{C}}^{u,v})^c)&\leq \P\left(\sum_{z\in{\calC}}{e_{\Sigma(\calC)}(\hat{C}_z)}\sum_{x\in{\hat{C}_z}}\bar{e}_{\hat{C}_z}(x)\ell_{x,v}\geq u\mathrm{cap}(\Sigma(\calC))\right)
    \\&\leq \P\left(\sum_{x\in{\Sigma(\calC)}}V(x)\ell_{x,v}\geq u\mathrm{cap}(\Sigma(\calC))\right)
    \\&\leq\exp\left( \theta v \langle V, (I- \theta\mathbf{GV})^{-1} 1 \rangle_{\ell^2(G)}-\theta u\mathrm{cap}(\Sigma(\calC))\right),
\end{split}
\end{equation}
where in the last inequality we used a Chernov bound combined with \eqref{eq:lap2}. Now take $\theta=\frac{\sqrt{u}-\sqrt{v}}{(1+\eta)\sqrt{u}},$ then $\theta<1/(1+\eta)$ and by \eqref{eq:boundgV}
 we thus have $\|\theta \mathbf{GV}\|_{\infty}\leq \theta(1+\eta)<1$ when $K\geq \Cr{cchoiceV}(\eta).$ Moreover, 
 \begin{equation*}
     \langle V, (I- \theta\mathbf{GV})^{-1} 1 \rangle_{\ell^2(G)}\leq \frac{1}{1-\theta (1+\eta)} \langle V, 1 \rangle_{\ell^2(G)}=\frac{\mathrm{cap}(\Sigma(\calC))}{1-\theta (1+\eta)}.
 \end{equation*}
 Plugging this in \eqref{eq:boundprobaE} and replacing $\theta$ by its value, \eqref{eq:thm4.2disco} follows readily.
 
 We now turn to the case $v>u$, for which we will actually need a version of \eqref{eq:lap2} valid for $\theta$ large. As we now explain, following ideas similar to the loop soup case, see for instance the proof of \cite[(3.45)]{MR2932978}, for all $\theta>0$ (and actually for any choice of $V\geq0$ with finite support) we have
 \begin{equation}
 \label{eq:laplacetheta}
     {\E} \Big[ \exp \Big( -\theta\sum_{x\in G}  V(x) {\ell}_{x,v} \Big) \Big] 
 = \exp\left( -\theta v \langle \sqrt{V}, (I+ \theta\mathbf{\sqrt{V}G\sqrt{V}})^{-1} \sqrt{V} \rangle_{\ell^2(\text{supp}(V))}\right).
 \end{equation}
In the previous equation we denote by $\text{supp}(V)$ the set $\{x\in{G}:V(x)>0\}$, by $\sqrt{V}$ the vector $(\sqrt{V(x)})_{x\in{\text{supp}(V)}}$, by $\mathbf{\sqrt{V}G\sqrt{V}}$ the matrix $(\sqrt{V(x)}g(x,y)\sqrt{V(y)})_{x,y\in{\text{supp}(V)}}$, and by $I$ as the identity on $\text{supp}(V)$. Since $\mathbf{G}$ has only positive eigenvalues by \cite[Proposition~1.3]{MR2932978}, and so $\mathbf{\sqrt{V}G\sqrt{V}}$ as well, the matrix $I+ \theta\mathbf{\sqrt{V}G\sqrt{V}}$ is invertible and so the right-hand side of \eqref{eq:laplacetheta} is well-defined. One can easily check that $\langle V,(\mathbf{GV})^n1\rangle_{\ell^2(G)}=\langle \sqrt{V}, (\mathbf{\sqrt{V}G\sqrt{V}})^{n} \sqrt{V} \rangle_{\ell^2(\text{supp}(V))}$ for all $n\geq0$, and thus if $\theta$ is small enough, \eqref{eq:laplacetheta} follows from \eqref{eq:lap2}. For $\theta\in{\mathbb{C}}$ with $\text{Re}(\theta)>0$, the left-hand side of \eqref{eq:laplacetheta} is analytic by dominated convergence theorem on compact subdomains, and the right-hand side can also easily be shown to be analytic by writing it in terms of the eigenvalues of $\mathbf{\sqrt{V}G\sqrt{V}}$, which  is symmetric and hence diagonalizable, and with positive eigenvalues. One then readily deduces \eqref{eq:laplacetheta} by uniqueness of analytic continuation.

The rest of the proof is now similar to the case $v<u$: \eqref{eq:boundprobaE} becomes by \eqref{eq:laplacetheta}
 \begin{equation}
 \label{eq:boundprobaE2}
     \P((E_{\mathcal{C}}^{u,v})^c)
    \leq\exp\left( -\theta v \langle \sqrt{V}, (I+ \theta\mathbf{\sqrt{V}G\sqrt{V}})^{-1} \sqrt{V} \rangle_{\ell^2(\text{supp}(V))}+\theta u\mathrm{cap}(\Sigma(\calC))\right),
 \end{equation}
 for any $\theta>0$. We now take  $\theta=\frac{\sqrt{v}-\sqrt{u}}{(1+\eta)\sqrt{u}}$ (which can be much larger than one, see for instance the proof of Theorem~\ref{the:FPPnointro} below), then by \eqref{eq:boundgV} the matrix $\theta \mathbf{GV}$ restricted to $\text{supp}(V)$ has eingenvalues smaller than $\theta(1+\eta)$ for all $K\geq \Cr{cchoiceV}(\eta).$ Hence, the matrix $\theta\mathbf{\sqrt{V}G\sqrt{V}}$ also has eigenvalues smaller than $\theta(1+\eta)$, and using that it is diagonalizable one can easily deduce that
  \begin{equation*}
     \langle \sqrt{V}, (I+ \theta\mathbf{\sqrt{V}G\sqrt{V}})^{-1} \sqrt{V} \rangle_{\ell^2(\text{supp}(V))}\geq \frac{1}{1+\theta (1+\eta)} \langle V, 1 \rangle_{\ell^2(\text{supp}(V))}=\frac{\mathrm{cap}(\Sigma(\calC))}{1+\theta (1+\eta)},
 \end{equation*}
and \eqref{eq:thm4.2disco} follows readily.
\end{proof}

\subsection{Proof of Proposition~\ref{pro:softlocaltimes}}
\label{sec:softlocaltimes}
The main result from this section will be Lemma~\ref{lem:inversesoftlocaltimes}, which corresponds to a version of Proposition~\ref{pro:softlocaltimes} in terms of excursions of random interlacements, instead of local times, and is thus a generalization of \cite[Proposition~5.1]{MR3602841} from $\Z^d$, $d\geq3$, to the class of graphs from \eqref{eq:standingassumptionintro}. The main tool in the proof of \cite[Proposition~5.1]{MR3602841} was the  soft local times method introduced in \cite{MR3420516,MR3126579}, but  it turns out that it will be faster to use the conditional soft local times method from \cite{AlvesPopov}, or more precisely the recent inverse soft local time technique from \cite{PreRodSou}, see Remark~\ref{rk:softlocaltimes},\ref{rk:whyinversesoftlocaltimes}) as to why. Once Lemma~\ref{lem:inversesoftlocaltimes} is proved, one can easily finish the proof of Proposition~\ref{pro:softlocaltimes} by rebuilding the interlacements trajectories from their excursions, combined with Lemma~\ref{lem:inversesoftlocaltimes} and some concentration bounds on the number of excursions from \cite{PreRodSou}.

We start by briefly recalling the setup to apply the soft local times method, which is similar (but not identical) to \cite[Section~5]{PreRodSou}. Fix some $L,K\geq1$, and introducing
\begin{equation*}
   \tilde{U}_z\stackrel{\textnormal{def.}}{=}B(z,8\sqrt{K}L)\text{ and }U_z\stackrel{\textnormal{def.}}{=}B(z,8KL)
\end{equation*}
additionally to the sets in \eqref{eq:defCU}, let us denote by $X^1,X^2,\dots$ the successive trajectories of interlacements hitting $\tilde{U}_z$, started at their first hitting time of $\tilde{U}_z$, as defined above \eqref{eq:defRI} when $A=\tilde{U}_z$. We then let $Y^z=(Y^z_k)_{k\in{\mathbb{N}}}$ be the Markov process obtained by concatenating the trajectories $(X^i_k)_{0\leq k\leq L_{T_{U_z}}^i}$ in increasing order, where $L_{T_{U_z}}^i$ is the last exit time of $U_z$ by $X^i$, which is finite by transience. We define for all $i\in{\N},$ $L\geq1$ and $z\in{\Lambda(L)}$ the successive times for $Y$ of return to $\tilde{U}_z$ and departure from $U_z$ by $R_0^z\leq D_0^z\leq\dots\leq R_k^z\leq D_k^z\leq \dots,$ that is $R_0^z=0$,  and for each $k\geq0,$ $D_{k}^z$ is the first exit time of $U_z$ after time $R_k^z$ for $Y$ and $R_{k+1}^z$ is the first return time to $\tilde{U}_z$ after time $D_k^z$ for $Y$, see \cite[(5.4)]{PreRodSou} with $x=Y$, $B_2=\tilde{U}_z$ and $B_3=U_z$ therein for a more detailed definition. 

Let $H_i^{z}=\inf\{k\in\{R_i^z,\dots,D_{i}^z\}:\,Y_k^z\in{\hat{C}_z}\}$ be the first hitting time of $\hat{C}_z$ after time $R_i^z$, with the convention $\inf\varnothing=\infty$. We define the sequence of excursions hitting $\hat{C}_z$ by letting $i_0=-1$ and $i_{j+1}=\inf\{k> i_j:\,H_k^{z}<\infty\}$ for all $j\geq0$ (depending implicitly on $z$), and taking
\begin{equation}
\label{eq:defWiz}
    W_{j}^{z}=\big\{Y_k^z:H_{i_j}^{z}\leq k\leq L_{i_j}^{z}\big\}\text{ for all }j\in{\mathbb{N}},
\end{equation}
where $L_i^{z}=\sup\{k\in{\{R_i^{z},\dots,D_i^z\}}:\,Y_k^z\in{\tilde{U}_{z}}\}$ is the last hitting time of $\tilde{U}_z$ before $D_i^z$ for each $i\in{\N}$. We also define the clothesline process associated to those excursions hitting $\hat{C}_z$ via
\begin{equation}
\label{eq:clothesline}
    \lambda_j^{z}=(Y_{R_{i_j}}^{z},Y_{D_{i_j}}^{z})\text{ for all }j\in{\mathbb{N}}.
\end{equation}
In other words, $W^{z}$ and $\lambda^{z}$ correspond to the processes introduced in \cite[(5.11)]{PreRodSou} for $B_1=\hat{C}_{z}$, $B_2=\tilde{U}_z$ and $B_3=U_z$, but restricted to the excursions hitting $\hat{C}_z$. We refer to Remark~\ref{rk:softlocaltimes},\ref{rk:whyinversesoftlocaltimes}) as to why this restriction is useful in our proof. The crux of the proof of Proposition~\ref{pro:softlocaltimes} lies in the following lemma.

\begin{Lemme}
\label{lem:inversesoftlocaltimes}
Upon extending the underlying probability space, there exist constants $C,\Cr{CsoftlocaltimesK},\Cl[c]{csoftlocaltimesK2}<\infty$ and $c>0$ such that for any $\eps\in{(0,1/2]}$, $k_0\in{\mathbb{N}}$, $L\geq1$ and $K\geq \Cr{csoftlocaltimesK2}\eps^{-\Cr{CsoftlocaltimesK}}$,  there exist a sequence $(\hat{\lambda}^z,\hat{W}^{z})_{z\in{G}}$ with the same marginals as $(\lambda^z,W^{z})_{z\in{G}}$, and a sequence of events $(\tilde{F}_{z}^{k,\eps})_{k\in{\mathbb{N}},z\in{G}}$ such that $(\hat{\lambda}^z,\hat{W}^{z},(\tilde{F}_{z}^{k,\eps})_{k\in{\mathbb{N}}})$, $z\in{\calC}$, are independent for each $\calC$ as in \eqref{eq:defcalC}, $(\tilde{F}_{z}^{k,\eps})_{k\in{\mathbb{N}}}$ is independent of $\hat{\lambda}^z$ for each $z\in{G}$, and for any $z\in{G}$ and $k\in{\{1,\dots,k_0\}}$, 
    \begin{equation}
        \label{eq:probasoftlocaltimes}
        \P\left(\tilde{F}_{z}^{k,\eps}\right)
        \geq 1-C\exp(-c\eps^2k)
    \end{equation}
    and, letting $k_\pm=\lceil k(1\pm\eps)\rceil$,
    \begin{equation}
    \label{eq:deftildeF}
        \tilde{F}_{z}^{k,\eps}\subset \{\hat{W}_1^z,\dots,\hat{W}_{k_-}^z\}\subset \left\{W_1^z,\dots,W_{k}^z\right\}\subset \{\hat{W}_1^z,\dots,\hat{W}_{k_+}^z\}.
    \end{equation}
\end{Lemme}

An important step to apply the soft local time technique from \cite{MR3420516}  is to verify that the density of the Markov chain considered therein does not fluctuate too much. We first collect this result, and will resume the proof of Lemma~\ref{lem:inversesoftlocaltimes} afterward. Let us define for all $L,K\geq1$, $z\in{G}$,  $x,w\in{\partial\tilde{U}_z}$, $v\in{\partial \hat{C}_z}$  and $y\in{\partial U_z^c}$ such that $\P_x(H_{\hat{C}_z}<T_{U_z},X_{T_{U_z}}=y)>0$ 
\begin{equation}
\label{eq:defbfg}
\mathbf{g}_{(x,y)}^{z}(v,w)\stackrel{\textnormal{def.}}{=}\P_x(X_{H_{\hat{C}_z}\wedge T_{U_z}}=v,X_{L_{\tilde{U}_z}(T_{U_z})}=w\,|\,H_{\hat{C}_z}<T_{U_z},X_{T_{U_z}}=y)
\end{equation}
and
\begin{equation}\label{eq:defbftildeg}
\mathbf{\tilde{g}}^{z}(v,w)\stackrel{\textnormal{def.}}{=}\bar{e}_{\hat{C}_z,\tilde{U}_z}(v)g_{U_z}(v,w)e_{\tilde{U}_z,U_z}(w)\text{ for all }v\in{\partial \hat{C}_z}\text{ and }w\in{\partial\tilde{U}_z},
\end{equation}
where for $A\subset B\subset\subset G$, we write $g_{B}(v,w)\stackrel{\textnormal{def.}}{=}\frac1{\lambda_w}\E_v[\sum_{k=0}^{T_{B}-1}1\{X_k=w\}]$  and $e_{A,B}(w)\stackrel{\textnormal{def.}}{=}\lambda_w\P_{w}(\tilde{H}_{A}>T_{B})$ for the Green function and equilibrium measure associated to the graph with infinite killing on $B^c$, and ${L}_{A}(n)=\sup\{k\in{\{0,\dots,n\}}:X_k\in{A}\}$ is the last exit time of $A$ before time $n$ for each $n\in{\N}$. We abbreviate $\Sigma^{(z)}=\partial \hat{C}_z\times \partial\tilde{U}_z$ for the space on which $\mathbf{g}^{z}_{(x,y)}$ and $\tilde{\mathbf{g}}^{z}$ are defined. To simplify notation, we will use the convention $x/0=\infty$ if $x\neq 0$ and $0/0=1$ in the rest of this section. Let us start with the following consequence of Harnack's inequality, which improves the bound one could deduce from \cite[Lemma~6.5]{DrePreRod2} taking advantage of \eqref{EHIK} and \cite[(2.21),(2.25)]{MR3602841}.
\begin{Lemme}
\label{lem:condsoftlocaltimes}
There exist constants $\Cl{CgzA},\Cl{cgzA}<\infty$ such that for all $L\geq1$, $K\geq \Cr{cgzA}$, $z\in{G}$, for all $x\in{\partial\tilde{U}_z}$, $\sigma\in{\Sigma^{(z)}}$ and $y\in{\partial U_z^c}$ such that $\P_x(H_{\hat{C}_{z}}<T_{U_z},X_{T_{U_z}}=y)>0$,
\begin{equation}
\label{eq:condsoftlocaltimes}
   \Big|\frac{\mathbf{g}_{(x,y)}^{z}(\sigma)}{\mathbf{\tilde{g}}^{z}(\sigma)}-1\Big|\leq \frac{\Cr{CgzA}}{K^{\Cr{cHarnackprecise}/2}}.
\end{equation}
\end{Lemme}
\begin{proof}
By the strong Markov property, one can decompose for all $v\in{\hat{C}_z}$ and $w\in{\partial \tilde{U}_z}$
\begin{equation*}
    \begin{split}
        &\P_x(X_{H_{\hat{C}_z}\wedge T_{U_z}}=v,X_{L_{\tilde{U}_z}(T_{U_z})}=w,H_{\hat{C}_z}<T_{U_z},X_{T_{U_z}}=y)
        \\&=\P_x(H_{\hat{C}_z}<T_{U_z})\P_x(X_{H_{\hat{C}_z}}=v\,|\,H_{\hat{C}_z}<T_{U_z})\P_v(X_{L_{\tilde{U}_z}(T_{U_z})}=w,X_{T_{U_z}}=y)
    \end{split}
\end{equation*}
and
\begin{equation*}
\begin{split}
    \P_v(X_{L_{\tilde{U}_z}(T_{U_z})}=w,X_{T_{U_z}}=y)&=\sum_{k=0}^{\infty}\P_v(L_{\tilde{U}_z}(T_{U_z})=k,X_k=w,X_{T_{U_z}}=y)
    \\&=\lambda_{w}{g}_{U_z}(v,w)\P_w(T_{U_z}<\tilde{H}_{\tilde{U}_z},X_{T_{U_z}}=y),
\end{split}
\end{equation*}
where we used the Markov property at time $k$ in last equality. Moreover, we have
\begin{equation*}
    \P_x(H_{\hat{C}_z}<T_{U_z},X_{T_{U_z}}=y)=\sum_{v'\in{\hat{C}_z}}\P_x(H_{\hat{C}_z}<T_{U_z},X_{H_{\hat{C}_z}}=v')\P_{v'}(X_{T_{{U}_z}}=y)
\end{equation*}
Combining the last three equations, in order to prove \eqref{eq:condsoftlocaltimes} it is sufficient to prove that there exists $\Cr{CgzA},\Cr{cgzA}<\infty$ such that for all $K\geq \Cr{cgzA}$, $v'\in{\hat{C}_z}$, $w\in{\partial \tilde{U}_z}$ and $y\in{\partial U_z^c}$
\begin{equation}
\label{eq:resttoproveharnack1}
    \Big|\frac{\lambda_{w}\P_w(T_{U_z}<\tilde{H}_{\tilde{U}_z},X_{T_{U_z}}=y)}{\P_{v'}(X_{T_{U_z}}=y)e_{\tilde{U}_z,U_z}(w)}-1\Big|\leq \frac{\Cr{CgzA}}{K^{\Cr{cHarnackprecise}/2}}
\end{equation}
and for all $x\in{\partial\tilde{U}_z}$ and $v\in{\partial \hat{C}_z}$
\begin{equation}
\label{eq:resttoproveharnack2}
    \Big|\frac{\P_x(X_{H_{\hat{C}_z}}=v\,|\,H_{\hat{C}_z}<T_{U_z})}{\bar{e}_{\hat{C}_z,\tilde{U}_z}(v)}-1\Big|\leq \frac{\Cr{CgzA}}{K^{\Cr{cHarnackprecise}/2}}.
\end{equation}
Note that $\tilde{U}_z\subset B(v',10\sqrt{K}L)$ and $B(v',KL)\subset U_z$. Therefore, on $\Z^d$, $d\geq3$, \eqref{eq:resttoproveharnack1} corresponds to \cite[Lemma~2.3]{MR3602841} (with $0$, $A$, $U$, $L$ and $K$ replaced by $v'$, $\tilde{U}_z$, $U_z$, $10\sqrt{K}L$ and $\sqrt{K}/10$ respectively therein),  whose proof can easily be adapted to graphs satisfying \eqref{eq:standingassumptionintro}. Indeed, in this proof \cite[(A.2)]{MR3602841} can be replaced by \eqref{EHIK} (and hence we need to replace $|x|/(KL)$ by $(d(x,v')/(KL))^{\Cr{cHarnackprecise}}$ therein), and \cite[(A.7)]{MR3602841} still holds by \eqref{eq:boundentrance} as long as $K\geq \Cr{cgzA}$ for $\Cr{cgzA}$ large enough. Note that the equilibrium measures that appear in \eqref{eq:resttoproveharnack1} and \cite[(2.21)]{MR3602841} are different, but the measure $e_{\tilde{U}_z,U_z}(w)$ is actually what the proof of \cite[Lemma~2.3]{MR3602841} naturally gives, see \cite[(A.11)]{MR3602841}.

Finally, \eqref{eq:resttoproveharnack2} corresponds to \cite[(2.25)]{MR3602841} (with $0$, $A$, $B$, $L$ and $K$ replaced by $z$, $\hat{C}_z$, $\hat{C}_z\cup U_z^c$, $8L$ and $\sqrt{K}/2$ respectively therein) on $\Z^d$, $d\geq3$. This proof can also be easily adapted to graphs satisfying \eqref{eq:standingassumptionintro} since it is a consequence of \cite[Lemma~2.3]{MR3602841} combined with the formula \cite[(2.28)]{MR3602841} which is valid on any transient graphs. The exact bound on the right-hand side of \eqref{eq:resttoproveharnack2} can be found below \cite[(2.28)]{MR3602841}, replacing $K$ by $(\sqrt{K}/2)^{\Cr{cHarnackprecise}}$ therein and assuming $K\geq \Cr{cgzA}$ for $\Cr{cgzA}$ large enough. The equilibrium measure appearing in \eqref{eq:resttoproveharnack2} and \cite[(2.25)]{MR3602841} for similar reason as in \eqref{eq:resttoproveharnack1}.  We leave the details to the reader.
\end{proof} 

\begin{proof}[Proof of Lemma~\ref{lem:inversesoftlocaltimes}]
Let us finish to introduce the setup we will use to apply the soft local times strategy, which is similar to \cite[Section~5.1]{PreRodSou} with $B_1=\hat{C}_z$, $B_2=\tilde{U}_z$, $B_3=U_z$, but with additional conditioning on excursions hitting $\hat{C}_z$. For $z\in{\calC}$, conditionally on $\lambda^{z}$, one can easily show that the process $(W_j^{z})_{j\geq1}$ is an independent process with densities $(\mathbf{g}^{z}_{\lambda_j^{z}})_{j\geq1}$ with respect to $\mu^{z},$ where
\begin{equation*}
    \mu^{z}(A)\stackrel{\textnormal{def.}}{=}\sum_{v\in{\partial \hat{C}_z},w\in{\partial\tilde{U}_z}}\P_v((X_k)_{0\leq k\leq L_{\tilde{U}_z}(T_{U_z})}\in{A}\,|\,X_{L_{\tilde{U}_z}(T_{U_z})}=w)
\end{equation*}
is a measure on the space $S^z$ of finite paths from $\partial \hat{C}_z$ to $\partial\tilde{U}_z$. We refer to below \cite[(5.16)]{PreRodSou} for a very similar computation. Alternatively, one can see $(W_j^{z})_{j\geq1}$ as an inhomogeneous Markov process with transition densities $(\mathbf{g}^{z}_{\lambda_j^{z}})_{j\geq1}$. Here with a slight abuse of notation we identify $\mathbf{g}^{z}_{\lambda}((x_i)_{i\leq k})$ with $\mathbf{g}_{\lambda}^{z}(x_1,x_k)$ for any nearest neighbor path $x_0,\dots,x_k$ in $S^z$, as well as $\mathbf{g}^{z}_{\lambda}((x'_i)_{1\leq i\leq k'},(x_i)_{1\leq i\leq k})$ with $\mathbf{g}_{\lambda}^{z}(x_1,x_k)$ for any other path $x'_0,\dots,x'_k$ in $S^z$ depending on the context. Under $\P(\cdot\,|\,\lambda^{z})$, and up to extending the probability space $\P$ to carry an independent family $({\xi}^z_i)_{i\in{\mathbb{N}},z\in{G}}$ of i.i.d.\ Exp($1$) random variables and additional independent Poisson point processes as in \cite[(4.10)]{PreRodSou}, let us denote by $\eta^z$ the Poisson point process on $\Sigma^{(z)}\otimes [0,\infty)$ obtained in \cite[(4.11)]{PreRodSou} with $T$, $\hat{\xi}_i$, $Z_i$ and $g_i$ replaced by $2k_0$, $\xi_i^z$, $W_i^z$ and $\mathbf{g}_{\lambda_i^{z}}^{z}$ respectively therein. One can prove similarly as in the proof of \cite[Lemma~5.3 and~Proposition~5.4]{PreRodSou} that for each $z$ as above, $\eta^{z}$ is independent of $(\eta^{z'})_{z'\notin{B(z,16KL)}}$, which directly implies that $(\eta^z)_{z\in{\calC}}$ are independent for each $\calC$ as in \eqref{eq:defcalC}. Furthermore, applying the soft local time technique under $\P(\cdot\,|\,\lambda^{z})$ for the densities $g_{\lambda}^z$, one obtains recursively in $k$ the random variables $\xi_k^z\in{[0,\infty)}$ and $(W_k^z,v_k^z)\in{\text{supp}(\eta)}$, and we refer to \cite[Proposition~4.3]{MR3420516} for the precise procedure (with $W_k^z$ playing the role of $z_k$ therein). Note that the process $(W_k^z)_{k\in{\mathbb{N}}}$ obtained via this procedure is exactly the process from \eqref{eq:defWiz} by our definition of $\eta^z$ and \cite[Lemma~4.2]{PreRodSou}, and similarly for $\xi$. The only fact from this construction we will need here is that letting
\begin{equation}
    \label{eq:softlocaltimes}
        G_{n,z}^{\text{RI}}(\sigma)\stackrel{\textnormal{def.}}{=}\sum_{i=1}^n\xi_i^z\mathbf{g}^z_{\lambda_i^z}(\sigma)\text{ for all }n\in{\mathbb{N}},z\in{\calC}\text{ and }\sigma\in{S^{z}},
\end{equation}
one has for any $n\geq1$ and $(\sigma,v)\in{\text{supp}(\eta^z)}$ that $(\sigma,v)\in{\{(W_1^z,v_1^z),\dots,(W_n^z,v_n^z)\}}$ if and only if $G_{n,z}^{\text{RI}}(\sigma)\geq v$. In particular for any $t>0$ and $n\in{\mathbb{N}}$, letting $N^z(t)\stackrel{\textnormal{def.}}{=}\eta^z(\{(\sigma,v):v\leq t\tilde{g}^z(\sigma)\})$ one has by contraposition
\begin{equation}
\label{eq:Nzbig}
    N^z(t)> n\Longrightarrow \inf_{\sigma\in{S^z}} G_{n,z}^{\text{RI}}(\sigma)\leq t\tilde{g}^z(\sigma)\Longrightarrow \sum_{i=1}^n\xi_i^z\leq \frac{t}{1-\tfrac{\eps}{10}},
\end{equation}
where the last inequality follows from \eqref{eq:condsoftlocaltimes} if $\Cr{CsoftlocaltimesK}=2/\Cr{cHarnackprecise}$ and $\Cr{csoftlocaltimesK2}$ is large enough so that $K\geq \Cr{csoftlocaltimesK2}\eps^{-\Cr{CsoftlocaltimesK}}$ implies $K\geq \Cr{cgzA}$ and $\Cr{CgzA}K^{-\Cr{cHarnackprecise}/2}\leq \eps/10$.
Similarly,
\begin{equation}
\label{eq:Nzsmall}
    N^z(t)< n\Longrightarrow \sup_{\sigma\in{S^z}} G_{n,z}^{\text{RI}}(\sigma)\geq t\tilde{g}^z(\sigma)\Longrightarrow \sum_{i=1}^n\xi_i^z\geq \frac{t}{1+\tfrac{\eps}{10}}.
\end{equation}
Upon extending the underlying probability space, for each $z\in{G}$, let $\hat{\lambda}^z\stackrel{\textnormal{def.}}{=}(\hat{\lambda}_i^{z})_{i\geq1}$ be a copy of $\lambda^z\stackrel{\textnormal{def.}}{=}(\lambda_i^{z})_{i\geq1},$ which is independent in $z\in{G}$, and independent of $(\eta^z)_{z\in{G}}$. Let us then denote by $(\hat{W}_j^z)_{j\geq1}$ the Markov chain obtained from applying \cite[Proposition~4.4]{PreRodSou} under $\P(\cdot\,|\,\lambda^{z},\hat{\lambda}^z)$ with $T$, $Z_i$ $g_i$ and $\tilde{g}_i$ replaced by $2k_0$, $W_i^z,$ $\mathbf{g}_{\lambda_i^{z}}^{z}$ and ${\mathbf{g}}_{\hat{\lambda}_i^{z}}^{z}$ respectively therein. Then under $\P(\cdot\,|\,\lambda^{z},\hat{\lambda}^z)$, there exists another family of i.i.d.\ Exp($1$) random variables $(\hat{\xi}^z_i)_{i\in{\mathbb{N}},z\in{G}}$, such that defining $\hat{G}_{n,z}^{\text{RI}}$ as in \eqref{eq:softlocaltimes} but replacing $\xi$ and $\lambda$ by $\hat{\xi}$ and $\hat{\lambda}$ respectively, the event
\begin{equation}
    \label{eq:probasoftlocaltimes2}
        \hat{G}_{k_-,z}^{\text{RI}}(\sigma)\leq G_{k,z}^{\text{RI}}(\sigma)\leq \hat{G}_{k_+,z}^{\text{RI}}(\sigma)\text{ for all }\sigma\in{S^z}
\end{equation}
is included in the event on the right hand side of \eqref{eq:deftildeF}. Moreover, since $\hat{W}^z$ and $\hat{\xi}^z$ are also obtained by applying the soft local time method to $\eta^z$ but for $g_{\hat{\lambda}_i^z}^z$ instead of $g_{\lambda_i^z}^z$, see the proof of \cite[Proposition~4.4]{PreRodSou}, the implications \eqref{eq:Nzbig} and \eqref{eq:Nzsmall} are still valid when replacing $G^{\text{RI}}_{\cdot}$ by $\hat{G}^{\text{RI}}_{\cdot}$ and $\xi^z$ by $\hat{\xi}^z$. We now define the event 
\begin{equation*}
     \tilde{F}_{z}^{k,\eps}\stackrel{\textnormal{def.}}{=}\left\{
     \begin{gathered}
     N^z\big(k\big(1-\tfrac{\eps}{10}\big)\big)<k<N^z\big(k\big(1+\tfrac{\eps}{10}\big)\big),
     \\ N^z\big(k_+\big(1-\tfrac{\eps}{10}\big)\big)<k_+,N^z\big(k_-\big(1+\tfrac{\eps}{10}\big)\big)>k_-
     \end{gathered}
     \right\}.
\end{equation*}
Note that since $\tilde{F}_{z}^{k,\eps}$ is $\eta^z$-measurable and $\hat{W}^z$ is measurable with respect to $\eta^z$ and $\hat{\lambda}^z$, we have that $(\hat{\lambda}^z,\hat{W}^{z},(\tilde{F}_{z}^{k,\eps})_{k\in{\mathbb{N}}})_{z\in{\calC}}$ are independent for each $\calC$ as in \eqref{eq:defcalC}, and $(\tilde{F}_{z}^{k,\eps})_{k\in{\mathbb{N}}}$ is independent of $\hat{\lambda}^z$ for each $z\in{G}$. Furthermore, noting that $N^z(t)$ is Poi($t$) distributed, one can easily deduce \eqref{eq:probasoftlocaltimes} from large deviations bounds on Poisson random variables, see for instance \cite[p.21-23]{MR3185193}. Finally, on the event $\tilde{F}_z^{k,\eps}$ by \eqref{eq:condsoftlocaltimes} (if $\Cr{csoftlocaltimesK2}$ is large enough), \eqref{eq:Nzbig} and \eqref{eq:Nzsmall} we have for all $\sigma\in{S^z}$
\begin{equation}
\label{eq:boundsoftlocaltimes}
    G_{k,z}^{\text{RI}}(\sigma)\leq \big(1+\tfrac{\eps}{10}\big)\tilde{g}^z(\sigma)\sum_{i=1}^k\xi_i^z\leq \frac{(1+\tfrac{\eps}{10})^2}{1-\tfrac{\eps}{10}}k\tilde{g}^z(\sigma)\leq \big(1-\tfrac{\eps}{10}\big)\tilde{g}^z(\sigma)\sum_{i=1}^{k_+}\xi_i^z\leq \hat{G}_{k_+,z}^{\text{RI}}(\sigma),
\end{equation}
where in the third inequality we additionally used that $k(1+\tfrac\eps{10})^3\leq k_+(1-\tfrac{\eps}{10})^3$ for $k$ large enough. Using similar lower bounds, one deduces that $\tilde{F}_z^{k,\eps}$ implies \eqref{eq:probasoftlocaltimes2}, and so \eqref{eq:deftildeF} is satisfied.
\end{proof}

The hastened reader might be tempted to directly deduce \eqref{eq:deftildeF} from \eqref{eq:condsoftlocaltimes} and \eqref{eq:probasoftlocaltimes2} for an event $\tilde{F}_z^{k,\eps}$ consisting of bounds of the type $(1-\tfrac\eps{10})\leq \sum_{i=1}^k\xi_i^z\leq(1+\tfrac\eps{10})$, as well as similar bounds on $\sum_i\hat{\xi}^z_i$. However, although one can easily prove that both $(\xi_i^z)_{i\in{\mathbb{N}}}$ and $(\hat{\xi}_i^z)_{i\in{\mathbb{N}}}$ are independent in $z\in{\calC}$ for any $\calC$ as in \eqref{eq:defcalC}, it is not clear if they are jointly independent in $z\in{\calC}$, and \eqref{eq:Nzbig} and \eqref{eq:Nzsmall} solve this problem by replacing these bounds on $\xi$ and $\hat{\xi}$ by bounds on $N^z$, which is independent in $z\in{\calC}$, similarly as in \cite[Lemma~2.1]{MR3126579}.

Lemma~\ref{lem:inversesoftlocaltimes} lets us decouple the excursions of random interlacements, whereas we want to decouple the local times of interlacements at times of order $T^z(u)$ in Proposition~\ref{pro:softlocaltimes}. As we now explain, one can deduce from \cite[Proposition~A.9]{PreRodSou} and Lemma~\ref{lem:inversesoftlocaltimes} for $\mathcal{C}=\{z\}$ that the two statements are essentially equivalent.

\begin{proof}[Proof of Proposition~\ref{pro:softlocaltimes}]
Using additional randomness independently for each $z\in{G}$, one can construct from $(\hat{W}_i^z,\hat{\lambda}_i^z)_{i\geq1}$ an interlacement process $\hat{\omega}^z$ such that $(\hat{W}_i^z)_{i\geq1}$ are the parts of the excursions of $\hat{\omega}^z$ from $\tilde{U}_z$ to $\partial U_z^c$ from their first hitting time of $\hat{C}_z$ to their last exit time of $\tilde{U}_z$, and $(\hat{\omega}^z)_{z\in{\calC}}$ are independent for each $\calC$ as in \eqref{eq:defcalC}. In other words, we want that $(\hat{W}_i^z,\hat{\lambda}_i^z)_{i\geq1}$ can be obtained from $\hat{\omega}^z$ in exactly the same way as $(W_i^z,\lambda_i^z)_{i\geq1}$ was obtained from $\omega$ in \eqref{eq:defWiz} and \eqref{eq:clothesline}. One can construct $\hat{\omega}^z$ from $(\hat{W}_i^z)_{i\geq1}$ and $\hat{\lambda}^z$ as follows: first define an interlacement process $\tilde{\omega}^z$ such that $\hat{\lambda}^z$ is the clothesline process associated with $\tilde{\omega}^z$, similarly as in \eqref{eq:clothesline}, and then writing $\tilde{X}^z_1,\tilde{X}^z_2,\dots$ for the trajectories of $\tilde{\omega}^z$ hitting $\hat{C}_z$, one can replace the excursions of $\tilde{X}^z_k$ between $\partial\tilde{U}_z$ and $\partial U_z^c$ hitting $\hat{C}_z$ by new excursions for all $k\in{\mathbb{N}}$, so that the successive subexcursions started when hitting $\hat{C}_z$ and stopped when last visiting $\tilde{U}_z$ are exactly $(\hat{W}^z_i)_{i\in{\mathbb{N}}}$. We refer to below \cite[(5.34)]{PreRodSou} for a precise description of this procedure, which only needs to be slightly modified to take into account that each of the excursions $\hat{W}_i^z$ we consider here is non-empty, and thus one only needs to modify the excursions of $\tilde{X}^z_k$ from $\partial \tilde{U}_z$ to $\partial{U}_z^c$ which hit $\hat{C}_z$, instead of all the excursions from $\partial \tilde{U}_z$ to $\partial{U}_z^c$. Note also that in this construction since $(\tilde{F}_{z}^{k,\eps})_{k\in{\mathbb{N}}}$ is independent of $\hat{\lambda}^z$ for each $z\in{G}$, one can also assume that $(\tilde{F}_{z}^{k,\eps})_{k\in{\mathbb{N}}}$ is independent of $\tilde{\omega}^z$ for each $z\in{G}$. We leave the details to the reader.

Let us denote $\hat{\omega}_u^z$ the restriction of the point process $\hat{\omega}^z$ to trajectories with label at most $u$, defined similarly as above \eqref{eq:defRI}, and by $(\hat{\ell}_{x,u}^z)_{x\in{G},u>0}$ the local times defined exactly as in \eqref{eq:defLxu} and \eqref{eq:localtimes} but for the process $\hat{\omega}^z$ instead of $\omega$. For each $z\in{G}$, let ${M}_u^z$ be the number of $k\geq1$ such that the excursion $W_k^z$ corresponds to a trajectory in $\omega_u$, i.e.\ the total number of excursions from $\partial\tilde{U}_z$ to $\partial{U}_z^c$ which hit $\hat{C}_z$ for random interlacements at level $u$, that one can formally define as in \cite[(5.12)]{PreRodSou}.  In other words defining $V_i^z$ as the number of $k\geq1$ such that the excursion $W_k^z$ is an excursion of the $i$-th trajectory of the random interlacement process $\omega$ hitting $\hat{C}_z$,  for each $i\geq1$, then $M_u^z=\sum_{i=1}^{N^{\hat{C}_z}_u}V_i^z$, see above \eqref{eq:defRI}. Define similarly $\hat{M}_u^z$ but replacing $\omega_u$ by $\hat{\omega}^z_u$. By construction $\hat{M}_u^z$ actually depends only on $\tilde{\omega}^z$, and is thus independent of $(\tilde{F}_{z}^{k,\eps})_{k\in{\mathbb{N}}}$ for each $z\in{G}$. Let us also define $\hat{T}^z(u)$ similarly as $T^z(u)$ in \eqref{eq:defTz} but replacing $\ell$ by $\hat{\ell}^z$. Abbreviating $u_z\stackrel{\textnormal{def.}}{=}u\mathrm{cap}(\hat{C}_z)\E[V_1^z]$ and $\hat{M}_{u,\pm}^z\stackrel{\textnormal{def.}}{=}\frac1{1\pm\eps}\hat{M}^z_{u(1\pm\eps)^6}$, we then let for each $z\in{G}$
\begin{equation}
    \label{eq:defFuz}
    F_z^{u,7\eps}\stackrel{\textnormal{def.}}{=}\tilde{F}_{z}^{\hat{M}_{u,+}^z,\eps}\cap \tilde{F}_{z}^{\hat{M}_{u,-}^z,\eps}\cap \left\{\!\!
    \begin{array}{c}
    \hat{M}_{u,+}^z\leq u_z(1+\eps)^6,\hat{M}_{u,-}\geq u_z(1-\eps)^6,\\
    u(1-\eps)< \hat{T}^z(u)\leq u(1+\eps),
    \\(1-\eps)\hat{M}_{u,+}^z\geq \hat{M}_{u(1+\eps)}^z,(1+\eps)\hat{M}_{u,-}^z\leq \hat{M}_{u(1-\eps)}^z
    \end{array}\!\!\right\}.
\end{equation}
It follows from Lemma~\ref{lem:inversesoftlocaltimes} that $((F_{z}^{u,7\eps})_{u>0},\hat{\ell}^z_{\cdot})$, $z\in{\calC}$, are independent for each $\calC$ as in \eqref{eq:defcalC}, and let us now verify that $F_z^{u,7\eps}$ satisfies \eqref{eq:boundprobaFzueps} and \eqref{eq:defeventF} for each $z\in{G}$. Similarly as in the proof of \cite[Lemma~5.6]{PreRodSou}, it follows from \cite[Proposition~A.9]{PreRodSou}, whose proof is valid on any transient graphs, that for all $v>0$, writing $v_z=v\mathrm{cap}(\hat{C}_z)\E[V_1^z]$
\begin{equation}
\label{eq:boundVuz}
\P\big(|M_v^z-v_z|\geq {\eps}v_z\big)\leq C\exp\big(-c\eps^2v\mathrm{cap}(\hat{C}_z)\big)
\leq C\exp(-c'\eps^2uL^{\nu})
\end{equation} 
for some constants $C<\infty$ and $c,c'>0$, where we used \eqref{eq:capball} in the last inequality. Note that condition \cite[(A.39)]{PreRodSou} is indeed satisfied since $V_1^z$ is larger than $1$ and dominated by a geometric random variable with parameter $\inf_{x\in{\partial U_z^c}}P_x(H_{\hat{C}_z}=\infty)\geq  c$ for $K$ large enough in view of \eqref{eq:boundentrance}, which we can assume by increasing the value of the constant $\Cr{csoftlocaltimesK}$ if needed. Using a union bound, Lemma~\ref{lem:inversesoftlocaltimes} twice, the independence of $\hat{M}_u^z$ and $(\tilde{F}_{z}^{k,\eps})_{k\in{\mathbb{N}}}$, the bound $u_z\geq u\Cr{ccapball}(8L)^{\alpha}$ which follows from \eqref{eq:capball}, \eqref{eq:boundVuz} four times for $v=u(1+\eps)^6,u(1+\eps),u(1-\eps)$ and $u(1-\eps)^6$, \eqref{eq:thm4.2disco} (which is in force up to increasing the value of the constant $\Cr{csoftlocaltimesK}$), replacing $u$, $v$ and  $\calC$ therein by $u(1\pm\eps)$, $u$ and $\{z\}$, and the inequalities $(1-\eps)^2(1+\eps)^3\geq1$ and $(1+\eps)^2(1-\eps)^3\leq1$ (for $\eps$ small enough), we obtain \eqref{eq:boundprobaFzueps} for $u\eps^2L^{\nu}$ large enough and $\eps$ small enough, which can be assumed without loss of generality up to increasing the constant $\Cr{Csoftlocaltimes}$. Let us define $L_{x}(W_k^z)$, resp.\ $L_{x}(\hat{W}_k^z)$, as the number of times $W_k^z$, resp.\ $\hat{W}_k^z$ visits $x$, similarly as in \eqref{eq:defLxu}, and let 
$$\ell_{x}(W_k^z)\stackrel{\textnormal{def.}}{=}\frac1{\lambda_x}\sum_{n=1}^{L_x(W_k^z)}\mathcal{E}_x^{n+\sum_{i=1}^{k-1}L_x(W_i^z)}\text{ for all }x\in{\hat{C}_z},$$
and define similarly $\ell_{\cdot}(\hat{W}_k^z)$ but for $L_{\cdot}(\hat{W}_k^z)$ instead of $L_{\cdot}(W_k^z)$. Then in view of \eqref{eq:localtimes},  $\ell_{x,v}$ is equal to the sum of $\ell_x(W_k^z)$ over $k\leq M_v^{z}$ by definition, and similarly $\hat{\ell}^z_{x,v}$ is the sum of $\ell_x(\hat{W}_k^z)$ over $k\leq \hat{M}_v^z$. Therefore on the event ${F}_z^{u,7\eps}$ we have by \eqref{eq:deftildeF} 
\begin{equation*}
\begin{split}
    &\sum_{x\in{\hat{C}_z}}\bar{e}_{\hat{C}_z}(x)\sum_{k=1}^{\hat{M}_{u,+}^z}\ell_x(W_k^z)\geq \sum_{x\in{\hat{C}_z}}\bar{e}_{\hat{C}_z}(x)\sum_{k=1}^{(1-\eps)\hat{M}_{u,+}^z}\ell_x(\hat{W}_k^z)
    \\&\geq \sum_{x\in{\hat{C}_z}}\bar{e}_{\hat{C}_z}(x)\sum_{k=1}^{\hat{M}_{u(1+\eps)}^z}\ell_x(\hat{W}_k^z)=\sum_{x\in{\hat{C}_z}}\bar{e}_{\hat{C}_z}(x)\hat{\ell}^z_{x,u(1+\eps)}\geq u,
\end{split}
\end{equation*}
where we abbreviate $\sum_{k=1}^x=\sum_{k=1}^{\lceil x\rceil}$ for any $x\geq0$, and we used \eqref{eq:defTz} and $\hat{T}^z(u)\leq u(1+\eps)$ in the last inequality. In view of \eqref{eq:defTz}, defining $M_{v}^{z,-}=\sum_{i=1}^{N_v^{\hat{C}_z}-1}V_i^z$ for any $v>0$, one deduces that ${M}_{T^z(u)}^{z,-}\leq \hat{M}_{u,+}^z$.  Moreover, in view of \eqref{eq:localtimes}, if a trajectory of $\omega$ hits $\hat{C}_z$ at level $v$, then $\ell_{x,v}^-$ is equal to the sum of $\ell_x(W_k^z)$ over $k\leq M_v^{z,-}$ by definition. Therefore, on the event ${F}_z^{u,7\eps}$ for any $x\in{\hat{C}_z}$
\begin{equation}
\label{eq:finalbonudsoftlocaltimes}
    \ell^-_{x,T^z(u)}= \sum_{k=1}^{{M}_{T^z(u)}^{z,-}}\ell_x(W_k^z)\leq \sum_{k=1}^{\hat{M}_{u,+}^z}\ell_x(W_k^z)\leq \sum_{k=1}^{(1+\eps)\hat{M}_{u,+}^z}\ell_x(\hat{W}_k^z)\leq \hat{\ell}^z_{x,u(1+\eps)^6}.
\end{equation}
Using the inequality $(1+\eps)^6\leq 1+7\eps$, valid for $\eps$ small enough, and proceeding similarly for the lower bound, we readily deduce \eqref{eq:defeventF} up to a change of variable for $\eps$.
\end{proof}

\begin{Rk}
\phantomsection
\label{rk:softlocaltimes}
    \begin{enumerate}
        \item\label{rk:whyinversesoftlocaltimes} One does not need to use the conditional soft local times method from \cite{AlvesPopov}, and in fact one could instead use the arguments from \cite[Section~2]{MR3126579} similarly as in \cite[Section~5]{MR3602841}. The main difference between these two methods is the random walk quantity which needs to be controlled. Here we need to bound the quantity $\mathbf{g}_{(x,y)}^{z}(v,w)$ from \eqref{eq:defbfg}, and it turns out that on the class of graphs from \eqref{eq:standingassumptionintro}, this is slightly easier than to prove the corresponding bound from \cite[(2.5)]{MR3126579}, see \cite[(5.6)]{MR3602841} for a proof on $\Z^d$, which is the main reason we used conditional soft local times. Furthermore, the bounds on soft local times, see \eqref{eq:boundsoftlocaltimes}, let us directly compare $\ell_{\cdot}$ with a process $\hat{\ell}^z_{\cdot}$ which has the same law as $\ell_{\cdot}$, instead of a process $\hat{\ell}^z$ with constant clothesline as in \cite[Lemma~2.1]{AlvesPopov}.
        
        Similarly, one never really needs to use the inverse soft local times technique from \cite[Section~4]{PreRodSou}. The main reason we use it in this article is that it does not add any complexity to the current proof, and otherwise one would need to use a different coupling between $\ell$ and the variables $\hat{\ell}^z$, $z\in{\calC}$, for each possible choice of the set $\calC$ as in \eqref{eq:defcalC} that we are trying to decouple, which would make equations such as \eqref{eq:replaceFPPbyiid} or \eqref{eq:defeventH} a bit more cumbersome to write. Note that the inverse soft local times we use here slightly differs from the one in \cite[Section~5]{PreRodSou} (or in fact the one from \cite[Section~3]{AlvesPopov}) since in \eqref{eq:defWiz} we restrict the excursion process to excursions hitting $\hat{C}_z$, contrary to \cite[(5.6)]{PreRodSou}. When trying to follow the strategy from \cite{PreRodSou}, one would need to change the definition of the density $\mathbf{g}^z$ in \eqref{eq:defbfg} by removing the conditioning on $H_{\hat{C}_z}<T_{U_z}$, and add a new term of the type $\P_x(H_{\hat{C}_z}>T_{U_z}\,|\,X_{T_{U_z}}=y)$ when $v$ is equal to some cemetery state, see \cite[(5.15),(5.16)]{PreRodSou}, but a bound similar to \eqref{eq:condsoftlocaltimes} for this new definition of \eqref{eq:defbfg} seems out of reach.
        \item\label{rk:othersoftlocaltimes} The main interest of the bound \eqref{eq:condsoftlocaltimes} is that the right-hand side can be made arbitrarily small as $K\rightarrow\infty,$ which implies \eqref{eq:Nzbig} and \eqref{eq:Nzsmall} and thus allows us to bound all the soft local times $G_{n,z}^{\text{RI}}(\sigma)$, $\sigma\in{\Sigma^{(z)}}$, at once, instead of bounding each of them individually. One could still bound each of these $CL^{2\alpha}$ soft local times individually by the exact same reasoning as in \cite[Lemma~5.5]{PreRodSou} (the only time the fact that $G=\Z^d$ therein was used is in the proof of \cite[Lemmas~A.2~and A.3]{PreRodSou}, but these can be replaced by Lemma~\ref{lem:condsoftlocaltimes} for $K= \Cr{cgzA}$ in our context), which would lead to a bound similar to \eqref{eq:boundprobaFzueps} but with an additional term $L^{2\alpha}$ in front of the exponential, and which is now valid for any $K\geq \Cr{csoftlocaltimesK}$, for some constant $\Cr{csoftlocaltimesK}<\infty$ not depending on $\eps$. If  $L^{\nu}\geq C(u\eps^2)^{-1}\log(1/(u\eps^2))$ for some constant $C<\infty$ large enough, this means that \eqref{eq:boundprobaFzueps} remains valid for $K\geq \Cr{csoftlocaltimesK}$. However, the additional polynomial term $L^{2\alpha}$ is actually detrimental when $L$ is of order $(u\eps^2)^{-\frac1\nu}$ and would eventually weaken Theorem~\ref{the:FPP} in certain cases by logarithmic factors, but the fact that the previous bound holds for $K\geq \Cr{csoftlocaltimesK}$, instead of $K\geq \Cr{csoftlocaltimesK}\eps^{-\Cr{CsoftlocaltimesK}}$, could actually strengthen Theorem~\ref{the:FPP} in other cases, see Remark~\ref{rk:endthm},\ref{rk:removingdependenceonzeta}) for details. It is an interesting question whether one can prove a version of \eqref{eq:boundprobaFzueps} when $K$ is of constant order and without any additional polynomial factor in front of the exponential, that is a bound which would still be relevant when $L$ is of order $(u\eps^2)^{-\frac1\nu}$. 
        \item\label{rk:roleTzu} The role of the time $T^z(u)$ in Proposition~\ref{pro:softlocaltimes} is to provide a level at which the local times of $\omega$ can be bounded by the local times of $\hat{\omega}^z$ on an event only depending on $\hat{\omega}^z$, and thus independent in $z\in{\calC}$ for any $\calC$ as in \eqref{eq:defcalC}, see \eqref{eq:finalbonudsoftlocaltimes} and above. Actually, Proposition~\ref{pro:softlocaltimes} is not true when replacing the level $T^z(u)$ by $u$, otherwise one could prove Theorem~\ref{the:FPP} without ever using Proposition~\ref{pro:thm4.2disco}, and would obtain a version of \eqref{eq:bounddtvfinal} with $F_{\nu}(L|\sqrt{v}-\sqrt{u}|^{\frac2\nu})$ replaced by $L(\sqrt{v}-\sqrt{u})^{\frac2{\nu}}$, which is not true at least for $\nu\leq 1$ by the lower bound in \eqref{eq:boundonPEintro} (see also \eqref{eq:bounddtvnointro} below). In other words, the levels $T^z(u)$, $u>0$, carry most of the long-range dependence of $\omega$, in a somewhat similar way that the harmonic average of the Gaussian free field was carrying the long-range dependence of the Gaussian free field in \cite{GRS21,MR3417515}, see Remark~\ref{rk:endthm},\ref{rk:extensiontoGFF}) for details.
    \end{enumerate}
\end{Rk}

\subsection{Proof of Proposition~\ref{pro:entropy}}
\label{sec:entropy}
We use the same general strategy as in the case $G=\Z^d$, $d\geq4$, from \cite[Proposition~4.3]{GRS21}, which we now very roughly describe. Recalling the notation from \eqref{eq:defCU} and following a renormalization scheme inspired by \cite[Section~7]{MR3420516}, define a sequence of dyadic scales $L_n\approx 2^n$, see \cite[(4.29)]{GRS21}, so that each path $\gamma$ from $[0,L_n)^d$ to $\partial [-L_n,2L_n)^d$ can be split into a path from $y+[0,L_{n-1})^d$ to $y+\partial [-L_{n-1},2L_{n-1})^d$ and a path from $z+[0,L_{n-1})^d$ to $z+\partial [-L_{n-1},2L_{n-1})^d$ for some $y\in{[0,L_n)^d}$ and $z\in{[-L_n+L_{n-1},2L_n-L_{n-1})^d}$ suitably chosen so that the parts of $\gamma$ in $y+[-L_{n-1},2L_{n-1})^d$ and $z+[-L_{n-1},2L_{n-1})^d$ are at distance at least $L_{n}/n^2$ from one another. Iterating this procedure gives for a fixed $k$ large enough (depending on $KL$) and all $n\geq k$ a sequence of $2^{n-k}$ points $x_i$ such that $\gamma$ connects $x_i+[0,L_k)^d$ to $x_i+\partial [-L_k,2L_k)^d$, and the sets $x_i+[-L_k,2L_k)^d$ are at distance at least $2KL$ from one another (and in fact even further apart for most of them, see \cite[(4.32)]{GRS21}). We refer to Figure~\ref{fig:renscheme} for an illustration of the construction of these boxes (with Euclidean balls instead of boxes therein). Defining $\mathcal{C}=\{x_i,i\leq 2^{n-k}\}$ eventually yields \cite[Proposition~4.3]{GRS21} on $\Z^d$, $d\geq4$. There are however several small differences between the previously described proof from \cite[Proposition~4.3]{GRS21} and our proof of Proposition~\ref{pro:entropy}, and in order to give insights on the strategy of this rather technical proof, we now motivate these differences.

The main difficulty when considering graphs that satisfy \eqref{eq:intro_sizeball}, \eqref{eq:intro_Green}, $\nu>1$, and \eqref{p0} instead of $\Z^d$, $d\geq4$, is to extend the version of the renormalization scheme from \cite[Section~7]{MR3420516} used in \cite{GRS21} in this context. In particular, one needs to find a replacement for the "square boxes" $z+[0,L_{n-1})^d$ previously used to cover $\partial [-L_n+L_{n-1},2L_n-L_{n-1})^d$ with a finite number of boxes included in $[-L_n,2L_n)^d$, as these square boxes do not exist on general graphs. This can be  achieved by taking "round balls" (balls for $d$, which is the Euclidean distance on $\Z^d$) and only asking to cover using "small" balls of size $L_{n-1}/n^2$, see  \eqref{eq:reasonIkxEkx}. In turn, this means that we need to consider paths from $B(x,L_n/(n+1)^2)$ to $B(x,2L_n)^c$, instead of $[0,L_n)^d$ to $[-L_n,2L_n)^d$. Moreover, the number of boxes needed to cover is of order $n^2$ instead of constant, but this is actually not a problem when computing the entropy cost \eqref{eq:cardA}, see \eqref{eq:recursiveentropy} and below. 

In \cite[(4.11)]{GRS21}, on $\Z^d$, $d\geq4$, there is an additional term $u(KL)^{-1}\stackrel{\textnormal{def.}}{=}\log(KL)^{-2}$ in the number of points in each set $\mathcal{C}\in{\A}$ compared to \eqref{eq:cardC}. This is essentially due to the fact that they take $L_0=1$ and then actually start their renormalization scheme at a scale $k$ such that the typical space $L_{k-1}k^{-2}$ between two boxes at scale $k$ is of order $KL$, see \cite[(4.46)]{GRS21}, and then the number of points $2^{n-k}$ in $\calC$ at scale $n\geq k$ is of order $L_n/(KL\log(LK)^2)$ for $n\geq k$.
Using an idea from \cite[(2.16)]{AndPre}, we will introduce an additional parameter $\kappa$ in the definition of $L_n$ in \eqref{eq:defLkMk}, and start the renormalization at $L_0=\kappa LK$ instead of $L_0=1$, which makes the typical space between two boxes at level $0$ larger than $\kappa L_0$ which is much larger than $16KL$ when $\kappa$ is large enough, and then the number of points $2^n$ in $\calC$ at scale $n$  is of order $L_n/(KL)$. We refer to the proof of \eqref{eq:distancetau1tau0} below to determine where this is necessary.

Finally, Proposition~\ref{pro:entropy} is satisfied for any $\nu>0$, and not only for $\nu>1$ as in the case $d\geq4$ of \cite[Proposition~4.3]{GRS21}, which answers positively a question from \cite[Remark~4.4,2)]{GRS21}. The main reason for this improvement is that we show that there are at most of order $L_i$ points in $\calC$ at distance $L_i$ from some fixed $y\in{\calC}$  for each $i\leq n$, see \eqref{eq:Byi} and below, whereas this was only proved for the number of points at distance $L_i/\log(L_i)^2$ from $y$ in \cite[(4.32)]{GRS21}. This is achieved using \eqref{eq:defproperembedding2} below, which controls more precisely the distance between two boxes at scale $i$, and additionally allows us to avoid the use of "shapes" from \cite[Section~4.2]{GRS21}. This method yields \eqref{eq:capC} but without the correct constant $\tfrac{\Cr{cbeta}}{1+\eta}$ in the limit as $x\rightarrow\infty$ of $H_{\nu}(x,y,\eta)$ when $\nu\leq1$, see Corollary~\ref{cor:entropy}, and we will explain after the proof of Corollary~\ref{cor:entropy} how to recover this constant.

Further differences are that we compute the capacity of $\Sigma(\calC)$ at scale $n$ using the information of all scales $k\leq n$ at once, see \eqref{eq:computecapacity} and above, instead of computing the capacity recursively in $n$ as in \cite[Lemma~4.9]{GRS21}, and that when $\nu>1$ we also bound the capacity of $\Sigma(\calC)$ by $cNL^{\nu-1}$ in \eqref{eq:capC}, for some constant $c=c(K)$, instead of $cN$ in \cite[(4.16)]{GRS21}, but it would actually be possible to adapt the proof of \cite{GRS21} to also obtain $cNL^{\nu-1}$ therein.

Let us now formally start the proof of Proposition~\ref{pro:entropy}, and we first introduce a renormalization scheme similar to the one from \cite[Section~2.2]{AndPre}, which is inspired by \cite[Section~7]{MR3420516}. Define recursively
\begin{equation}
\label{eq:defLkMk}
 L_0\stackrel{\textnormal{def.}}{=}\kappa LK,\ l_k\stackrel{\textnormal{def.}}{=}\frac{L_k}{(k+1)^2}\text{ and }L_{k+1}\stackrel{\textnormal{def.}}{=}2(L_k+\kappa l_k)\text{ for all }k\geq0.
\end{equation}
Here $\kappa\geq1$ is a parameter on which $l_k$ and $L_k$ depend implicitly and that we will fix later. Note that there exists a constant $C=C(\kappa)<\infty$ such that
\begin{equation}
\label{eq:boundLk}
    LK2^k\leq 2^nL_{k-n}\leq L_k\leq CLK2^k\text{ for all }0\leq  n\leq k.
\end{equation}
For all $x\in{G}$ and $k\in{\N_0}$ we abbreviate 
\begin{equation}
\label{eq:defIkx}
\begin{split}
    &I_k(x)
   \stackrel{\textnormal{def.}}{=}\Lambda\big(l_k\big)\cap  B\big(x,l_{k+1}+l_k\big)
\end{split}
\end{equation}
and 
\begin{equation}
\label{eq:defEkx}
\begin{split}
&E_k(x)
\stackrel{\textnormal{def.}}{=}\Lambda\big(l_k\big)\cap B\big(x,2l_k+L_{k+1}\big)\setminus B(x,L_{k+1}-l_k).
\end{split}
\end{equation}
The sets $I_k$ and $E_k$ are chosen so that if $\kappa$ is large enough, then all $k\in{\N_0}$ and $x\in{G}$
\begin{equation}
\begin{gathered}
    \label{eq:reasonIkxEkx}
    B(x,l_{k+1})\subset B\big(I_k(x),l_k\big),\ 
    B(x,L_{k+1}+2L+\Cr{Cdvsdgr}^{-1})\setminus B(x,L_{k+1})\subset B\big(E_k(x),l_k\big)
\end{gathered}
\end{equation}
and
\begin{equation}\label{eq:reasonIkxEkx2}
\begin{gathered}
    B\big(I_k(x),2L_k\big)\subset B(x,L_{k+1}-2L-\Cr{Cdvsdgr}^{-1}),\ 
     B\big(E_k(x),2L_k\big)\subset B(x,2L_{k+1}).
\end{gathered}
\end{equation}
Indeed, since $l_k\geq2L+\Cr{Cdvsdgr}^{-1}$ upon choosing $\kappa$ large enough, \eqref{eq:reasonIkxEkx} follows easily from \eqref{eq:defLambda}, and \eqref{eq:reasonIkxEkx2} is an easy consequence of \eqref{eq:defLkMk} if $\kappa$ is large enough. The properties \eqref{eq:reasonIkxEkx} and \eqref{eq:reasonIkxEkx2} imply that any ($L$-nearest neighbor) path from $B(x,l_{k+1})$ to $B(x,2L_{k+1})^c$ can be decomposed into a path from $B(y,l_{k})$ to $B(y,2L_k)^c$ for some $y\in{I_k(x)}$, and a disjoint path from $B(z,l_k)$ to $B(z,2L_{k})^c$ for some $z\in{E_k(x)}$, which will be the starting block in the proof of Lemma~\ref{lem:entropy} below. We refer to Figure~\ref{fig:renscheme} for an illustration of this construction.

Let us denote by $T_k=\bigcup_{n=0}^kT^{(n)},$ where $T^{(n)}=\{0,1\}^n$, the binary tree of depth $k\in{\mathbb{N}_0}$, with the convention $T^{(0)}=\{\varnothing\}$. For $\sigma\in{T^{(i)}}$, $\sigma'\in{T^{(j)}}$, $i+j\leq k$, we denote by $\sigma\sigma'\in{T_k}$ the concatenation of $\sigma$ and $\sigma'$. We call $k$-proper embedding with base $x\in{G}$ any map $\tau:T_k\rightarrow G$ such that $\tau(\varnothing)=x$, for all $0\leq n\leq k-1$ 
\begin{equation}
\label{eq:defproperembedding1}
    \tau(\sigma0)\in{I_{k-n-1}(\tau(\sigma))},\tau(\sigma1)\in{E_{k-n-1}(\tau(\sigma))}\text{ for all }\sigma\in{T^{(n)}},
\end{equation}
and for all $0\leq n\leq k$
\begin{equation}
\label{eq:defproperembedding2}
d(\tau(\sigma),\tau(\sigma'))\geq 2L_{k-n}+\kappa l_{k-n}\text{ for all }\sigma\neq\sigma'\in{T^{(n)}},
\end{equation}
as well as for all $0\leq n\leq k$
\begin{equation}
\label{eq:defproperembedding3}
    \tilde{C}_{\tau(\sigma)}^{L_{k-n}}\cap \tilde{C}_{\tau(\sigma\sigma')}^{L_0}\neq\varnothing\text{ for all }\sigma\in{T^{(n)}}\text{ and }\sigma'\in{T^{(k-n)}}.
\end{equation}

\begin{figure}
\centering
\includegraphics[scale=0.65]{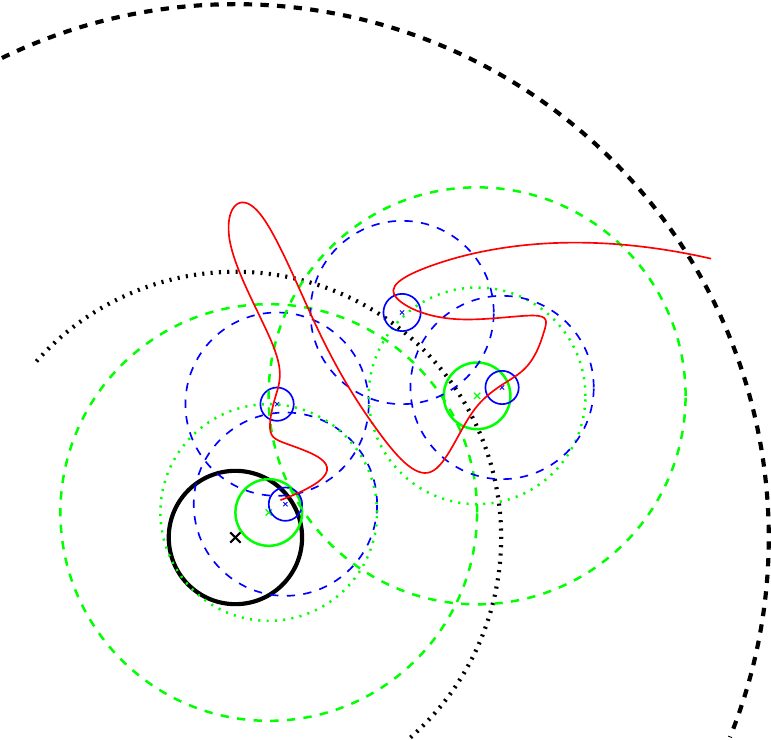}
\caption{Illustration of the proper embedding generated by a path $\gamma$ from $C_x^{l_k}$ to $\tilde{C}_x^{L_k}$ (in red). Normal lines represent $C_{\cdot}^{l_n}$, dashed lines $\tilde{C}_{\cdot}^{L_n}$, and dotted lines ${C}_{\cdot}^{L_n}$, with $n=k$ in black, $n=k-1$ in green, and $n=k-2$ in blue.}
\label{fig:renscheme}
\end{figure}

\begin{Lemme}
\label{lem:entropy}
There exists a constant $\Cl{Ckappa}\in{(1,\infty)}$, such that for all $L\geq1$, $K\geq 2$, $\kappa\geq\Cr{Ckappa}$, $k\in{\mathbb{N}_0}$, $x\in{G}$ and $L$-nearest neighbor paths  $\gamma$ between $C_x^{l_k}$ and $(\tilde{C}_x^{L_k})^c$, there exists a $k$-proper embedding $\tau$ with base $x$ such that 
\begin{equation}
\label{eq:gammahitCtau}
C_{\tau(\sigma)}^{l_{0}}\cap\gamma\neq\varnothing\text{ for all }\sigma\in{T^{(k)}}.
\end{equation} 
\end{Lemme}
\begin{proof}
We proceed recursively on $k\in{\N_0}$. If $k=0$, we simply take $\tau(\varnothing)=x$, which is clearly a $0$-proper embedding. Let us now assume that the lemma is true for $k\in{\mathbb{N}_0}$, let $x\in{G}$, and $\gamma$ be an $L$-nearest neighbor path between $C_x^{l_{k+1}}$ and $(\tilde{C}_x^{L_{k+1}})^c$, stopped on hitting $(\tilde{C}_x^{L_{k+1}})^c$. Let $\gamma_0$ be the part of $\gamma$ until its last visit in $C_x^{L_{k+1}}$, and $\gamma_1$ be the part of $\gamma$ after its last visit in $C_x^{L_{k+1}}$. By \eqref{eq:dvsdgr} and our definition of $L$-nearest neighbor paths, see below \eqref{eq:defLambda}, the first vertex of $\gamma_0$ is in $B(x,l_{k+1})$ and the first vertex of $\gamma_1$ is in $B(x,L_{k+1}+2L+\Cr{Cdvsdgr}^{-1})\setminus B(x,L_{k+1})$.  We let $\tau(0)$, resp.\ $\tau(1)$, be the first vertex in $I_k(x)$, resp.\ $E_k(x)$, such that $C_{\tau(0)}^{l_k}\cap \gamma_0(0)\neq\varnothing$, resp.\ $C_{\tau(1)}^{l_k}\cap \gamma_1(0)\neq\varnothing$, which exists by \eqref{eq:reasonIkxEkx} (upon taking $\Cr{Ckappa}$ large enough). Then since $\tilde{C}_{\tau(0)}^{L_k}\subset B(x,L_{k+1}-2L-\Cr{Cdvsdgr}^{-1})$ by \eqref{eq:reasonIkxEkx2}, using \eqref{eq:dvsdgr} again we know that $\gamma_0$ is an $L$-nearest neighbor path from $C_{\tau(0)}^{l_k}$ to $(\tilde{C}_{\tau(0)}^{L_k})^c$, and similarly $\gamma_1$ is an $L$-nearest neighbor path from $C_{\tau(1)}^{L_k}$ to $(\tilde{C}_{\tau(1)}^{L_k})^c$. By our inductive hypothesis, there exists for each $i\in{\{0,1\}}$ a $k$-proper embedding $\sigma\mapsto\tau(i\sigma)$ satisfying \eqref{eq:gammahitCtau} for $\gamma_i$. It is then clear that $\tau$ also verifies \eqref{eq:gammahitCtau} for $\gamma$  by writing $\sigma=i\sigma',$ $i\in{\{0,1\}}$ and $\sigma'\in{T^{(k)}}$.

Let us now verify that $\tau$ is a $k+1$-proper embedding. 
We start with \eqref{eq:defproperembedding1} for $\tau$, which is clear when $n\neq0$ by using \eqref{eq:defproperembedding1} for $\tau(i\sigma)$, $i\in{\{0,1\}}$. When $n=0$, \eqref{eq:defproperembedding1} is also satisfied by definition of $\tau(0)$ and $\tau(1)$. One can similarly show that \eqref{eq:defproperembedding3} is satisfied for $n\neq0$. Moreover, when $n=0$, for each $\sigma'\in{T^{(k+1)}}$ since $C_{\tau(\sigma')}^{l_0}\cap \gamma=\varnothing$ and $\gamma\subset B(x,2L_{k+1}+2L+\Cr{Cdvsdgr}^{-1})$, we have $\tilde{C}_{\tau(\sigma')}^{L_0}\cap \tilde{C}_{x}^{L_{k+1}}\neq\varnothing$ upon taking $\Cr{Ckappa}$ large enough, which finishes the proof of \eqref{eq:defproperembedding3}. It is also easy to verify \eqref{eq:defproperembedding2} when $\sigma$ and $\sigma'$ are both descendant of the same $i\in\{0,1\}$.

It thus remains to prove that $d(\tau(0\sigma),\tau(1\sigma'))\geq 2L_{k-n}+\kappa l_{k-n}$ for all $\sigma\neq\sigma'\in{T^{(n)}}$ and $n\leq k$. By \eqref{eq:defIkx}, \eqref{eq:defEkx} and \eqref{eq:defproperembedding1}, we have for all $0\leq j\leq k-1$ and $\sigma\in{T^{(j)}}$
\begin{equation*}
d(\tau(0\sigma i),\tau(0\sigma))\leq 2l_{k-j-1}+L_{k-j}\leq 2^{-j}L_k\left(1+\frac{1}{(k-j)^2}\right),
\end{equation*}
where we used \eqref{eq:boundLk} in the last inequality. Summing over $j$, we deduce that for all $0\leq n\leq k$ and $\sigma\in{T^{(n)}}$
\begin{align*}
    d(\tau(0\sigma),\tau(0))\leq L_k\sum_{j=0}^{n-1}2^{-j}\left(1+\frac{1}{(k-j)^2}\right)
    \leq L_k\left(2-2^{-n+1}+\frac{C'}{(k+1)^2}\right),
\end{align*}
where in the last inequality we used the bound $\sum_{j=0}^{k-1}2^{-j}/(k-j)^2\leq \frac{C}{(k+1)^2}$ for some absolute constant $C$, which can for instance be proved by changing the index $j$ in the sum to $k-j$, which after replacing the sum by an integral using monotonicity show that the previous sum is smaller than $C'2^{-k}\int_1^{k+1}x^{-2}2^x\mathrm{d}x$. This last integral can then be bounded by $C''2^k/(k+1)^2$ after an integration by parts. Combining with \eqref{eq:defIkx} we deduce
\begin{equation}
\label{distancetau0sigma}
    d(\tau(0\sigma),x)\leq L_k(2-2^{-n+1})+\frac{CL_k}{(k+1)^2}
\end{equation}
for some finite constant $C$.
Moreover, for any $0\leq n\leq k$ and $\sigma'\in{T^{(n)}}$ since $C_{\tau(1\sigma')}^{l_{0}}\cap\gamma_1\neq \varnothing$ and $\gamma_1\subset B(x,L_{k+1})^c$ we have
\begin{equation}
\label{distancetau1sigma}
    d(\tau(1\sigma'),x)\geq L_{k+1}-l_{0}\geq 2L_k+\frac{(2\kappa-C) L_k}{(k+1)^2},
\end{equation}
where we used \eqref{eq:defLkMk} in the last inequality. Combining \eqref{distancetau0sigma} and \eqref{distancetau1sigma} one can take $\Cr{Ckappa}$ large enough so that for all $\kappa\geq \Cr{Ckappa}$, $0\leq n\leq k$ and $\sigma,\sigma'\in{T^{(n)}}$
\begin{equation}
\label{eq:distancetau1tau0}
    d(\tau(1\sigma'),x)-d(\tau(0\sigma),x)\geq 2^{-n+1}L_{k}+\frac{2(\kappa-C) L_k}{(k+1)^2}\geq 2L_{k-n}+\frac{\kappa L_{k-n}}{(k-n+1)^2}
\end{equation}
by \eqref{eq:defLkMk}, and we can easily conclude.
\end{proof}

Let us now define for each $x\in{G}$ and $k\in{\mathbb{N}_0}$
\begin{equation}
\label{eq:defAtilde}
    \tilde{\mathcal{A}}_{x,k}\stackrel{\textnormal{def.}}{=}\left\{(\mathbf{y}_{\sigma})_{\sigma\in{T^{(k)}}}\in{\Lambda(L)}^{T^{(k)}}\!\!\!:\,\mathbf{y}_{\sigma}\in{C_{\tau(\sigma)}^{l_0}}\text{ for some }k\text{-proper embedding }\tau\text{ with base }x\right\}.
\end{equation}

\begin{Lemme}
\label{lem:properembedding}
    For any $\kappa\geq9$, there exist constants $c=c(\kappa)>0$ and $C=C(\kappa)<\infty$ such that  for all $L\geq 1$, $K\geq2$, $k\in{\mathbb{N}}$ and $x\in{G}$
\begin{equation}
\label{eq:numberproperembedding}
    \big|\tilde{\mathcal{A}}_{x,k}\big|
    \leq \exp\left(\frac{CL_k\log(K)}{KL}\right),
\end{equation}
and for $\mathbf{y}\in{\tilde{\mathcal{A}}_{x,k}}$, letting $\mathcal{C}(\mathbf{y})=\{\mathbf{y}_{\sigma},\sigma\in{T^{(k)}}\}$, we have 
\begin{equation}
\label{eq:capproperembedding}
    \frac{1}{|\Sigma(\tilde{\mathcal{C}})|^2}\sum_{x,y\in{\Sigma(\tilde{\mathcal{C}})}}g(x,y)\leq ({cF_{\nu}(L_k,LK)\overline{H}_{\nu}(\tfrac{L_k}{LK},K)})^{-1}\text{ for all }\tilde{\mathcal{C}}\subset\mathcal{C}(\mathbf{y})\text{ with }|\tilde{\mathcal{C}}|\geq 2^{k-1},
\end{equation}
where $\overline{H}_{\nu}(x,y)$ is equal to $1/(1+x^{\nu-1}y^{\nu})$ if $\nu<1$, to $1/(1+y/\log(x))$ if $\nu=1$, and to $y^{-\nu}$ if $\nu>1$.
\end{Lemme}
\begin{proof}
    Let us start with \eqref{eq:numberproperembedding}, and denote by $M_k$ the maximum over $x\in{G}$ of $|\tilde{\mathcal{A}}_{x,k}|$. For each $k$-embedding $\tau$ we have that $\sigma\mapsto \tau(0\sigma)$ and $\sigma\mapsto\tau(1\sigma)$ are two $k-1$ proper embedding and so it follows from \eqref{eq:defLambda}, \eqref{eq:defIkx}, \eqref{eq:defEkx} and the inequality $2l_k+L_{k+1}\leq C\kappa (k+1)^2 l_k$ that 
\begin{equation}
\label{eq:recursiveentropy}
 M_0\leq \Cr{CLambda}(\kappa K)^{\alpha}
 \text{ and }M_k\leq (\Cr{CLambda}(C\kappa k^2)^{\alpha}M_{k-1})^2\text{ for all }k\geq1,
\end{equation}
Using that $\prod_{i=1}^{k}(k-i+1)^{\alpha 2^{i+2}}\leq \exp(C2^k)$, \eqref{eq:numberproperembedding} follows readily from \eqref{eq:boundLk}.

We now turn to the proof of \eqref{eq:capproperembedding}. Let $\tau$ be a proper embedding with base $x$, $\mathbf{y}\in{T^{(k)}}$ be such that $\mathbf{y}_{\sigma}\in{C_{\tau(\sigma)}^{l_0}}$ for all $\sigma\in{T^{(k)}}$, 
and let $\tilde{\calC}\subset\calC(\mathbf{y})$ with $|\tilde{\calC}|\geq 2^{k-1}$. For any $y\in{\Sigma(\tilde{\calC})}$ and $1\leq i\leq k$, let us denote by ${A}_i^y$ the set $\tilde{\calC}\cap B(y,L_i)\setminus B(y,L_{i-1})$. Denoting by $B^y_i$ the set of $\sigma\in{T^{(k-i)}}$ such that $\tau(\sigma)\in{B(y,3L_i+3L_0)},$ we see by \eqref{eq:defproperembedding2} that the balls $B(\tau(\sigma),L_i)$, $\sigma\in{B^y_i}$, are disjoint and all included in $B(y,6L_i)$. Therefore by \eqref{eq:intro_sizeball}  
\begin{equation}
\label{eq:Byi}
|B^y_i|\leq \frac{\Cr{Csizeball}(6L_i)^{\alpha}}{\Cr{csizeball}L_i^{\alpha}}\leq \frac{\Cr{Csizeball}6^{\alpha}}{\Cr{csizeball}}.
\end{equation}
For each $\sigma\in{T^{(k)}}$ such that $\mathbf{y}_{\sigma}\in{A_i^y}$, letting $\sigma_i\in{T^{(k-i)}}$ be the $i$-th ancestor of $\sigma$, we have $d(\tau(\sigma),\tau(\sigma_i))\leq 2L_i+2L_0$ by \eqref{eq:defproperembedding3}, $d(\tau(\sigma),\mathbf{y}_{\sigma})\leq L_0$, and so $d(y,\tau(\sigma_i))\leq 3L_i+3L_0$, that is $\sigma_i\in{B_i^y}$. Therefore $|A_i^y|\leq 2^i|B_i^y|\leq 2^i\Cr{Csizeball}6^{\alpha}\Cr{csizeball}^{-1}$. Using \eqref{eq:intro_sizeball}, \eqref{eq:intro_Green} and \eqref{eq:boundonlambda}, we deduce that for each $y\in{\Sigma(\tilde{\calC})}$
\begin{equation}
\label{eq:sumginSigma}
\begin{split}
\sum_{x\in{\Sigma(\tilde{\calC})\setminus C_y^{L_0}}}g(x,y)\leq \sum_{i=1}^k\sum_{x\in{{A}_i^y}}\sum_{z\in{\hat{C}_x^L}}g(z,y)
&\leq C\sum_{i=1}^kL^{\alpha}2^i(L_{i-1}-8L)^{-\nu}
\\&\leq C'L^{\alpha-\nu}K^{-\nu}\sum_{i=1}^k2^{(1-\nu)i},
\end{split}
\end{equation}
for some constants $C,C'<\infty$, and where we used \eqref{eq:boundLk} in the last inequality, together with $L_{i-1}\geq 9L$ for all $i\geq1$ since $\kappa\geq 9$. Since ${C}_y^{L_0}\cap\Sigma(\tilde{\calC})=\hat{C}_{\mathbf{y}_{\sigma}}^L$ for some $\sigma\in{T^{(k)}}$ by \eqref{eq:defSigma} and \eqref{eq:defproperembedding2} for $n=k$, we can similarly bound 
\begin{equation*}
\begin{split}
\sum_{x\in{\Sigma(\tilde{\calC})}\cap C_y^{L_0}}g(x,y)&\leq g(y,y)+\sum_{i=1}^{\lceil\log_2(8L)\rceil}\sum_{z\in{B(\mathbf{y}_{\sigma},2^i)\setminus B(\mathbf{y}_{\sigma},2^{i-1})}}g(z,y)
\\&\leq C\left(1+\sum_{i=1}^{\lceil\log_2(8L)\rceil}2^{i(\alpha-\nu)}\right)\leq C'L^{\alpha-\nu},
\end{split}
\end{equation*}
where the last inequality follows from \eqref{eq:condnualpha}. Note that by \eqref{eq:defSigma}, \eqref{eq:intro_sizeball}, \eqref{eq:boundonlambda} and  \eqref{eq:defproperembedding2} for $n=k$ we have $|\Sigma(\tilde{\calC})|\geq cL^{\alpha}2^{(k-1)}$. Therefore 
\begin{equation}
\label{eq:computecapacity}
\begin{split}
\frac{1}{|\Sigma(\tilde{\mathcal{C}})|}\sum_{x\in{\Sigma(\tilde{\mathcal{C}})}}g(x,y)&\leq C2^{-k}L^{-\alpha}\left(L^{\alpha-\nu}+L^{\alpha-\nu}K^{-\nu}\sum_{i=1}^k2^{(1-\nu)i}\right)\\&\leq
 \begin{cases}
C2^{-k\nu}(LK)^{-\nu}(2^{k(\nu-1)}K^\nu+1)&\text{ if }\nu<1,
\\Ck2^{-k}(LK)^{-1}(K/k+1)&\text{ if }\nu=1,
\\C2^{-k}(LK)^{-\nu}K^{\nu}&\text{ if }\nu>1.
\end{cases}
\end{split}
\end{equation}
Combining this with \eqref{eq:boundLk} and \eqref{eq:defGnu} yields \eqref{eq:capproperembedding}.
\end{proof}

Combining Lemmas~\ref{lem:entropy} and \ref{lem:properembedding}, one obtains a version of Proposition~\ref{pro:entropy} but with $\overline{H}_{\nu}$ instead of $H_{\nu}$ in \eqref{eq:capC}. As this will be the first step in the proof of Proposition~\ref{pro:entropy} and is of independent interest if one is not interested in the exact value of the constant $\Cr{cFPP2}$ in Theorem~\ref{the:FPP} when $\nu\leq1$, let us now explicitly state this intermediate result.

\begin{Cor}
\label{cor:entropy}
    There exist $\Cl[c]{ccardcalC},\Cl[c]{ccapcalC2}>0,$ and $\Cr{CcardA},\Cl{CKL2},\Cl{CcardcalC}<\infty$ such that for all $L\geq1$, $K\geq2,$  $x\in{G}$ and $N\geq \Cr{CKL2}KL,$ there exist $p\in{[\Cr{ccardcalC}N/(KL),\Cr{CcardcalC}N/(KL)]}$ and a family $\overline{\mathcal{A}}=\overline{\mathcal{A}}_{x,N}^{L,K}$ of collections $\mathcal{C}\subset B(x,N)$ as in \eqref{eq:defcalC} such that \eqref{eq:cardC}, \eqref{eq:gammaincludedC} and \eqref{eq:cardA} are satisfied for $\overline{\A}$ and, recalling $\overline{H}_{\nu}$ from below \eqref{eq:capproperembedding},
\begin{equation}
\label{eq:capC2}\begin{aligned}
         &\mathrm{cap}(\Sigma(\tilde{\mathcal{C}}))
         \geq  {\Cr{ccapcalC2}F_{\nu}(N,KL)\overline{H}_{\nu}(\tfrac{N}{KL},K)}
         \text{ for all }\mathcal{C}\in{\overline{\mathcal{A}}}\text{ and }\tilde{\mathcal{C}}\subset\mathcal{C}\text{ with }|\tilde{\mathcal{C}}|\geq n/2.
\end{aligned}
\end{equation}
\end{Cor}
\begin{proof}
Let  $\kappa=\Cr{Ckappa}\vee 9$, $\Cr{CKL2}=4\kappa$, $N\geq \Cr{CKL2}KL$, and choose some $k\in{\mathbb{N}_0}$ such that $N/(16\kappa)\leq L_k\leq N/4$, which exists by \eqref{eq:defLkMk}. We let $\overline{\mathcal{A}}_{x,N}^{L,K}\stackrel{\textnormal{def.}}{=}\tilde{\mathcal{A}}_{x,k}$, see \eqref{eq:defAtilde}, and note that any $\mathcal{C}\in{\overline{\mathcal{A}}_{x,N}^{L,K}}$ is as in \eqref{eq:defcalC} by \eqref{eq:defproperembedding2} for $n=k$, and since $\kappa\geq9$. Moreover, for any $k$-proper embedding $\tau$ with base $x$, $\sigma\in{T^{(k)}}$ and $\mathbf{y}_{\sigma}\in{C_{\tau(\sigma)}^{L_0}}$ we have $d(x,\tau(\sigma))\leq 2(L_k+L_0)$ by \eqref{eq:defproperembedding3},  which implies $\mathcal{C}\subset B(x,N)$. Since $|\mathcal{C}(\mathbf{y})|=2^k$ for any $\mathbf{y}\in{\tilde{\mathcal{A}}_{x,k}}$, letting $p=2^k$ \eqref{eq:cardC} is satisfied and so by \eqref{eq:boundLk} one can choose $\Cr{ccardcalC}$ small enough and $\Cr{CcardcalC}$ large enough so that $p\in{[\Cr{ccardcalC}N/(KL),\Cr{CcardcalC}N/(KL)]}$ is satisfied. The inequality \eqref{eq:cardA} follows directly from \eqref{eq:numberproperembedding}, and  \eqref{eq:capC2} from \eqref{variational} and \eqref{eq:capproperembedding}. Finally, for any $L$-nearest neighbor path $\gamma$ from $x$ to $B(x,N-2L)^c$, since $N-2L\geq 2L_k$, we have $\gamma\cap (\tilde{C}_{x}^{L_k})^c\neq \varnothing$. By Lemma~\ref{lem:entropy}, there is a $k$-proper embedding $\tau$ with base $x$ such that $C_{\tau(\sigma)}^{l_0}\cap\gamma$ contains some vertex $\mathbf{y}_{\sigma}$ for any $\sigma\in{T^{(k)}}$, which directly implies \eqref{eq:gammaincludedC}.
\end{proof}

\begin{Rk}
\label{rk:otherschemenu<1} When $\nu<1$, one could replace the scales from \eqref{eq:defLkMk} by $L_k=\ell_0^k(\kappa KL)$, for any $\ell_0\in{(2,2^{\frac1\nu})}$, and the proof would still work. This yields a new version of Corollary~\ref{cor:entropy} but with $p(N/(KL))^{-\frac{\log2}{\log\ell_0}}\in{[\Cr{ccardcalC},\Cr{CcardcalC}]}$, which corresponds to a version of Proposition~\ref{pro:entropy} with the same leading order for $p$ as in \eqref{eq:choicep}, but without the correct constant in \eqref{eq:capC}. One could even take $\ell_0=2^{\frac1\nu}$, but one then additionally needs to replace the right-hand side of \eqref{eq:capC2} by ${\Cr{ccapcalC2}N^{\nu}\overline{H}_{\nu}(\tfrac{N}{KL},K)}/\log(N/(KL))$ (since the sum on the last line of \eqref{eq:sumginSigma} is then replaced by $\sum_{i=0}^k2^i\ell_0^{-\nu i}=k+1$). This new scheme is actually easier to implement, since a bound such as \eqref{eq:defproperembedding2} becomes easy to prove with more space between the different scales, but since we need the choice $p\approx N/(KL)$ in the proof of Lemma~\ref{lem:coarsegraining}  we are not going to pursue this here.
\end{Rk}

As explained in Remark~\ref{rk:otherschemenu<1}, one could also prove a version of Corollary~\ref{cor:entropy} so that $p$ can be chosen with a similar leading order as in \eqref{eq:choicep}. This is however not enough to deduce Proposition~\ref{pro:entropy} since \eqref{eq:capC2} and \eqref{eq:capC} are not equivalent. Indeed, the function $\Cr{ccapcalC2}\overline{H}_{\nu}(x,y)$ does not have the same explicit limit as the function $H_{\nu}(x,y,\eta)$ as $x\rightarrow\infty$ when $\nu\leq 1$. In order to obtain this exact limit we will use a similar strategy to \cite[Proposition~4.3]{GRS21} on $\Z^3$, see also \cite[Lemma~5.2]{MuiSev}, which essentially consists in defining $\mathcal{A}$ as the set of all possible sequence of vertices $x_i$ in $\Lambda(L)\cap \partial B(x,16iKL)$, $1\leq i\leq N/(16KL)$, and is the simplest way to obtain \eqref{eq:gammaincludedC}. This strategy yields the correct constant in \eqref{eq:capC} via Lemma~\ref{lem:capatube}, but it is easily seen that the entropy bound \eqref{eq:cardA} would have to be replaced by $|\A|\leq (N/L)^{CN/L}$. In order to obtain the correct bound in \eqref{eq:cardA}, our new strategy will essentially be to consider vertices $x_i$ in $\Lambda(M)\cap \partial B(x,iM)$, $1\leq i\leq N/M$, for some $M\in{[16KL,N}]$, and use Corollary~\ref{cor:entropy} (for $N=M$) to define $\A\cap \tilde{C}_{x_i}^M$ for each $1\leq i\leq N/M$. The set $\A$ then still has approximately $N/L$ vertices, as required in \eqref{eq:cardC}, and  the good entropy bound in Corollary~\ref{cor:entropy} will yield \eqref{eq:cardA} for a suitable choice of $M$.

\begin{proof}[Proof of Proposition~\ref{pro:entropy}]
When $\nu>1$, we can simply take $\Cr{crho}=1/2$, $H_{\nu}(x,y,\eta)=\Cr{ccapcalC2}\overline{H}_{\nu}(x,y)$ and $\Cr{CNKLp}=\Cr{CNKL}/2=\Cr{CcardcalC}$, and Proposition~\ref{pro:entropy} then follows from Corollary~\ref{cor:entropy} by defining $\mathcal{A}_{x,N}^{L,K}$ as the set of $\mathcal{C}$ which correspond to elements of $\overline{\mathcal{A}}_{x,N}^{L,K}$ plus a few additional vertices chosen arbitrarily so that $\mathcal{C}$ is still as in \eqref{eq:defcalC} and $|\mathcal{C}|=p$, which is always possible if $N\geq \Cr{CKL}KL$ upon choosing $\Cr{CKL}\geq \Cr{CKL2}$ large enough. 

Let us now assume $\nu\leq 1$. We define $M\stackrel{\textnormal{def.}}=KL\log(N/(KL))$, $P\stackrel{\textnormal{def.}}=\lfloor pKL/(\Cr{CcardcalC}M)\rfloor=\lfloor p/(\Cr{CcardcalC}\log(N/(KL))\rfloor$ and $S_{x,i}^{N,M}\stackrel{\textnormal{def.}}{=}(B(x,(5i-1)N/(5P))\setminus B(x,(5i-2)N/(5P)))\cap \Lambda(L)$ for all $1\leq i\leq P$ and $x\in{G}$. Note that $P\geq p/(2\Cr{CcardcalC}\log(N/(KL)) $ if $\Cr{CNKLp}$ is large enough by \eqref{eq:choicep}. If $N\geq \Cr{CKL}KL$ for $\Cr{CKL}$ large enough, then $M\geq \Cr{CKL2}KL$, and so Corollary~\ref{cor:entropy} with $N=M$ therein is in force. Let us define
\begin{equation*}
\mathcal{A}_{x,N}^{L,K}\stackrel{\textnormal{def.}}{=}\left\{\bigcup_{i=1}^{{P}}\mathcal{C}_{i}:\text{ for each } i\in\left\{1,\dots,P\right\},\mathcal{C}_i\in{\overline{\mathcal{A}}_{x_i,M}^{L,K}}\text{ for some }x_i\in{S_{x,i}^{N,M}} \right\}.
\end{equation*}
Each $\mathcal{C}\in{\mathcal{A}_{x,N}^{L,K}}$ is as in \eqref{eq:defcalC} since $d(y,z)\geq (i-j)N/(5P)-2M\geq 16KL$ for each $y\in{\mathcal{C}_i}$ and $z\in{\mathcal{C}_j}$, $i\neq j$, upon taking $p\leq \Cr{CNKL}N/(KL)$ and $\Cr{CNKL}$ small enough. Moreover, by \eqref{eq:defLambda} and \eqref{eq:cardA}
\begin{equation*}
\begin{split}
|\mathcal{A}_{x,N}^{L,K}|&\leq \left(|\sup_{y}\overline{\mathcal{A}}_{y,M}^{L,K}|\left(\frac{\Cr{CLambda}N}{L}\right)^{\alpha}\right)^{P}
\\&\leq \exp\left(\Cr{CcardA}p\log(K)+\frac{\alpha p}{\Cr{CcardcalC}\log(N/(KL))}\log\left(\frac{\Cr{CLambda}N}{L}\right)\right),
\end{split}
\end{equation*}
and we easily deduce that \eqref{eq:cardA} is still satisfied upon changing the constant $\Cr{CcardA}$. For each $L$-nearest neighbor path $\gamma$ from $x$ to $B(x,N-2L)^c$ and  for each $1\leq i\leq P$  we denote by $x_i$ the first point of $B(x,(5i-2)N/(5P))^c\cap \Lambda(L)$ visited by $\gamma$, then $d(x,x_i)\leq (5i-2)N/(5P)+L+\Cr{Cdvsdgr}^{-1}\leq (5i-1)N/(5P)$ by \eqref{eq:defLambda}, \eqref{eq:dvsdgr} and upon taking $\Cr{CNKL}$ small enough, and so $x_i\in{S_{x,i}^{N,M}}$. Moreover, we have ${C}_{x_i}^M\subset B(x,(5i-1)N/(5P)+M)\subset B(x,N-2L)$ if $\Cr{CNKL}$ is small enough, and so $\gamma$ can be decomposed into a path from $x_i$ to $({C}_{x_i}^{M})^c$, from which \eqref{eq:gammaincludedC} follows easily. Note also that $|\mathcal{C}|\leq P(\Cr{CcardcalC}M/(KL))\leq  p$ for any $\mathcal{C}\in{\mathcal{A}_{x,N}^{L,K}}$, and so \eqref{eq:cardC} is satisfied up to adding arbitrarily a few points in $\mathcal{C}$, in such a way that $\calC$ is still as in \eqref{eq:defcalC}, which is always possible if $\Cr{CKL}$ is large enough. 

It thus only remains to verify \eqref{eq:capC}, which will follow from Lemma~\ref{lem:capatube}. Let $\mathcal{C}=\cup_{i=1}^{P}\mathcal{C}_i\in{\A_{x,N}^{L,K}}$ and $\tilde{\mathcal{C}}\subset\mathcal{C}$ with $|\tilde{\calC}|\geq (1-\rho)|\calC|$. Note that since $|\mathcal{C}_k|\geq \Cr{ccardcalC}M/(KL)$ and $|\mathcal{C}|\leq p$, there are at most $2\rho pKL/(\Cr{ccardcalC}M)\leq 4\rho \Cr{CcardcalC}P/\Cr{ccardcalC}$ possible $k$ such that $|\tilde{\mathcal{C}}\cap\mathcal{C}_k|\leq \tfrac12|\mathcal{C}_k|$. Therefore defining $A\stackrel{\textnormal{def.}}{=}\{k\in\{1,\dots,p\}:\,|S_k|\geq \tfrac12|\mathcal{C}_k|\}$, we have $n\stackrel{\textnormal{def.}}{=}|A|\geq (1-4\rho\Cr{CcardcalC}/\Cr{ccardcalC})P\geq \Cr{ccardAP} P$, with $\Cr{ccardAP}$ as in Lemma~\ref{lem:capatube} for $\delta=1/2$, where the last inequality is satisfied upon taking $\rho\leq \Cr{crho}$ for some constant $\Cr{crho}$ small enough. Let us also define $S_k\stackrel{\textnormal{def.}}{=}\Sigma(\tilde{\mathcal{C}}\cap\mathcal{C}_k)$ for each $1\leq k\leq P$. Noting that $S_k\subset \Sigma(\mathcal{C}_k)\subset B(\mathcal{C}_k,8L)\subset B(S_{x,k}^{N,M},8L+M)\subset B(x,iN/P)\setminus B(x,(i-1/2)N/P)$ if $\Cr{CNKL}$ is small enough, and so $d(S_i,S_j)\geq (|i-j|-1/2)N/P$ for all $i\neq j$, and that $P\geq  \Cr{CboundP}$ is satisfied if $\Cr{CNKLp}=\Cr{CNKLp}(\eta)$ is large enough, it follows from \eqref{eq:capC2} and Lemma~\ref{lem:capatube} that \eqref{eq:capunionbound} is satisfied with $\kappa={\Cr{ccapcalC2}F_{\nu}(M,KL)\overline{H}_{\nu}(\tfrac{M}{KL},K)}$. Finally by \eqref{eq:defGnu}, one has
\begin{equation}
\label{eq:condMN}
\begin{split}
\frac{1}{\Cr{cconstantabovekappa}\kappa P}&=
\frac{\Cr{CcardcalC}M^{1-\nu}\log(M/KL)^{1\{\nu=1\}}}{\Cr{cconstantabovekappa}\Cr{ccapcalC2}pKL\overline{H}_{\nu}(\tfrac{M}{KL},K)}
\leq \frac{\eta\log(N/KL)^{1\{\nu=1\}}}{\Cr{cbeta}N^{\nu}\overline{H}_{\nu}(\log(N/(KL)),K)},
\end{split}
\end{equation}
where the last inequality holds if $p\geq \Cr{CNKLp}(N/(KL))^{\nu}\log(N/(KL))^{1-\nu}$ for some constant $\Cr{CNKLp}=\Cr{CNKLp}(\eta)$ large enough. Letting $H_{\nu}(x,y,3\eta)=\overline{H}_{\nu}(\log(x),y)\Cr{cbeta}/(1+3\eta)$, which satisfies $H_{\nu}(x,y,\eta)\rightarrow \frac{\Cr{cbeta}}{1+3\eta}$ as $x\rightarrow\infty$, see below \eqref{eq:capproperembedding}, noting that $\log(P)\leq (1+\eta)\log(N/(KL))$ if  $\Cr{CKL}$ is large enough, and doing a change of variable for $\eta$, we can conclude by \eqref{eq:capunionbound}.
\end{proof}

\begin{Rk}
\label{rk:lowercardinality} Replacing $M$ by $CKL$ for a large enough constant $C$ in the proof of Proposition~\ref{pro:entropy}, one can alternatively choose $p\geq \Cr{CNKLp}(N/(KL))^{\nu}$ if $\nu<1$, and $p\geq  \Cr{CNKLp}(N/(KL))/\log(N/(KL))$ if $\nu=1$ in \eqref{eq:choicep}. However, the entropy bound \eqref{eq:cardA} would then need to be replaced by $|\mathcal{A}|\leq \exp\left(\Cr{CcardA}p\log(N/L)\right)$, which would make our proof of Theorem~\ref{the:FPP} more complicated as explained in Remark~\ref{rk:decayVu},\ref{rk:noapriori}). It is an interesting question whether one can prove a version of Proposition~\ref{pro:entropy} but for $p$ as above and with the same entropy bound as in \eqref{eq:cardA}, see  Remark~\ref{rk:decayVu},\ref{rk:othercardinality}) as to why.
\end{Rk}

\section{Proof of Theorems~\ref{the:limitVuintro2} and (\ref{eq:boundonPEintro})}
\label{sec:proofofFPPintro}

In this short section, we first deduce  from Theorem~\ref{the:FPP} a generalization to graphs satisfying \eqref{eq:standingassumptionintro} of the respective upper bounds in Theorem~\ref{the:limitVuintro2} and \eqref{eq:boundonPEintro}, as well as obtain almost matching lower bounds, see Theorems~\ref{the:limitVunointro} and~\ref{the:FPPnointro}. We start with the setup from Theorem~\ref{the:limitVuintro2}, which correspond to considering the distance $d_{\mathbf{t},u}$ from \eqref{eq:defdzL} for the choices $R=1$ and $t_z^u=1\{\ell_{z,u}>0\}$, which fits the setup from \eqref{eq:defbft}.

\begin{The}
\label{the:limitVunointro}
Take $R=1$, $t_z^u=1\{\ell_{z,u}>0\}$ for all $z\in{G}$ and fix some $\eta\in{(0,1)}$. For all $u>u_{*}$, there exist $\Cl[c]{cFPPVu}=\Cr{cFPPVu}(\eta,u)>0$, $\Cl{clowFPPVu}<\infty$, $\Cl[c]{cdtu}=\Cr{cdtu}(u)>0$ and $\Cl{Cdtu}=\Cr{Cdtu}(u)<\infty$  such that: for all $x\in{G}$
\begin{equation}
\label{eq:lowerboundVunointro}
    \liminf_{N\rightarrow\infty}\frac{\log(\P(d_{\mathbf{t},u}(x;N)\leq \Cr{cFPPVu}N))}{N^{\nu}}\geq\liminf_{N\rightarrow\infty}\frac{\log(\P(x\leftrightarrow B(x,N)^c))}{N^{\nu}}\geq -{\Cr{clowFPPVu}(\sqrt{u}-\sqrt{u_{*}})^2};
\end{equation}
moreover if on the one hand $\nu\leq 1$ and $u>u_{**}$, then for all $x\in{G}$
\begin{equation}
\label{eq:limitVunointro}
   \limsup_{N\rightarrow\infty}\frac{\log(N)^{1\{\nu=1\}}\log(\P(d_{\mathbf{t},u}(x;N)\leq \Cr{cFPPVu}N))}{N^{\nu}}\leq -(1-\eta)\Cr{cbeta}(\sqrt{u}-\sqrt{u_{**}})^2;
\end{equation}
and if on the other hand $\nu>1$ and $u>u_{**}$, then for all $x\in{G}$
\begin{equation}
\label{eq:limitVunu>1nointro}
     -\Cr{Cdtu}\leq \liminf_{N\rightarrow\infty}\frac{\log(\P(d_{\mathbf{t},u}(x;N)\leq \Cr{cFPPVu}N))}{N}\leq \limsup_{N\rightarrow\infty}\frac{\log(\P(d_{\mathbf{t},u}(x;N)\leq \Cr{cFPPVu}N))}{N}\leq -\Cr{cdtu}.
\end{equation}
\end{The}
\begin{proof}
    For $R\geq u_{**}^{-1/\nu}$ to be fixed later, let us define $g_z((L_x)_{x\in{B(z,3R)}})=1\{B(z,R)\not\leftrightarrow B(z,2R)^c\text{ in }\{x:\,L_x=0\}\}$ for each $L\in{[0,\infty)^{B(z,3R)}}$ and $z\in{\Lambda(R)},$ which is an increasing function. Take also $v>u_{**}$ and $u\in{(u_{**},v)}$ small enough (depending on $\eta$ and $v$) so that $(\sqrt{v}-\sqrt{u})^2\geq (1-\eta)(\sqrt{v}-\sqrt{u_{**}})^2$, and let $\Theta=v/u$ and $\xi=(\sqrt{v}/\sqrt{u})-1$. Then, $t_z^u=0$ if and only if $B(z,R)\leftrightarrow B(z,2R)^c\text{ in }\mathcal{V}^u$, and so \eqref{eq:assumptiontzu} (note that $\Cr{ccondFPP}=\Cr{ccondFPP}(\eta,\xi,\Theta)$ therein for $\eta$, $\xi$ and $\Theta$ as before now depends only on $\eta$ and $v$) for $s=1/2$ is satisfied for $R=R(v,\eta)$ large enough. Then, taking $\zeta=R(\sqrt{v}-\sqrt{u})^{2/\nu}$, the conditions of Theorem~\ref{the:FPP} are fulfilled, and so \eqref{eq:bounddtvfinal} is satisfied, and the constants therein now depend on $v$ and $\eta$. Note also that each time some $z\in{\Lambda(R)}$ with $g_z(\ell_{\cdot,u})>0$ is visited by a path $\gamma$ starting or ending in $B(z,R)$, then the path $\gamma$ also visits $\I^u$ in $B(z,2R)$. Using \eqref{eq:defLambda} and taking $\Cr{cFPPVu}=\Cr{cFPP}/(2^{\alpha}\Cr{CLambda}R)$, in view of \eqref{eq:defGnu}, for $\nu\leq 1$ the left-hand side of \eqref{eq:limitVunointro} (for $v$ instead of $u$) can thus be upper bounded by $-\Cr{cbeta}(\sqrt{v}-\sqrt{u_{**}})^2(1-\eta)/(1+\eta)$. After a change of variable for $\eta$, we obtain \eqref{eq:limitVunointro}. The proof of the upper bound in \eqref{eq:limitVunu>1nointro} for $\nu>1$ and $\eta=1/2$ is similar (but the constant $\Cr{cFPP2}$ from \eqref{eq:bounddtvfinal} now depends on $v$), and for the lower bound one simply asks that each of the vertices of the path $(x=x_1,x_2,\dots,x_{p})$ from \eqref{eq:intro_shortgeodesic}, with $p=(N+\Cr{Cgeo})/\Cr{cgeo}$, are such that $B(x_i,\Cr{cgeo}+\Cr{Cgeo})\subset \V^u$, and uses the FKG inequality.

    It remains to prove \eqref{eq:lowerboundVunointro}. Abbreviate $B=B(x,2N)$ and $\delta=u-u_*$, recall the definition of $N_v^{B}$, $v>0$, from above \eqref{eq:defRI}, and let $\tilde{N}_v^{B}$ be an independent copy of $N_v^{B}$ for any $v>0$. Then letting $\tilde{\V}^v=B\setminus \cup_{k=1}^{\tilde{N}_v^B}\{\text{trace}(X^k)\cap B\}$ for any $v>0$, where $(X^i)_{i\in{\N}}$ are the trajectories of $\omega$ hitting $B$, see above \eqref{eq:defRI}, we have that $\tilde{\V}^v\stackrel{\textnormal{law.}}{=}\V^v\cap B$, and is independent of $N_v^{B}$. Therefore
\begin{equation}
\label{eq:boundonprobatoconnectVu}
\begin{split}
    &\P(x\leftrightarrow   B(x,N)^c\text{ in }\V^u)
    \\&\geq \P\big(N_u^{B}\leq (u_*-2\delta)\mathrm{cap}(B),\tilde{N}_{u_*-\delta}^{B}\geq (u_*-2\delta)\mathrm{cap}(B),x\leftrightarrow  B(x,N)^c\text{ in }\tilde{\V}^{u_*-\delta}\big)
    \\&\geq \P(N_u^{B}\leq (u_*-2\delta)\mathrm{cap}(B))\Big(\P\big(x\stackrel{\tilde{\V}^{u_*-\delta}}{\longleftrightarrow}  B(x,N)^c\big)-\P\big(\tilde{N}_{u_*-\delta}^{B}\leq (u_*-2\delta)\mathrm{cap}(B)\big)\Big).
    \end{split}
\end{equation}
If $P\in{[\Cr{CboundP},N]}$, it follows from large deviations bounds on Poisson random variables, which can for instance be deduced from \cite[(3.7)]{chernoff1952measure} applied to $\text{Poi}(1)-1$ random variables and the infinite divisibility of Poisson random variables,  that if $\delta\leq u_*/2$ and $N$ is large enough then
\begin{equation}
\label{eq:boundNuB}
    \P(N_u^{B}\leq (u_*-2\delta)\mathrm{cap}(B)))\geq \exp(-C\delta^2\mathrm{cap}(B))\geq\exp\big(-{C'\delta^2N^{\nu}}\big)
\end{equation}
for some constants $C,C'<\infty$, where the last inequality follows from \eqref{eq:capball}. We moreover have that $\P(x\leftrightarrow   B(x,N)^c\text{ in }\tilde{\V}^{u_*-\delta})\geq \P(x\leftrightarrow \infty\text{ in }\tilde{\V}^{u_*-\delta})>0$  by definition of $u_*$, and thus by Chebyshev's inequality one can easily show that $\P(\tilde{N}_{u_*-\delta}^{B}\leq (u_*-2\delta)\mathrm{cap}(B))\leq \P(x\leftrightarrow \infty\text{ in }\tilde{\V}^{u_*-\delta})/2$ if $N$ is large enough. Combining this with \eqref{eq:boundonprobatoconnectVu} and \eqref{eq:boundNuB}, and noting that $\delta^2\leq C(\sqrt{u}-\sqrt{u_*})^2$ when $\delta\leq u_*$, and that $d_{\mathbf{t},u}(x;N)=0$ when $x\leftrightarrow B(x,N)^c$ in $\V^u$, finishes the proof of \eqref{eq:lowerboundVunointro}.
\end{proof}

\begin{Rk} 
\label{rk:generalizationmainFPP}
As should be clear from the proof of Theorem~\ref{the:limitVunointro}, one can use Theorem~\ref{the:FPP} to deduce that for any $\eta>0$, if  $\P(d_{\mathbf{t},u}(x;L;2L)=0)\rightarrow0$ as $L\rightarrow\infty$ for some $\mathbf{t}$ as in \eqref{eq:defbft}, then $\P(d_{\mathbf{t},v}(x;N)\leq cN)$ can be upper bounded by $\exp(-\Cr{cbeta}(1-\eta)(\sqrt{v}-\sqrt{u})^2N^{\nu}\log(N)^{-1\{\nu=1\}})$  as $N\rightarrow\infty$ if $\nu\leq1,$ and $\exp(-c'N)$ if $\nu>1$, where $(u,v)\in{\L}$, see \eqref{eq:defcalL}, and $c,c'$ are positive constants, depending on $u$, $v$ and $\eta$.  Theorem~\ref{the:limitVunointro} is thus just an application of this when $R=1$ and $t_z^u=1\{z\in{\I^u}\}$, $u>u_{**}$.
\end{Rk}

In particular, \eqref{eq:limitVunointro} and \eqref{eq:limitVunu>1nointro} show that $d_{\mathbf{t},u}(x;N)$ grows linearly in $N$ with high probability, and so the time constant associated to $d_{\mathbf{t},u}$ is positive when it exists. Note that one knows this on $\Z^d$ from \cite{AndPre}, but that we could not simply adapt the method used therein to general graphs satisfying \eqref{eq:standingassumptionintro}, see Remark~\ref{rk:apriori},\ref{rk:previousFPPbounds}) as to why. Proposition~\ref{pro:apriori} provides a shorter proof of this, but the bounds therein are not optimal. On the contrary, in view of \eqref{eq:lowerboundVunointro} and \eqref{eq:limitVunu>1nointro}, the upper bounds from Theorem~\ref{the:limitVunointro} that we deduced from Theorem~\ref{the:FPP} are sharp up to constants (not depending on $u$ when $\nu<1$ if the phase transition is sharp) whenever $\nu\neq1$, and up to a logarithmic correction in $N$ when $\nu=1$. Actually on $\Z^3$, for which $\nu=1$, the upper bound in \eqref{eq:limitVunointro} is also sharp using \cite{GosRod} as indicated in Theorem~\ref{the:limitVuintro2}, and we believe this should also be true on graphs satisfying \eqref{eq:standingassumptionintro} with $\nu\leq 1$.  However, proving this on general graphs would require at least the existence of a strongly percolative phase for $\V^u$ similar to the one established on $\Z^d$ in \cite{MR3269990}, which is an interesting open question in this context. It is interesting to note that when $\nu<1$, our proof of \eqref{eq:lowerboundVunointro} is much easier than the method used in \cite{GosRod}, but at the cost of not obtaining matching constants with \eqref{eq:limitVunointro}, even if the phase transition was sharp.

\begin{proof}[Proof of Theorem~\ref{the:limitVuintro2}]
Recall that $d_u(N)=d_{\mathbf{t},u}(0;N)$, see \eqref{eq:defduN} and \eqref{eq:defdzL}, when $R=1$ and $t_z^u=1\{\ell_{z,u}>0\}$ for all $z\in{G}$. In particular, the left-hand side of \eqref{eq:limitVuintro} can easily be upper bounded by its right-hand side using \eqref{eq:limitVunointro}, \eqref{eq:assumptionZd} and that $u_*=u_{**}$ by \cite{sharpnessRI1,sharpnessRI2,sharpnessRI3}. The corresponding lower bound is proved in \cite{GosRod}. 
\end{proof}

We now turn to the proof of \eqref{eq:boundonPEintro}. For the lower bound, we will need the following technical lemma, which uses a reasoning similar to \cite[(6.14)]{DrePreRod5} and above.

\begin{Lemme}
\label{lem:exitviatubes}
    There exist constants $\Cl[c]{Cexittube3}>0$, $1< \Cl{Cexittube2}<\infty$ and $C<\infty$  such that for all $N\geq 1$, integers $P\leq N$ and $x\in{G}$, there exists a sequence $x=x'_1,\dots,x'_P$ such that $|i-j|N/P-1/\Cr{Cexittube3}\leq d(x'_i,x'_{j})\leq (|i-j|+1/\Cr{Cexittube3})N/P$ for all $1\leq i,j\leq P$, $B(x'_P,N/P)\cap B(x,N)=\varnothing$, and letting $\mathcal{L}_{x}^{P,N}\stackrel{\textnormal{def.}}{=}\bigcup_{i=1}^PB(x'_i,\Cr{Cexittube2}N/P)$ we have for all $1\leq i,j\leq P$ and $y\in{B(x'_i,N/P)}$
    \begin{equation}
    \label{eq:exitviatubes}
        \P_y(H_{B(x'_j,N/P)}<T_{\mathcal{L}_x^{P,N}})\geq \exp(-C|i-j|).
    \end{equation}
\end{Lemme}
\begin{proof}
Let $(x=x_0,x_1,\dots)$ be the infinite path from \eqref{eq:intro_shortgeodesic}, and let $x'_1=x$ and $x'_i=x_{\lceil (i+\Cr{Cgeo}+1)N/(P\Cr{cgeo})\rceil }$ for all $i\geq2$.  Then  by \eqref{eq:intro_shortgeodesic} one can easily check that $|i-j|N/P-\Cr{cgeo}-\Cr{Cgeo}\leq d(x'_i,x'_j)\leq  (|i-j|+2+\Cr{cgeo}+2\Cr{Cgeo})N/P$ for each $1\leq i,j\leq P$, and $d(x,x'_P)\geq (P+1)N/P$. Therefore, the sequence $(x'_k)$ satisfies the conditions of the lemma when taking $\Cr{Cexittube3}=(2+\Cr{cgeo}+2\Cr{Cgeo})^{-1}$. It remains to verify \eqref{eq:exitviatubes} for this sequence, and first note that  for all $0\leq k\leq P-1$
\begin{equation}
\label{eq:xkxk+1sequence}
   B(x'_{k},N/P)\subset{B(x'_{k+1},(3+\Cr{cgeo}+2\Cr{Cgeo})N/P)}.
\end{equation}
In particular by \eqref{eq:capball} and \eqref{eq:probahittinglower}, writing $\Cr{Cexittube2}=(3+\Cr{cgeo}+2\Cr{Cgeo})\Cr{Chittinglower}$, we have for all $1\leq k\leq P-1$ and $z\in{B(x'_{k+1},(3+\Cr{cgeo}+2\Cr{Cgeo})N/P)}$
\begin{equation}
\label{eq:boundxkxk+1}
    \P_{z}(H_{B(x'_{k+1},N/P)}<T_{B(x'_{k+1},\Cr{Cexittube2}N/P)})\geq \frac{\Cr{ccapball}\Cr{cGreen}}{4^{\nu}(3+\Cr{cgeo}+2\Cr{Cgeo})^{\nu}}.
\end{equation}
If $i\leq j$, using \eqref{eq:boundxkxk+1} recursively on $k\in{\{i,\dots,j-1\}}$ combined with \eqref{eq:xkxk+1sequence} and the strong Markov property, one easily deduces \eqref{eq:exitviatubes}. If $j\leq i$ the proof is similar.
\end{proof}

Note that we will actually never use the bound \eqref{eq:exitviatubes} in this section, but only the fact that $\mathcal{L}_{x}^{P,N}$ is connected. The reason we included \eqref{eq:exitviatubes} is that it will later be useful in the proof of the lower bounds in Theorems~\ref{the:localuniquenessintro} and \ref{the:boundcapacityintro}, see \eqref{eq:prooflowerboundcapa} and \eqref{eq:probaE+} below.  Indeed the set $\mathcal{L}_x^{P,N}$ is a union of $P$ balls of size $\Cr{Cexittube2}N/P$ intersecting each other, starting in $x$ and ending in $B(x,N)^c$, and thus essentially corresponds to a generalization of the tube $\mathcal{R}_{N,\Cr{Cexittube2}N/P}$ from \eqref{eq:RNp} to more general graphs. 

Recall the function $F_{\nu}$ from \eqref{eq:defGnu} and below, as well as the families of weights $\mathbf{t}$ from \eqref{eq:defbft}, which depend on the functions $(g_z)_{z\in{\Lambda(R)}}$ which we will additionally assume to satisfy
\begin{equation}
\label{eq:additionalassumptiont}
    \text{ for each }z\in{\Lambda(R)},\,g_z\text{ is increasing and }g_z(\mathbf{0})=0,
\end{equation}
where $\mathbf{0}$ denotes the constant vector equal to $0$. The reason we assume $g_z(\mathbf{0})= 0$ is that if $g_z(\mathbf{0})>0$, then the weights $t_z^u$ are lower bounded by $g_z(\mathbf{0})$, and so the distance $d_{\mathbf{t},u}(x;N)$ is trivially at least linear with probability one. As we will soon explain, the following statement is a rigorous version of the result \eqref{eq:boundonPEintro}, generalized to graphs satisfying \eqref{eq:standingassumptionintro}.

\begin{The}
\label{the:FPPnointro}
Fix $\eta\in{(0,1)}$ and $\zeta\geq 1$. There exist constants $\Cl[c]{ccondFPPintro}>0$, depending only on $\eta$, as well as $\Cl{cRlarge},\Cl{cboundNintro}<\infty$ and $\Cl[c]{cFPPintro}>0,$ depending on $\eta$ and $\zeta$, such that for all $s>0$, $u>0$ satisfying $R\stackrel{\textnormal{def.}}{=}\zeta u^{-\frac1\nu}\geq \Cr{cRlarge}$, families of weights $\mathbf{t}$ as in \eqref{eq:defbft} and \eqref{eq:additionalassumptiont} satisfying
\begin{equation}
\label{eq:assumptiontzunointro}
    \P(t_z^{\eta u}\leq s)\leq \Cr{ccondFPPintro}\text{ for all }z\in{\Lambda(R)},
\end{equation}
and for all $N\geq \Cr{cboundNintro}u^{-\frac1\nu}$ and $x\in{\Lambda(R)}$, we have
\begin{equation}
\label{eq:bounddtvnointro}
   \exp\left(-{\Cl{CFPP2intro}}(1+\eta)F_{\nu}\big(Nu^{\frac1\nu}\big)\right)\leq \P\left(d_{\mathbf{t},u}(x;N)\leq \frac{\Cr{cFPPintro}sN}{R}\right)\leq 
    \exp\left(-{\Cl[c]{cFPP2intro}}(1-\eta)F_{\nu}\big(Nu^{\frac1\nu}\big)\right),
\end{equation}
where $\Cr{cFPP2intro}=\Cr{cbeta}$ and $\Cr{CFPP2intro}=\Cr{Cbeta}$ if $\nu\leq 1$, and $\Cr{cFPP2intro}=\Cr{cFPP2intro}(\zeta)>0$ and $\Cr{CFPP2intro}=\Cr{CFPP2intro}(\zeta)<\infty$ if $\nu>1$.
\end{The}
\begin{proof}
   We start with the upper bound, which is a simple consequence of Theorem~\ref{the:FPP}, replacing $u$ and $v$ therein by $\eta u$ and $u$ respectively, and so $(\sqrt{v}-\sqrt{u})^{\frac{2}{\nu}}$ is replaced by $u^{\frac1\nu}(1-\sqrt{\eta})^{\frac2\nu}$, taking $\Theta=1/\eta$, $\xi=1-\sqrt{\eta}$, $\Cr{ccondFPPintro}=\Cr{ccondFPP}(1\wedge \zeta)=\Cr{ccondFPP}$ (which now only depend on $\eta$), $\Cr{cFPPintro}=\Cr{cFPP}/(1\vee\zeta)$, $\Cr{cFPP2intro}=\Cr{cFPP2}$, taking $\Cr{cboundNintro}$ large enough so that the constant $\Cr{cboundN}$ from \eqref{eq:bounddtvfinal} can be absorbed in the exponential up to changing $1+\eta$ in $(1+\eta)^2$, and doing a change of variable for $\eta$.
    
    For the lower bound, one simply notes by \eqref{eq:exitviatubes} that the set $\L_x^{P,N}$ from Lemma~\ref{lem:exitviatubes} is connected and intersects $B(x,N)^c$, and letting $P=N/R$,  upon possibly further increasing $\Cr{Cexittube2}$, one further has by \eqref{eq:defLambda} that it contains an $R$-nearest neighbor path $(y_1,\dots,y_m)$ from $x$ to $B(x,N)^c$ such that $B(y_i,3R)\subset \L_x^{P,N}$ for all $i\leq m$. Therefore, since $g_z(\mathbf{0})=0$,
    \begin{equation}
    \label{eq:prooflowerFPPintro}
        \P\left(d_{\mathbf{t},u}(x;N)\leq \frac{\Cr{cFPPintro}sN}{R}\right)\geq \P\left(\I^u\cap \mathcal{L}_x^{P,N}=\varnothing\right)\geq\exp\Big(- \Cr{Ccapunionball}(1+\eta)uF_{\nu}\big(N,\frac{N}{P}\big)\Big),
    \end{equation}
    where we used \eqref{eq:defIuintro} and \eqref{eq:capunionbound2} in the last inequality, which is satisfied when taking $S_i=B(x'_i,\Cr{Cexittube2}N/P)$ and $\delta=1/(2\Cr{Cexittube3}^{-1}+2\Cr{Cexittube2})$ therein, as well as $Nu^{\frac1\nu}\geq \Cr{cboundNintro}$ and $R\geq \Cr{cRlarge}$ for some finite constants $\Cr{cboundNintro}$ and $\Cr{cRlarge}$ large enough. Since,  by \eqref{eq:defGnu}, \eqref{eq:defGnu} and our choice of $P$ and $R$ we have for $Nu^{\frac1\nu}$ large enough that $uF_{\nu}(N,N/P)\leq (1+\eta)F_{\nu}(Nu^{\frac1\nu})$ when $\nu\leq 1$, and $uF_{\nu}(N,N/P)\leq CF_{\nu}(Nu^{\frac1\nu})$ when $\nu>1$, we can easily conclude.
\end{proof}

Note that a path $\gamma$ from $x$ to $B(x,N)^c$ with length of order $N$ will typically be at distance less than $cu^{-\frac1\nu}$ from $cF_{\nu}(u^{\frac1\nu}N)$ different trajectories of interlacements, see Lemma~\ref{lem:capatube} as to why, and hence in view of \eqref{eq:bounddtvnointro} this path  will typically  be at distance less than $cu^{-\frac1\nu}$ from a given trajectory as before at least $c'u^{\frac1\nu}N/F_{\nu}(u^{\frac1\nu}N)$ different times, which is of order $1$ when $\nu>1$, but goes to infinity as $u^{\frac1\nu}N$ goes to infinity when $\nu\leq1$. In other words, when a trajectory of interlacements hits $B(\gamma,cu^{-\frac1\nu})$, it will in fact typically hit $B(\gamma,cu^{-\frac1\nu})$ of order $(u^{\frac1\nu}N)^{1-\nu}\log(u^{\frac1\nu}N)^{1\{\nu=1\}}$ many times when $\nu\leq 1$, but only a constant number of times when $\nu>1$.

In order to deduce \eqref{eq:boundonPEintro} from Theorem~\ref{the:FPPnointro}, one simply replaces $R$ by $R\sqrt{d}$ (since $\Lambda(R)=(R/\sqrt{d})\Z^d$ by our choice of the Euclidean distance), and take 
\begin{equation}
\label{eq:defgzZd}
    g_z\big((\ell_x)_{x\in{B(z,3R\sqrt{d})}}\big)=1\big\{\{x\in{z+(-\tfrac{R}{2},\tfrac{R}2]^d}:\,\ell_x>0\}\text{ is typical}\big\}
\end{equation}
in \eqref{eq:defbft}. The precise definition of a typical event is then simply that the function $g_z$ above satisfies \eqref{eq:defbft}, \eqref{eq:additionalassumptiont} and \eqref{eq:assumptiontzunointro} for $s=1/2$. The event $E$ from \eqref{eq:defEintro} then implies the event that  $d_{\mathbf{t},u}(x;N)\leq \frac{sN}{R}$, and the upper bound in \eqref{eq:boundonPEintro} then follows from \eqref{eq:bounddtvnointro}, \eqref{eq:defGnu} and \eqref{eq:assumptionZd}. For the lower bound, one can simply proceed similarly as in \eqref{eq:prooflowerFPPintro}.

\begin{Rk}
\phantomsection
\label{rk:endthmintro}
\begin{enumerate}
    \item\label{rk:2ptfunction}
  The bounds \eqref{eq:limitVunu>1nointro} and \eqref{eq:bounddtvnointro} are also satisfied when replacing $d_{\mathbf{t},u}(x;N)$ by $d_{\mathbf{t},u}(x,y)$ for any $y\in{G}$ with $d(x,y)> N$, as long as one replaces \eqref{eq:intro_shortgeodesic} by the condition
    \begin{equation}
\label{eq:intro_shortgeodesic2}
\begin{array}{l}
\text{for all }x,y\in{G},\text{ there exists a path }(x=x_0,x_1,x_2,\dots,x_P=y) \\[0.3em]
\text{such that }\Cr{cgeo}|k-p|-\Cr{Cgeo}\leq d(x_k,x_p)\leq \Cr{cgeo}|k-p|+\Cr{Cgeo}\text{ for all }k,p\in{\{0,\dots,P\}}.
\end{array}
\end{equation}
The upper bounds are trivial since $d_{\mathbf{t},u}(x;N)\leq d_{\mathbf{t},u}(x,y)$, see \eqref{eq:defdzL} and below, and it is easy to adapt the proof of the lower bound in \eqref{eq:limitVunu>1nointro} using condition \eqref{eq:intro_shortgeodesic2} instead of \eqref{eq:intro_shortgeodesic} therein. For the lower bound in \eqref{eq:bounddtvnointro} one simply proceeds similarly as in \eqref{eq:prooflowerFPPintro} but replacing $d_{\mathbf{t},u}(x;N)$ by $d_{\mathbf{t},u}(x,y)$ and $\mathcal{L}_x^{P,N}$ by a set $\mathcal{L}_{x,y}^{P,N}$ similar to the one from Lemma~\ref{lem:exitviatubes} but with the condition $x'_P=y$ instead of $B(x'_P,N/P)\cap B(x,N)=\varnothing$ therein. The existence of such a set follows from similar arguments as in Lemma~\ref{lem:exitviatubes}, replacing the use of \eqref{eq:intro_shortgeodesic} therein by \eqref{eq:intro_shortgeodesic2}. Condition \eqref{eq:intro_shortgeodesic2} is for instance satisfied on $\Z^d$, $d\geq3$, endowed with the Euclidean distance, see \cite[Remark~2.3]{GRS21}, but there are graphs that satisfy \eqref{eq:standingassumptionintro} and not \eqref{eq:intro_shortgeodesic2}, for instance the graph $G_2$ from \cite[(1.4)]{DrePreRod2} (since $d$ is not proportional to the graph distance on $G_2$, see the proof of \cite[Corollary~3.9]{DrePreRod2}). More generally, even if \eqref{eq:intro_shortgeodesic2} is not satisfied, one can still obtain \eqref{eq:bounddtvnointro} when replacing $d_{\mathbf{t},u}(x;N)$ by $d_{\mathbf{t},u}(x,y)$ as long as $d(x,y)> N$ and $y$ is in the geodesic $x=(x_0,x_1,\dots)$ from \eqref{eq:intro_shortgeodesic}.

\item\label{rk:annuluscrossing} Another possibility is to replace $d_{\mathbf{t},u}(x;N)$ by $d_{\mathbf{t},u}(x;N;2N)$ in \eqref{eq:limitVunointro}, \eqref{eq:limitVunu>1nointro} and \eqref{eq:bounddtvnointro}, and it is in fact possible to adapt our proof to this case for graphs that satisfy \eqref{eq:standingassumptionintro}. For the lower bounds one simply needs to replace in \eqref{eq:prooflowerFPPintro} the set $\L_x^{P,N}$ by a similar tube of length smaller than $(1+\eta)N$, but now starting in $B(x,N)$ and ending in $B(x,2N)^c$, which can easily be constructed from \eqref{eq:intro_shortgeodesic} similarly as in Lemma~\ref{lem:exitviatubes}. For the upper bounds, one needs a version of Proposition~\ref{pro:entropy} where in \eqref{eq:gammaincludedC} one asks instead that $\gamma$ is a path from $B(x,N-2L)$ to $B(x,2N+2L)$, and the rest of the proof is similar. This new version of Proposition~\ref{pro:entropy} can in fact be deduced from the current version of Proposition~\ref{pro:entropy} by considering the union of the families $\mathcal{A}_{y,\rho,p}^{L,K,N}$ over $y\in{B(x,N)\cap \Lambda(L)}$, for which \eqref{eq:cardA} is still satisfied up to changing the constant $\Cr{CcardA}$ therein by our choice of $p$ in \eqref{eq:choicep}.

\item\label{rk:EvsE'} One could replace the event $E$ from \eqref{eq:defEintro} by a similar event $E'$ where we replace $R\Z^d$ by $\Z^d$ and $sN/R$ by $sN$ therein. This corresponds intuitively to counting any point $x$ such that $\I^u\cap \big(x+(-\tfrac{R}{2},\tfrac{R}2]^d\big)$ is typical once, instead of counting the whole region $\I^u\cap \big(x+(-\tfrac{R}{2},\tfrac{R}2]^d\big)$ once.  One can easily show that $E'$ implies that $d_{\mathbf{t},u}(x;N)\leq \frac{csN}{R}$ for $\mathbf{t}$ as in \eqref{eq:defbft} with 
\begin{equation}
    \label{eq:defgz2}
    g_z\big((\ell_x)_{x\in{B(z,3R\sqrt{d})}}\big)=1\big\{\{x\in{y+(-\tfrac{R}{2},\tfrac{R}2]^d}:\,\ell_x>0\}\text{ is typical for all }y\in{z+(-\tfrac{R}{2},\tfrac{R}2]^d}\big\}
\end{equation}
since to any path $\gamma$ of length $k$ from $0$ to $B(N)^c$ in $\Z^d$, one can associate an $R\sqrt{d}$-nearest neighbor path $\gamma'$ from $0$ to $B(N)^c$ in $R\Z^d$ such that $\gamma$ spends a total time larger than $3^{-d}R$ in $x+(-\frac{R}{2},\frac{R}{2}]^d$ for each $x$ in ${\gamma'}$ (except possibly the last point). Using a reasoning similar to \eqref{eq:prooflowerFPPintro} for the lower bound, one deduces that \eqref{eq:boundonPEintro} still holds when replacing $E$ by $E'$ as long as the function $g_z$ from \eqref{eq:defgz2} satisfies \eqref{eq:additionalassumptiont} and \eqref{eq:assumptiontzunointro}. This is for instance the case when $R=\zeta u^{-\frac1{d-2}}$ for some $\zeta\geq C$ and we call $\{x\in{y+(-\tfrac{R}{2},\tfrac{R}2]^d}:\,\ell_x>0\}$ typical when it is non-empty, which essentially corresponds to the FPP distance naturally associated with $(\I^u+(-\frac{R}{2},\frac{R}{2}]^d)^c$. If the condition $R=\zeta u^{-\frac1{d-2}}$ for some $\zeta\geq C$ is not satisfied but \eqref{eq:assumptiontzunointro} is, one can still obtain an upper bound similar to the one in \eqref{eq:limitVuintro} but replacing the probability therein by either the probability of $E$ or $E'$, and $u_*$ by the smallest parameter such that \eqref{eq:assumptiontzunointro} is satisfied, for $g_z$ either as in \eqref{eq:defgzZd} or \eqref{eq:defgz2}. The additional interest of \eqref{eq:defEintro} is that the result is true as soon as $uN^{d-2}$ is large enough, instead of only asymptotically as $N\rightarrow\infty$ in \eqref{eq:limitVuintro}, at the cost of assuming that $uR^{d-2}$ is large.
\end{enumerate}
\end{Rk}

\section{Local uniqueness and capacity bounds}
\label{sec:localuniqueness}
In this section, we prove the two applications of our FPP result: Theorem~\ref{the:boundcapacitynointro} on the capacity of random walks, and Theorem~\ref{the:localuniquenessnointro} on the local uniqueness of random interlacements, which are versions of Theorems~\ref{the:localuniquenessintro} and \ref{the:boundcapacityintro} on general graphs satisfying \eqref{eq:standingassumptionintro}. As explained below \eqref{eq:interetmarcheintro}, the first step will be to deduce from Theorem~\ref{the:FPPnointro} the following bounds on the probability that a random walk hits the interlacement set at level $u$ before exiting a ball of size $N$ around its starting point. To simplify notation, let us abbreviate $\log_*(x)=\log(x)\vee1$, and recall the function $F_{\nu}$ from below \eqref{eq:defGnu}, as well as the notation $\overline{X}_A$ from above \eqref{eq:oldboundcapacity}.

\begin{Lemme}
\label{lem:hittinginter}
Fix some $\eta\in{(0,1/4]}$ and $u_0<\infty$. There exists a constant  $C<\infty$, depending on $\eta$ and $u_0$, such that for all $u\in{(0,u_0]}$, $N\geq Cu^{-\frac1\nu}$,  and $x\in{G}$, 
\begin{equation}
    \label{eq:interetmarche}
    \begin{split}
        \exp\left(-\frac{\Cl{cinteretmarche2}}{1-\eta}F_{\nu}(Nu^{\frac1\nu})\right)&\leq \P\otimes\P_x(\mathcal{I}^u\cap \overline{X}_{B(x,N)}=\varnothing)
        \\&\leq \begin{cases}\exp\left(-\frac{\Cl[c]{cinteretmarche}}{1+\eta}F_{\nu}(Nu^{\frac1\nu})\log_*(\tfrac{1}{u})^{-1_{\alpha=2\nu}}\right)&\text{ if }\alpha\geq 2\nu,
        \\\exp\left(-{\Cr{cinteretmarche}}Nu^{\frac{1-\alpha+2\nu}{\nu}}\right)&\text{ otherwise,}
        \end{cases}
    \end{split}
\end{equation}
where $\Cr{cinteretmarche2}=\Cr{Cbeta}$ and $\Cr{cinteretmarche}=\Cr{cbeta}$ if $\nu\leq 1$, and $\Cr{cinteretmarche2}=\Cr{cinteretmarche2}(u_0)<\infty$ and $\Cr{cinteretmarche}=\Cr{cinteretmarche}(u_0)>0$ if $\nu>1$. 
\end{Lemme}
\begin{proof}
    We start with the proof of the case $\alpha\geq 2\nu$ for the upper bound in \eqref{eq:interetmarche}. Fix $\eta\in{(0,1/4)}$, and let $\Cr{ccondFPPintro}=\Cr{ccondFPPintro}(\eta)$ be the constant from \eqref{eq:assumptiontzunointro}. First note that for $\zeta\geq2$, denoting by $X^1$ the first trajectory of interlacements hitting $B(x,\zeta u^{-\frac1\nu}/2)$, see above \eqref{eq:defRI}, we have by \cite[Lemma~5.3]{DrePreRod5}, \eqref{eq:defIuintro} and \eqref{eq:capball}
    \begin{equation}
    \label{eq:initializationFPP}
    \begin{split}
        \P\big(&\mathcal{I}^{\eta u}\cap B(x,\zeta u^{-\frac1\nu}/2)\neq\varnothing,\mathrm{cap}(\overline{X}^1_{B(x,\zeta u^{-\frac1\nu})})\geq u^{-1}\log_*(\tfrac{1}{u})^{-1_{\alpha=2\nu}}\big)\\&\geq \big(1-C\exp(-c\zeta\log(\zeta)^{-\frac1\nu1\{\alpha=2\nu\}} )\big)(1-\exp(-c\eta \zeta^{\nu}))
        \geq 1-\Cr{ccondFPPintro}
    \end{split}
    \end{equation}
    for some constants $c>0$ and $C<\infty$, and where the last inequality holds when $\zeta= \Cl{cmathfrakC}$ for some finite constant $\Cr{cmathfrakC}=\Cr{cmathfrakC}(\eta)\geq2$. Let $R=\Cr{cmathfrakC}u^{-\frac1\nu},$ and for any $y\in{\Lambda(R)}$ and $u'>0$
    \begin{equation}
    \label{eq:choicetlocaluniqueness}
    {t}_{y}^{u'}=1\{\mathcal{I}^{u'}\cap B(y,R)\text{ contains a connected component }\mathfrak{C}\text{ with }\mathrm{cap}(\mathfrak{C})\geq \tfrac{1}{u}\log_*(\tfrac{1}{u})^{-1_{\alpha=2\nu}}\}.
    \end{equation}
    It then follows from \eqref{eq:initializationFPP} that \eqref{eq:assumptiontzunointro} is fulfilled for $s=1/2$, and so by Theorem~\ref{the:FPPnointro}, letting $x'\in{\Lambda(R)}$ be such that $x\in{B(x',R)}$, if $N\geq \Cr{cboundNintro}u^{-\frac1\nu}/(1-\eta)$ then
    \begin{equation}
\label{eq:bounddtvfinalappli}
    \P\left(d_{\mathbf{t},u}(x';N(1-\eta))\leq \frac{\Cr{cFPPintro}N}{4R}\right)\leq 
    \exp\left(-\frac{\Cr{cFPP2intro}}{1+\eta}F_{\nu}\big(N(1-\eta)u^{\frac1\nu})\big)\right).
\end{equation}
    Note that for all $y\in{G}$, $z\in{B(y,R)}$ and sets $\mathfrak{C}\subset B(y,R)$ with $\mathrm{cap}(\mathfrak{C})\geq u^{-1}\log_*(\tfrac{1}{u})^{-1_{\alpha=2\nu}}$ by \eqref{eq:probahittinglower} we have
    \begin{equation}
    \label{eq:probahittingC}
    \begin{split}
        \P_z(H_{\mathfrak{C}}<T_{B(y,\Cr{Chittinglower}R)})\geq \frac{\Cr{cGreen}u^{-1}\log_*(\tfrac{1}{u})^{-1_{\alpha=2\nu}}}{4^{\nu}R^{\nu}}\geq\frac{\Cr{cGreen}}{4^{\nu}\Cr{cmathfrakC}^{\nu}}\log_*(\tfrac{1}{u})^{-1_{\alpha=2\nu}}.
    \end{split}
    \end{equation}
     Let $y_1,\dots,y_M$ be the sequence of vertices defined recursively as follows: $y_{i}$ is the first vertex $y\in{\Lambda(R)\cap B(x,N-\Cr{Chittinglower}R)}$ such that $B(y,R)$ is visited by $X$, with  $t_{y}^u=1$ and $B(y,R)\cap B(y_j,\Cr{Chittinglower}R)=\varnothing$ for all $1\leq j\leq i-1$. We stop this recursion at the first time $M$ after which there does not exist such a vertex $y$ anymore, and then if $N\eta\geq (\Cr{Chittinglower}+1)R$, on the event $d_{\mathbf{t},u}(x';N(1-\eta))\geq \frac{\Cr{cFPPintro}N}{4R}$ we must have $M\geq \frac{\Cr{cFPPintro}N}{4\Cr{Chittinglower}^{\alpha}\Cr{CLambda}R}$ in view of \eqref{eq:defLambda}, below \eqref{eq:defdzL}, and our definition of $x'$.
    We also let $H_i\stackrel{\textnormal{def.}}{=}\inf\{k\geq0:\,X_k\in{B(y_i,R)}\}$ and $T_i\stackrel{\textnormal{def.}}{=}\inf\{k\geq H_i:\,X_k\in{B(y_i,\Cr{Chittinglower}R)^c}\}$, and then $H_i\leq T_i\leq H_{i+1}$ are stopping times by definition. Conditionally on $\mathcal{I}^u$ and on the event $d_{\mathbf{t},u}(x';N(1-\eta))\geq \frac{\Cr{cFPPintro}N}{4R}$, we thus have for $Nu^{\frac1\nu}$ large enough
    \begin{equation}
    \label{eq:IuhitXrecursively}
    \begin{split}
        \P_x\big(\mathcal{I}^u\cap \overline{X}_{B(x,N)}=\varnothing\big)&\leq \P_x\Big(\forall i\in\big{\{1,\dots,\frac{\Cr{cFPPintro}N}{4\Cr{Chittinglower}^{\alpha}\Cr{CLambda}R}\big\}},\mathcal{I}^u\cap \{X_k,H_i\leq k\leq T_i\}=\varnothing\Big)
        \\&\leq \left(1-\frac{\Cr{cGreen}}{4^{\nu}\Cr{cmathfrakC}^{\nu}}\log_*(\tfrac{1}{u})^{-1_{\alpha=2\nu}}\right)^{\frac{\Cr{cFPPintro}Nu^{\frac1\nu}}{4\Cr{Chittinglower}^{\alpha}\Cr{CLambda}\Cr{cmathfrakC}}},
    \end{split}
    \end{equation}
where the last inequality follows from \eqref{eq:choicetlocaluniqueness} and \eqref{eq:probahittingC} by applying the strong Markov property at time $H_i$ recursively on $i$. Combining this with \eqref{eq:bounddtvfinalappli}, and noting that by \eqref{eq:defGnu}, the second line of \eqref{eq:IuhitXrecursively} is smaller than $\exp(-\Cr{cbeta}F_{\nu}(Nu^{\frac1\nu}))$ if $\nu\leq 1$ and $Nu^{\frac1\nu}$ is large enough, we obtain the upper bound in \eqref{eq:interetmarche} when $\alpha\geq 2\nu$ after a change of variable for $\eta$. When $\alpha<2\nu$, we note that by \cite[(3.17) and Lemma~4.4]{DrePreRod2}, the bound \eqref{eq:initializationFPP} can be replaced by
\begin{equation*}
    \begin{split}
        \P\big(&\mathcal{I}^{\eta u}\cap B(x,\zeta u^{-\frac1\nu}/2)\neq\varnothing,\mathrm{cap}(\overline{X}^1_{B(x,\zeta u^{-\frac1\nu})})\geq u^{-\tfrac{\alpha}{\nu}+1}\big)\\&\geq (1-\zeta^{-\frac{\alpha-\nu}{2}})(1-\exp(-c\eta \zeta^{\nu}))
        \geq 1-\Cr{ccondFPPintro}
    \end{split}
\end{equation*}
when taking $\zeta=\Cr{cmathfrakC}$ for a large enough constant $\Cr{cmathfrakC}<\infty$, and the rest of the proof is similar.

Let us now turn to the lower bound, and we let $ P\leq N$ be an integer which we will fix later. By Lemma~\ref{lem:exitviatubes} we have 
\begin{equation}\label{eq:IuinterXbar}
\begin{split}
    \P_x\big(\mathcal{I}^u\cap \overline{X}_{B(x,N)}=\varnothing\big)&\geq \P\big(\mathcal{I}^u\cap \mathcal{L}_{x}^{P,N}=\varnothing\big)\P_x(T_{\mathcal{L}_{x}^{P,N}}>H_{B(x_P,N/P)})
    \\&\geq \exp\big(-u\mathrm{cap}(\mathcal{L}_x^{P,N})\big)\exp(-CP),
\end{split}
\end{equation}
where in the last inequality we used \eqref{eq:defIuintro}. Moreover, by \eqref{eq:capunionbound2} with $n=P$, which is satisfied for $S_i=B(x'_i,\Cr{Cexittube2}N/P)$ and $\delta=1/(2\Cr{Cexittube3}^{-1}+2\Cr{Cexittube2})$ therein, if $N\geq P\geq \Cr{CboundP}$, one has $\mathrm{cap}(\mathcal{L}_x^{P,N})\leq  \Cr{Ccapunionball}(1+\eta)F_{\nu}(N,N/P)$. If $\nu>1$, we now choose $P=\lceil N(u/u_0)^{\frac1\nu}\rceil$, and then $uF_{\nu}(N,N/P)\leq CuNu^{\frac{1-\nu}{\nu}}=CF_{\nu}(Nu^{\frac1\nu})$. By \eqref{eq:IuinterXbar}, this yields the lower bound in \eqref{eq:interetmarche} when $Nu^{\frac1\nu}$ is large enough. If $\nu\leq 1$, we choose $P=\lceil (u/u_0)N^{\nu}\log(N(u/u_0)^{\frac1\nu})^{-2}\rceil$, which if $Nu^{\frac1\nu}$ is large enough satisfies  $N\geq P\geq \Cr{CboundP}$, $CP\leq \eta \Cr{Ccapunionball}F_{\nu}(Nu^{\frac1\nu})$ and $uF_{\nu}(N,N/P)=uN^{\nu}\log(P)^{-1\{\nu=1\}}\leq (1+\eta)F_{\nu}(Nu^{\frac1\nu})$, and we can conclude similarly as before.
\end{proof}

Let us now explain how to deduce a generalization of Theorem~\ref{the:boundcapacityintro} to graphs as in \eqref{eq:standingassumptionintro} from Lemmas~\ref{lem:exitviatubes} and~\ref{lem:hittinginter}, which improves \cite[Lemma~5.3]{DrePreRod5}. It holds whenever $\nu\geq1$ and $\alpha\geq 2\nu$, which is the case on $\Z^d$ if and only if $d\in\{3,4\}$ by \eqref{eq:assumptionZd}.

\begin{The}
\label{the:boundcapacitynointro}
Fix some $\eta\in{(0,1)}$. There exist constants $\Cl[c]{ccapwalk3},\Cl[c]{ccapwalk2},\Cl[c]{ccapwalk}>0$ as well as $\Cl{Ccapwalk},\Cl{Ccapwalk2}<\infty$, depending on $\eta$, such that for all $x\in{G}$ and $N,t\geq1$, if $\nu>1$ and $t\leq \Cr{ccapwalk3}N^{\nu-1}$, 
    \begin{equation}
    \begin{split}
    \label{eq:boundcapacitynu>1nointro}
    \Cr{ccapwalk2}\exp\big(-\Cr{Ccapwalk}t^{\frac1{\nu-1}}\big)&\leq \P_x\Big(\mathrm{cap}(\overline{X}_{B(x,N)})\leq \frac{N^{\nu}}{t}\Big)
    \\&\leq
    \begin{cases}
    \Cr{Ccapwalk2}\exp\big(-\Cr{ccapwalk}t^{\frac1{\nu-1}}\big),&\text{ if }\alpha>2\nu,
    \\\Cr{Ccapwalk2}\exp\big(-\Cr{ccapwalk}t^{\frac{1}{\nu-1}}\log(\tfrac{N^{\nu-1}}t)^{-\frac{\nu}{\nu-1}}\big),&\text{ if }\alpha=2\nu,
    \end{cases}
    \end{split}
\end{equation}
and if $\nu=1$ and $t\leq \log(N)/(\Cr{Cbeta}(1+\eta/2))$
    \begin{equation}
        \label{eq:boundcapacitynu=1nointro}
\Cr{ccapwalk2}\exp\big(-\exp(\Cr{Cbeta}(1+\eta)t)\big)\leq \P_x\Big(\mathrm{cap}(\overline{X}_{B(x,N)})\leq \frac{N}{t}\Big)
\leq
\Cr{Ccapwalk2}\exp\big(-\exp(\Cr{cbeta}(1-\eta)t)\big).
\end{equation}
\end{The}
\begin{proof}
    We start with the upper bounds. It follows from \eqref{eq:interetmarche} and \eqref{eq:linkcapRI} that for all $u>0$ (note that the condition $u\leq u_0$ is not used in the proof of the upper bound in \eqref{eq:interetmarche}) and $N\geq Cu^{-\frac1\nu}$
    \begin{equation*}
            \E_x\big[\!\exp(-u\mathrm{cap}(\overline{X}_{B(x,N)})\big]\leq \exp\left(-\frac{\Cr{cinteretmarche}}{1+\eta}F_{\nu}(Nu^{\frac1\nu})\log_*(\tfrac{1}{u})^{-1_{\alpha=2\nu}}\right).
    \end{equation*}
    Hence by a Chernov bound for all $t,u>0$ and $N\geq Cu^{-\frac1\nu}$
    \begin{equation}
    \label{eq:Chernovcap}
        \P_x\left(\mathrm{cap}(\overline{X}_{B(x,N})\leq\frac{N^{\nu}}{t}\right)\leq \exp\left(\frac{uN^{\nu}}{t}-\frac{\Cr{cinteretmarche}}{1+\eta}F_{\nu}(Nu^{\frac1\nu})\log_*(\tfrac{1}{u})^{-1_{\alpha=2\nu}}\right).
    \end{equation}
    If $\nu>1$ and $\alpha>2\nu$, we take $u=N^{-\nu}({\Cr{cinteretmarche}}t/(2(1+\eta)))^{\frac\nu{\nu-1}}$, and if $\alpha=2\nu$, we take $u=N^{-\nu}(\Cr{cinteretmarche}t/(2\nu\log(Nt^{-\frac1{\nu-1}})))^{\tfrac{\nu}{\nu-1}}$, which yields the upper bound in \eqref{eq:boundcapacitynu>1nointro}. If $\nu=1$, we take $u=\exp(\Cr{cinteretmarche}t(1+\eta)^{-2})/N$, which directly yields the upper bound in \eqref{eq:boundcapacitynu=1nointro} for $t$ large enough (which can be assumed w.l.o.g.\ by changing the constant $\Cr{Ccapwalk2}$) up to replacing $1-\eta$ by $(1+\eta)^{-3}$ therein. Note that the condition  $N\geq Cu^{-\frac1\nu}$ is indeed fulfilled when $\alpha>2\nu$ and $t$ is large enough, as well as when $\alpha=2\nu$ and $t/\log(Nt^{-\frac1{\nu-1}})$ is large enough, which in both cases can be assumed without loss of generality up to changing the constant $\Cr{Ccapwalk2}$.

    Let us now turn to the lower bounds. Recalling the set $\mathcal{L}_{x}^{P,N}$ from Lemma~\ref{lem:exitviatubes} we have by Lemma~\ref{lem:capatube}, for $\delta=1/(2\Cr{Cexittube3}^{-1}+2\Cr{Cexittube2})$ and $\eta/2$ instead of $\eta$, that $\mathrm{cap}(\mathcal{L}_{x}^{P,N})\leq N^{\nu}/t$ if $P=\lceil(\Cr{Ccapunionball}(1+\eta/2)t)^{\frac1{\nu-1}}\rceil$ when $\nu>1$, or if $P=\lceil\exp(\Cr{Ccapunionball}(1+\eta/2)t)\rceil$ when $\nu=1$. Note that $N\geq P\geq \Cr{CboundP}$ is satisfied, as required in Lemma~\ref{lem:capatube}, if $t$ is large enough, which can be assumed without loss of generality by changing the constant $\Cr{ccapwalk2}$, and if we take $\Cr{ccapwalk3}=(\Cr{Ccapunionball}(1+\eta/2))^{-1}$. Then by \eqref{eq:exitviatubes} we have
    \begin{equation}
    \label{eq:prooflowerboundcapa}
        \P_x\Big(\mathrm{cap}(\overline{X}_{B(x,N)})\leq\frac{N^{\nu}}{t}\Big)\geq \P_x(H_{B(x'_P,N)}<T_{\mathcal{L}_x^{P,N}})\geq \exp(-CP),
    \end{equation}
    which gives the lower bound \eqref{eq:boundcapacitynu>1nointro} when $\nu>1$, as well as the lower bound \eqref{eq:boundcapacitynu=1nointro} when $\nu=1$ and $t$ is large enough.
\end{proof}

We did not consider the case $\nu<1$ in Theorem~\ref{the:boundcapacitynointro} since $\mathrm{cap}(\overline{X}_{B(x,N)})$ is always deterministically larger than $cN^{\nu}$ for some constant $c$, see \cite[Lemma~3.2]{DrePreRod2}, and hence it is trivially never much smaller than $N^{\nu}$. Furthermore, $\mathrm{cap}(\overline{X}_{B(x,N)})$ is always larger than $cN$ when $\nu>1$, resp.\ $cN/\log(N)$ when $\nu=1$, and so conditions of the type $t\leq \Cr{ccapwalk3}N^{\nu-1}$, resp.\  $t\leq \log(N)/(\Cr{Cbeta}(1+\eta/2))$ in the lower bounds of Theorem~\ref{the:boundcapacitynointro} are necessary.  Let us now quickly explain how to deduce Theorem~\ref{the:boundcapacityintro} from Theorem~\ref{the:boundcapacitynointro} and \eqref{eq:capRnp}.

\begin{proof}[Proof of Theorem~\ref{the:boundcapacityintro}]
    The upper bound in \eqref{eq:illustrationcapacityintro} follows from \eqref{eq:capRnp}, \eqref{eq:boundcapacitynu>1nointro}, \eqref{eq:boundcapacitynu=1nointro} and \eqref{eq:assumptionZd} after a change of variable for $\eta$.  Note to that effect that the conditions on $t$ and $N$ above \eqref{eq:boundcapacitynu>1nointro} and \eqref{eq:boundcapacitynu=1nointro} are never used in the proofs of the respective upper bounds, and that for $N\leq cp$, the upper bound in \eqref{eq:illustrationcapacityintro} trivially holds upon changing the constants therein. For the lower bound, one first notes that on $\Z^d$, one can simply take $(0=x_0,x_1,\dots)$ in \eqref{eq:intro_shortgeodesic} as multiples of $e_1$, and thus also $x'_1,\dots,x'_P$ in Lemma~\ref{lem:exitviatubes}. In particular $\mathcal{R}_{N,p}$ contains the set $\L_0^{\lfloor cN/p\rfloor,N}\cap B(N)$, and hence one can lower bound the probability in \eqref{eq:illustrationcapacityintro} by the one in \eqref{eq:exitviatubes} for $P=\lfloor cN/p\rfloor$, $i=1$ and $j=P$.
\end{proof}

Interestingly, the bounds from Theorem~\ref{the:boundcapacityintro} become sharp up to constants in the regime $1<\nu<\alpha/2$ (which never occurs on $\Z^d$). Indeed, in this regime of parameters  Lemma~\ref{lem:capatube} implies that $cN^{\nu}P^{1-\nu}\leq  \mathrm{cap}(\mathcal{L}_{x}^{P,N})\leq CN^{\nu}P^{1-\nu}$ for $P\in{[\Cr{CboundP},N]}$, and  \eqref{eq:boundcapacitynu>1nointro} then yields 
\begin{equation}
\label{eq:illustrationcapacitynointro}
     c\exp\big(-CP\big)\leq \P_0\Big(\mathrm{cap}(\overline{X}_{B(N)})\leq \mathrm{cap}\big(\mathcal{L}_{x}^{P,N}\big)\Big)\leq
        C\exp\big(-cP\big)\text{ if }1<\nu<\frac{\alpha}{2}.
\end{equation}
Since on $\Z^d$ one can choose the sequence $(x'_1,\dots)$ from Lemma~\ref{lem:exitviatubes} so that $\mathcal{L}^{\lceil CN/p\rceil,N}_0\subset \mathcal{R}_{N,p}\subset\mathcal{L}^{\lfloor cN/p\rfloor,N}_0$ by a similar reasoning as in the proof of Theorem~\ref{the:boundcapacityintro}, \eqref{eq:illustrationcapacitynointro} is indeed a generalization of \eqref{eq:illustrationcapacityintro} to graphs satisfying \eqref{eq:standingassumptionintro} with $1<\nu<\alpha/2$. When $\nu=\alpha/2$, one obtains on general graphs a bound similar to \eqref{eq:illustrationcapacitynointro} for $G=\Z^4$, replacing $\mathcal{R}_{N,p}$ by $\mathcal{L}_x^{N/p,N}$ and $\log(p)^2$ by $\log(p)^{\nu/(\nu-1)}$ therein. When $\nu=1$, it is however not clear if a bound similar to \eqref{eq:illustrationcapacityintro} for $G=\Z^3$ holds, since the proof on $\Z^3$ used the equality $\Cr{cGreenasymp}=\Cr{CGreenasymp}$ in \eqref{eq:intro_Green_asymp}. If $\nu=1$ and $\Cr{cGreenasymp}<\Cr{CGreenasymp}$, one has to change $1-\eta$ in the upper bound of \eqref{eq:illustrationcapacityintro} by some positive constant, which still improves on the bounds from \cite[Lemma~3.2]{DrePreRod2}.

\begin{Rk}
\label{rk:overdefXbar}
    Theorem~\ref{the:boundcapacitynointro} remains true, and hence Theorem~\ref{the:boundcapacityintro} as well, when replacing $\overline{X}_{B(x,N)}$ by $\hat{X}_{B(x,N)}\stackrel{\textnormal{def.}}{=}\{X_k,k\geq0\}\cap B(x,N)$. Since $\mathrm{cap}(\hat{X}_{B(x,N)})\geq \mathrm{cap}(\overline{X}_{B(x,N)})$ by monotonicity of capacity, it suffices to prove the lower bounds \eqref{eq:boundcapacitynu>1nointro} and \eqref{eq:boundcapacitynu=1nointro} for $\overline{X}_{B(x,N)}$. For \eqref{eq:boundcapacitynu>1nointro}, one simply notes that $\mathrm{cap}(\hat{X}_{B(x,N)})\leq \mathrm{cap}(\overline{X}_{B(x,KN)})$ on the event that $X$ never returns to $B(x,N)$ after time $T_{B(x,KN)},$ for any $K\geq1$. Taking $K$ large enough so that $\P_z(H_{B(x,N)}=\infty)\geq 1/2$ for any $z\in{B(x,KN)^c}$, which is possible by \eqref{eq:boundentrance}, and using the strong Markov property finishes the proof when $\nu>1$.
    
    In order to obtain the correct constant $\Cr{Cbeta}$ in the lower bound \eqref{eq:boundcapacitynu=1nointro} when $\nu=1$, one can proceed as follows : let $x=x'_1,x'_2,\dots$ be the sequence from Lemma~\ref{lem:exitviatubes} for $N(1+\eta/5)$ instead of $N$ and $P=\lceil\exp(\Cr{Cbeta}(1+\eta/5)^2t)\rceil$($\leq N$ under the conditions of Theorem~\ref{the:boundcapacitynointro}) therein, then by a reasoning similar to the case $\nu>1$ and to \eqref{eq:prooflowerboundcapa} one has
    \begin{equation*}
    \begin{split}
        \P_x\Big(\mathrm{cap}(\hat{X}_{B(x,N)})\leq\frac{N^{\nu}}{t}\Big)&\geq \P_x(H_{B(x'_P,\frac{N}{P})}<T_{\mathcal{L}_x^{P,N(1+\eta)/5}},H_{B(x,N)}\circ\theta_{H_{B(x'_P,N)}}=\infty)
        \\&\geq \exp(-CP)\inf_{z\in{B(x'_P,\frac{N}{P})}}\P_z(H_{B(x,N)}=\infty),
    \end{split}
    \end{equation*}
    where $\theta_t$ is the time-shift by $t$. To conclude, it will be enough to show that
    \begin{equation}
    \label{eq:betterlowerbound}
        \inf_{z\in{B(x'_P,\frac{\eta N}{5\Cr{Cexittube2}})}}\P_z(H_{B(x,N)}=\infty)\geq \Cl[c]{cbetterbound}
    \end{equation}
    for some constant $\Cr{cbetterbound}>0$, since $\frac{N}{P}\leq \frac{\eta N}{5\Cr{Cexittube2}}$ upon taking $t$ large enough w.l.o.g.
   The proof of \eqref{eq:betterlowerbound} uses a standard chaining argument similar to \eqref{eq:exitviatubes}, that we now detail since it will actually also be useful in the proof of Theorem~\ref{the:localuniquenessintro}.  By \eqref{eq:intro_shortgeodesic} and the definition of $x'_P$ in the proof of Lemma~\ref{lem:exitviatubes}, one can find a subsequence $x'_P=y_0,y_1,\dots$ of the sequence $x=x_0,x_1,\dots$ such that $d(x,y_i)\geq \Cr{cgeo}(i+\lceil (P+\Cr{Cgeo}+1)N(1+\eta/5)/(P\Cr{cgeo})\rceil)-\Cr{Cgeo}\geq N(1+\eta/5)$ and $\Cr{cgeo}|i-j|-\Cr{Cgeo}\leq d(y_j,y_i)\leq \Cr{cgeo}|i-j|+\Cr{Cgeo}$ for each $i,j\in{\N_0}$. Letting now $y'_1=x_P$ and $y'_i=y_{\lceil (i+\Cr{Cgeo}+1)\eta N/(5\Cr{Cexittube2}\Cr{cgeo})\rceil}$ for all $i\geq2$, \eqref{eq:boundxkxk+1} still holds when replacing $x'_k$ by $y'_k$ and $P$ by $5\Cr{Cexittube2}/\eta$ therein. Let $K'$ be the first $i\in{\N_0}$ such that $B(y'_i,\eta N/(5\Cr{Cexittube2}))\subset B(x,KN)^c$, and let ${\L}_x^{P,N,\eta}=\cup_{i=1}^{K'}B(y'_i,\eta N/5)$. Then, similarly as in \eqref{eq:exitviatubes} we have that $\P_{z}(H_{B(y'_{K'},\eta N/(5\Cr{Cexittube2}))}<T_{{\L}_x^{P,N,\eta}})\geq \exp(-cK')$ for any $z\in{B(x'_P,\eta N/(5\Cr{Cexittube2})}$. Since $K'\leq CK$ for a constant $C=C(\eta)$ by definition and $\L_x^{P,N,\eta}\subset B(x,N)^c$, one easily deduces \eqref{eq:betterlowerbound} by combining this with \eqref{eq:boundentrance} for $K$ large enough.
\end{Rk}

It remains to deduce Theorem~\ref{the:localuniquenessintro} from Theorem~\ref{the:boundcapacityintro} for the upper bound, and from ideas similar to Lemma~\ref{lem:hittinginter} for the lower bound. For each $x\in{G}$, $\xi>1$, $u>0$ and $N\geq1$, let us define the local uniqueness event in $B(x,\xi N)$ for interlacements at level $u$ in $B(x,N)$ by 
\begin{equation}
\label{eq:deflocaluniqnointro}
    \text{LocUniq}_{u,N,\xi}(x)\stackrel{\textnormal{def.}}{=}\bigcap_{x,y\in{\mathcal{I}^u}\cap B(x,N)}\left\{x\leftrightarrow y\text{ in }\mathcal{I}^u\cap B(x,\xi N)\right\}.
\end{equation}
Note that by convention if $\I^u\cap B(x,N)=\varnothing$, then the event $\text{LocUniq}_{u,N,\xi}(x)$ occurs. Sometimes, it is additionally required that $\I^u\cap B(x,N)\neq \varnothing$ in the definition of local uniqueness, and our results would actually still be satisfied with this alternative definition, which can easily be proved using a union bound, \eqref{eq:defIuintro} and \eqref{eq:capball}. Moreover, the event $\text{LocUniq}_{u,N,\xi}(x)$ has only good chances to occur if $B(x,N)$ is connected in $B(x,\xi N)$, that is if for every $y,y'\in{B(x,N)}$ there is a nearest neighbor path in $B(x,\xi N)$ connecting $y$ to $y'$. We thus define $\xi_0$ as the infimum of all $\xi\geq1$ such that $B(x,N)$ is connected in $B(x,\xi N)$ for all $x\in{G}$ and $N$ sufficiently large, which is finite and does not depend on $N$ by \cite[(3.4)]{DrePreRod2}. Note that we actually have $\xi_0=1$ if $d$ is the graph distance, and in fact also for any of the graphs from \cite[(1.4)]{DrePreRod2} or \cite[Theorem~2]{MR2076770}. Let us now prove a generalization of Theorem~\ref{the:localuniquenessintro} to graphs satisfying \eqref{eq:standingassumptionintro}, which improves  \cite[Theorem~5.1]{DrePreRod5}. Recall the function $F_{\nu}$ from below \eqref{eq:defGnu}.

\begin{The}
\label{the:localuniquenessnointro}
    Fix some $\eta\in{(0,1)}$, $u_0>0$ and $\xi>\xi_0$. There exists a  constant $\Cl{cuNlocuniq}<\infty$, depending on $\eta,$ $u_0$ and $\xi$, such that for all $N\geq1$ and $u\in{(0,u_0/2]}$ with $uN^{\nu}\geq \Cr{cuNlocuniq}\log(u_0/u)^{1\{\alpha=2\nu\}/\nu}$
     \begin{equation}
    \label{eq:boundlocaluniqnointro}
    \begin{split}
        \exp\left(-{\Cl{clocaluniq}}F_{\nu}\big(Nu^{\frac1\nu}\big)\right)\leq \sup_{x\in{G}}\P\big(\textnormal{LocUniq}_{u,N,\xi}(x)^c\big)\leq 
        \begin{cases}
            \exp\left(-{\Cl[c]{clocaluniq2}}F_{\nu}\big(Nu^{\frac1\nu}\big)\right)&\text{ if }\alpha>2\nu,
            \\\exp\left(-\frac{\Cr{clocaluniq2}Nu^{\frac1\nu}}{\log(u_0/u)}\right)&\text{ if }\alpha=2\nu,
        \end{cases}
    \end{split}
    \end{equation}
    where $\Cr{clocaluniq}=(1+\eta)(\xi-1)^{\nu}\Cr{Cbeta}$ if $\nu\leq 1$, $\Cr{clocaluniq}=\Cr{clocaluniq}(u_0,\xi)$ if $\nu>1$, $\Cr{clocaluniq2}=(1-\eta)(\xi-1)^{\nu}\Cr{cbeta}$ if $\nu=1$, and $\Cr{clocaluniq2}=\Cr{clocaluniq2}(u_0,\xi)$ if $\nu\neq 1$.
\end{The}
\begin{proof}
        We start with the upper bound, and by monotonicity we can assume w.l.o.g.\ that $1/\eta\in{\{20,21,\dots\}}$. Let us decompose $\I^u$ into independent copies $\I_k^{\eta u}$, $k\in{\{1,\dots,1/\eta\}}$ of $\I^{\eta u}$, so that $\I^u= \cup_{k=1}^{1/\eta}\I_k^{\eta u}$, and fix $i,j\in{\{1,\dots,1/\eta\}}$. Let $\overline{\xi}\stackrel{\textnormal{def.}}{=}(1-\eta)^2\xi$, and we assume w.l.o.g.\ that $\eta$ is small enough so that $\overline{\xi}>\xi_0$. Let $Y$ be the trace of an excursion from $B(x,N)$ to $B(x,\overline{\xi} N)^c$ for $\I_i^{\eta u}$, that is the trace of an excursion started on hitting $B(x,N)$ and stopped on reaching $B(x,\overline{\xi}N)^c$ for one of the random walks in the random interlacement process associated to $\I_i^{\eta u}$, and let $Z$ be the trace of an excursion from $B(x,N)$ to $B(x,\overline{\xi} N)^c$ for $\I_j^{\eta u}$.  For $t\geq1$ to be chosen later, abbreviating $\overline{N}\stackrel{\textnormal{def.}}{=}(\overline{\xi}-1)N$, let us define the event
    \begin{equation*}
        E\stackrel{\textnormal{def.}}{=}\Big\{\mathrm{cap}(Y)\geq \frac{\overline{N}^{\nu}}{t},\mathrm{cap}(Z)\geq \frac{\overline{N}^{\nu}}{t}\Big\}.
    \end{equation*}
    Furthermore, we arbitrarily fix some distinct numbers $i',j'\in{\{1,\dots,1/\eta \}\setminus \{i,j\}}$, and define $\overline{\I}^{(1-4\eta)u}_{i,j,i',j'}=\cup_{\substack{k=1,k\neq i,j,i',j'}}^{1/\eta}\I_k^{\eta u}$. Note that $\overline{\I}^{(1-4\eta)u}_{i,j,i',j'}$ contains a set with the same law as $\I^{u(1-4\eta)}$ (but does not necessarily exactly have the same law since it is possible that $i=j$). For $\hat{\xi}\stackrel{\textnormal{def.}}{=}(1-\eta)\xi\in{(\overline{\xi},\xi)}$, let us denote by $Y^1,\dots,Y^{M_Y}$, resp.\  $Z^1,\dots,Z^{M_Z}$, the traces of the trajectories in the interlacement process associated to $\overline{\I}^{(1-4\eta)u}_{i,j,i',j'}$  hitting $Y$, resp.\ $Z$, started on first hitting $Y$, resp.\ $Z$, and stopped on first exiting $B(x,\hat{\xi} N)$. We have on the event $E$
    \begin{equation}
    \label{eq:boundMYMZ}
        \P\Big(M_Y\geq \frac{\eta u\overline{N}^{\nu}}{t},M_Z\geq \frac{\eta u\overline{N}^{\nu}}{t}\,\big|\,\I_i^{\eta u},\I_j^{\eta u}\Big)\geq 1-2\exp\Big(-\frac{(1-6\eta\log(1/\eta))^2u\overline{N}^{\nu}}{t}\Big)
    \end{equation}
    since $M_Y$ and $M_Z$ both dominate Poisson random variables with parameter larger than $u(1-4\eta)\overline{N}^{\nu}/t$ on the event $E$, and we used large deviation bounds on Poisson random variables, see for instance \cite[p.21-23]{MR3185193}. Let us denote by $K_Y$, resp.\ $K_Z$, the index $k\in{\{1,\dots,M_Y\}}$, resp.\ $k\in{\{1,\dots,M_Z\}}$, for which $\mathrm{cap}(Y^k)$, resp.\ $\mathrm{cap}(Z^k)$, is maximal. Abbreviating $\hat{N}\stackrel{\textnormal{def.}}{=}(\hat{\xi}-\overline{\xi})N$ and for $s\geq1$ to be fixed later, let
    \begin{equation*}
        \hat{E}\stackrel{\textnormal{def.}}{=}\Big\{\mathrm{cap}(Y^{K_Y})\geq \frac{\hat{N}^{\nu}}{s},\mathrm{cap}(Z^{K_Z})\geq \frac{\hat{N}^{\nu}}{s}\Big\}.
    \end{equation*}
    By \eqref{eq:boundMYMZ} on the event $E$
    \begin{equation}
    \label{eq:boundPhatE}
        \P(\hat{E}^c\,|\,\I_i^{\eta u},\I_j^{\eta u})\leq 2\exp\Big(-\frac{(1-6\eta\log(1/\eta))^2u\overline{N}^{\nu}}{t}\Big)+2\P\Big(\mathrm{cap}(Y^1)\leq \frac{\hat{N}^{\nu}}{s}\Big)^{\frac{\eta u\overline{N}^{\nu}}{t}}.
    \end{equation}
    Let us further denote by $\hat{Y}^{1},\dots,\hat{Y}^{\hat{M}_Y}$, resp.\ $\hat{Z}^{1},\dots,\hat{Z}^{\hat{M}_Z}$, the traces of the trajectories in the interlacements process associated to $\I_{i'}^{\eta u},$ resp.\ $\I_{j'}^{\eta u}$,  hitting $Y^{K_Y}$, resp.\ $Z^{K_Z}$, started on first hitting $Y^{K_Y}$, resp.\ $Z^{K_Z}$, and stopped on first exiting $B(x,\xi N)$. On the event $\hat{E}$, using a reasoning similar to \eqref{eq:boundMYMZ} to bound the probability that $\hat{M}_Y$ or $\hat{M}_Z$ is smaller than $\eta u\hat{N}^{\nu}/(10s)$, we have
    \begin{equation}
    \label{eq:interynzm}
    \begin{split}
        \P&\left(\bigcap_{n=1}^{\hat{M}_Y}\bigcap_{m=1}^{\hat{M}_Z}\big\{\hat{Y}^n\cap \hat{Z}^m =\varnothing\big\}\,\Big|\,\I_i^{\eta u},\I_j^{\eta u},\overline{\I}^{(1-4\eta)u}_{i,j,i',j'}\right)
        \\&\leq 2\exp\Big(-\frac{\eta u\hat{N}^{\nu}}{2s}\Big)+\P(\hat{Y}^1\cap \hat{Z}^1=\varnothing)^{\frac{\eta u\hat{N}^{\nu}}{10s}},
    \end{split}
    \end{equation}
    where we used that the events $\{\hat{Y}^n\cap \hat{Z}^n=\varnothing\}$, $n\leq \eta u\hat{N}^{\nu}/(10s)$, are conditionally i.i.d.\ since $i'\neq j'$. Let us also abbreviate $\tilde{N}=(\xi-\hat{\xi})N$, and let us denote by $y^1$ and $z^1$ the respective starting points of $\hat{Y}^1$ and $\hat{Z}^1$. Since $y^1$ and $z^1$ are both in the connected component of $B(x,N)$ in $B(x,\hat{\xi} N)$ by definition and $\hat{\xi}>\xi_0$, see below \eqref{eq:deflocaluniqnointro}, there is a nearest neighbor path from $y^1$ to $z^1$ included in $B(x,\hat{\xi}N)$. In particular by \eqref{eq:defLambda} and \eqref{eq:dvsdgr}, and assuming w.l.o.g.\ that $\Cr{Chittinglower}^{-1}(\xi-\hat{\xi})N/4\geq \Cr{Cdvsdgr}^{-1}$,  one can find a sequence $x_1,\dots,x_p\in{\Lambda(\Cr{Chittinglower}^{-1}\tilde{N}/8)\cap B(x,\hat{\xi}N+\Cr{Chittinglower}^{-1}\tilde{N}/8))}$ such that $z^1\in{B(x_1,\Cr{Chittinglower}^{-1}\tilde{N}/8)}$, $y^1\in{B(x_p,\Cr{Chittinglower}^{-1}\tilde{N}/8)}$, $d(x_i,x_{i+1})\leq \Cr{Chittinglower}^{-1}\tilde{N}/2$ and $p\leq \Cr{CLambda}(8\xi\Cr{Chittinglower}/(\xi-\hat{\xi}))^{\alpha}$. Since $B(x_i,\tilde{N}/2)\subset B(x,\xi N)$ for all $1\leq i\leq p$, it follows from using \eqref{eq:capball} and \eqref{eq:probahittinglower} iteratively that the probability that $\hat{Z}^1$ hits $B(y^1,\Cr{Chittinglower}^{-1}\tilde{N})$ before leaving $B(x,\xi N)$ is larger than a constant, depending on $\xi$ and $\hat{\xi}$. Using \eqref{eq:probahittinglower} again, one deduces that for all $r\geq1$,
    \begin{equation}
    \label{eq:interRW}
    \begin{split}
        \P(\hat{Y}^1\cap \hat{Z}^1\neq\varnothing)&\geq c\P\Big(\mathrm{cap}\big(\hat{Y}^1\cap B\big(y^1,\Cr{Chittinglower}^{-1}\tilde{N}\big)\big)\geq \frac{\Cr{Chittinglower}^{-\nu}\tilde{N}^{\nu}}{2^{\nu}r\log(N)^{1_{\alpha=2\nu}}}\Big)\frac{\Cr{cGreen}}{8^{\nu}r\log(N)^{1_{\alpha=2\nu}}}
        \\&\geq c(1-C\exp(-cr^{\frac1\nu}))\frac{\Cr{cGreen}}{8^{\nu}r\log(N)^{1_{\alpha=2\nu}}}\geq \frac{c'}{\log(N)^{1_{\alpha=2\nu}}}
        \end{split}
    \end{equation}
    for some constants $C<\infty$ and $c,c'>0$, depending on $\hat{\xi}$ and $\xi$, where the second inequality follows from \cite[Lemma~5.3]{DrePreRod5} and the inequality $\log(\Cr{Chittinglower}^{-1}\tilde{N})\leq 2^{\nu}\log(N)$ which can be assumed to hold w.l.o.g, and the last inequality follows from taking $r$ a large enough constant. Combining this with \eqref{eq:interynzm} and using a union bound, we deduce that on the event $\hat{E}$
    \begin{equation}
    \label{eq:localuniqueness2traj}
        \P\big(Y^{K_Y}\nleftrightarrow Z^{K_Z}\text{ in }\I^{u}\cap B(x,\xi N)\,|\,\I_i^{\eta u},\I_j^{\eta u},\overline{\I}^{(1-4\eta)u}_{i,j,i',j'}\big)\leq C\exp\Big(-\frac{cuN^{\nu}}{s\log(N)^{1_{\alpha=2\nu}}}\Big)
    \end{equation}
    for some constant $C<\infty$ and $c>0$, depending on $\xi$ and $\eta$. Note also that the total number of traces of excursions $Y$ and $Z$ as before is smaller than $C(uN^{\nu})^2$ with probability at least $1-CuN^{\nu}\exp(-cuN^{\nu})$ for some positive and finite constant $c,C$ by \cite[(5.15)]{DrePreRod5}. Since $\textnormal{LocUniq}_{u,N,\xi}(x)^c$ implies that there exist $i,j,i',j'$ and trajectories $Y$ and $Z$ as before such that the event in the probability on the left-hand side of \eqref{eq:localuniqueness2traj} occurs, using \eqref{eq:boundPhatE} and noting that if the excursion associated to ${Y}$ hit $B(x,N)$ in $y$ for the first time, then $\mathrm{cap}(Y)$ stochastically dominates $ \mathrm{cap}(\overline{X}_{B(y,\overline{N})})$ under $\P_y$, and similarly for ${Z}$, as well as $Y^1$ with $\hat{N}$ instead of $\overline{N}$, we have for all $t,s\geq1$ that the probability of $\textnormal{LocUniq}_{u,N,\xi}(x)^c$ can be upper bounded by
    \begin{equation}
    \label{eq:finalupperloluniq}
    \begin{split}
         C(uN^{\nu})^2\bigg(&\exp\Big(-\frac{(1-6\eta\log(1/\eta))^2u\overline{N}^{\nu}}{t}\Big)+\exp\Big(-\frac{cuN^{\nu}}{s\log(N)^{1_{\alpha=2\nu}}}\Big)
         \\&+\sup_{y\in{G}}\P_y\Big(\mathrm{cap}(\overline{X}_{B(y,\overline{N})})\leq \frac{\overline{N}^{\nu}}{t}\Big)+\P_y\Big(\mathrm{cap}(\overline{X}_{B(y,\hat{N})})\leq \frac{\hat{N}^{\nu}}{s}\Big)^{\frac{\eta u\overline{N}^{\nu}}{t}}\bigg).
    \end{split}
    \end{equation}
    Using Theorem~\ref{the:boundcapacitynointro}, and taking  $t=s=(u^{\frac1\nu}N)^{\nu-1}$ if $\nu>1$ and $\alpha> 2\nu$, one obtains the upper bound in \eqref{eq:boundlocaluniqnointro} in this case. Note that the condition $t\leq \Cr{ccapwalk3}N^{\nu-1}$, or $t\leq \log(N)/(\Cr{Cbeta}(1+\eta/2))$ if $\nu=1$, is actually never used in the proof of the upper bounds of Theorem~\ref{the:boundcapacitynointro}. When $\alpha=2\nu$, one takes $t=(u^{\frac1\nu}N)^{\nu-1}\log(u_0/u)$ and $s=C\log(N)$, where the constant $C$ is chosen large enough so that the last probability in \eqref{eq:finalupperloluniq} is smaller than $1/2$, which is possible by \cite[Lemma~5.3]{DrePreRod5}.  Using Theorem~\ref{the:boundcapacitynointro} again, one can then bound \eqref{eq:finalupperloluniq} by
    \begin{equation*}
        C(uN^{\nu})^2\Big(\exp\Big(-\frac{cu^{\frac1\nu}N}{\log(u_0/u)}\Big)+\exp\Big(-\frac{cuN^{\nu}}{\log(N)^{2}}\Big)\Big).
    \end{equation*}
    Noting that $\nu\geq2$ when $\alpha=2\nu$ by \eqref{eq:condnualpha}, one easily deduces \eqref{eq:boundlocaluniqnointro} if $uN^{\nu}\geq \Cr{cuNlocuniq}\log(u_0/u)^{1/\nu}$ for some constant $\Cr{cuNlocuniq}$ large enough. When $\nu<1$, one notices that by \cite[Lemma~3.2]{DrePreRod2} (or in a more complicated fashion by \eqref{eq:capunionbound}) there exists a constant $c$ so that $\mathrm{cap}(\overline{X}_{B(y,M)})\geq cM^{\nu}$ for all $y\in{G}$ and $M\geq1$, and so taking $t=s=1/c$ we can also conclude. Finally when $\nu=1$ one chooses $t=\log(u\overline{N})/(\Cr{cbeta}(1-\eta))$ and $s=\log(t)/(\Cr{cbeta}(1-\eta))$. The second line of \eqref{eq:finalupperloluniq} can then be bounded using \eqref{eq:boundcapacitynu=1nointro} by $\Cr{Ccapwalk2}\exp\big(-u\overline{N}\big)$. Since $cuN^{\nu}/s\geq u\overline{N}^{\nu}/t\geq \Cr{cbeta}u(\xi-1)(1-2\eta\xi/(\xi-1))(1-\eta)^2N/\log(uN)$ if $uN^{\nu}$ is large enough, one can easily conclude by \eqref{eq:finalupperloluniq}, and after a change of variable for $\eta$
    \medskip
    
    Let us now turn to the lower bound. Let $x_0\in{G}$,  $N'\stackrel{\textnormal{def.}}{=}(1+\eta)\xi N$, $P\in{[2,N']}$ be an integer that we will fix later, $P'\stackrel{\textnormal{def.}}{=}2(1+\eta)P$, and let $x_0=x'_1,\dots,x'_{P'}$ as in Lemma~\ref{lem:exitviatubes},  but for $2(1+\eta)N'$ instead of $N$, $P'$ instead of $P$, and $x_0$ instead of $x$ therein. Let $M\in{\N}$ that we will fix later, fix $I=M+\lceil 1+1/\Cr{Cexittube3}+P\rceil$ and $x\stackrel{\textnormal{def.}}{=}x'_{I}$, and note that $B(x'_M,N'/P)\cup B(x'_{P'},N'/P)\subset B(x,N')^c$ upon additionally assuming that $P\geq \Cl{cNPball}$ for some large enough constant $\Cr{cNPball}=\Cr{cNPball}(\eta)$. Let us also define $J\stackrel{\textnormal{def.}}{=}M+\lceil P(1+1/((1+\eta)\xi))\rceil$ and $z\stackrel{\textnormal{def.}}{=}x'_{J+1}$, which satisfies $B(z,N'/P)\subset B(x,N)$. Let us further abbreviate $\tilde{\L}_{x}^{P,N'}\stackrel{\textnormal{def.}}{=}\cup_{i=M}^{I}B(x'_i,\Cr{Cexittube2}N'/P)$ and $\tilde{\L}_{z}^{P,N'}\stackrel{\textnormal{def.}}{=}\cup_{i=J+1}^{P'}B(x'_i,\Cr{Cexittube2}N'/P)$, and note that $\tilde{\L}_{x}^{P,N'}\cap \tilde{\L}_{z}^{P,N'}=\varnothing$ if $P\geq \Cr{cNPball}$, upon increasing the constant $\Cr{cNPball}$ if necessary. The sets $\tilde{\L}_{x}^{P,N'}$ and $\tilde{\L}_{z}^{P,N'}$ will serve as two disjoint "tubes" starting in $B(x,N)$ and ending in $B(x,N')^c$, in which we will construct two interlacement trajectories hitting $B(x,N)$ and not connected in $\I^u\cap B(x,\xi N)$, see Figure~\ref{fig:locuniqlower}.  

    \begin{figure}[ht]
\centering
\includegraphics[scale=0.26]{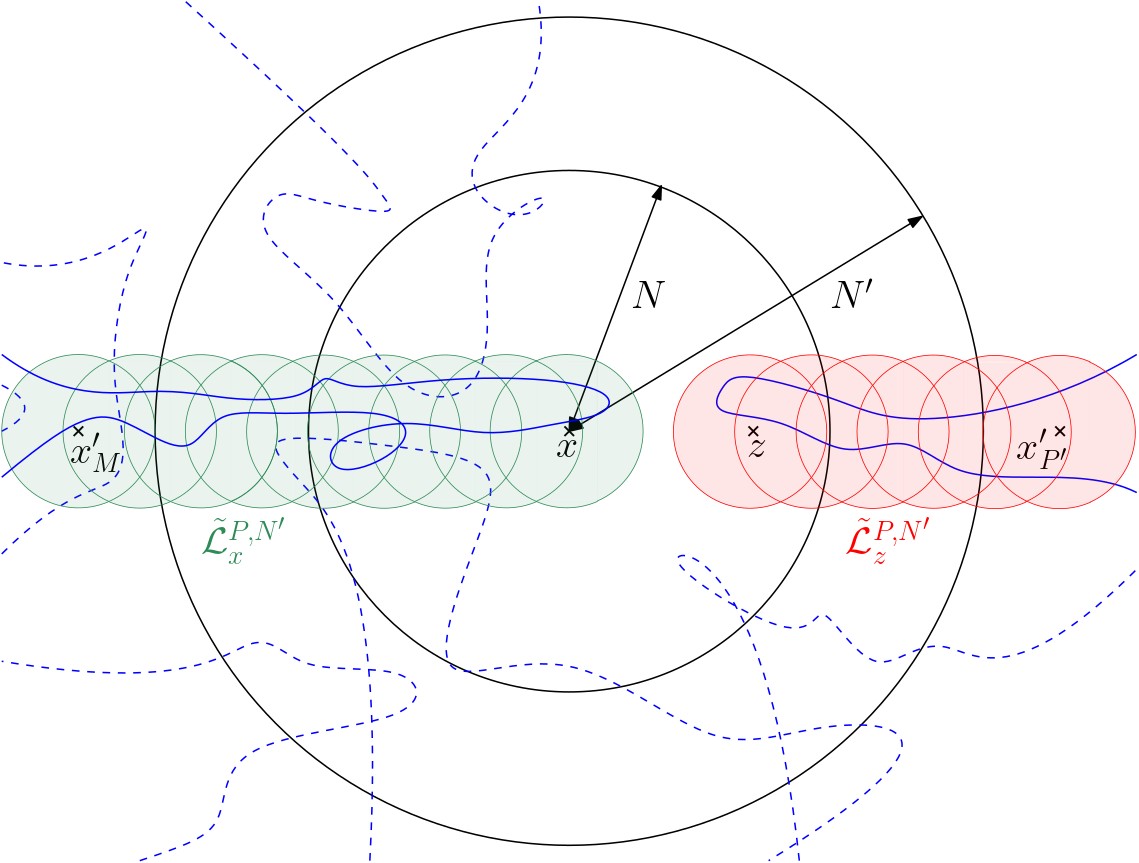}
\caption{The strategy to ensure that the local uniqueness event does not occur: a) There is exactly one trajectory hitting $B(x'_{P'},N'/P)$, and this trajectory visits $B(z,N'/P)$ while staying confined in the tube $\tilde{\L}^{P,N'}_z$. b) There are trajectories hitting $B(x'_M,N'/P)$, the first such trajectory visits $B(x,N'/P)$ while staying confined in the tube $\tilde{\L}^{P,N'}_x$, whereas all the other such (dashed) trajectories never visit $B(x,\xi N)$. c) Any other (dashed) trajectory never visits $\tilde{\L}_z^{P,N'}$.}
    \label{fig:locuniqlower}
\end{figure}

   Let us abbreviate $\Cl{cabbrev}\stackrel{\textnormal{def.}}{=}(2^{2+\nu}\Cr{CGreen}\Cr{Ccapball})^{-\frac1\nu}$. Recall that $N_u^A$ is the number of trajectories in the interlacement process at level $u$ hitting $A$, let $\tilde{N}_u^{\tilde{\mathcal{L}}_z^{P,N'}}$ be the number of trajectories hitting $\tilde{\mathcal{L}}_z^{P,N'}$ but not $B(x'_M,N'/P)\cup B(x'_{P'},N'/P)$, and $\tilde{N}_u^{B(x'_M,N'/P)}$ be the number of trajectories hitting $B(x'_M,N'/P)$ but not $B(x'_{P'},N'/P)$. Then in view of \eqref{eq:defequicap} and \eqref{eq:defRI}, and by thinning property of Poisson random variables,  $\tilde{N}_u^{B(x'_M,N'/P)}$ is a Poisson random variable with parameter 
   \begin{equation*}
   \begin{split}
       &u\sum_{y\in{B(x'_M,N'/P)}}\lambda_y\P_y(H_{B(x'_{P'},N'/P)}=\infty)\P_y(H_{B(x'_{P'},N'/P)}=\infty,\tilde{H}_{B(x'_M,N'/P)}=\infty)
       \\&\geq u\inf_{y\in{\partial  B(x'_M,\Cr{cabbrev}N'/P)^c}}\P_y(H_{B(x'_{P'},N'/P)\cup B(x'_M,N'/P)}=\infty)^2\mathrm{cap}(B(x'_M,N'/P))
    \end{split}
   \end{equation*} 
   where we used the strong Markov property at the first exit time of $B(x'_M,\Cr{cabbrev}N'/P)$ to bound the last probability on the first line. Moreover if $P\geq \Cr{cNPball}$ and upon increasing $\Cr{cNPball}$ (which depends on the choice of $M$), by a union bound, \eqref{eq:boundentrance} our choice of $\Cr{cabbrev}$ we have  for all $y\in{\partial  B(x'_M,\Cr{cabbrev}N'/P)^c}$
   \begin{equation*}
       \P_y\big(H_{B(x'_{P'},N'/P)\cup B(x'_M,N'/P)}<\infty\big)\leq \frac12,
   \end{equation*}
   where we used that $d(x'_{P'},y)\geq (2P-M-1/\Cr{Cexittube3}-\Cr{cabbrev})N'/P-\Cr{Cdvsdgr}^{-1}>\Cr{cabbrev}N'/P$. In particular, if $P\leq (u/u_0)^{\frac1\nu}N$, it follows from \eqref{eq:capball} that $\tilde{N}_u^{B(x'_M,N'/P)}$ dominates a Poisson random variable with constant parameter, and note that by \eqref{eq:capball} it is also clearly dominated by a Poisson random variable with parameter $\Cr{Ccapball}u(N'/P)^{\nu}$. 
   
   Moreover, $\tilde{N}_u^{\tilde{\mathcal{L}}_z^{P,N'}}$ is by definition dominated by a Poisson random variable with parameter $u\mathrm{cap}(\tilde{\mathcal{L}}_z^{P,N'})-u\mathrm{cap}(B(x'_{P'},N'/P))$. To bound $\mathrm{cap}(\tilde{\mathcal{L}}_z^{P,N'})$, note that if we take $S_i=B(x'_{i+J},\Cr{Cexittube2}N'/P)$  for all $1\leq i\leq P'-J$, then since $d(x'_i,x'_j)\leq (j-i+ 1/\Cr{Cexittube3})N'/P\leq (j-i+ 1/\Cr{Cexittube3})N'(3\eta+(\xi-1)/\xi)/(P'-J)$ for all $J\leq i\leq j\leq P'$, and similarly for the lower bound, the inequality \eqref{eq:capunionbound2} is in force for some finite constant $\delta$, $P'-J$ instead of $P$ and $N'(3\eta+(\xi-1)/\xi)$ instead of $N$ therein, as long as $P\geq \Cr{cNPball}$ upon increasing $\Cr{cNPball}$. Using \eqref{eq:defGnu} and by our choice of $N'$, the function $F_{\nu}(N'(3\eta+(\xi-1)/\xi),N'(3\eta+(\xi-1)/\xi)/(P'-J))$ appearing therein can moreover be upper bounded by $F_{\nu}((\xi-1)N,\xi N/P)(1+C\eta)$ for some constant $C<\infty$, upon increasing $\Cr{cNPball}$ again. Combining the previous observations and using that $\tilde{N}_u^{\tilde{\mathcal{L}}_z^{P,N'}}$, $\tilde{N}_u^{B(x'_M,N'/P)}$ and $N_u^{B(x'_{P'},N'/P)}$  are independent, we deduce that there exists a large enough constant $\Cl{ccapN'}<\infty$ such that for all $P\in{[\Cr{cNPball},(u/u_0)^{\frac1\nu}N]}$ 
    \begin{equation}
    \label{eq:NuLPN'}
    \begin{split}&\P\big(\tilde{N}_u^{\tilde{\mathcal{L}}_{z}^{P,N'}}=0,N_u^{B(x'_{P'},N'/P)}=1,\tilde{N}_u^{B(x'_M,N'/P)}\in{[1,\Cr{ccapN'}u(N'/P)^{\nu}]}\big)
    \\&\geq cu\mathrm{cap}(B(x'_{P'},N'/P))\exp\big(-u\Cr{Ccapunionball}(1+C\eta)F_{\nu}((\xi-1)N,\xi N/P))\big),
    \end{split}
    \end{equation}
    for some constant $c>0$ and $C<\infty$ depending on $\eta$, $\xi$ and $u_0$. Let $X^{1,x'_{P'}}$ be the first doubly infinite trajectory of random interlacements hitting $B(x'_{P'},N'/P)$ started at its first hitting time of $B(x'_{P'},N'/P)$, see above \eqref{eq:defRI}, and denote by $E_{z}^+$ the event that $(X^{1,x'_{P'}}_k)_{k\geq0}$ first visits $B(z,N'/P)$ and then returns to $B(x'_{P'},N'/P)$ before exiting $\tilde{\L}_{z}^{P,N'}$, and afterwards never returns to $B(x,\xi N)$ anymore. 
    By the strong Markov property at time $H_{B(z,N'/P)}$ and at the first hitting time of $B(x'_{P'},N'/P)$ after $H_{B(z,N'/P)}$
    we have if $P\geq \Cr{cNPball}$
    \begin{equation}
    \label{eq:probaE+}
    \begin{split}
        \P(E_{z}^+)&\geq \inf_{y\in{B(x'_{P'},N'/P)}}\P_y(H_{B(z,N'/P)}<T_{\tilde{\L}_{z}^{P,N'}})\inf_{y\in{B(z,N'/P)}}\P_y(H_{B(x'_{P'},N'/P)}<T_{\tilde{\L}_{z}^{P,N'}})
         \\&\phantom{\geq}\times\inf_{y\in{\partial B(x'_{P'},N'/P)}}\P_{y}(H_{B(x,\xi N)}=\infty)
        \\&\geq \Cr{cbetterbound}\exp(-CP),
    \end{split}
    \end{equation}
     where  in the last inequality we used  \eqref{eq:exitviatubes} twice and  the fact that $P'\leq 4P$ to bound the probabilities on the first line, as well as a reasoning similar to \eqref{eq:betterlowerbound} (with $5\eta$ instead of $\eta$ and $\xi N$ instead of $N$ therein) to bound the probability on the second line, together with the inequality $\frac{N'}{P}\leq \frac{\xi \eta N}{\Cr{Cexittube2}}$ upon increasing $\Cr{cNPball}$. Moreover, let $E_{z}^-$ be the event that $(X_{-k}^{1,x'_{P'}})_{k\geq0}$ never visits $B(x,\xi N)$, then, using the convention $0/0=\infty$, by \eqref{eq:defRI} and the Markov property at time $T_{B(x'_{P'},{K}N'/P)}$ for all $K>1$ and $P\geq  \Cr{cNPball}$, assuming that $\Cr{cNPball}=\Cr{cNPball}(K)$ is large enough so that $B(x'_{P'},{K} N'/P)\subset B(x,\xi N)^c$,
    \begin{equation}
    \label{eq:probaE-}
    \begin{split}
        \P(E_{z}^-)&\geq \inf_{y\in{B(x'_{P'},N'/P)}}\frac{\P_y(\tilde{H}_{B(x'_{P'},N'/P)}>T_{B(x'_{P'},{K} N'/P)})}{\P_y(\tilde{H}_{B(x'_{P'},N'/P)}=\infty)}
        \\&\phantom{\geq}\times\inf_{y\in{\partial B(x'_{P'},{K} N'/P)^c}}\P_y(H_{B(x'_{P'},N'/P)\cup B(x,\xi N)}=\infty)
        \\&\geq \inf_{y\in{\partial B(x'_{P'},{K} N'/P)^c}}\P_y(H_{B(x,\xi N)}=\infty)-\P_y(H_{B(x'_{P'},N'/P)}<\infty)\geq \frac{\Cr{cbetterbound}}{2}.
        %\\&\geq 1-\frac{2^{2\nu+1}\Cr{CGreen}\Cr{Ccapball}}{K^{\nu}},
    \end{split}
    \end{equation}
    In the last inequality we used  \eqref{eq:boundentrance} to upper bound the second probability by $\Cr{cbetterbound}/2$ if $K$ is a large enough constant, and a reasoning similar to \eqref{eq:betterlowerbound} (with $5\eta$ instead of $\eta$ and $\xi N$ instead of $N$ therein) to lower bound the first probability by $\Cr{cbetterbound}$ if $2KN'/P\leq \xi\eta N/\Cr{Cexittube2}$, which holds upon increasing  $\Cr{cNPball}$.
    
    Let us also denote by $X^{i,x'_{1}}$ the $i$-th doubly infinite trajectory of random interlacements hitting $B(x'_{M},N'/P)$ but not $B(x'_{P'},N'/P)$, started at its first hitting time of $B(x'_{M},N'/P)$, and by $E_{x}^\pm$ the events defined similarly as $E_z^{\pm}$, but replacing $x'_{P'}$ and $z$ by  $x'_M$ and $x$ respectively. One can lower bound $\P(E_{x}^{\pm})$ similarly as in \eqref{eq:probaE+} and \eqref{eq:probaE-}, with small adaptations to ensure that the walk avoids $B(x'_{P'},N'/P)$. Note that to ensure that the walk starting in $B(x'_M,N'/P)$ never visits $B(x,\xi N)$, one needs to slightly adapt the proof of \eqref{eq:betterlowerbound} by forcing the walk to visits successively $B(x'_{M-1},N'/P),\dots,B(x'_1,N'/P)$ while avoiding $B(x,\xi N)$, and then never come back in $B(x,\xi N)$, which is possible upon fixing $M$ a large enough constant (which plays the same role as $K'$ in the proof of \eqref{eq:betterlowerbound}). Moreover by the same reasoning,  one has
    \begin{equation}
    \label{eq:internothitblambda}
    \P\Big(\bigcap_{i=2}^{\Cr{ccapN'}u(N'/P)^{\nu}}X^{i,x'_{1}}\cap B(x,\xi N)=\varnothing\Big)\geq \exp\big(-Cu(N/P)^{\nu}\big).
    \end{equation}
    Note that the intersection of the events on the first line of \eqref{eq:NuLPN'} and on the left-hand side of \eqref{eq:internothitblambda} with $E_{z}^+\cap E_{z}^-\cap E_{x}^+\cap E_{x}^-$ implies that there is a unique trajectory in  the interlacements process at level $u$ visiting $\tilde{\L}_z^{P,N'}\cap B(x,\xi N)$, which visits $B(z,N'/P)\subset B(x,N)$ and then exits $B(x,\xi N)$ without ever hitting $B(x,\xi N)\setminus \tilde{\L}_z^{P,N'}$, and that there is a trajectory hitting $B(x,N'/P)\subset B(x,N)$ but not $\tilde{\L}_z^{P,N'}\cap B(x,\xi N)$, which together implies $\textnormal{LocUniq}_{u,N,\xi}(x)^c$, see Figure~\ref{fig:locuniqlower}. Hence we have by independence of the previous events and \eqref{eq:capball} that if $P\in{[\Cr{cNPball},(u/u_0)^{\frac1\nu}N]}$, then one can lower bound $\P\big(\textnormal{LocUniq}_{u,N,\xi}(x)^c\big)$ by 
    \begin{equation*}
    \begin{split}
        cu(N/P)^{\nu}\exp\big(-u\Cr{Ccapunionball}(1+C\eta)F_{\nu}((\xi-1)N,\xi N/P)-CP-Cu(N/P)^{\nu}\big)
    \end{split}
    \end{equation*}
    for some constants $c>0$ and $C<\infty$. We now take $P=\lceil uN^{\nu}/\log(uN^{\nu})^2\rceil$ if $\nu\leq 1$ and $P=\lfloor(u/u_0)^{\frac1\nu}N\rfloor$ if $\nu>1$, which satisfies $P\in{[\Cr{cNPball},(u/u_0)^{\frac1\nu}N]}$ if  $uN^{\nu}\geq \Cr{cuNlocuniq}$ for some constant $\Cr{cuNlocuniq}=\Cr{cuNlocuniq}(\xi,u_0,\eta)$ large enough,  and one readily deduces the lower bound in \eqref{eq:boundlocaluniqnointro} up to a change of variable for $\eta$.
    \end{proof}

Note that the condition $uN^{\nu}\geq \Cr{cuNlocuniq}$ in Theorem~\ref{the:localuniquenessnointro} when $\alpha>2\nu$ is necessary for the lower bound in \eqref{eq:boundlocaluniqnointro} to hold since when $Nu^{\frac1\nu}\rightarrow0$, the probability of non-local uniqueness converges to $0$, because this probability is smaller than the probability that $\I^u\cap B(x,N)\neq\varnothing$, and thus smaller than $1-\exp(-\Cr{Ccapball}uN^{\nu})$ by \eqref{eq:capball} and \eqref{eq:defIuintro}.

\begin{Rk}
\phantomsection\label{rk:endlocuniq}
\begin{enumerate}
    \item\label{rk:differentlocaluniq} The technique we use to deduce the upper bound of \eqref{eq:boundlocaluniqnointro} from Theorem~\ref{the:boundcapacitynointro} could actually be slightly simplified if $\alpha>2\nu$ and one does not want to compute the exact value of the constant $\Cr{clocaluniq2}$ when $\nu=1$. Currently, the proof connects two trajectories of interlacements $Y$ and $Z$ by first proving that with high probability there is a trajectory $Y^i$ hitting $Y$ and with large capacity, and a trajectory $Z^i$ hitting $Z$ and with large capacity, see \eqref{eq:boundPhatE}, and that still with high probability there is a trajectory of interlacements hitting $Y^i$ which intersects another trajectory of interlacements hitting $Z^i$, see \eqref{eq:interynzm}. One could instead simply ask that there is a trajectory $Y^i$ hitting $Y$ which intersects another trajectory $Z^i$ hitting $Z$, and the probability of this event can be bounded similarly as in \eqref{eq:interynzm} and \eqref{eq:interRW}, and is of the correct order, but only up to constants. In the proof of \cite[Theorem~5.1]{DrePreRod5}, one uses an even simpler approach which consists in connecting $Y$ and $Z$ directly by a single trajectory of random interlacements, and the probability of this event is bounded in \cite[Lemma~4.3]{DrePreRod2} following ideas from \cite[Lemma~12]{MR2819660}. However this simpler strategy would not be sufficient to yield the upper bound in \eqref{eq:boundlocaluniqnointro} when $\nu\geq1$ (even with Theorem~\ref{the:boundcapacitynointro} at hand), as can for instance be seen by comparing the first exponential in \eqref{eq:finalupperloluniq} with \cite[(4.4)]{DrePreRod2} when $\mathrm{cap}(U)$ and $\mathrm{cap}(V)$ are both of order $N^{\nu}/t$ and $L=N$ therein.
    \item \label{rk:suplowerbound}
    One can prove a lower bound on $\P(\textnormal{LocUniq}_{u,N,\xi}(x)^c)$ similar to the one from \eqref{eq:boundlocaluniqnointro} for any $x\in{G}$ with the following property: $B(x,N/2)$ contains a point $z$ such that there exist geodesics $(z=z_0,z_1,\dots,z_{CN})$ and $(x=x_0,x_1,\dots,x_{CN})$ as in \eqref{eq:intro_shortgeodesic} with $B(x_i,cN)\cap B(z_j,cN)=\varnothing$ for all $1\leq i, j\leq CN$, where $c$ is a constant that can be taken arbitrarily small up to changing the constant $\Cr{cuNlocuniq}$ in Theorem~\ref{the:localuniquenessnointro}, and $C<\infty$ is a large enough constant, depending on $\xi$. Indeed one can then proceed similarly as in Lemma~\ref{lem:exitviatubes} to build from the previous geodesics two sets $\tilde{\L}_z^{P,N'}$ and $\tilde{\L}_x^{P,N'}$ which consist respectively of balls $B(z_i,\Cr{Cexittube2}N'/P)$ and $B(x_i,\Cr{Cexittube2}N'/P)$, and are thus disjoint if $P$ is large enough, and which can each contain a random walk starting respectively in $B(z,N'/P)$ and $B(x,N'/P)$, until reaching $B(x,N')^c$, and which then never come back in $B(x,\xi N)$. One can then use these sets to force non-local uniqueness similarly as in the proof of Theorem~\ref{the:localuniquenessnointro}, and we leave the details to the reader. For example, one can thus prove the previous lower bound if $x$ belongs to a doubly infinite geodesic for the graph distance, and that $d=\Cr{cgeo}d_{\mathrm{gr}}$ along this geodesic. In particular, the previous lower bound is satisfied for all $x\in{G}$ if $G$ is any of the graphs from \cite[(1.4)]{DrePreRod2}, see \cite[Theorem~4.1]{MR864581} as to why this is true for $G_4$, and for $G_3$ one simply considers geodesics from $x$ to a point at distance at least $CN$ in the direction of $e_1$, where $e_1$ is the unit vector in the first Cartesian direction, take $z$ at distance $N/4$ from $x$ in the direction $-e_1$, and consider a second geodesic from $z$ to a point at distance at least $CN$ in the direction of $-e_1$.
    
    Actually, one can also prove that the previous lower bound is satisfied for all $x\in{G}$ for any graph $\G$ as in \eqref{eq:standingassumptionintro} with $\nu\geq1$. To prove this, one simply takes $z=x$ in the proof of Theorem~\ref{the:localuniquenessnointro} and replaces the set $\tilde{\L}_z^{P,N'}$ therein by the set ${\L}_x^{P,N'}$ from Lemma~\ref{lem:exitviatubes}. One can then simply require that $\L_{x}^{P,N'}$ contains a unique trajectory of interlacement hitting both $B(x,N)$ and $B(x,N')^c$, which afterwards never returns in $B(x,\xi N)$ for both its backwards and forward parts similarly as in the events $E_z^+\cap E_z^-$, and that there is another random walk starting in $B(x,N)\setminus B(\L_x^{P,N'},\Cl[c]{cfinalball}N)$ which never hits $\L_x^{P,N'}$, where $\Cr{cfinalball}>0$ is a small enough constant. Note that this last event happens with probability at least $1-CN^{-\nu}\mathrm{cap}(\L_x^{P,N'})$ by \eqref{entrancegreenequi}, which is larger than $1/2$ when $\nu\geq 1$ and $P$ is large enough in view of \eqref{eq:defGnu} and \eqref{eq:capunionbound2}, and that $P$ can indeed be chosen large enough for the choice of $P$ from the proof of Theorem~\ref{the:localuniquenessnointro} up to increasing $\Cr{cuNlocuniq}$. One can also show that there exists a walk in $\I^u$ intersecting $B(x,N)\setminus B(\L_x^{P,N'},\Cr{cfinalball}N)$ with high probability if $\Cr{cuNlocuniq}$ is large enough and $\Cr{cfinalball}$ is small enough since by \eqref{eq:intro_sizeball} and \eqref{eq:boundonlambda} we then have $|B(\L_{x}^{P,N'},2\Cr{cfinalball}N)|\leq C(\Cr{Cexittube2}/P+2\Cr{cfinalball})^{\alpha}(N')^{\alpha}P\leq |B(x,N-\Cr{cfinalball}N)|$. We leave the details to the reader.
    \item  \label{rk:constants}
    Let us quickly comment on the value of the constant $\Cr{clocaluniq2}$ from \eqref{eq:boundlocaluniqnointro}, which is not explicit when $\nu<1$, contrary to the case $\nu=1$. This is due to the fact that when $\nu=1$, one can take $s$ large enough so that the second exponential in \eqref{eq:finalupperloluniq} is smaller than the first, and at the same time $\mathrm{cap}(\overline{X}_{B(y,\hat{N})})\geq \frac{\hat{N}^{\nu}}{s}$ occurs with large enough probability by Theorem~\ref{the:boundcapacitynointro}. However when $\nu<1$, $\mathrm{cap}(\overline{X}_{B(y,\hat{N})})\geq \frac{\hat{N}^{\nu}}{s}$ can occur only if $s$ is constant, and then the second exponential in \eqref{eq:finalupperloluniq} could be larger than the first. It is an interesting question whether this is only a technical problem, or if $\Cr{clocaluniq2}$ can in fact be taken smaller than $\Cr{cbeta}(\xi-1)^{\nu}$ when $\nu<1$. 
     \item\label{rk:betterctegffcablesystem} In the proofs of \cite[Theorem~1.7 and Lemma~6.4]{DrePreRod5}, one uses the bounds on local uniqueness from \cite[Theorem~5.1]{DrePreRod5}. If one instead now uses the better bounds from Theorem~\ref{the:localuniquenessnointro}, one can in fact replace $\ell$ by $M\xi\log(r/\xi)\log(\xi)^{1\{\alpha=2\nu\}}$ when $\alpha\geq 2\nu$ and $\nu>1$ (which is the most interesting case) in \cite[(6.19) and (8.5)]{DrePreRod5}. In turn, this means that one can replace the constants $c_7$ and $c_8$ in \cite[Theorem~1.7 and Proposition~6.1]{DrePreRod5} by $\nu-1$, and only assume that $r/\xi\geq \log(\xi)\log\log(\xi)^C$ for $C$ large enough when $\alpha=2\nu$ in \cite[(1.46)]{DrePreRod5}. Note that when $\alpha=2\nu$, this new version of \cite[(1.46)]{DrePreRod5} is still better than the estimate one could obtain from adapting the proof of \cite[Proposition~6.1]{DrePreRod5} to this case, see \cite[Remark~8.1,1)]{DrePreRod5} as to why. We also refer to \cite{DrePreRod8} for further applications of Theorem~\ref{the:FPPnointro} to the percolation of the Gaussian free field on the cable system.
    \item \label{rk:sharpalpha<2nu}
    When $\alpha<2\nu$, one cannot obtain the same upper bounds as in the case $\nu>1$ and $\alpha>2\nu$ from \eqref{eq:boundcapacitynu>1nointro}. Indeed, it follows from subadditivity of capacity and \eqref{eq:intro_Green} that  for all $s,N\geq1$ and $x\in{G}$
\begin{equation*}
    \P_x(\mathrm{cap}(\overline{X}_{B(x,N)})\leq {sN^{\alpha-\nu}})\geq \P_x(T_{B(x,N)}\leq sN^{\alpha-\nu}\Cr{cGreen})\geq 1-C/s
\end{equation*}
for some finite constant $C<\infty$, where in the last inequality we used Markov's inequality combined with \cite[Lemma~A.1]{DrePreRod2}. In particular, this probability is larger than $1/2$ if $s$ is a large enough constant, which implies that upper bounds similar to \eqref{eq:boundcapacitynu>1nointro} cannot be true anymore when $\alpha<2\nu$. This indicates that an upper bound similar to the case $\nu>1$ and $\alpha>2\nu$ in \eqref{eq:boundlocaluniqnointro} should also not be true anymore when $\alpha<2\nu$. In other words, the lower bounds in \eqref{eq:boundlocaluniqnointro} and \eqref{eq:boundcapacitynu>1nointro} are not sharp when $\alpha<2\nu$. One can still use our method to deduce some non-optimal upper bounds on the quantities in \eqref{eq:boundcapacitynu>1nointro} and \eqref{eq:boundlocaluniqnointro} when $\alpha<2\nu$, but it turns out that these bounds would not be better than the ones from \cite[Remark~5.4]{DrePreRod5}. It is an interesting open question to prove matching upper and lower bounds similar to  \eqref{eq:boundlocaluniqnointro} and \eqref{eq:boundcapacitynu>1nointro} when $\alpha<2\nu$ but with different scaling in the exponential, as well as to remove the logarithm discrepancy  when $\alpha=2\nu$.
\end{enumerate}
\end{Rk}

Let us now quickly explain how to deduce Theorem~\ref{the:localuniquenessintro} from Theorem~\ref{the:localuniquenessnointro}.
\begin{proof}[Proof of Theorem~\ref{the:localuniquenessintro}]
    One simply combines Theorem~\ref{the:localuniquenessnointro} for $u_0=1$ and $\xi=2$ (note that $\xi_0=1$ on $\Z^d$) with \eqref{eq:assumptionZd} and \eqref{eq:defGnu}, noting that the supremum in \eqref{eq:boundlocaluniqnointro} can be removed in view of Remark~\ref{rk:endlocuniq},\ref{rk:suplowerbound}).
\end{proof}

We conclude with a remark on the necessity of the condition \eqref{eq:intro_shortgeodesic}, which did not appear in \cite{DrePreRod2,DrePreRod5}.

\begin{Rk}
\label{rk:whyintroshortgeodesic} The condition \eqref{eq:intro_shortgeodesic} is only used in the proof of the lower bound in \eqref{eq:limitVunu>1nointro} and in Lemma~\ref{lem:exitviatubes}, and hence it is also indirectly used for the proof of the lower bounds in  \eqref{eq:bounddtvnointro}, \eqref{eq:boundlocaluniqnointro}, \eqref{eq:boundcapacitynu>1nointro} and \eqref{eq:boundcapacitynu=1nointro}, as well as the same bounds on $\Z^d$ in Section~\ref{sec:intro}. The corresponding upper bounds, as well as Theorem~\ref{the:FPP}, hold without this assumption. We refer to \cite[Lemma~8.1,3)]{DrePreRod5} to see why such an additional condition is in fact necessary to obtain the previous lower bounds when $\nu>1$. When $\nu<1$, one can however still prove the lower bound in \eqref{eq:bounddtvnointro} even without the condition \eqref{eq:intro_shortgeodesic}, and in fact even without \eqref{eq:intro_sizeball}, as long as one does not care about the exact value of the constant $\Cr{CFPP2intro}$ in \eqref{eq:bounddtvnointro}, similarly as in \cite[(1.25)]{DrePreRod5}. Indeed, one can replace in \eqref{eq:prooflowerFPPintro} the use of the set $\L_{x}^{P,N}$ from  Lemma~\ref{lem:exitviatubes} (which is where we used \eqref{eq:intro_shortgeodesic}) by the set $B(x,N)$, and then use that the capacity of $B(x,N)$ is smaller than $\Cr{Ccapball}N^{\nu}$ (i.e.\ it is the same as the capacity of $\L_{x}^{P,N}$ up to constants when $\nu<1$), see \eqref{eq:capball} (whose proof does not require \eqref{eq:intro_sizeball}). Note that condition \eqref{eq:intro_sizeball} was also never used in the proof of the upper bound in Theorem~\ref{the:localuniquenessnointro} when $\nu<1$. 
\end{Rk}

\bibliography{bibliographie}

\begin{thebibliography}{10}

\bibitem{adhikariokada}
A.~Adhikari and I.~Okada.
\newblock Moderate deviations for the capacity of the random walk range in
  dimension four.
\newblock {\em Preprint, to appear in Ann.\ Probab.}, 2024.

\bibitem{AlvesPopov}
C.~Alves and S.~Popov.
\newblock Conditional decoupling of random interlacements.
\newblock {\em ALEA Lat. Am. J. Probab. Math. Stat.}, 15(2):1027--1063, 2018.

\bibitem{andres2016harnack}
S.~Andres, J.-D. Deuschel, and M.~Slowik.
\newblock Harnack inequalities on weighted graphs and some applications to the
  random conductance model.
\newblock {\em Probab. Theory Related Fields}, 164(3-4):931--977, 2016.

\bibitem{AndPre}
S.~Andres and A.~Pr{\'e}vost.
\newblock {First passage percolation with long-range correlations and
  applications to random Schrödinger operators}.
\newblock {\em Ann. Appl. Probab.}, 34(2):1846 -- 1895, 2024.

\bibitem{asselah2020deviations}
A.~Asselah and B.~Schapira.
\newblock Deviations for the capacity of the range of a random walk.
\newblock {\em Electron. J. Probab.}, 25:Paper No. 154, 28, 2020.

\bibitem{AsselahAmine2018Cotr}
A.~Asselah, B.~Schapira, and P.~Sousi.
\newblock Capacity of the range of random walk on {$\Bbb Z^d$}.
\newblock {\em Trans. Amer. Math. Soc.}, 370(11):7627--7645, 2018.

\bibitem{AsselahAmine2019COTR}
A.~Asselah, B.~Schapira, and P.~Sousi.
\newblock Capacity of the range of random walk on {$\Bbb{Z}^4$}.
\newblock {\em Ann. Probab.}, 47(3):1447--1497, 2019.

\bibitem{ADH17}
A.~Auffinger, M.~Damron, and J.~Hanson.
\newblock {\em 50 years of first-passage percolation}, volume~68 of {\em
  University Lecture Series}.
\newblock American Mathematical Society, Providence, RI, 2017.

\bibitem{MR2076770}
M.~T. Barlow.
\newblock Which values of the volume growth and escape time exponent are
  possible for a graph?
\newblock {\em Rev. Mat. Iberoamericana}, 20(1):1--31, 2004.

\bibitem{MR3185193}
S.~Boucheron, G.~Lugosi, and P.~Massart.
\newblock {\em Concentration inequalities}.
\newblock Oxford University Press, Oxford, 2013.
\newblock A nonasymptotic theory of independence, With a foreword by Michel
  Ledoux.

\bibitem{ChangYinshan2017Toot}
Y.~Chang.
\newblock Two observations on the capacity of the range of simple random walks
  on {$\Bbb Z^3$} and {$\Bbb Z^4$}.
\newblock {\em Electron. Commun. Probab.}, 22:Paper No. 25, 9, 2017.

\bibitem{chernoff1952measure}
H.~Chernoff.
\newblock A measure of asymptotic efficiency for tests of a hypothesis based on
  the sum of observations.
\newblock {\em Ann. Math. Statistics}, 23:493--507, 1952.

\bibitem{MR3126579}
F.~Comets, C.~Gallesco, S.~Popov, and M.~Vachkovskaia.
\newblock On large deviations for the cover time of two-dimensional torus.
\newblock {\em Electron. J. Probab.}, 18:no. 96, 18, 2013.

\bibitem{dembo2022capacity}
A.~Dembo and I.~Okada.
\newblock Capacity of the range of random walk: the law of the iterated
  logarithm.
\newblock {\em Ann. Probab.}, 52(5):1954--1991, 2024.

\bibitem{DrePreRod}
A.~Drewitz, A.~Pr\'{e}vost, and P.-F. Rodriguez.
\newblock The sign clusters of the massless {G}aussian free field percolate on
  {$\Bbb{Z}^d, d \geqslant 3$} (and more).
\newblock {\em Comm. Math. Phys.}, 362(2):513--546, 2018.

\bibitem{DrePreRod8}
A.~Drewitz, A.~Pr\'{e}vost, and P.-F. Rodriguez.
\newblock Arm exponent for the {G}aussian free field on metric graphs in
  intermediate dimensions.
\newblock {\em Preprint, available at
  \href{https://arxiv.org/abs/2312.10030}{arXiv:2312.10030}}, 2023.

\bibitem{DrePreRod5}
A.~Drewitz, A.~Pr\'{e}vost, and P.-F. Rodriguez.
\newblock Critical exponents for a percolation model on transient graphs.
\newblock {\em Invent. Math.}, 232(1):229--299, 2023.

\bibitem{DrePreRod2}
A.~Drewitz, A.~Pr\'{e}vost, and P.-F. Rodriguez.
\newblock Geometry of {G}aussian free field sign clusters and random
  interlacements.
\newblock {\em Probab. Theory Related Fields}, 2024.

\bibitem{MR3269990}
A.~Drewitz, B.~R\'{a}th, and A.~Sapozhnikov.
\newblock Local percolative properties of the vacant set of random
  interlacements with small intensity.
\newblock {\em Ann. Inst. Henri Poincar\'{e} Probab. Stat.}, 50(4):1165--1197,
  2014.

\bibitem{MR3390739}
A.~Drewitz, B.~R\'ath, and A.~Sapozhnikov.
\newblock On chemical distances and shape theorems in percolation models with
  long-range correlations.
\newblock {\em J. Math. Phys.}, 55(8):083307, 30, 2014.

\bibitem{MR4568695}
H.~Duminil-Copin, S.~Goswami, P.-F. Rodriguez, and F.~Severo.
\newblock Equality of critical parameters for percolation of {G}aussian free
  field level sets.
\newblock {\em Duke Math. J.}, 172(5):839--913, 2023.

\bibitem{sharpnessRI1}
H.~Duminil-Copin, S.~Goswami, P.-F. Rodriguez, F.~Severo, and A.~Teixeira.
\newblock Finite range interlacements and couplings.
\newblock {\em Preprint, to appear in Ann. Probab.}, 2023.

\bibitem{sharpnessRI3}
H.~Duminil-Copin, S.~Goswami, P.-F. Rodriguez, F.~Severo, and A.~Teixeira.
\newblock Phase transition for the vacant set of random walk and random
  interlacements.
\newblock {\em Preprint, available at
  \href{https://arxiv.org/abs/2308.07919}{arXiv:2308.07919}}, 2023.

\bibitem{sharpnessRI2}
H.~Duminil-Copin, S.~Goswami, P.-F. Rodriguez, F.~Severo, and A.~Teixeira.
\newblock A characterization of strong percolation via disconnection.
\newblock {\em Proceedings of the London Mathematical Society}, 129(2):e12622,
  2024.

\bibitem{GRS21}
S.~Goswami, P.-F. Rodriguez, and F.~Severo.
\newblock On the radius of {G}aussian free field excursion clusters.
\newblock {\em Ann. Probab.}, 50(5):1675--1724, 2022.

\bibitem{GosRod}
S.~Goswami, P.-F. Rodriguez, and Y.~Shulzhenko.
\newblock {\em In preparation}, 2024.

\bibitem{MR1853353}
A.~Grigor'yan and A.~Telcs.
\newblock Sub-{G}aussian estimates of heat kernels on infinite graphs.
\newblock {\em Duke Math. J.}, 109(3):451--510, 2001.

\bibitem{hernandez2023chemical}
S.~Hern\'{a}ndez-Torres, E.~B. Procaccia, and R.~Rosenthal.
\newblock The chemical distance in random interlacements in the low-intensity
  regime.
\newblock {\em Comm. Math. Phys.}, 400(3):1697--1737, 2023.

\bibitem{landkof1972foundations}
N.~S. Landkof.
\newblock {\em Foundations of modern potential theory}, volume 180.
\newblock Springer, 1972.

\bibitem{MR1117680}
G.~F. Lawler.
\newblock {\em Intersections of random walks}.
\newblock Probability and its Applications. Birkh\"{a}user Boston, Inc.,
  Boston, MA, 1991.

\bibitem{li2014lower}
X.~Li and A.-S. Sznitman.
\newblock A lower bound for disconnection by random interlacements.
\newblock {\em Electron. J. Probab.}, 19:no. 17, 26, 2014.

\bibitem{li2022large}
X.~Li and Z.~Zhuang.
\newblock {On large deviations and intersection of random interlacements}.
\newblock {\em Bernoulli}, 30(3):2102 -- 2126, 2024.

\bibitem{MuiSev}
S.~Muirhead and F.~Severo.
\newblock Percolation of strongly correlated {G}aussian fields, {I}: {D}ecay of
  subcritical connection probabilities.
\newblock {\em Probab. Math. Phys.}, 5(2):357--412, 2024.

\bibitem{nitzschner2020solidification}
M.~Nitzschner and A.-S. Sznitman.
\newblock Solidification of porous interfaces and disconnection.
\newblock {\em J. Eur. Math. Soc. (JEMS)}, 22(8):2629--2672, 2020.

\bibitem{MR4499015}
C.~Panagiotis and F.~Severo.
\newblock Analyticity of {G}aussian free field percolation observables.
\newblock {\em Comm. Math. Phys.}, 396(1):187--223, 2022.

\bibitem{MR3420516}
S.~Popov and A.~Teixeira.
\newblock Soft local times and decoupling of random interlacements.
\newblock {\em J. Eur. Math. Soc. (JEMS)}, 17(10):2545--2593, 2015.

\bibitem{PreRodSou}
A.~Prévost, P.-F. Rodriguez, and P.~Sousi.
\newblock Phase transition for the late points of random walk.
\newblock {\em Preprint, available at
  \href{https://arxiv.org/abs/2309.03192}{arXiv:2309.03192}}, 2023.

\bibitem{MR2819660}
B.~R\'ath and A.~Sapozhnikov.
\newblock On the transience of random interlacements.
\newblock {\em Electron. Commun. Probab.}, 16:379--391, 2011.

\bibitem{MR2889752}
B.~R\'{a}th and A.~Sapozhnikov.
\newblock Connectivity properties of random interlacement and intersection of
  random walks.
\newblock {\em ALEA Lat. Am. J. Probab. Math. Stat.}, 9:67--83, 2012.

\bibitem{MR3024098}
B.~R\'ath and A.~Sapozhnikov.
\newblock The effect of small quenched noise on connectivity properties of
  random interlacements.
\newblock {\em Electron. J. Probab.}, 18:no. 4, 20, 2013.

\bibitem{saloff2002aspects}
L.~Saloff-Coste.
\newblock {\em Aspects of Sobolev-type inequalities}, volume 289.
\newblock Cambridge University Press, 2002.

\bibitem{MR3650417}
A.~Sapozhnikov.
\newblock Random walks on infinite percolation clusters in models with
  long-range correlations.
\newblock {\em Ann. Probab.}, 45(3):1842--1898, 2017.

\bibitem{MR2520124}
A.-S. Sznitman.
\newblock Random walks on discrete cylinders and random interlacements.
\newblock {\em Probab. Theory Related Fields}, 145(1-2):143--174, 2009.

\bibitem{MR2561432}
A.-S. Sznitman.
\newblock Upper bound on the disconnection time of discrete cylinders and
  random interlacements.
\newblock {\em Ann. Probab.}, 37(5):1715--1746, 2009.

\bibitem{MR2680403}
A.-S. Sznitman.
\newblock Vacant set of random interlacements and percolation.
\newblock {\em Ann. Math. (2)}, 171(3):2039--2087, 2010.

\bibitem{MR2778797}
A.-S. Sznitman.
\newblock On the critical parameter of interlacement percolation in high
  dimension.
\newblock {\em Ann. Probab.}, 39(1):70--103, 2011.

\bibitem{MR2891880}
A.-S. Sznitman.
\newblock Decoupling inequalities and interlacement percolation on
  {$G\times\mathbb{Z}$}.
\newblock {\em Invent. Math.}, 187(3):645--706, 2012.

\bibitem{MR2892408}
A.-S. Sznitman.
\newblock An isomorphism theorem for random interlacements.
\newblock {\em Electron. Commun. Probab.}, 17:no. 9, 9, 2012.

\bibitem{MR3050507}
A.-S. Sznitman.
\newblock Random interlacements and the {G}aussian free field.
\newblock {\em Ann. Probab.}, 40(6):2400--2438, 2012.

\bibitem{MR2932978}
A.-S. Sznitman.
\newblock {\em Topics in Occupation Times and {G}aussian free fields}.
\newblock Zurich Lectures in Advanced Mathematics. European Mathematical
  Society (EMS), Z\"urich, 2012.

\bibitem{MR3417515}
A.-S. Sznitman.
\newblock Disconnection and level-set percolation for the {G}aussian free
  field.
\newblock {\em J. Math. Soc. Japan}, 67(4):1801--1843, 2015.

\bibitem{MR3602841}
A.-S. Sznitman.
\newblock Disconnection, random walks, and random interlacements.
\newblock {\em Probab. Theory Related Fields}, 167(1-2):1--44, 2017.

\bibitem{MR4607724}
A.-S. Sznitman.
\newblock On the cost of the bubble set for random interlacements.
\newblock {\em Invent. Math.}, 233(2):903--950, 2023.

\bibitem{MR2525105}
A.~Teixeira.
\newblock Interlacement percolation on transient weighted graphs.
\newblock {\em Electron. J. Probab.}, 14:no. 54, 1604--1628, 2009.

\bibitem{MR3076674}
A.~Teixeira and J.~Tykesson.
\newblock Random interlacements and amenability.
\newblock {\em Ann. Appl. Probab.}, 23(3):923--956, 2013.

\bibitem{MR2838338}
A.~Teixeira and D.~Windisch.
\newblock On the fragmentation of a torus by random walk.
\newblock {\em Comm. Pure Appl. Math.}, 64(12):1599--1646, 2011.

\bibitem{MR3563197}
J.~\v{C}ern\'{y} and A.~Teixeira.
\newblock Random walks on torus and random interlacements: macroscopic coupling
  and phase transition.
\newblock {\em Ann. Appl. Probab.}, 26(5):2883--2914, 2016.

\bibitem{MR864581}
M.~E. Watkins.
\newblock Infinite paths that contain only shortest paths.
\newblock {\em J. Combin. Theory Ser. B}, 41(3):341--355, 1986.

\end{thebibliography}
\bibliographystyle{abbrv}

\end{document}